\documentclass[12pt]{article}
\pdfoutput=1

\usepackage{Maxalt}
\usepackage{mathrsfs}
\definecolor{coolcolor}{rgb}{0.6,0,1}
\newcommand\omicron{o}

\usepackage{amscd}
\usepackage{tikz-cd}
\newsavebox{\pullback}
\sbox\pullback{%
\begin{tikzpicture}%
\draw (0,0) -- (1ex,0ex);%
\draw (1ex,0ex) -- (1ex,1ex);%
\end{tikzpicture}}

\hypersetup{hypertexnames=false}

\title{Orbifold resolution via hyperkahler quotients:\\
\Large{the $D_2$ ALF manifold}}
\author[1]{Arnav Tripathy\thanks{tripathy@stanford.edu}}
\author[2]{Max Zimet\thanks{mzimet@fas.harvard.edu}}
\affil[1]{Department of Mathematics,

Stanford University, Stanford, CA 94305 USA

~}
\affil[2]{Black Hole Initiative and Department of Physics,

Harvard University, Cambridge, MA 02138 USA}

\date{}

\newcommand{\kq}{\mathord{/ \!\! /_{\xi}}}
\newcommand{\kqnoxi}{\mathord{/ \!\! /}}
\newcommand{\hqnoxi}{\mathord{/ \!\! / \! \! /}}
\newcommand{\hq}{\mathord{/ \!\! / \! \! /_{\xi}}}

\DeclareMathOperator{\slope}{slope}

\DeclareMathOperator{\im}{im}
\DeclareMathOperator{\coker}{coker}
\DeclareMathOperator{\Ad}{Ad}
\DeclareMathOperator{\pt}{pt}

\usepackage[OT2,T1]{fontenc}
\DeclareSymbolFont{cyrletters}{OT2}{wncyr}{m}{n}
\DeclareMathSymbol{\Sha}{\mathalpha}{cyrletters}{"58}

\definecolor{greencolor}{rgb}{0,0.6,0.5}

 \theoremstyle{plain}
 \newinttheorem[theTheorem]{thm}{Theorem}
\newinttheorem{cor}{Corollary}
\newinttheorem{lem}{Lemma}
\newinttheorem{prop}{Proposition}
\newinttheorem{conj}{Conjecture}
\newinttheorem{fact}{Fact}
\newinttheorem{thrm}{Theorem}
 \theoremstyle{definition}
 \newinttheorem{defn}{Definition}
 \newinttheorem{notation}{Notation}
 \newinttheorem{ex}{Example}
 \newinttheorem{constr}{Construction} 
 \newtheorem*{thm*}{Theorem}
 \theoremstyle{remark}
 \newinttheorem{rmk}{Remark}
 \newinttheorem{rem}{Remark}
 \newinttheorem{ass}{Assumption}
  \newinttheorem{conv}{Convention}
 \newinttheorem{axiom}{Axiom}

\def\beq{\begin{eqnarray}}
\def\eeq{\end{eqnarray}}
 \newcommand{\bp}{\begin{proof}[Proof]}
 \newcommand{\ep}{\end{proof}}
 
\newcommand\nc{\newcommand}
\nc\mc\mathcal
\nc\mb\mathbb

\usepackage{amssymb} 
\def\acts{\ \rotatebox[origin=c]{-90}{$\circlearrowright$}\ }

\usepackage{stmaryrd}
\newcommand{\Wedge}{\mathchoice{{\textstyle\bigwedge}}%
    {{\bigwedge}}%
    {{\textstyle\wedge}}%
    {{\scriptstyle\wedge}}}

\begin{document}

\maketitle

%Singular equivariant instantons
%math.DG math.AG hep-th math.AP?

\begin{abstract}
    We propose an infinite-dimensional generalization of Kronheimer's construction of families of hyperkahler manifolds resolving flat orbifold quotients of $\mb{R}^4$. As in~\cite{kronheimer:construct}, these manifolds are constructed as  hyperkahler quotients of affine spaces. This leads to a study of \emph{singular equivariant instantons} in various dimensions. In this paper, we study singular equivariant Nahm data to produce the family of $D_2$ asymptotically locally flat (ALF) manifolds as a deformation of the flat orbifold $(\mb{R}^3 \times S^1)/Z_2$. We furthermore introduce a notion of stability for Nahm data and prove a Donaldson-Uhlenbeck-Yau type theorem to relate real and complex formulations. We use these results to construct a canonical Ehresmann connection on the family of non-singular $D_2$ ALF manifolds. In the complex formulation, we exhibit explicit relationships between these $D_2$ ALF manifolds and corresponding $A_1$ ALE manifolds. We conjecture analogous constructions and results for general orbifold quotients of $\mb{R}^{4-r} \times T^r$ with $2 \le r \le 4$. The case $r = 4$ produces K3 manifolds as hyperkahler quotients.
\end{abstract}

\newpage
\tableofcontents
\hypersetup{linkcolor=blue}

\newpage
\section{Introduction}

%Should we condense the \newpages above?

Hyperkahler manifolds, or Riemannian manifolds $(M, g)$ of real dimension $4n$ with holonomy contained within $Sp(n)$, are extremely special. For example, this holonomy condition forces the metric $g$ to have vanishing Ricci curvature. And yet, there remains much to be understood even for this nice class of manifolds. For example, to date, there is no known exact analytic expression for any Ricci-flat metric on a (non-toroidal) compact manifold -- even on the simplest nontrivial example of a compact hyperkahler manifold, namely the K3 manifold in real dimension four. It is hence of interest to find further constructions of hyperkahler manifolds that may be of use for this and other questions.

Kronheimer, in~\cite{kronheimer:construct}, gives a beautiful construction of a family of asymptotically locally Euclidean (ALE) hyperkahler manifolds deforming the singular quotient $\mb{R}^4/\Gamma$ for any finite group $\Gamma \subset Sp(1) \simeq SU(2) \subset SO(4)$ preserving the flat hyperkahler structure of $\mb{R}^4$. He does so by writing down a hyperkahler quotient presentation of these deformations to realize each of these (generically) smooth manifolds as moduli spaces of linear algebraic data -- explicit matrices, in fact -- solving algebraic equations.

Our central proposal in this paper is to suggest a generalization when $\Gamma$ is replaced by an \emph{infinite} discrete group $\hat{\Gamma}$ acting on $\mb{R}^4$, following the physics literature~\cite{wati:dBraneT,ho:noncomm,waldram:lol,greene:lol,mz:K3HK,mz:K3HK3}.\footnote{These papers study D-brane probes of orbifolds, following Douglas and Moore~\cite{moore:ALEinst}, who were in turn translating Kronheimer's construction into string theory. So, the mathematical and physical sources of motivation are one and the same.} The corresponding orbifold quotient is of the form $(\mb{R}^{4-r} \times T^r) / \Gamma$ for some finite $\Gamma$; the case $r = 0$ recovers Kronheimer's construction while the cases $1 \le r \le 4$ are new. Denote by $\{p_i\}$ the finite set of points in $(\mb{R}^{4-r} \times T^r) / \Gamma$ which have nontrivial stabilizer groups $\Gamma_i \hookrightarrow \Gamma$. We then expect to find a family of deformations of $(\mb{R}^{4-r} \times T^r) / \Gamma$ whose local geometry about $p_i$ approximates Kronheimer's ALE family of deformations of $\mb{R}^4/\Gamma_i$ near the flat orbifold limit.\footnote{Such families have been explored from a differential geometric point of view in \cite{page:orbifold,gibbons:k3,topiwala:orb,lebrun:orb,minerbe:Dk}.} Explaining this general proposal takes up the first part of the paper, culminating in our main Conjecture~\ref{conj:mainconj}. This conjecture makes manifest many other pleasant features of the $r = 0$ case such as surjectivity of the period map and explicit diffeomorphisms between the manifolds for generic parameter values. 

In general, the linear-algebraic data of the $r = 0$ case is replaced by the gauge-theoretic data of a connection and $4 - r$ (matrix-valued) functions on the dual torus $\hat{T^r}$.\footnote{Such upgrades have proved useful in a number of related contexts, beginning with Nahm's upgrade of the ADHM \cite{adhm} construction of instanton moduli spaces to a construction of monopole moduli spaces \cite{nahm:monopoles1}.} These data are moreover required (i) to be equivariant for the dual action of $\Gamma$ on $\hat{T^r}$, (ii) to have prescribed singular behavior about each of the nontrivially-stabilized points for this dual action, and (iii) to satisfy an anti-self-duality (ASD) equation away from the nontrivially-stabilized points. We are hence studying moduli spaces of \emph{singular equivariant instantons}. As usual, we would also like theorems of Donaldson-Uhlenbeck-Yau (DUY)~\cite{donaldson:duy,yau:duy} type that characterize these moduli spaces, at least with their complex (and holomorphic symplectic) structure, in terms of holomorphic data satisfying a stability condition.

We briefly mention two desired applications of the above construction. First, a means of deforming solutions of the ASD equations in the orbifold limit to corresponding singular solutions, for small values of the deformation parameters, would yield explicit analytic formulae for the Ricci-flat metric on the hyperkahler quotient.\footnote{Near the resolved orbifold points the starting point should instead be related to solutions of the relevant equations of~\cite{kronheimer:construct}.} In the $r = 4$ case, one would thereby find explicit Ricci-flat metrics on K3 manifolds. One may hope to construct these deformed solutions via a modified Banach contraction procedure. Second, in light of conjectures of Strominger-Yau-Zaslow mirror symmetry~\cite{strominger:mirrorT} and Gaiotto-Moore-Neitzke~\cite{GMN:walls}, one may hope that for $r \ge 2$, the hyperkahler structure encodes open Gromov-Witten invariants.\footnote{In the ALF case, the metrics produced by the present construction also encode the effects of instantons -- namely, those of the 3d $\N=4$ $SU(2)$ $N_f=2$ gauge theory, which were studied in \cite{tong:AHmatter}.} These ideas were made precise in \cite{mz:K3HK,mz:K3HK3}, which carried out the first iteration of the Banach contraction procedure and resummed the resulting formulae to those of \cite{GMN:walls,mz:k3}. This resummation procedure should yield all of the relevant open Gromov-Witten invariants. Given these expected applications, we write the present paper for a wider audience of geometric analysts and possibly even enumerative geometers; we hence include sufficient detail on our gauge-theoretic constructions so as to be self-contained. Many of the technical results we establish should be familiar to gauge theory experts. However, we note that we pay greater attention both to the structure of the non-generic singular moduli spaces near their singularities and to the relationships between moduli spaces with different parameter values than is typical in the literature.

Indeed, while this paper is aimed at a mathematical audience, we hope that physically-inclined readers may also appreciate the presentation here of the rigorous construction of gauge-theoretic moduli spaces and the study of their properties with only mild discomfort. We particularly direct such readers' attention to (i) Propositions~\ref{prop:realRG} and~\ref{prop:cplxRG} and the map $RG_\CC$ of equation \eqref{eq:rgC} as elegant manifestations of the renormalization group, (ii) the formal gauge transformations of Definition~\ref{defn:for} that feature in the main results stated in~\S\ref{sec:nahm}, (iii) the Donaldson-type functional~\cite{donaldson:duy,donaldson:duy2} which features prominently in the proof of the DUY-type theorem below, and (iv) the Ehresmann connection on the family of non-singular $D_2$ asymptotically locally flat (ALF) manifolds constructed in \S\ref{sec:duy}.

The bulk of this paper is devoted to rigorously studying the case $r = 1$ and $\Gamma = Z_2$ in order to produce the family of $D_2$ ALF hyperkahler manifolds.\footnote{We refer readers to \cite{chen:ALF} for the classification of ALF manifolds and a discussion of their properties and to \cite{cherkis:dkInst,ioanas:legendre}, which gave fairly explicit descriptions of the geometry of $D_k$ ALF manifolds using twistor techniques.} There have been a number of papers in the physics and math literature studying ALF manifolds from the point of view of Nahm's equations \cite{kronheimer:cotangent,dancer:alf,dancer:alf2,cherkis:dkGravInst,w:alfQuotient,cherkis:alfInst3,takayama:bow}, but there are a number of differences in our discussion. As we discuss in \S\ref{sec:prelim}, the strategy for constructing $D_k$ ALF manifolds in these references involves a series of hyperkahler quotients, whereas it is both physically important and mathematically useful that the ALF manifolds may be constructed as single hyperkahler quotients of infinite-dimensional affine spaces; we explain in this section how the two approaches may be easily related to each other when $k=2$. Furthermore, our approach naturally leads us to introduce an interesting notion of stability for Nahm data with respect to which we prove a new DUY-type theorem. In turn, this leads us to discover interesting relationships between $D_2$ ALF manifolds. We also find an interesting relationship with the $A_1$ ALE manifold: specifically, in any complex structure, the $D_2$ ALF manifold may be covered by two open submanifolds which are isomorphic, as holomorphic symplectic manifolds, to submanifolds of $A_1$ ALE manifolds equipped with corresponding complex structures. The latter submanifolds are deformation retracts of the full $A_1$ ALE manifold, and so this provides an elegant means of computing the periods of the $D_2$ ALF manifold.

We now outline the paper. \S\ref{sec:setup} is the first main part of the paper, where we motivate and define the necessary objects to state our main Conjecture~\ref{conj:mainconj}. We begin by stating a generalization of Kronheimer's construction that applies to (deformations of) $\RR^4/\hat\Gamma$ for $\hat\Gamma$ an infinite discrete group. We then apply Fourier duality to arrive at an equivalent gauge-theoretic formulation over $\hat{T^r}$; in this framework, we calculate the relevant space of moment map parameters, discuss some symmetries of the construction, and verify that the hyperkahler quotient with vanishing moment map parameters is a flat orbifold. There is a surprising new difficulty here in the $r > 0$ case, thanks to the existence of infinitely many points with nontrivial $\hat\Gamma$ stabilizers: besides the expected infinite-dimensional functional analysis, there is a new essentially topological difficulty in showing that there are no extraneous components besides $\mb{R}^4/\hat{\Gamma}$. We offer two proofs of quite different flavor, one more analytic using heat-kernel regularized traces and the other more topological and essentially using equivariant Chern characters. We then end~\S\ref{sec:setup} by stating the main conjecture and offering a few appendical subsections on related algebro- (or \mbox{complex-)}geometric and algebro-topological aspects of the construction.

The remainder of the paper then specializes to $r = 1$ and $\Gamma=Z_2$. \S\ref{sec:nahm} is a brief section that specializes the above to this setting and gives a full statement of results we prove in the following sections. Our main theorems are then proved in turn. \S\ref{sec:hkq} proves that in this case, our moduli space of singular equivariant Nahm data is indeed a family of hyperkahler real four-manifolds deforming the flat orbifold $(\mb{R}^3 \times S^1)/Z_2$. \S\ref{sec:complex} is devoted to a holomorphic analog which is shown by a DUY-type theorem proved in \S\ref{sec:duy} to be equivalent to the construction of \S\ref{sec:hkq}. We also note in \S\ref{sec:duy} that the DUY-type theorem yields a canonical Ehresmann connection on the family of non-singular $D_2$ ALF manifolds. This is an interesting strengthening of the familiar fact that these manifolds are diffeomorphic. The construction of this Ehresmann connection also works in Kronheimer's original ALE setting.

\vspace{.2cm}

\noindent\textbf{Statement of Results.}
Our main results are as follows:
\begin{itemize}
\item Theorem \ref{thm:noFI}, with independent proofs in \S\S\ref{subsec:flatorb} and \ref{subsec:flatorb2}, which gives a gauge-theoretic hyperkahler quotient construction of an arbitrary flat hyperkahler orbifold of the form $\RR^4/\hat\Gamma$, where $\hat\Gamma$ is a discrete group of affine linear automorphisms of $\RR^4$ which preserves its hyperkahler structure.
\item Our explicit description in Proposition \ref{prop:compNahm} of the `complex Nahm data' that populates the $D_2$ ALF moduli spaces and our exploitation of this description in Proposition \ref{prop:RGitems} and Remark \ref{rmk:periods} in order to provide interesting relationships between the $D_2$ ALF manifold and the $A_1$ ALE manifold.
\item Our gauge-theoretic constructions of the $D_2$ ALF manifold from both a hyperkahler quotient point of view (Theorem \ref{thm:modHK}) and a holomorphic symplectic quotient point of view (Theorem \ref{thm:modCpx}).
\item Our DUY-type theorem, Theorem \ref{thm:duy}, which relates these constructions.
\item Our description in Theorem \ref{thm:relXi} of gauge-theoretic birational equivalences between $D_2$ ALF manifolds.
\item Our combination of the previous two results in Theorems \ref{thm:res} and \ref{thm:conn} to show that the non-singular $D_2$ ALF manifolds produced by our constructions are diffeomorphic, and that furthermore there is a natural Ehresmann connection on the family thereof.
\end{itemize}

\vspace{.2cm}

\noindent\textbf{Conventions.} We will use the following conventions and notation throughout the paper:
\begin{itemize}
\item We denote by $Z_n$ the finite cyclic group of order $n$, i.e., the abelian group variously denoted elsewhere by $\mb{Z}/n, C_n$, etc. 
\item Given a set $X$ with action by group $G$, we denote by $X^G$ the $G$-invariant subset. 
\item Given a finite rank lattice $\Lambda$, we define its dual lattice to be $\Hom(\Lambda,2\pi\ZZ)$ and denote it by $\Lambda^\vee$.
\item We denote by $H^s = W^{s,2}$ the space of Sobolev $s$-regular functions. One may take $s \in \mb{R}$; there is no great loss in this paper with taking $s \in \mb{Z}$. 
\item In the above, if $s$ is sufficiently large that $H^s \hookrightarrow C^0$ by Sobolev embedding, we will throughout assume that we are working with the continuous representative within the a.e.-defined equivalence class of functions. 
\item Given a manifold $M$ and vector space $V$, we denote by $H^s(M, V)$ the space of Sobolev $s$-regular $V$-valued functions on $M$; slightly more generally, we use this notation to denote Sobolev $s$-regular sections if $V$ is a vector bundle (or sheaf of vector spaces) over $M$. For more general target spaces $N$, the notation $H^s(M, N)$ is more ambiguous. Suppose $s \ge 0$ and $N$ is a sublocus of a vector space $V$; we will follow the usual convention of taking the subset of $H^s(M, V)$ consisting of functions (a.e.) valued in $N$. Note, however, that we will typically assume that (i) $M$ is compact and (ii) $s > \dim_{\mb{R}}M / 2$, so that not only do we have an a.e. defined map, we may assume its continuity. Suppose now that $N$ is a submanifold of $V$; it is then straightforward to see one may equally well define this space $H^s(M, N)$ by demanding that our continuous map $M \to N$ in fact be in $H^s_{\mathrm{loc}}$ locally using open charts for $N$ and their preimages in $M$. Note that within this good range of $\dim_{\mb{R}}M/2 < s \le \infty$, where we again stress that $M$ is compact, the mapping space $H^s(M, N)$ is a Hilbert (or Fr\'echet) manifold~\cite{Eells}. %for any $s < \infty$ ($s = \infty$). \MZ{not sure this is true for small $s$ -- eg for $U^0$, our discussion suggests that tangent space is modelled on $L^0$, not $L^2$. the `with this definition' thing sort of makes this ok, but the exposition clearly says everything is equivalent so i think its best to restrict to the sensible range}

% https://mathoverflow.net/questions/102018/manifold-valued-sobolev-spaces and https://webspace.science.uu.nl/~meier007/Hilbertmflt3.pdf are references. https://mathoverflow.net/questions/108808/differential-of-a-sobolev-map-between-manifolds I don't really like either.

\item We will work throughout with infinite-dimensional manifolds, but all our model topological vector spaces shall be \emph{separable} Banach or Fr\'echet spaces. As such, our manifolds will be manifestly second-countable. Note that second-countability is inherited under taking subsets or open quotients, and so all our topological spaces will be second-countable. We may hence analyze point-set topological properties by means of convergence of sequences (as opposed to more general nets or filters). 

\item We often work with spaces of the form $H^s(M, N)$ for sufficiently large $s \le \infty$; these spaces are, of course, Hilbert for $s < \infty$ and Fr\'echet for $s = \infty$. It becomes repetitive to keep repeating that our spaces are Hilbert and/or Fr\'echet, and so we will simply refer to Hilbert spaces, Hilbert manifolds, Hilbert Lie groups, and so on, meaning always that they are instead Fr\'echet manifolds, Fr\'echet Lie groups, etc., if $s = \infty$.

\item We will use $z \mapsto \overline{z}$ to denote complex conjugation and $A \mapsto A^{\dagger}$ to denote the adjoint (i.e., conjugate transpose) of a matrix or linear operator. However, we will use $(-)^*$ to denote formal $L^2$ adjoints of linear differential operators. We follow standard mathematical conventions for Lie algebras; e.g., elements of $\mf{u}(n)$ are skew-Hermitian matrices.

\item We will take gauge groups to act on the \emph{right} so that the action of a gauge transformation $g$ on a connection $\nabla$ is $\nabla \mapsto g^{-1}\nabla g$. \end{itemize}

\vspace{.2cm}

\noindent\textbf{Acknowledgements.} We thank L. Fredrickson, S. Kachru, R. Mazzeo, and R. Melrose for enjoyable and helpful conversations. We also thank A. Kupers, P. Kronheimer, H. Tanaka, and R. Thomas for comments on a draft version of this paper. The research of A.T. was supported by the National Science Foundation under NSF MSPRF grant number 
1705008. This publication is funded in part by the Gordon and Betty Moore Foundation through grant GBMF8273 to Harvard University to support the work of the Black Hole Initiative. This publication was also made possible through the support of a grant from the John Templeton Foundation. The research of M.Z. was also supported by DOE grant DE-SC0007870.

\section{The Kronheimer construction for infinite groups}\label{sec:setup}

\subsection{General setup}\label{subsec:setup}

We begin by explaining some general constructions for discrete subgroups of groups of affine linear transformations before eventually specializing to the case of interest, namely subgroups preserving the hyperkahler structure of flat $\mb{R}^4$.

So, let $Q$ be any finite-dimensional $\mb{R}$-vector space endowed with a positive-definite inner product $\langle -, -\rangle$, which induces a Riemannian metric on $Q$. Let $\mathrm{Aff}(Q)$ denote the group of affine linear automorphisms preserving this metric, i.e., the semidirect product $$\mathrm{Aff}(Q) := Q \rtimes O(Q)\,.$$ $\mathrm{Aff}(Q)$ sits in the canonically split short exact sequence
\beq \label{eq:ses} 1 \to Q \to \mathrm{Aff}(Q) \to O(Q) \to 1\,.\eeq
Suppose now that $\hat{\Gamma}$ is some discrete subgroup of $\mathrm{Aff}(Q)$, and suppose $\Lambda \subset \hat{\Gamma}$ is some finite-index normal subgroup whose image in $\mathrm{Aff}(Q)$ in fact lands within the subgroup of translations $Q \hookrightarrow \mathrm{Aff}(Q)$. Let $\Gamma$ then denote the quotient group $\hat{\Gamma}/\Lambda$ so that we have a diagram of short exact sequences \beq\label{eq:sesg} \begin{tikzcd} 1 \ar[r] & \Lambda \ar[r] \ar[d,hookrightarrow] & \hat{\Gamma} \ar[r,"\pi"] \ar[d,hookrightarrow] & \Gamma \ar[r] \ar[d] & 1 \\ 1 \ar[r] & Q \ar[r] & \mathrm{Aff}(Q) \ar[r] & O(Q) \ar[r] & 1\makebox[0pt][l]{\,.} \end{tikzcd} \eeq
Note that $\Lambda$, as a discrete subgroup of $Q$, is free abelian of some rank we shall denote $r$ for $0 \le r \le \dim_{\mb{R}}Q$. Next, as $\Lambda$ is abelian, the short exact sequence above induces an automorphism $\Gamma \acts{\Lambda}$. Compatibility with~\eqref{eq:ses} and the analogous (defining) action $O(Q)\acts{Q}$ implies that $\Lambda_{\mb{R}} := \Lambda \otimes_{\mb{Z}} \mb{R} \subset Q$ is a subrepresentation under the action of $\Gamma \to O(Q)$. As $\Gamma$ acts preserving the inner product structure of $Q$, we hence have that $Q$ splits as a direct sum
\beq Q \simeq \Lambda_{\mb{R}} \oplus \Lambda_{\mb{R}}^{\perp}\eeq
of $\Gamma$-representations. %finite subgroup of SU(2) and falls into the McKay classification \ftodo{[cite something]}.

We will say $\hat{\Gamma} \hookrightarrow \mathrm{Aff}(Q)$ is \emph{linear} if there exists a (group-theoretic) splitting of $\hat{\Gamma} \stackrel{\pi}{\twoheadrightarrow} \Gamma$ compatible with the canonical splitting of $\mathrm{Aff}(Q) \twoheadrightarrow O(Q)$. In other words, $\hat{\Gamma}\acts{Q}$ is linear if we have a splitting $\hat{\Gamma} \simeq \Gamma \ltimes \Lambda$, with $\Gamma$ acting by linear automorphisms and $\Lambda$ acting by translations. If $\hat{\Gamma}\acts{Q}$ is not linear, we term its action \emph{affine}. In either case, when we think of the space $Q$ above as an (affine) representation of $\hat{\Gamma}$, we will refer to it as the \emph{defining} representation.

\begin{rmk}\label{rmk:crystal}
The examples that will ultimately interest us will admit an alternative definition of $\Gamma$ as the image of $\hat{\Gamma}$ within the $O(Q)$ quotient group of $\mathrm{Aff}(Q)$ (so that automatically, $\Gamma$ as a discrete subgroup of the compact group $O(Q)$ is finite and the kernel $\Lambda$ of $\hat{\Gamma}\stackrel{\pi}{\to} \Gamma$ is a subgroup of the translation group $Q$). But, the above definition allows for slightly more general examples as well, such as $\hat{\Gamma} \simeq \mb{Z}$ acting on $Q \simeq \mb{R}$ by translation and $\Lambda$ the index $2$ subgroup of $\hat{\Gamma}$; note that this example would be termed \emph{affine} in the terminology above. 

Let us comment briefly on why one may be interested in enlarging one's family of examples to include ones such as the above where $\Gamma \to O(Q)$ is not injective; for simplicity of nomenclature in this discussion, let us restrict to the case of $Q \simeq \Lambda_{\mb{R}}$. The end goal of this section is to reconstruct $Q/\hat{\Gamma}$ as a moduli space of connections on $Q^*/\Lambda^{\vee}$ (before then attempting to produce \emph{deformations} of $Q/\hat{\Gamma}$ as moduli spaces of \emph{singular} connections on $Q^*/\Lambda^{\vee}$). In the example above, suppose $Q/\Lambda$ is a circle of radius $2R$, so that then $Q/\hat{\Gamma}$ is a circle of radius $R$. Alternatively, we could have chosen $\Lambda' = \hat{\Gamma}$ and $\Gamma' = \{1\}$, in which case we would simply reconstruct the circle of radius $R$ as the moduli of flat $U(1)$ connections on the dual circle of radius $1/2\pi R$. The surprising statement is that there is an equivalent reconstruction of our original circle in terms of a moduli of $U(2)$ connections on a dual circle of radius $1/4\pi R$ subject to a $Z_2$-equivariance condition -- where the $Z_2$ acts trivially on the underlying circle\footnote{For physically-inclined readers, note that this equivariance condition is that of an \emph{asymmetric orbifold}, as is clear by tracing the $Z_2$-action on $Q/\Lambda$ through $T$-duality. Hence, one may parse our discussion in~\S\ref{subsec:asymorb} as a determination of the worldvolume gauge theory on a class of D-branes on asymmetric orbifolds.} of radius $1/4\pi R$ (as the action of $\Gamma$ on $Q^*$ factors through $\Gamma \to O(Q)$). More generally, there are related equivalences of all manner of moduli spaces on the two circles, where those on the \emph{smaller} circle are subject to a $Z_2$-equivariance condition -- e.g., of instanton moduli spaces on $\mb{R}^3 \times Q^*/\Lambda^{\vee}$ and on $\mb{R}^3 \times Q^*/(\Lambda')^{\vee}$. Making sharp mathematical sense of these equivalences is one of the main motivations for not supposing $\Gamma \hookrightarrow O(Q)$; the other motivation is that the above example is an excellent toy model for the case of affine actions we treat in~\S\ref{subsec:asymorb}.

It may also be of interest to compare the above setup to the definition of a \emph{crystallographic space group} in crystallography; indeed, the two are essentially identical, although in crystallography one typically asks that $Q \simeq \Lambda_{\mb{R}}$. Notwithstanding this mild difference, the \emph{linear} case we demarcate above is known as a \emph{symmorphic} space group in crystallography. This entire section may thus be thought of as an exposition of crystallographic T-duality, as explored in, e.g., the recent papers~\cite{GT,GKT}, as well as older papers~\cite{wati:dBraneT,ho:noncomm,FreedMoore}. However, in contrast with some of these references, we are more concerned here with the differential geometry of instantons than the algebraic topology of charge lattices. 
\end{rmk}

It is slightly more straightforward to pass to the Fourier-dual description in the case of linear actions, and so we will restrict to that case for the purposes of the present discussion. The affine case will be treated in \S\ref{subsec:asymorb}; passing from the linear case to the affine case simply involves upgrading from instantons on equivariant bundles to instantons on twisted equivariant bundles (whose underlying non-equivariant bundles are untwisted). Almost all cases that arise in the present context -- and certainly the case that occupies the bulk of this paper -- are of this linear form.

\vspace{0.5cm}
Recall that in Kronheimer~\cite{kronheimer:construct}, one now encounters the following constructions in the case that $\hat{\Gamma} = \Gamma$ is finite: \begin{align*} R_{\Gamma} &:= \mb{C}[\Gamma]\text{ is the regular representation of $\Gamma$ over $\mb{C}$ with its standard Hermitian structure}, \\ U(R_{\Gamma}) &:= \{\text{unitary matrices acting on }R_{\Gamma}\} \\ \mf{u}(R_{\Gamma}) &:= \{\text{skew-symmetric matrices acting on }R_{\Gamma}\} \\ \tilde{\mc{G}} &:= U(R_{\Gamma})^{\Gamma} = \{g \in U(R_{\Gamma}) \bigm| g \gamma = \gamma g \ \forall \gamma \in \text{Im}(\Gamma \hookrightarrow U(R_{\Gamma}))\} \\ \mc{G} &:= \tilde{\mc{G}} / Z(\tilde{\mc{G}}) \\ \mc{A} &:= \Big( \mf{u}(R_{\Gamma}) \underset{\mb{R}}{\otimes} Q \Big)^{\Gamma} = \{v \in \mf{u}(R_{\Gamma}) \underset{\mb{R}}{\otimes} Q \bigm| \gamma v = v \ \forall \gamma \in \Gamma \acts{\mf{u}(R_{\Gamma}) \underset{\mb{R}}{\otimes} Q}\}\,.
\end{align*}
A few points require clarification. $\Gamma\hookrightarrow U(R_\Gamma)$ is the natural permutation action on the basis vectors by left-multiplication. This yields an adjoint action on $\mf{u}(R_{\Gamma})$. The ambient space $\mc{A}$ is then defined as the $\Gamma$-invariants for the diagonal action on $\mf{u}(R_{\Gamma}) \underset{\mb{R}}{\otimes} Q$. Finally, $Z(\tilde{\mc{G}}) \simeq U(1)$ is the center of $\tilde{\mc{G}}$.

We would like to generalize the above to the case that $\hat{\Gamma}$ is of the more general form introduced above. The ensuing infinite-dimensional linear algebra naturally requires some functional analysis. We first note that $\Lambda$ naturally inherits a norm $| \cdot |\colon \Lambda \to \mb{R}_{\ge 0}$ from its embedding in $Q$. Most of the definition of $\mc{G}$ and $\mc{A}$ above may be repeated \emph{verbatim} once we construct a suitable regular representation $R_{\hat{\Gamma}}$ carrying a Hilbert space structure. The natural idea is to replace $\mb{C}[\Lambda]$ by the Hilbert space of square-summable sequences $\ell^2(\Lambda)$, or slightly more generally, the Sobolev $s$-regular versions thereof. Recall that throughout, we are assuming we have a canonical splitting $S$ of $1 \to \Lambda \to \hat{\Gamma} \to \Gamma \to 1$, as we assume we are in the case of a linear $\hat\Gamma\acts{Q}$ action. Such a splitting $S\colon \Gamma \to \hat{\Gamma}$ also yields $\tilde{S}\colon \hat{\Gamma} \to \Lambda$, a left inverse to $\Lambda \hookrightarrow \hat{\Gamma}$ defined by $\tilde S(\hat\gamma) = \hat\gamma \cdot S(\pi(\hat\gamma))^{-1}$. Equivalently, $\tilde S(\hat\gamma)$ is the result of acting by $\hat\gamma$ on $0\in Q$. We then have the following definition:

\begin{defn} \label{def:originalDefs} \phantom{this is here for spacing reasons}
\begin{itemize}
\item Let $\ell^2(\Lambda) := \{(a_{\lambda})_{\lambda \in \Lambda} \in \mathsf{Maps}(\Lambda,\mb{C}) \bigm| \sum_{\lambda \in \Lambda} |a_{\lambda}|^2 < \infty\}$ be the usual Hilbert space of $\ell^2$ sequences. 
\item More generally, for any $s \in [0, \infty)$, we let $h^s(\Lambda) \subset \ell^2(\Lambda)$ denote the subspace $$h^s(\Lambda) := \{(a_\lambda)_{\lambda \in \Lambda} \bigm| \sum_{\lambda \in \Lambda} |\lambda|^{2s} |a_\lambda|^2 < \infty\}.$$ For $s = \infty$, we define $h^{\infty}(\Lambda)$, which we may also denote $c^{\infty}(\Lambda)$, as $$h^{\infty}(\Lambda) := \bigcap_{s \ge 0} h^s(\Lambda).$$
\item Let $\iota\colon \Lambda \hookrightarrow \ell^2(\Lambda)$ denote the canonical Schauder basis $\iota(\lambda) = (\delta_{\lambda \lambda'})_{\lambda' \in \Lambda}$. Note that $$\iota(\Lambda) \subset h^s(\Lambda)\ \forall s \in [0, \infty].$$
\item For $0 \le s < \infty$, we define ``the (Sobolev $s$-regular) regular representation of $\hat{\Gamma}$,'' $R^s_{\hat{\Gamma}}$, by
\beq R^s_{\hat{\Gamma}} := \{(a_{\tilde{\gamma}})_{\tilde{\gamma} \in \hat{\Gamma}} \in \mathsf{Maps}(\hat{\Gamma},\mb{C}) \bigm| \forall\tilde{\gamma}\in\hat{\Gamma}, \sum_{\lambda \in \Lambda} |\lambda|^{2s} |a_{\lambda+\tilde{\gamma}}|^2 < \infty \}\ . \eeq
Similarly, $R^{\infty}_{\hat{\Gamma}} := \cap_{s \ge 0} R^s_{\hat{\Gamma}}$. Note that for any $s \in [0, \infty]$, the splitting $S$ yields an isomorphism
\begin{eqnarray} \label{eq:sregrep} 
R^s_{\hat{\Gamma}} &\simeq& h^s(\Lambda) \underset{\mb{C}}{\otimes} R_{\Gamma} \\ 
(a_{\hat{\gamma}} := a_{\tilde{S}(\hat{\gamma})}a_{\pi(\hat{\gamma})}) &\mapsfrom& (a_{\lambda}) \otimes (a_{\gamma}) \nonumber 
\end{eqnarray} which allows us to endow $R^0_{\hat{\Gamma}} \simeq \ell^2(\Lambda) \otimes_{\mb{C}} R_{\Gamma}$ with its induced Hermitian\footnote{Of course, these spaces naturally have Hilbert space structures for all $s \in [0, \infty)$. We emphasize the Hilbert space structure for $s = 0$ as it is specifically for the $\ell^2$ norm that we would like to consider unitary matrices.} structure. %One may immediately verify this Hilbert space structure on $R_{\hat{\Gamma}}$ is canonical, i.e., independent of $S$.
\item By abuse of notation, let $\iota\colon\hat{\Gamma} \to R^s_{\hat{\Gamma}}$ continue to denote the canonical topological basis and let $\iota\colon \Gamma\to R_\Gamma$ denote the analogous basis of $R_\Gamma$.
%The section $S$ determines a set-theoretic isomorphism $\hat{\Gamma} \simeq \Lambda \times \Gamma$; by abuse of notation, let $\iota$ also denote the composition $$\hat{\Gamma} \stackrel{\sim}{\to} \Lambda \times \Gamma \hookrightarrow h^s(\Lambda) \otimes_{\mb{C}} \mb{C}[\Gamma]$$ so that $\iota(\hat{\Gamma})$ once again provides a topological basis for $R^s_{\hat{\Gamma}}$. Note that while $\iota$ depends on the choice of section $S$, the image $\iota(\hat{\Gamma})$ is canonical.
\item Let $U^0(R_{\hat{\Gamma}})$ be the group of unitary operators on $R^0_{\hat{\Gamma}}$; more generally,
$$U^s(R_{\hat{\Gamma}}) := \{g \in U^0(R_{\hat{\Gamma}}) \bigm| g\text{ preserves }R^s_{\hat{\Gamma}} \subset R^0_{\hat{\Gamma}}\}\,.$$
We do not put topologies on these groups.
\item Let $j\colon \hat{\Gamma} \hookrightarrow U^s(R_{\hat{\Gamma}})$ be the group homomorphism defined by mandating for $\tilde{\gamma} \in \hat{\Gamma}$ that $j(\tilde{\gamma})$ map $\iota(\tilde{\gamma'})$ to $\iota(\tilde{\gamma}\tilde{ \gamma'})$.  Indeed,  having defined the action on a topological basis for $R^s_{\hat{\Gamma}}$, at most one continuous, linear extension exists, and it is not difficult to see (i) that such an extension does exist for all $s \in [0, \infty]$ and (ii) that it is unitary for the $\ell^2$ inner product. We also let $j\colon \Gamma\to U(R_\Gamma)$ denote the analogous permutation representation.
%Moreover, while the map $i: \hat{\Gamma} \hookrightarrow U^s(R_{\hat{\Gamma}})$ again depends on the precise choice of section $S$, the image subgroup $i(\hat{\Gamma})$ is once again independent.
\item Define the gauge group as
\beq \label{eq:tildeGD0}
\widetilde{\mc{G}}^s_{\hat{\Gamma}} := U^s(R_{\hat{\Gamma}})^{j(\hat{\Gamma})} = \{ g \in U^s(R_{\hat{\Gamma}}) \bigm| g j(\tilde{\gamma}) = j(\tilde{\gamma}) g\ \forall \tilde{\gamma} \in \hat{\Gamma}\}\,.\eeq
As we will see below, for $s>r/2$ this is naturally a Hilbert Lie group.
\item Define
$$\mf{u}^s(R_{\hat{\Gamma}}) := \{ X \in \mathrm{End}(R^0_{\hat{\Gamma}}) \bigm| X\text{ is skew-Hermitian and preserves }R^s_{\hat{\Gamma}} \subset R^0_{\hat{\Gamma}}\}\,.$$
This `looks like' a Lie algebra for $U^s(R_{\hat\Gamma})$, but we stress that we did not put a topology on this group, let alone a smooth structure. 
%Some warnings are in order. Namely, while $U^s(R_{\hat{\Gamma}})$ has a natural Hilbert (or Fr\'echet, for $s = \infty$) Lie group structure, it is \emph{not} true that $\mf{u}^s(R_{\hat{\Gamma}})$ is the Lie algebra (i.e., the tangent space at the identity $T_{\text{id}}U^s(R_{\hat{\Gamma}})$) for this Lie group structure. Moreover, it is not even always true that $\mf{u}^s(R_{\hat{\Gamma}})$ is always a Lie algebra unless $s$ is sufficiently large that the Lie bracket is well-defined (instead of only densely-defined). We return to these points in Remark~\ref{rmk:reg} and Proposition~\ref{prop:reg} below, in the Fourier dual presentation. \MZ{does any of this ever matter or if we worked in the sensible range of $s$'s like everyone else in this field seems to could we not spend a page (adding up the suffering at various places) talking about this?}
\item Define $R^{-\infty}_{\hat{\Gamma}}$ by duality, and define $\mf{u}^{-\infty}(R_{\hat{\Gamma}})$ as skew-Hermitian operators $R^{\infty}_{\hat{\Gamma}} \to R^{-\infty}_{\hat{\Gamma}}$. Using $\iota$ to provide a topological basis of $R_{\hat{\Gamma}}$, one may represent elements of $R^{-\infty}_{\hat{\Gamma}}$ and $\mf{u}^{-\infty}(R_{\hat{\Gamma}})$, respectively, by vectors and skew matrices ``with at most polynomial growth.''

\vspace{.2cm} 

We then define the canonical element $d \in \mf{u}^{-\infty}(R_{\hat{\Gamma}}) \underset{\mb{R}}{\otimes} Q$ by demanding that $d$ (regarded as a matrix as in the last sentence) map $\iota(\tilde{\gamma})$ to $i \, \iota(\tilde\gamma) \otimes \tilde S(\tilde\gamma)$, where we use the inclusions $\tilde S(\gamma) \in \Lambda \hookrightarrow \Lambda_\RR \hookrightarrow Q$. We observe the important commutation relation 
\be \label{eq:dComm} [j(\lambda), d] = -i \, {\rm Id}\otimes \lambda \ee
for $\lambda\in \Lambda$. (When $j(\lambda)$ is to the left of $d$, it is interpreted as $j(\lambda)\otimes {\rm Id}$.)

\item %Similarly to the warning above, the obvious action of $U^s(R_{\hat{\Gamma}})\acts{\mf{u}^s(R_{\hat{\Gamma}})}$ is only densely-defined.
Note the action of $\hat{\Gamma}\stackrel{j}{\hookrightarrow}U^\infty(R_{\hat{\Gamma}})$ is well-defined on $\mf{u}^s(R_{\hat{\Gamma}})$ even for $s = -\infty$. We now enlarge this $\hat{\Gamma}\acts{\mf{u}^{-\infty}(R_{\hat{\Gamma}})}$ action to a $\hat{\Gamma}$-action on $\mf{u}^{-\infty}(R_{\hat{\Gamma}}) \otimes_{\mb{R}} Q$ as follows: let $\hat{\gamma}$ act linearly by $\hat{\gamma} \otimes \pi(\hat{\gamma})$ and then add $i\, \mathrm{Id} \otimes \tilde{S}(\hat{\gamma})$. Explicitly, $\sum_a X_a\otimes q_a \mapsto \parens{\sum_a j(\hat\gamma) X_a j(\hat\gamma)^{-1} \otimes \pi(\hat\gamma)(q_a)} + i \, {\rm Id}\otimes \tilde S(\hat\gamma)$, where $a$ indexes a basis $q_a$ of $Q$.

We may then define the ambient space 
\beq \mc{A}^{-\infty}_{\hat{\Gamma}} := \Big( \mf{u}^{-\infty}(R_{\hat{\Gamma}}) \otimes_{\mb{R}} Q \Big)^{\hat{\Gamma}}\,.\eeq
We may verify $d \in \mc{A}^{-\infty}_{\hat{\Gamma}}$, which in turn allows us to define $\mc{A}^s_{\hat{\Gamma}}$ for $s \in [0, \infty]$ as
\beq \mc{A}^s_{\hat{\Gamma}} := \Big( d + \mf{u}^s(R_{\hat{\Gamma}}) \otimes_{\mb{R}} Q \Big)^{\hat{\Gamma}}\,.\eeq
Here $d + \mf{u}^s(R_{\hat{\Gamma}})\otimes_{\mb{R}}Q$ is an affine subspace of $\mf{u}^{-\infty}(R_{\hat{\Gamma}})\otimes_{\mb{R}}Q$. Note that one may have \emph{a priori} wished to define $\mc{A}^s_{\hat{\Gamma}}$ as simply $(\mf{u}^s(R_{\hat{\Gamma}})\otimes_{\mb{R}}Q)^{\hat{\Gamma}}$, but the proof of Prop. \ref{prop:nontrans} will make it clear that this space would be empty.
\end{itemize}
\end{defn}

We now pass to the Fourier-dual presentations. We first observe that the automorphism $\Gamma\acts \Lambda$ induces two natural actions on $\Lambda^\vee$. The first is the transpose action, and the second is defined by using the inner product on $Q$ to identify $\Lambda_\RR$ and $\Lambda_\RR^*$. These actions are inverses of each other, and we equip $\Lambda^\vee$ with the second action. We similarly define $\Gamma$ actions on $\Lambda_\RR^*$, $(\Lambda_\RR^*)^\perp$, and $Q^*$.

\begin{prop} \label{prop:nontrans} Let $T_\Lambda^{\vee}$ denote $\mathrm{Hom}(\Lambda, U(1))$, i.e. $T_\Lambda \simeq \Lambda_{\mb{R}}^* / \Lambda^{\vee}$, the dual torus of $T_{\Lambda} := \Lambda \underset{\mb{Z}}{\otimes} U(1) \simeq \Lambda_{\mb{R}} / \Lambda$. Then, for all $0 \le s \le \infty$, we have the following: 
\begin{enumerate}[(i)] 
\item $h^s(\Lambda) \simeq H^s(T_\Lambda^{\vee}, \mb{C})$,  
\item $R^s_{\hat{\Gamma}} \simeq H^s(T_\Lambda^{\vee}, R_{\Gamma})$.
\item For $s = 0$ or $r/2 < s \le \infty$, $U^s(R_{\hat{\Gamma}})^{j(\Lambda)} \simeq H^s(T_{\Lambda}^{\vee},U(R_{\Gamma}))$. 
\item For $r/2 < s \le \infty$, $\Big(d + \mf{u}^s(R_{\hat{\Gamma}})\otimes_{\mb{R}}Q\Big)^\Lambda \simeq \mathrm{Conn}^s(T_{\Lambda}^{\vee},R_{\Gamma}) \times H^s(T_{\Lambda}^{\vee}, \mf{u}(R_{\Gamma}) \otimes \Lambda_{\mb{R}}^{\perp})$, where\footnote{\label{ft:trivconn} Of course, as connections on a trivial bundle, the space $\mathrm{Conn}^s(T_{\Lambda}^{\vee},R_{\Gamma})$ may be immediately identified with $H^s(T_{\Lambda}^{\vee}, \mf{u}(R_{\Gamma}) \otimes \Lambda_{\mb{R}})$. The reason to not do so throughout is for purposes of recalling the natural group action, as will be explained below. That being said, we will frequently use that $\mathrm{Conn}^s(T^{\vee}_{\Lambda},R_{\Gamma})$ is a torsor over $H^s(T^{\vee}_{\Lambda},\mf{u}(R_{\Gamma})\otimes \Lambda_{\mb{R}})$ -- and specifically its presentation as an affine translate thereof by the de~Rham differential $d_{\mathrm{dR}}$.} 
\begin{eqnarray*} \mathrm{Conn}^s(T_{\Lambda}^{\vee},R_{\Gamma}) &:=& \{H^s\text{-regular connections on the trivial principal $U(R_{\Gamma})$-bundle on $T_{\Lambda}^{\vee}$}\} \\ &\simeq& \{H^s\text{-regular Hermitian connections on the trivial $R_{\Gamma}$-bundle on $T_{\Lambda}^{\vee}$}\}
\end{eqnarray*} 
and $H^s(T_{\Lambda}^{\vee}, \mf{u}(R_{\Gamma}) \otimes \Lambda_{\mb{R}}^{\perp})$ denotes $H^s$-regular functions on $T_{\Lambda}^{\vee}$ valued in $\mf{u}(R_{\Gamma}) \otimes \Lambda_{\mb{R}}^{\perp}$. 
\end{enumerate}

The automorphism $\Gamma\acts{\Lambda^\vee}$ induces an action of $\Gamma$ on $T_{\Lambda}^{\vee}$, while $\Gamma \hookrightarrow U(R_{\Gamma})$ acts on $U(R_{\Gamma})$ by conjugation and hence similarly on $\mf{u}(R_{\Gamma})$; one also obtains a $\Gamma$-action on connections. Finally, $\Lambda_{\mb{R}}^{\perp}$ also carries a natural $\Gamma$-action.  \begin{enumerate}[(i)]\setcounter{enumi}{4}

\item For $s = 0$ or $r/2 < s \le \infty$, $\widetilde{\mc{G}}_{\hat{\Gamma}}^s \simeq H^s_{\Gamma}(T_{\Lambda}^{\vee},U(R_{\Gamma}))$, i.e., $\Gamma$-equivariant $H^s$-regular maps $T_{\Lambda}^{\vee} \to U(R_{\Gamma})$, and \item for $r/2 < s \le \infty$, $\mc{A}^s_{\hat{\Gamma}} \simeq \mathrm{Conn}^s_{\Gamma}(T_{\Lambda}^{\vee},R_{\Gamma}) \times H^s_{\Gamma}(T_{\Lambda}^{\vee}, \mf{u}(R_{\Gamma}) \otimes \Lambda_{\mb{R}}^{\perp})$, i.e., the space of $\Gamma$-equivariant connections and functions, respectively. \end{enumerate} \end{prop}

\begin{defn} Given the splitting of statement~\emph{(vi)} above, we make the definitions \begin{align*} (\mc{A}^s_{\hat{\Gamma}})' &:= \mathrm{Conn}^s_{\Gamma}(T^{\vee}_{\Lambda},R_{\Gamma}) \\ (\mc{A}^s_{\hat{\Gamma}})'' &:= H^s_{\Gamma}(T^{\vee}_{\Lambda},\mf{u}(R_{\Gamma}) \otimes \Lambda^{\perp}_{\mb{R}})\,,\end{align*} so that we have the canonical splitting \beq \mc{A}^s_{\hat{\Gamma}} \simeq (\mc{A}^s_{\hat{\Gamma}})' \times (\mc{A}^s_{\hat{\Gamma}})''\,.\eeq\end{defn}

\bp The first statement is immediate from Fourier theory. So too is the second from the isomorphism~\eqref{eq:sregrep}. %, which also makes it clear that statement \emph{(ii)} holds in general but only noncanonically so. 
We now show~\emph{(iii)} and \emph{(iv)}. 

Suppose $g \in U^0(R_{\hat{\Gamma}})$ indeed commutes with $j(\Lambda)$.  Label the elements of $\hat{\Gamma}$ by pairs $(m, \gamma) \in \Lambda \times \Gamma$ using the splitting, and let $g_{m_1, m_2}^{\gamma_1, \gamma_2}$ denote the matrix elements $\langle (m_1, \gamma_1), g (m_2, \gamma_2) \rangle$. For $k\in \Lambda$, we denote the matrix elements $\avg{(m_1,\gamma_1),j(k)(m_2,\gamma_2)}$ by $j(k)_{m_1,m_2}^{\gamma_1,\gamma_2} = j(k)_{m_1,m_2} \delta^{\gamma_1,\gamma_2}$, where $j(k)_{m_1,m_2} = \delta_{m_1,m_2+k}$. Finally, write the matrix elements $\avg{(m_1,\gamma_1),d(m_2,\gamma_2)} \in Q\otimes \CC$ as $d^{\gamma_1,\gamma_2}_{m_1,m_2} = d_{m_1} \delta_{m_1,m_2} \delta^{\gamma_1,\gamma_2}$, where $d_{m_1} = i m_1$. That $g$ commutes with $j(\Lambda)$ for $m \in \Lambda$ yields the condition
\beq \label{eq:trans} g_{m_1, m_2}^{\gamma_1, \gamma_2} = g_{m_1+m, m_2+m}^{\gamma_1,\gamma_2}\text{ for all }m_1,m_2,m\in \Lambda, \gamma_1,\gamma_2 \in \Gamma\,.\eeq
This constraint is solved by those $g$ of the form
\be g^{\gamma_1,\gamma_2}_{m_1,m_2} = \sum_{k\in \Lambda} g^{\gamma_1,\gamma_2}_{k,0} j(k)_{m_1,m_2} \ . \ee
Similarly, the $\Lambda$-invariance constraint on $A\in \A^{-\infty}_{\hat\Gamma}$ is solved by those $A$ satisfying
\be A^{\gamma_1,\gamma_2}_{m_1,m_2} = d^{\gamma_1,\gamma_2}_{m_1,m_2} + \sum_{k\in \Lambda} a_{k,0}^{\gamma_1,\gamma_2} j(k)_{m_1,m_2} \ ; \ee
the second term on the right side satisfies a shift invariance condition analogous to~\eqref{eq:trans}. Between these expressions and \eqref{eq:dComm}, we are led to identify $j(k)$ for $k\in\Lambda$ with the function $e^{i\vec k\cdot \vec\theta}$ on $T^\vee_\Lambda$, where $\vec{\theta}\in \Lambda_\RR^*$ provides coordinates on $T_{\Lambda}^{\vee}$, and $d$ with the exterior derivative on $T^\vee_\Lambda$. We therefore define $L(g)\colon T^\vee_\Lambda \to \End(R_\Gamma)$ as the matrix-valued function with $(\gamma_1, \gamma_2)$ matrix entry given by
\beq\label{eq:Ldefn}
L(g)^{\gamma_1, \gamma_2} := \sum_{k \in \Lambda} g_{k,0}^{\gamma_1,\gamma_2} \exp(i\vec{k}\cdot\vec{\theta}) \ , \eeq
where this sum indicates the element of $L^2(T_{\Lambda}^{\vee},\mb{C})$ provided by Fourier theory. (As usual, we may for simplicity perform our manipulations below by assuming sufficient regularity that all our series converge uniformly -- and then simply extend by continuity to $\ell^2(\Lambda) \stackrel{\sim}{\to} L^2(T_{\Lambda}^{\vee})$.) Similarly, we define $L(A)\colon T^\vee_\Lambda \to \End(R_\Gamma)\otimes Q$ by
\be L(A)^{\gamma_1,\gamma_2} := \sum_{k\in\Lambda} a_{k,0}^{\gamma_1,\gamma_2} \exp(i\vec k\cdot\vec\theta) \ . \ee
From here, we focus on completing the verification of \emph{(iii)}, as that of \emph{(iv)} proceeds similarly.

Using that the $(m, \gamma)$ provide an orthonormal basis for $R_{\hat{\Gamma}}^0$ and the unitarity of $g$, we obtain 
$$\int_{T_{\Lambda}^{\vee}} \sum_{\gamma_1 \in \Gamma} \overline{L(g)^{\gamma_1,\gamma_3}} L(g)^{\gamma_1,\gamma_2} \exp(i\vec{m}\cdot\vec{\theta}) d\vec{\theta} = \hat V \delta_{0,m}\delta_{\gamma_2,\gamma_3}\ \forall m \in \Lambda, \gamma_2,\gamma_3 \in \Gamma$$
$$\implies \sum_{\gamma_1 \in \Gamma} \overline{L(g)^{\gamma_1,\gamma_3}}L(g)^{\gamma_1,\gamma_2}  = \delta_{\gamma_2,\gamma_3}\text{ a.e. }\forall \gamma_2,\gamma_3 \in \Gamma$$
$$ \implies L(g) \in L^2(T_{\Lambda}^{\vee}, U(R_{\Gamma}))\,.$$
Here, $\hat V$ is the volume of $T^\vee_\Lambda$. Conversely, given an element of $g^{\vee} \in L^2(T_{\Lambda}^{\vee},U(R_{\Gamma})) \subset L^2(T_{\Lambda}^{\vee},\End(R_{\Gamma}))$, expanding its matrix elements in Fourier series as in \eqref{eq:Ldefn} before extending by~\eqref{eq:trans} will yield an element $\ell(g^{\vee}) \in U^0(R_{\hat{\Gamma}})^{j(\Lambda)}$ in the inverse association to the map $L$ defined above. In particular, such a $g^\vee$ is bounded in $L^\infty$, since it is unitary; hence, it does indeed define an endomorphism of $R^0_{\hat\Gamma}$. It is a similar exercise in Fourier analysis to verify that $L, \ell$ are homomorphisms so that we have an isomorphism $U^0(R_{\hat{\Gamma}})^{j(\Lambda)} \simeq L^2(T^{\vee}_{\Lambda},U(R_{\Gamma}))$.

Finally, for $s > r/2$, it similarly follows from Fourier duality that $U^s(R_{\hat{\Gamma}})^{j(\Lambda)}$ and $H^s(T^{\vee}_{\Lambda},U(R_{\Gamma}))$ are carried to each other under the inverse morphisms, i.e., that the subgroup of $L^2(T^{\vee}_{\Lambda},U(R_{\Gamma}))$ which preserves $H^s(T^{\vee}_{\Lambda},R_{\Gamma}) \subset L^2(T^{\vee}_{\Lambda}, R_{\Gamma})$ is $H^s(T^{\vee}_{\Lambda},U(R_{\Gamma}))$, which follows as $H^s$ is a Banach algebra in this Sobolev range.

%the columns of $g$ are forced to have sufficiently fast decay so as to lie in $h^s(\Lambda)$ so that $L(g)^{\gamma_1,\gamma_2} \in H^s(T_{\Lambda}^{\vee},\mb{C})$ for all $\gamma_1,\gamma_2$; similar comments hold for the morphism $\ell$ and so the Sobolev $s$-regular subgroups are carried to each other under these inverse morphisms.
%In particular, $g$ preserves $R^s_{\hat{\Gamma}}$, or equivalently $L(g)$ preserves $H^s(T^{\vee}_{\Lambda},R_{\Gamma})$, as $H^s$ is a Banach algebra in this Sobolev range.

Imposing the further $\Gamma$-invariance necessary to arrive at statements \emph{(v)} and \emph{(vi)} is relatively immediate. For instance, for gauge transformations the $\Gamma$-invariance condition reads $g_{\gamma k,0} = j(\gamma) g_{k,0} j(\gamma)^\dagger$, where these are all thought of as $|\Gamma|\times |\Gamma|$ matrices. We then have
\be L(g)(\gamma\vec\theta) = \sum_{k\in\Lambda} g_{k,0} e^{i\vec k\cdot (\gamma\vec\theta)} = \sum_{k\in\Lambda} g_{k,0} e^{i(\gamma^{-1} \vec k)\cdot \vec\theta} = \sum_{k\in\Lambda} g_{\gamma k, 0} e^{i\vec k\cdot\vec\theta} = j(\gamma) L(g)(\vec\theta) j(\gamma)^{-1} \ . \label{eq:gEquiv} \ee
We similarly have $\gamma^* L(A) = j(\gamma) L(A) j(\gamma)^{-1}$.
\ep 

The data of the Kronheimer construction for the infinite discrete group $\hat{\Gamma}$ has hence been recast as gauge-theoretic data on the dual torus $T_{\Lambda}^{\vee}$. As usual now in gauge theory, we will work in a range of Sobolev regularity such that various multiplication- and action-maps are well-defined (as opposed to merely densely-defined). Indeed, recall that in the range $s > r/2$, $H^s(\mb{R}^r)$ is a Banach algebra for $s > r/2$ and that Sobolev $s$-regular functions are continuous. We will henceforth assume at least this regularity unless specifically noted otherwise.

The following is a computation: 
\begin{lem}
The conjugation action of $U^{\infty}(R_{\hat{\Gamma}})^{\Lambda}$ on $\Big(\mf{u}^{\infty}(R_{\hat{\Gamma}})\underset{\mb{R}}{\otimes}Q\Big)^{\Lambda}$ is Fourier-dual under the isomorphisms of Proposition~\ref{prop:nontrans} above to the action \begin{align*} C^{\infty}(T_{\Lambda}^{\vee},U(R_{\Gamma}))&\acts{\mathrm{Conn}^{\infty}(T_{\Lambda}^{\vee},R_{\Gamma}) \times C^{\infty}(T_{\Lambda}^{\vee}, \mf{u}(R_{\Gamma}) \otimes \Lambda_{\mb{R}}^{\perp})}\\g&\mapsto\Big((\nabla,\phi)\mapsto(g^{-1}\nabla g,g^{-1}\phi g)\Big),\text{ or}\\g&\mapsto\Big((A,\phi)\mapsto(g^{-1}Ag+g^{-1}dg,g^{-1}\phi g)\Big)\,.\end{align*} Note that in passing from the second to the third line above, we follow footnote~\ref{ft:trivconn} in identifying the connection $\nabla$ on a trivial bundle with a one-form $A$; i.e., we write $\nabla = d + A$.

Finally, the above continuously extends to an action of $$H^{s+1}(T_{\Lambda}^{\vee},U(R_{\Gamma}))\acts{\mathrm{Conn}^s(T^{\vee}_{\Lambda},R_{\Gamma}) \times H^s(T^{\vee}_{\Lambda},\mf{u}(R_{\Gamma}) \otimes \Lambda_{\mb{R}}^{\perp})}$$ for any $r/2 < s \le \infty$.
\end{lem}

\begin{defn} Given $\hat{\Gamma}\acts{Q}$ as above, define
\beq (\widetilde{\mc{G}}^s_{\hat{\Gamma}})_{\mb{C}} := H^s_{\Gamma}(T_{\Lambda}^{\vee}, GL_{\mb{C}}(R_{\Gamma}))\,.\eeq
\end{defn}

We refer to $(\widetilde{\mc{G}}^s_{\hat{\Gamma}})_{\mb{C}}$ as the \emph{complexified} Lie group; indeed, in the good range of $s > r/2$, the natural embedding of $\widetilde{\mc{G}}^s_{\hat{\Gamma}}  \hookrightarrow (\widetilde{\mc{G}}^s_{\hat{\Gamma}})_{\mb{C}}$ of Hilbert (Fr\'echet) Lie groups induces an isomorphism $$T_{\mathrm{Id}}\widetilde{\mc{G}}^s_{\hat{\Gamma}} \underset{\mb{R}}{\otimes} \mb{C} \stackrel{\sim}{\leftarrow} T_{\mathrm{Id}}(\widetilde{\mc{G}}^s_{\hat{\Gamma}})_{\mb{C}}\,.$$ Once again, this complexified Lie group admits a definition in terms of $\hat{\Gamma}$-invariant endomorphisms of $R^0_{\hat{\Gamma}}$ with appropriate regularity conditions, which one may show agrees with the above via the Fourier transform.

\begin{prop} \label{prop:reg} Given a linear $\hat{\Gamma}\acts{Q}$ action as above, we have, for any $s > r/2$, that \begin{enumerate}[(i)] 
\item $(\widetilde{\mc{G}}^s_{\hat{\Gamma}})_{\mb{C}}$ is a group, 
\item $\widetilde{\mc{G}}^{s+1}_{\hat{\Gamma}}$ naturally acts smoothly on $\mc{A}^s_{\hat{\Gamma}}$, as does $(\widetilde{\mc{G}}^{s+1}_{\hat{\Gamma}})_{\mb{C}}$ if $Q$ carries a complex structure respected by $\Gamma$, 
\item $\mc{A}^s_{\hat{\Gamma}}$, as an affine space over (or torsor for) the vector space $$V^s:=H^s_{\Gamma}(T^{\vee}_{\Lambda},\mf{u}(R_{\Gamma})\underset{\mb{R}}{\otimes}Q),$$ carries a natural (weak) Riemannian structure\footnote{This Riemannian structure is weak because it does not yield an isomorphism between $V^s$ and its $L^2$ dual. Similar remarks hold for the infinite-dimensional K\"ahler and hyperkahler structures that we will introduce below, but for conciseness we will not call them weak K\"ahler or hyperkahler structures.} from the $L^2$ inner product on $V^s$,
\item $\widetilde{\mc{G}}^{s+1}_{\hat{\Gamma}}$ acts isometrically on said Riemannian structure, and
\item $\mc{A}^{s}_{\hat{\Gamma}}$ has a curvature\footnote{The \emph{curvature} nomenclature is due to the fact that if $r = \dim_{\mb{R}}Q$ so that $\mc{A}^s_{\hat{\Gamma}}$ is purely a space of connections on an $r$-dimensional torus, then this map is indeed a curvature.} map given by \begin{align*} \mc{A}^{s}_{\hat{\Gamma}} &\stackrel{\mathscr{F}}{\to} H^{s-1}_{\Gamma}(T^{\vee}_{\Lambda},\mf{u}(R_{\Gamma})\underset{\mb{R}}{\otimes}\Wedge^2Q) \\ (A,\phi)&\mapsto \Big(dA+\frac{1}{2}[A,A],d\phi+[A,\phi],\frac{1}{2}[\phi,\phi]\Big)\end{align*} Here $$\Wedge^2Q \simeq \Wedge^2(\Lambda_{\mb{R}}) \oplus (\Lambda_{\mb{R}} \otimes \Lambda_{\mb{R}}^{\perp}) \oplus \Wedge^2(\Lambda_{\mb{R}}^{\perp})$$ denotes the second exterior power of $Q$; the $d$ is a de~Rham differential, where we identify $k$-forms on $T_{\Lambda}^{\vee}$ with $\Wedge^k\Lambda_{\mb{R}}$-valued functions. \end{enumerate} \end{prop}

\bp Our regularity assumption makes the above straightforward. Only the Riemannian structure remains to be explained. Namely, in statement~\emph{(iii)}, we note that the tangent space at any point of $\mc{A}^s_{\hat{\Gamma}}$ is canonically identified with $V^s$ and so it is on this space that we are obliged to provide a metric, which we do via~(i) the Killing form on $\mf{u}(R_{\Gamma})$,~(ii) the inner product on $Q$, and~(iii) the $L^2$ inner product on $T^{\vee}_{\Lambda}$. Explicitly, we define $g_V(a,b) = - \frac{1}{|\Gamma| \hat V} \int_{T^\vee_\Lambda} \Tr(\avg{a,b})\, d{\rm Vol} = \frac{1}{|\Gamma| \hat V} (a,b)_{L^2}$, where $a,b\in V^s$, $\hat V$ is the volume of $T^\vee_\Lambda$, and the minus sign is required because $b$ is skew-Hermitian.
\ep 
%This is definitely not sharp -- e.g. [DK] explain how to get away with less in the last statement because we only need $H^s \times H^s \to H^{s-1}$ rather than $H^s$ being a Banach algebra. Whatever.

We are now ready to specialize to the cases of $\Gamma$ preserving more structure of $Q$, such as a K{\"a}hler or hyperkahler structure. So, suppose first that $Q$ carries the structure of a K{\"a}hler vector space, i.e. a nondegenerate two-form $\omega \in \Wedge^2 Q^{\vee}$ and complex structure $J \in \mathrm{End}_{\mb{R}}Q$ that together induce the original inner product on $Q$; let $U(Q) \hookrightarrow O(Q)$ denote the subgroup of linear transformations preserving this K{\"a}hler structure. 

We note that the isomorphism between $Q$ and $Q^*$ induced by the inner product on $Q$ allows us to translate this K\"ahler structure to $Q^*$ (and therefore to $T^\vee_\Lambda$); e.g., $J_{Q^*} w = (J w^\sharp)^\flat$, where $\sharp$ and $\flat$ are the usual musical isomorphisms that raise and lower indices using the metric. For the sake of clarity, we note that there is a second natural complex structure on $Q^*$, namely the transpose $J^T$, and this differs from the chosen complex structure on $Q^*$ by an overall sign. (This follows from the fact that $J$ is skew-Hermitian, which in turn follows from orthogonality and $J^2=-1$.)

We then have the following: 

\begin{prop}\label{prop:kahtime} Given $\hat{\Gamma}\acts{Q}$ such that $\Gamma \hookrightarrow U(Q)$, we have for all $r/2 < s \le \infty$ that
\begin{enumerate}[(i)]
\item the affine space $\mc{A}^s_{\hat{\Gamma}}$ carries a natural flat K{\"a}hler structure,
\item the $\tilde\G^{s+1}_{\hat{\Gamma}}$-action on $\mc{A}^s_{\hat{\Gamma}}$ is Hamiltonian, and 
\item the moment map for the above Hamiltonian action may be chosen to be 
$$\mu = \iota_\omega \circ \mathscr{F}\colon\mc{A}^s_{\hat{\Gamma}}\to H^{s-1}_{\Gamma}(T^{\vee}_{\Lambda},\mf{u}(R_{\Gamma})) \subset (\mathrm{Lie}\,\mc{G}^{s+1}_{\hat{\Gamma}})',$$ 
where $\iota_{\omega}\colon\Wedge^2Q\to\mb{R}$ is given by contracting with the K{\"a}hler form $\omega \in \Wedge^2 Q^{\vee}$. 
\end{enumerate}
\end{prop}

\begin{rmk} \label{rmk:mommap}
Note the phrasing of the last statement above -- in general, moment maps for Hamiltonian actions by some Lie group $\mc{G}$ are not uniquely determined, but rather are ambiguous up to the addition of some constant in $((\mathrm{Lie}\,\mc{G})')^{\mc{G}}$. Our identification of the moment map above is hence simply one (particularly nice) choice. Also, in order to avoid carrying around unpleasant constants throughout this paper, we will only require the equation characterizing a moment map (\eqref{eq:mommap} below) to be satisfied up to a non-zero multiplicative constant.
\end{rmk}
\bp
Our main task is to produce a K{\"a}hler structure on the vector space $V^s$, which we do by taking complex structure $J_V$ as induced by the complex structure $J\in\mathrm{End}_{\mb{R}}(Q)$ and two-form $\omega_V$ as induced by~(i) the Killing form on $\mf{u}(R_{\Gamma})$,~(ii) the two-form $\omega$ on $Q$, and~(iii) the $L^2$ inner product on $T^{\vee}_{\Lambda}$. Explicitly, we have $\omega_V(a,b) = -\frac{1}{|\Gamma|\hat V} \int_{T^\vee_\Lambda} \Tr(\omega(a,b))\, d{\rm Vol}$, where $a,b\in V^s$ and $\hat V$ is the volume of $T^\vee_\Lambda$. $J_V^2=-1$, orthogonality of $J_V$, and $g_V(a,b) = \omega_V(a, J_V b)$ (where the metric $g_V$ is from Prop. \ref{prop:reg}) all follow from the analogous results for $Q$. This formula for $\omega_V$ makes it clear that $\tilde\G^{s+1}_{\hat\Gamma}$ acts via symplectomorphisms. It is worth pointing out that, by regarding $a,b$ as one-forms on $T^\vee_\Lambda$ instead of elements of $Q$, one may recast this K\"ahler structure on $V^s$ in terms of that on $T^\vee_\Lambda$, and the formula $J_V a = - J_{T^\vee_\Lambda}^T a$ makes it clear that $(1,0)$ tangent vectors to $V^s$ are given by $(0,1)$-forms.

The remaining statements are then straightforward computations; we refer readers to the proof of Prop. \ref{prop:Hamtrick} for such a computation in the case $r=1$.
\ep

%It is likely of some interest to study the resulting K{\"a}hler quotients, especially on more general manifolds and upon deforming the moment map by nontrivial moment map parameters. We now, however, finally specialize to the case most of interest, namely that of $Q$ carrying a hyperkahler structure preserved by $\Gamma \hookrightarrow Sp(Q) \hookrightarrow O(Q)$. % finally ready to specialize to the case most of interest here, namely that of hyperkahler quotients. Logically speaking, we could first assume only that $Q$ carries a flat K{\"a}hler structure preserved by $\Gamma$ and study the ensuing natural K{\"a}hler structure on $\mc{A}^s_{\hat{\Gamma}}$ with Hamiltonian action by $\mc{G}^{s+1}_{\hat{\Gamma}}$, compute the moment map for this action, and so forth. These K{\"a}hler quotient moduli spaces (especially upon deforming by moment map parameters) are no doubt interesting. But to avoid repeating ourselves inordinately, we specialize immediately now to the hyperkahler case. In other words, assume now that $Q$ carries a flat hyperkahler structure and that $\Gamma \hookrightarrow O(Q)$ in fact preserves this structure, i.e. $\Gamma \hookrightarrow Sp(Q) \subset O(Q)$. We then have the following:

Let us specialize further now to the case that $Q$ carries a hyperkahler structure preserved by the group $Sp(Q) \subset O(Q)$. The inner product on $Q$ now induces a hyperkahler structure on $Q^*$. The relative sign between the $i$-th complex structure $J_{Q^*}^i$ ($i=1,2,3$) and $(J^i)^T$ is now quite important, as $J_{Q^*}^i$ obey the required unit quaternion algebra whereas $(J^i)^T$ do not. We denote the associated K\"ahler forms on $Q$ by $\omega^i$. After choosing a complex structure, we may equivalently describe them by providing a holomorphic symplectic form and a K\"ahler form.

The following proposition is once again a computation:

\begin{prop}\label{prop:hktime} Given $\hat{\Gamma}\acts{Q}$ with $\Gamma \hookrightarrow Sp(Q)$ as above, we have that, for all $r/2 < s \le \infty$,
\begin{enumerate}[(i)]
\item $\mc{A}^s_{\hat{\Gamma}}$ carries a natural flat hyperkahler structure,
\item the $\tilde\G^{s+1}_{\hat{\Gamma}}$ action on $\mc{A}^s_{\hat{\Gamma}}$ is tri-Hamiltonian, and 
\item the moment maps for said tri-Hamiltonian action may be identified with $\mu^i = \iota_{\omega^i} \circ \mathscr{F}$ where $\iota_{\omega^i}\colon\Wedge^2 Q\to\mb{R}$ contracts with the K\"ahler form $\omega^i$.
\end{enumerate}
\end{prop}

%\bp As in Proposition~\ref{prop:kah}, these claims are direct computations, with the only additional input as in Proposition~\ref{prop:reg} that the Sobolev regularity is sufficient to yield well-defined maps. \ep %and $F_{\mb{C}}$ is the complex curvature \begin{align*}H^s_{\Gamma}(T^{\vee}_{\Lambda},\mf{gl}_{\mb{C}}(\mc{L}_{\hat{\Gamma}})\otimes_{\mb{C}}Q) &\stackrel{F_{\mb{C}}}{\to} H^s_{\Gamma}(T^{\vee}_{\Lambda},\mf{gl}_{\mb{C}}(\mc{L}_{\hat{\Gamma}})\otimes_{\mb{C}}\Wedge^2_{\mb{C}}Q)\\ \nabla \oplus \phi &\mapsto \frac{1}{2}[\nabla \oplus \phi, \nabla \oplus \phi]\,,\end{align*} %what is the hyperkahler quotient in this case anyway? it's too big to just be Q/\hat{\Gamma}

We now define $\mc{G}^{s+1}_{\hat{\Gamma}}$ as the quotient of $\tilde{\mc{G}}^{s+1}_{\hat{\Gamma}}$ by a subgroup which acts trivially on $\mc{A}^{s}$:

\begin{defn} We define $\mc{G}^s_{\hat{\Gamma}}$ as the quotient group \beq \mc{G}^s_{\hat{\Gamma}} := H^s_{\Gamma}(T^{\vee}_{\Lambda},U(R_{\Gamma}))/U(1)\,,\eeq where $U(1)$ is included as constant maps to the central factor of $U(R_{\Gamma})$.
Note that the action of $\widetilde{\mc{G}}^{s+1}_{\hat{\Gamma}}$ on $\mc{A}^s_{\hat{\Gamma}}$ factors through $\mc{G}^{s+1}_{\hat{\Gamma}}$, and all the statements of Propositions~\ref{prop:reg}, \ref{prop:kahtime}, and \ref{prop:hktime} hold for $\mc{G}^{s+1}_{\hat{\Gamma}}\acts{\mc{A}^s_{\hat{\Gamma}}}$. \end{defn}

We then, of course, have $$\mathrm{Lie}\,\mc{G}^s_{\hat{\Gamma}} \simeq H^s_{\Gamma}(T^{\vee}_{\Lambda},\mf{u}(R_{\Gamma})) / \mf{u}(1) \simeq \{X \in H^s_{\Gamma}(T^{\vee}_{\Lambda},\mf{u}(R_{\Gamma})) \bigm| \int_{T^{\vee}_{\Lambda}} \mathrm{Tr}\,X d\mathrm{vol} = 0\}\,.$$

We now go on a small tangent concerning the distinction between $SU, U$, and $PU$ for the purposes of defining the gauge group. The mild vexation here is simply that the trace part of the connection is essentially irrelevant to our construction, and one may therefore wish to excise it. For example, the trace and trace-free parts of the moment maps only involve, respectively, the trace and trace-free parts of $A\in \A^s_{\hat\Gamma}$; the trace parts of the moment maps are linear; and the center of $\tilde\G^{s+1}_{\hat\Gamma}$ only acts on the trace parts of $A$. In light of this, one may define the variant $\bar{\mc{A}}^s_{\hat{\Gamma}}$ as an affine space over $H^s_{\Gamma}(T^{\vee}_{\Lambda},\mf{su}(R_{\Gamma}) \otimes_{\mb{R}} Q)$; $\widetilde{\mc{G}}^{s+1}_{\hat{\Gamma}}$ still acts via the composition $\bar{\mc{A}}^s_{\hat{\Gamma}} \hookrightarrow \mc{A}^s_{\hat{\Gamma}} \stackrel{\mathrm{act}}{\to} \mc{A}^s_{\hat{\Gamma}} \twoheadrightarrow \bar{\mc{A}}^s_{\hat{\Gamma}}$, where the last map eliminates the trace part, but now with a larger trivially-acting subgroup by which one may wish to quotient in order to arrive at another group $\bar{\mc{G}}^{s+1}_{\hat{\Gamma}}$. To that end, we note the following:

\begin{defn} \label{defn:gen} We say that the $\Gamma$-action on $\Lambda$ is \emph{generic} if it has no nonzero invariants on $\Lambda_{\mb{R}}$; i.e., if $$\Lambda_{\mb{R}}^{\Gamma} = \{0\}\,.$$ This condition is equivalent to $(\Lambda_{\mb{R}}^{\vee})^{\Gamma}=\{0\}$.\end{defn}

\begin{prop} \label{prop:Gdefn} For any $s > r/2$, we have the following injections of Lie groups: 
\beq\label{eq:Gdefn} \frac{H^s_{\Gamma}(T^{\vee}_{\Lambda},SU(R_{\Gamma}))}{H^s_{\Gamma}(T^{\vee}_{\Lambda},Z_{|\Gamma|})} \hookrightarrow \frac{H^s_{\Gamma}(T^{\vee}_{\Lambda},U(R_{\Gamma}))}{H^s_{\Gamma}(T^{\vee}_{\Lambda},U(1))} \hookrightarrow H^{s,p\ell}_{\Gamma}(T^{\vee}_{\Lambda},PU(R_{\Gamma}))\,,\eeq
where the $p \ell$ superscript in the last term denotes ``pointwise liftable'' maps, i.e. the subgroup of equivariant functions defined as \begin{align} H^{s,p\ell}_{\Gamma}(T^{\vee}_{\Lambda},PU(R_{\Gamma})) &:= \{f \in H^s_{\Gamma}(T^{\vee}_{\Lambda},PU(R_{\Gamma})) \bigm| \forall\,p \in T^{\vee}_{\Lambda}\text{ with stabilizer }\Gamma_p \subset \Gamma, \nonumber\\& \hspace{3cm} f(p) \in \mathrm{Im}(SU(R_{\Gamma})^{\Gamma_p} \to PU(R_{\Gamma})^{\Gamma_p})\}\,. \end{align} If the $\Gamma$ action on $\Lambda$ is generic, then all the arrows in~\eqref{eq:Gdefn} are isomorphisms. \end{prop}

\begin{defn} We define $\bar{\mc{G}}^s_{\hat{\Gamma}}$ as the quotient group \beq \bar{\mc{G}}^s_{\hat{\Gamma}} := H^s_{\Gamma}(T^{\vee}_{\Lambda},U(R_{\Gamma}))/H^s_{\Gamma}(T^{\vee}_{\Lambda},U(1))\,.\eeq 
The action of $\widetilde{\mc{G}}^{s+1}_{\hat{\Gamma}}$ on $\bar{\mc{A}}^s_{\hat{\Gamma}}$ now factors through $\bar{\mc{G}}^{s+1}_{\hat{\Gamma}}$, and all the statements of Propositions~\ref{prop:reg}, \ref{prop:kahtime}, and \ref{prop:hktime} hold for $\bar{\mc{G}}^{s+1}_{\hat{\Gamma}}\acts{\bar{\mc{A}}^s_{\hat{\Gamma}}}$. \end{defn}

Note that in the first term of~\eqref{eq:Gdefn}, the subgroup $H^s_{\Gamma}(T^{\vee}_{\Lambda},Z_{|\Gamma|})$ is simply a fanciful way of writing the finite cyclic group $Z_{|\Gamma|}$. Hence, one may immediately identify the Lie algebra of $\bar{\mc{G}}^s_{\hat{\Gamma}}$. 

\begin{cor} Given a linear $\hat{\Gamma}$-action on $Q$ such that the $\Gamma$-action on $\Lambda$ is generic, then for any $r/2 < s \le \infty$ we have $$\mathrm{Lie}\,\bar{\mc{G}}^s_{\hat{\Gamma}} \simeq H^s_{\Gamma}(T^{\vee}_{\Lambda},\mf{su}(R_{\Gamma}))\,.$$ \end{cor}

\begin{proof}[Proof of Proposition~\ref{prop:Gdefn}] Most of the assertions are immediate, and the maps of groups in~\eqref{eq:Gdefn} are immediately seen to be injective as indicated. Surjectivity of these arrows under the genericity assumption remains to be proven.

Note that 
$$\mathrm{Im}(SU(R_{\Gamma})^{\Gamma_p}\to PU(R_{\Gamma})^{\Gamma_p}) = \mathrm{Im}(U(R_{\Gamma})^{\Gamma_p}\to PU(R_{\Gamma})^{\Gamma_p})$$
for any subgroup $\Gamma_p \subset \Gamma$, so the pointwise lifting condition is identical whether phrased for lifting to $SU$ or $U$. Moreover, on any simply-connected $\Gamma_p$-equivariant open neighborhood of any point $p$, a pointwise liftable equivariant map to $PU(R_{\Gamma})$ may be locally equivariantly lifted to a map to $SU(R_{\Gamma})$. As such, we may locally lift our map to $SU(R_{\Gamma})$, and the ambiguities on overlaps of these local open neighborhoods are valued in $Z_{|\Gamma|}$. Hence, the obstruction to lifting a map in $H^{s,p\ell}_{\Gamma}(T^{\vee}_{\Lambda},PU(R_{\Gamma}))$ to $H^{s}_{\Gamma}(T^{\vee}_{\Lambda},SU(R_{\Gamma}))$ lies in the first \v{C}ech cohomology group $\mathrm{H}^1(T^{\vee}_{\Lambda}/\Gamma; Z_{|\Gamma|})$. Here by $T^{\vee}_{\Lambda}/\Gamma$ we mean the coarse quotient topological space, which we abbreviate $C$ with its coarse quotient map $T^{\vee}_{\Lambda} \stackrel{q}{\to} C$ in the following.

Using the image of the fixed point $0 \in T^{\vee}_{\Lambda}$ as a base-point in $C$, we readily see that $\pi_1(T^{\vee}_{\Lambda}) \twoheadrightarrow \pi_1(C)$; in fact, not only can loops in $C$ be lifted up to homotopy, they can be lifted on the nose. This surjection shows that $\pi_1(C)$ is abelian, and upon considering straight-line paths in the universal cover $\Lambda^{*}_{\mb{R}}$ of $T^{\vee}_{\Lambda}$, one may see that $\Lambda^{\vee} \cap (\Gamma - 1)\Lambda^{*}_{\mb{R}}$ is in the kernel of $\Lambda^{\vee} \simeq \pi_1(T^{\vee}_{\Lambda}) \to \pi_1(C)$, where we define the subspace $(\Gamma-1)\Lambda^{*}_{\mb{R}}$ to be generated as follows: $$(\Gamma-1)\Lambda^{*}_{\mb{R}} := \langle (\gamma-1)v \bigm| \gamma \in \Gamma, v \in \Lambda^{*}_{\mb{R}}\rangle\,.$$
But, any vector $w \in \Lambda_\RR^*$ orthogonal to the subspace $(\Gamma-1)\Lambda^{*}_{\mb{R}}$ must be preserved by $\Gamma$, contradicting the genericity assumption unless $w = 0$. Hence $(\Gamma-1)\Lambda^{*}_{\mb{R}} = \Lambda^{*}_{\mb{R}}$ and $\pi_1(C)$ is trivial, so in particular $\mathrm{H}^1(C; Z_{|\Gamma|}) = \mathrm{Hom}(\pi_1(C), Z_{|\Gamma|}) = 0$, completing the argument.
\ep

In the present section, we will mostly find it convenient to work with $\G^{s+1}_{\hat\Gamma}$ and $\A^s_{\hat\Gamma}$, but in the rest of the paper we will switch to $\bar\G^{s+1}_{\hat\Gamma}$ and $\bar\A^s_{\hat\Gamma}$. This will not affect the resulting hyperkahler quotient. To see this, we note that in tandem with the natural linear Coulomb gauge fixing condition, the trace parts of the moment maps yield the linear operator $d^+\oplus d^*$ on $T^\vee_\Lambda$, where the $+$ superscript indicates that we keep the part that survives contraction with $\omega^i$. Corollary \ref{cor:comm} of \S\ref{subsec:flatorb} will imply that the kernel of this operator coincides with the kernel of the standard operator $d\oplus d^*$ of Hodge theory. In the case of interest where $\Gamma$ acts on $\Lambda$ generically as per Definition~\ref{defn:gen}, there are no nonzero $\Gamma$-invariant harmonic 1-forms on $T^\vee_\Lambda$ and so this operator has trivial kernel.

We finally perform our final specialization: consider now the case of $Q$ four (real) dimensional, i.e. $Q \simeq \mb{C}^2$ with holomorphic symplectic form $\omega_{\mb{C}} = \omega^2 + i \omega^3 = dz^1 \wedge dz^2$ and K{\"a}hler form $\omega^1 = \frac{i}{2}(dz^1 \wedge d\overline{z^1} + dz^2 \wedge d\overline{z^2})$. We identify
\begin{align*}\mc{A}^s_{\hat{\Gamma}} &\simeq \Big(d+\mf{u}^s(R_{\hat{\Gamma}})\underset{\mb{R}}{\otimes}Q\Big)^{\hat{\Gamma}}\\&\simeq\Big(d+\mf{gl}^s_{\mb{C}}(R_{\hat{\Gamma}})\underset{\mb{C}}{\otimes}Q\Big)^{\hat{\Gamma}}\\&\simeq\Big(d+\mf{gl}^s_{\mb{C}}(R_{\hat{\Gamma}})\oplus\mf{gl}^s_{\mb{C}}(R_{\hat{\Gamma}})\Big)^{\hat{\Gamma}},
\end{align*}
%We will often \MZ{not sure this is ever the case any more}, hence, parametrize elements of $\mc{A}^s_{\hat{\Gamma}}$ in this case by pairs $\alpha, \beta \in \mf{gl}_{\mb{C}}^s(R_{\hat{\Gamma}})\simeq H^s(T^{\vee}_{\Lambda},\mf{gl}_{\mb{C}}(R_{\Gamma}))$. 

\begin{rmk}\label{rmk:gen} In this case, the finite subgroup $\Gamma \hookrightarrow Sp(1) \simeq SU(2)$ falls into the ADE classification as appears, e.g., in the McKay correspondence of~\cite{mckay37graphs}. We assume throughout that we are in the nontrivial case of $\Gamma \ne \{1\}$. Then one may directly confirm that a strong form of the genericity assumption of Definition~\ref{defn:gen} is satisfied: the only element of $Q$ with a nontrivial stabilizer is the origin. In particular, Proposition~\ref{prop:Gdefn} holds.\end{rmk}

\subsection{Hyperkahler quotients}

We recall the general setup of K{\"a}hler and hyperkahler quotients, following~\cite{MarsdenWeinstein, MumfordGIT, hitchin:hkSUSY}. 

\begin{constr}\label{constr:realKah} Suppose $(M, \omega)$ is a K\"{a}hler manifold with a Hamiltonian action by a Lie group $G$, i.e., an action $G\acts{M}$ such that there exists a $G$-equivariant \emph{moment map} $\mu\colon M \to \mf{g}^{\vee}$ such that for all $X \in \mf{g}$, 
\beq\label{eq:mommap}d(\langle \mu, X \rangle) = \iota_X \omega\,.\eeq 
Then, for any $\xi \in (\mf{g}^{\vee})^G$, $G$ acts on $\mu^{-1}(\xi)$ preserving the restriction of the K{\"a}hler form $\omega$, and we denote the quotient by
$$M\kq G := \mu^{-1}(\xi)/G\,.$$ 
If this quotient space is a manifold (or orbifold), it naturally carries a K\"{a}hler structure. The K\"ahler form is defined by observing that if $X\in \mf{g}$ determines a vertical tangent vector $x\in T_m M$ with $m\in \mu^{-1}(\xi)$ then $\omega(x,\cdot)=0$. The complex structure is induced from that of $M$ by using the fact that the horizontal subbundle over $\mu^{-1}(\xi)$ (the orthogonal complement of the vertical subbundle) is a complex subbundle of $TM|_{\mu^{-1}(\xi)}$. We will call the space $(\mf{g}^{\vee})^G$ the space of \emph{moment map} or \emph{Fayet-Iliopoulos} (FI) parameters; in good situations, one may expect to find a family of K\"ahler manifolds fibered over (an open sublocus within) this space of parameters. \end{constr}

\begin{rmk}\label{rmk:canstr} If $M\kq G$ is a manifold (or orbifold), said manifold structure is \emph{canonical}. In particular, it is a \emph{condition} that the quotient $M\kq G$ be a manifold -- as opposed to extra \emph{data} of the choice of some manifold structure. Indeed, we would like to say that a function $f$ on $M\kq G$ is smooth if and only if its pullback to $\mu^{-1}(\xi)$ is the restriction of a smooth function\footnote{We will frequently be applying this construction to the case that $M$ is an infinite-dimensional manifold. When $M$ is a Banach manifold, the definitions of a derivative, tangent space, smooth function, and so on agree \emph{verbatim} with the finite-dimensional treatment; see~\cite{Lang} for an exposition along these lines. When $M$ is a Fr\'echet manifold, however, the sheaf of smooth functions is still sensible but certainly requires more groundwork to develop in the absence of a norm. We hence stick to the usual approach in gauge theory of only using these constructions for the Banach case, accessing the Fr\'echet case afterwards only by comparison.} on $M$; it is then a question whether the topological space (or groupoid) $M\kq G$ endowed with this (sheaf of) functions is locally isomorphic to Euclidean space or not. See the end of~\cite[\S4.2.5]{DK} for a related discussion. \end{rmk}

\begin{constr}\label{constr:cplxKah} Given datum as above, one often has a complexified group $G_{\mb{C}}$ acting holomorphically on $M$. Given the choice of $\xi \in (\mf{g}^{\vee})^G$, one may define the $\xi$-polystable sublocus as $$M^{\mathrm{ps}} := \{m \in M \bigm| G_{\mb{C}}m \cap \mu^{-1}(\xi) \ne \emptyset\}\,.$$ Then, in good situations, one has an isomorphism 
\beq\label{eq:DUYabstract} \mu^{-1}(\xi)/G \stackrel{\sim}{\to} M^{\mathrm{ps}}/G_{\mb{C}}\,.\eeq 
If such an isomorphism holds, we call it a \emph{Donaldson-Uhlenbeck-Yau} (DUY) isomorphism\footnote{This isomorphism may also be called a Hitchin-Kobayashi correspondence.} or DUY theorem. \end{constr}

The equivalence of the two constructions above, when such a DUY theorem holds, is a pleasant way to learn more about both sides of the isomorphism. The left side of~\eqref{eq:DUYabstract} is more convenient for producing the K\"ahler structure while the right side is often easier to construct, albeit only `easily' yielding a structure as a complex manifold. We will often call the left side the ``real formulation'' of the quotient and the right side the ``complex formulation.'' Note that in the complex formulation $\xi$ enters only through the stability condition, i.e. in the characterization of the sublocus $M^{\mathrm{ps}}$. As such, while the space of FI parameters appears naturally on the left side of~\eqref{eq:DUYabstract}, the complex-analytic space on the right side of~\eqref{eq:DUYabstract} often depends on $\xi \in (\mf{g}^{\vee})^G$ only through a locally finite stratification\footnote{\label{footnote:VGIT}often called a \emph{wall-and-chamber decomposition}. See~\cite{DolgachevHuVGIT,ThaddeusVGIT} for a thorough explanation of this chamber structure in the finite-dimensional algebro-geometric context.} of this parameter space, with the complex structure of $M^{\mathrm{ps}}/G_{\mb{C}}$ constant within each stratum. We take this opportunity to also quickly introduce some other terminology which will be helpful:
\begin{defn}\label{defn:stab}
Given $G_{\mb{C}}\acts{M}$ as above, we similarly define the \emph{stable}, \emph{semistable}, and \emph{unstable} loci for a given $\xi \in (\mf{g}^{\vee})^G$ as
\begin{align*}
M^{\mathrm{st}} &:= \{m \in M^{\mathrm{ps}} \bigm| \mathrm{Stab}(m) \subset G_{\mb{C}}\text{ is finite}\} \\
M^{\mathrm{ss}} &:= \{m \in M \bigm| \overline{G_{\mb{C}}m} \cap M^{\mathrm{ps}} \ne \emptyset\} \\
M^{\mathrm{us}} &:= M \setminus M^{\mathrm{ss}} \ .
\end{align*}
\end{defn}

There is a standard approach to proving the DUY isomorphism if $G$ and $M$ are finite-dimensional following the Kempf-Ness theorem~\cite{kempfNess}; see the discussion after \cite[Theorem 1.1]{McGNcounter}, and the listed references therein, for careful treatments in a variety of cases. If $G$ and $M$ are instead infinite-dimensional manifolds, proving each case of a DUY isomorphism involves interesting analysis~\cite{donaldson:ns, donaldson:duy, yau:duy}. The complex formulation is often referred to as geometric invariant theory (GIT), an abbreviation we shall occasionally adopt as well. In this context, one often uses the notation $M\kqnoxi G_{\mb{C}}$ rather than $M\kqnoxi G$; this mild abuse of notation serves as a reminder that at least in the complex formulation, the construction is purely holomorphic. 

The definitions of $M^{\mathrm{st}}, M^{\mathrm{ss}}, M^{\mathrm{us}}$ above all depend on $M^{\mathrm{ps}}$, and therefore on $\xi \in (\mf{g}^{\vee})^G$. One may wish for a definition of these loci that is more intrinsically holomorphic. When such a (necessarily $\xi$-dependent) \emph{stability condition} may be found, the proof that this alternative characterization is equivalent to the definitions above is typically included as part of the statement of the DUY theorem. In such approaches, the definitions of (semi)stability are usually more privileged, with the polystable locus defined afterward in some way ensuring that an orbit is polystable if it is closed within the semistable locus. 

\vspace{.2cm}

We now discuss the analogous construction for hyperkahler manifolds. There are various ways to write the package of data with which a hyperkahler manifold is endowed; here, we will write the triple $(M, \omega_\CC, \omega)$ where $(M, \omega)$ is a K\"ahler manifold and $\omega_\CC$ is a holomorphic symplectic form on $M$. Later, we will also employ the more democratic notation $\omega = \omega^1$, $\omega_\CC = \omega^2 + i \omega^3$ which makes manifest the K\"ahler forms in the three primary complex structures.

\begin{constr}\label{constr:hkquot} If $(M, \omega_\CC, \omega)$ is a hyperkahler manifold with a tri-Hamiltonian action by a Lie group $G$ -- i.e., such that there exist both complex and real moment maps \begin{align*}
\mu_{\mb{C}}\colon &M \to \mf{g}_{\mb{C}}^{\vee} \ , \\ \mu_{\mb{R}}\colon &M \to \mf{g}^{\vee}
\end{align*} 
satisfying analogs of~\eqref{eq:mommap}, then for any $\vec{\xi} = (\xi_{\mb{C}},\xi_{\mb{R}})$ with $\xi_{\mb{C}} \in (\mf{g}_{\mb{C}}^{\vee})^G, \xi_{\mb{R}} \in (\mf{g}^{\vee})^G$, we define the quotient $$M\hq G := (\mu_{\mb{C}}^{-1}(\xi_{\mb{C}}) \cap \mu_{\mb{R}}^{-1}(\xi_{\mb{R}}))/G\,.$$ If this quotient is a manifold (or orbifold), it now naturally carries a hyperkahler structure. By definition, this quotient is the K\"ahler quotient of $\mu_{\mb{C}}^{-1}(\xi_{\mb{C}})$ by $G$ and, once again, if $G_\CC$ acts holomorphically symplectically then one could form a complex formulation and ask for a DUY isomorphism \beq\label{eq:hkDUY} (\mu_{\mb{C}}^{-1}(\xi_{\mb{C}}) \cap \mu_{\mb{R}}^{-1}(\xi_{\mb{R}}))/G \stackrel{\sim}{\to} \mu_{\mb{C}}^{-1}(\xi_{\mb{C}})^{\mathrm{ps}}/G_{\mb{C}}\,.\eeq 
The right side of~\eqref{eq:hkDUY} is naturally a complex manifold (orbifold) in good cases. Since the restriction of the holomorphic symplectic form $\omega_\CC$ to $\mu^{-1}_\CC(\xi_\CC)$ vanishes on vertical vectors and is $G_{\mb{C}}$-invariant, it then descends to yield the structure of a holomorphic symplectic complex manifold on the right side.
\end{constr}

Proposition~\ref{prop:hktime} gives a tri-Hamiltonian action of the Lie group $\mc{G}^{s+1}_{\hat{\Gamma}}$ on the hyperkahler space $\mc{A}^s_{\hat{\Gamma}}$, and so it is tempting to form the hyperkahler quotient as in Construction~\ref{constr:hkquot} above.\footnote{It is similarly tempting to study the K\"ahler quotients of Proposition~\ref{prop:kahtime}. We will do so in future work.} Indeed, this is precisely the main thrust of the current section.

First, however, it is reasonable to calculate the moment map parameter space $((\mf{g}^{\vee})^G)^3 \simeq (\mf{g}^{\vee}_{\mb{C}})^G \oplus (\mf{g}^{\vee})^G$. We turn to this task next, in \S\ref{subsec:FIparam}, before then describing some symmetries of our construction in \S\ref{sec:variants}. We then compute this hyperkahler quotient in the special case of $\vec{\xi} = 0$ in both \S\ref{subsec:flatorb} and~\S\ref{subsec:flatorb2} using different proof strategies. We will then finally state the main conjecture in Conjecture~\ref{conj:mainconj} of \S\ref{subsec:mainconj} about the expected behavior of this family of hyperkahler quotients over the space of parameters. We close with a pair of appendical subsections. \S\ref{subsec:sheaf} discusses the relation of the derived McKay correspondence to a DUY isomorphism for some cases of Conjecture~\ref{conj:mainconj}, while \S\ref{subsec:asymorb} explains the relevant gauge-theoretic problem when the action of $\hat{\Gamma}\acts{Q}$ in~\S\ref{subsec:setup} is affine. We then remark on the topology of the gauge group and its relation to the topology of the ensuing hyperkahler quotient in~\S\ref{subsec:topology}.

\subsection{FI parameter space}\label{subsec:FIparam}

The main task of this subsection is to calculate the space of moment map parameters relevant for the hyperkahler quotient construction of Construction~\ref{constr:hkquot} as applied to the data of Proposition~\ref{prop:hktime}. A priori, this means that we wish to compute the $\G^s_{\hat\Gamma}$-invariant subspace of $(\mathrm{Lie}\ \G^s_{\hat\Gamma})'$. However, in accordance with the discussion at the end of \S\ref{subsec:setup}, and as in \cite{kapustin:impurity}, the parts of this subspace with nonvanishing trace do not affect the resulting hyperkahler quotient: they just add a multiple of a Green's function for $d^+$ (in Coulomb gauge, say) to elements of $\A^s_{\hat\Gamma}$.

So, for $\mc{G} = H^s_{\Gamma}(T^{\vee}_{\Lambda},SU(R_{\Gamma})) / Z_{|\Gamma|}$, we wish to compute the $\mc{G}$-invariant subspace of $(\mathrm{Lie}\ \mc{G})^{\vee}$. Note that in the set-up of~\S\ref{subsec:setup}, this dual Lie algebra is largest when $s = \infty$, and so for the strongest possible statement, we will work throughout with $s = \infty$ in this subsection. We begin in greater generality, replacing $T^{\vee}_{\Lambda}$ by a more general manifold $M$ with a $\Gamma$ action. First, we introduce several definitions:

\begin{defn} Given a finite group $\Gamma$, define the vector space 
\beq\mf{z}_{\Gamma}:=\mathrm{Lie}\,Z(SU(R_{\Gamma})^{\Gamma})\,.\eeq
Slightly more generally, if $\Gamma$ is a finite group with subgroup $\Gamma_p \subset \Gamma$, define the vector space
\beq \mf{z}_{\Gamma_p;\Gamma} := \mathrm{Lie}\,Z(SU(R_{\Gamma})^{\Gamma_p})\,,\eeq
where $\Gamma_p \hookrightarrow \Gamma \hookrightarrow U(R_{\Gamma})$ is considered as a subgroup acting on the normal subgroup $SU(R_{\Gamma})$ by conjugation, $SU(R_{\Gamma})^{\Gamma_p}$ is the fixed subgroup of this conjugation action, and $Z(SU(R_{\Gamma})^{\Gamma_p})$ is the center of said fixed subgroup.

We term $\mf{z}_{\Gamma}$ the \emph{space of (FI) parameters associated to $\Gamma$}, and when $\Gamma_p \hookrightarrow \Gamma$ is the stabilizer group at $p \in M$ for some action $\Gamma\acts{M}$, we abbreviate $\mf{z}_{\Gamma_p;\Gamma}$ as $\mf{z}_p$ and term it the \emph{local space of parameters at $p$} for reasons that should become clear shortly. \end{defn}

\begin{rmk} Following~\cite[(2.5)]{kronheimer:construct}, we may describe these spaces rather concretely. If we use the standard Peter-Weyl type decomposition of the regular representation \begin{align*} R_{\Gamma} &\simeq \bigoplus_{V_i \in\,\mathrm{Irr}\,\Gamma} \mathrm{End}_{\mb{C}}(V_i) \\ &\simeq  \bigoplus_{V_i \in\, \mathrm{Irr}\,\Gamma} V_i \otimes V_i^*,\end{align*} then Schur's lemma immediately implies that \begin{align*}U(R_{\Gamma})^{\Gamma}&\simeq \prod_{V_i\in\,\mathrm{Irr}\,\Gamma}U(V_i^*)\\ \implies Z(U(R_{\Gamma})^{\Gamma})&\simeq\prod_{V_i\in\,\mathrm{Irr}\,\Gamma}U(1)\\ \implies \mf{z}_{\Gamma}&\simeq\Big(\bigoplus_{V_i\in\,\mathrm{Irr}\,\Gamma} \mf{u}(1)\Big)^0\,,\end{align*}where $\mathrm{Irr}\,\Gamma$ denotes the set of irreducible representations of $\Gamma$ and $(-)^0$ denotes the traceless subspace.\end{rmk}

More generally still, we have the following:

\begin{lem} Given $\Gamma_p \hookrightarrow \Gamma$ as above, we have a canonical isomorphism %\ftodo{I want to think a little more about this -- is there really a natural map, esp. after quotienting by the overall $\mf{u}(1)$? does it even matter at all?} 
\beq\mf{z}_{\Gamma_p} \stackrel{\sim}{\to}\mf{z}_{\Gamma_p;\Gamma}\,.
\eeq
\end{lem}
\bp The first step to analyzing $\mf{z}_{\Gamma_p;\Gamma}$ is to decompose $R_{\Gamma}$ as a representation of $\Gamma_p$, whereupon the above line of reasoning again contributes one factor of $\mf{u}(1)$ for every irreducible representation of $\Gamma_p$ that appears. Once again, if $V\in\mathrm{Irr}\,\Gamma_p$ appears $n$ times for any $n > 0$, Schur's lemma ensures that the relevant factor of the centralizer of $\Gamma_p$ in $U(R_\Gamma)$ is $U(n)$ with center $U(1)$. So, the content of the claim above is that every irreducible representation of $\Gamma_p$ appears in $R_{\Gamma}$. Equivalently, we wish to know that every irreducible representation of $\Gamma_p$ appears upon decomposing \emph{some} irreducible representation of $\Gamma$. But this claim is immediate from Frobenius reciprocity.\ep

\begin{rmk} One may interpret the main result of this subsection, the identification of the space of FI parameters, as an extension of the above propositions about finite groups $\Gamma$ and their subgroups $\Gamma_p$ to a much larger class of locally compact groups $\hat{\Gamma}$. After all, the original formulation of~\S\ref{subsec:setup} -- before passing through Fourier duality -- worked directly with $SU(R_{\hat{\Gamma}})^{\hat{\Gamma}}$ for infinite groups $\hat{\Gamma}$. This subsection, hence, may be interpreted as an extension to the class of finite groups extended by abelian lattices that appear in our applications of interest. It may be of interest to use the theory of harmonic analysis on locally compact groups to study this theory in still greater generality. \end{rmk}

Before introducing the main theorem of this section, we introduce one last piece of notation.

\begin{defn}\label{defn:isoaction} Suppose $M$ is a smooth manifold endowed with the action of a finite group $\Gamma$. Given $p \in M$, denote by $\Gamma_p \subset \Gamma$ the stabilizer of $p$, and fine the subset $F \subset M$ by 
\beq F := \{p \in M \bigm| \Gamma_p \ne \{1\}\}\,.\eeq 
We will say that the action of $\Gamma$ on $M$ is \emph{isolated} if $F \subset M$ is discrete. We will then typically write $F = \{p_i\}$ with corresponding stabilizer groups $\Gamma_i \subset \Gamma$. \end{defn}

\begin{rmk}\label{rmk:isoaction} Note that any linear action $\hat{\Gamma}\acts{\mb{R}^4}$ with $\Gamma \subset SU(2)$ automatically yields an isolated action of $\Gamma$ on $T^{\vee}_{\Lambda}$.\end{rmk}

\begin{thrm} \label{thm:FI} Suppose $M$ is a smooth manifold with an isolated action by some finite group $\Gamma$. Then the subspace of equivariant $\mf{su}(R_{\Gamma})$-valued distributions on $M$ invariant under the coadjoint action by the group $\mc{G} := C^{\infty}_{c,\Gamma}(M,U(R_{\Gamma}))$ of equivariant smooth, compactly supported\footnote{Note that for the group-valued functions we consider here, by ``compact support'' we mean identically equal to $\mathrm{Id} \in U(R_{\Gamma})$ outside some compact set.} functions is \begin{align} \Big( \mc{D}'_{\Gamma}(M,\mf{su}(R_{\Gamma})) \Big)^{C^{\infty}_{c,\Gamma}(M,U(R_{\Gamma}))} &\simeq \Big( \prod_i \mf{z}_{p_i} \Big)^{\Gamma} \\ \sum_i \xi_i \delta(p_i) &\mapsfrom (\xi_i)\,.\end{align} \end{thrm} %\ftodo{the right side above may require some parsing depending how we eventually decide to give it} Here's a computation of FI parameter space. I'm not sure exactly how we'll give the result -- certainly not case by case, although maybe we'll refer to [Plethora] for a rundown in the K3 case -- most likely by local models after recalling some details of McKay from [Kronheimer]. \end{thm} 

\bp We show this theorem in a number of steps. First, we claim that any distribution preserved under the coadjoint action as above must be supported only at points with a nontrivial stabilizer group. Indeed, suppose $T$ is some invariant distribution \emph{not} supported only at such points; i.e., suppose $q \in \mathrm{supp}\,T$ and that $q$ has trivial stabilizer in $\Gamma$. By assumption, we may find a neighborhood $q \in U \subset M$ such that $(\Gamma \setminus \{1\})(U) \cap U = \emptyset$ and that there exists some test function $\phi \in C^{\infty}_c(U, \mf{su}(R_{\Gamma}))$ with $\langle T, \phi \rangle \ne 0$. (Note that we are throughout using the Killing form on $\mf{su}(R_{\Gamma})$, i.e. the trace pairing, in order to ignore the distinction between $\mf{su}(R_\Gamma)$ and its dual.) Now, by assumption, we have that \begin{align} \forall g \in C^{\infty}_{c,\Gamma}(M, U(R_{\Gamma})), \langle T, g\phi \rangle &= \langle T, \phi \rangle \nonumber \\ \Longleftrightarrow \forall g \in C^{\infty}(U, U(R_{\Gamma})), \langle T, g \phi \rangle &= \langle T, \phi \rangle, \label{eq:groupuseful}\end{align} the equivalence of the two rows following as (i) only $g$ supported in $U$ are relevant and (ii) any $g$ supported in $U$ may be canonically extended to an equivariant $g$ supported on $\Gamma \cdot U$. (Rather, shrink $U$ slightly so that $U \subset U'$, still with $(\Gamma \setminus \{1\})(U') \cap U' = \emptyset$, and use a partition of unity and polar projection to the unitary group to extend group-valued functions on $U$ to compactly-supported group-valued functions on $U'$ to equivariant compactly-supported group-valued functions on $\Gamma \cdot U'$.) But now integrate the second equality in~\eqref{eq:groupuseful} over all constant $g$ to find $$0 \ne \langle T, \phi \rangle = \int_{U(R_{\Gamma})} \langle T, g\phi \rangle dg$$ for the Haar measure on $U(R_{\Gamma})$. But, given that $\phi$ was valued in $\mf{su}(R_{\Gamma})$, it follows that integrating $g\phi$ returns the identically vanishing test function, which is a contradiction. Hence $q\not\in\mathrm{supp}\,T$.

As we have now shown that the support of any invariant distribution is supported at the set of points with nontrivial stabilizers, which we are hypothesizing is some discrete set of isolated points in $M$, it is clear the space of invariant distributions is the product over spaces of local parameters for each $p_i$, up to an overall constraint relating the parameters associated to points in the same $\Gamma$ orbit. So, we now turn to classifying those invariant distributions supported at a point $p$ with stabilizer group $\Gamma_p \subset \Gamma$. 

Any distribution supported at a point $p$ is a linear combination of the delta distribution $\delta(p)$ and its derivatives. So, suppose $$T = \sum_{\alpha}v^{\alpha}\partial_{\alpha}\delta(p)$$ is an invariant distribution, where $\alpha$ runs over a finite list of multi-indices and $v^{\alpha} \in \mf{su}(R_{\Gamma})$ for all such $\alpha$. We may act by $g \in \mc{G}$ which are some constant value in $U(R_{\Gamma})^{\Gamma_p}$ near $p$, which suffices to show that $v^{\alpha}\in\mf{z}_p$ for all $\alpha$. Next, suppose that $T$ has order $m > 0$. Denoting $V = T_pM$ as a $\Gamma_p$-representation and $\mf{s} = \mf{su}(R_{\Gamma})$, we have that the order $m$ part of $T$ may be encoded in some $$v^{(m)} \in \mathrm{Sym}^m V \otimes \mf{s}\,.$$ Act now by means of $g \in \mc{G}$ that are the identity at the origin and have vanishing derivatives of all orders up to $m-1$. The $m^{\text{th}}$ derivatives of these group elements may be encoded as elements of $\mathrm{Hom}_{\Gamma_p}(\mathrm{Sym}^mV, \mf{s})$; conversely, any such $\Gamma_p$-equivariant linear map (or equivariant $m^{\text{th}}$ order polynomial on $V$) may be realized as the $m^{\text{th}}$ order part of the Taylor expansion of some such $g \in \mc{G}$. Acting with such group elements yields that $v^{(m)}$ is in the kernel of the natural bilinear form $$\Big( \mathrm{Sym}^mV^*\otimes\mf{s}\Big)^{\Gamma_p} \times \Big(\mathrm{Sym}^mV\otimes\mf{s}\Big) \stackrel{\langle -, - \rangle \otimes [-,-]}{\longrightarrow} \mf{s}\,.$$ 
Recall, however, that we already know all the components $v^{\alpha} \in \mf{z}_p \subset \mf{s}^{\Gamma_p}$ for all $\alpha$, so $v^{(m)}$ is in fact in the kernel against the pairing with the full $\Big(\mathrm{Sym}^mV^*\otimes\mf{s}\Big)$. As $\mf{s}$ is semisimple, this kernel is trivial. Hence $T$ is in fact of order $0$, as desired. \ep

\begin{defn} 
Given $\hat{\Gamma}\acts{Q}$ for $Q \simeq \mb{C}^2$ and $\Gamma \hookrightarrow Sp(1)$ nontrivial as above, let \emph{the FI parameter space} $\mf{z}$ denote the $\Gamma$-invariant part of the sum of all the local FI parameter spaces, i.e. 
\beq\mf{z} := \Big(\bigoplus_{p \in F \subset T^{\vee}_{\Lambda}} \mf{z}_p\Big)^{\Gamma}\,.\eeq As noted in the summation index, the space $\mf{z}_p$ is nontrivial only for $p \in F$ with nontrivial stabilizer. In particular, the sum above is finite. \end{defn}

\begin{rmk} Note that one could rewrite $\mf{z}$ above by dropping the $\Gamma$-equivariance and instead just taking a sum over representatives of $\Gamma$-orbits (of nontrivially-stabilized points). In other words, we have $$\mf{z} \simeq \bigoplus_{p_i \in F/\Gamma} \mf{z}_p\,.$$ \end{rmk}

\subsection{Gauge group variants} \label{sec:variants}

Let us return to the general setup of a finite group $\Gamma$ with an isolated action on a smooth $r$-manifold $M$; consider once more the Lie group $\tilde{\mc{G}}^s = H^s_{c, \Gamma}(M,U(R_{\Gamma}))$ for any $r/2 < s \le \infty$, where we recall that the compact-support condition demands that the group element be identically $\mathrm{Id}$ outside some compact set. It will be of interest to calculate the normalizer of $\tilde{\mc{G}}^s$ within a larger group. Specifically, consider the larger group (without the equivariance conditions) $H^s(M,U(R_{\Gamma}))$, and denote by $N(\tilde{\mc{G}}^s)$ the normalizer of $\tilde{\mc{G}}^s$ within $H^s(M,U(R_{\Gamma}))$. We begin by establishing some notation for a simple case of the above.

\begin{lem}\label{lem:ptwisenorm} 
Suppose $\Gamma$ is a finite group with subgroup $\Gamma_p \hookrightarrow \Gamma$. Let $R_j$ denote the irreducible representations of $\Gamma_p$, and suppose $R_{\Gamma}$ decomposes, as a $\Gamma_p$-representation, as $$R_{\Gamma} \simeq \bigoplus_j V_j \otimes R_j\,.$$ Denote $d_j = \dim V_j$, and for each $d \in \mb{N}$, let $R(d) \subset \mathrm{Irr}(\Gamma_p)$ denote the set of representations $R_j$ such that $d_j = d$. Then we have a split short exact sequence \beq 1 \to U(R_{\Gamma})^{\Gamma_p} \to N_{U(R_{\Gamma})}(U(R_{\Gamma})^{\Gamma_p}) \to \prod_d S_{|R(d)|} \to 1\,.\eeq  
\end{lem}

\begin{defn} In light of the above, we denote by $W(\Gamma_p; \Gamma)$ the group $\prod_d S_{|R(d)|}$ when in the situation of the proposition above. This group is so named in analogy with Weyl groups; note also that it carries a natural action on $\mf{z}_{\Gamma_p;\Gamma}$ which we think of as analogous to a Weyl action on a Cartan. Frequently, we abbreviate the above as $W_p\acts{\mf{z}_p}$. \end{defn}

\bp[Proof of Lemma~\ref{lem:ptwisenorm}] Recall that we previously noted that $V_j$ is nonzero for all $j$, by Frobenius reciprocity, and that $$U(R_{\Gamma})^{\Gamma_p} \simeq \prod_j U(V_j)\,,$$ by Schur's lemma. Now, we claim that the normalizer of the above in $U(R_{\Gamma})$ is precisely $$N_{U(R_{\Gamma})}(U(R_{\Gamma})^{\Gamma_p}) \simeq \Big( \prod_j U(V_j) \Big) \rtimes W(\Gamma_p;\Gamma)\,,$$ where $\prod_d S_{R(d)}$ is realized as block-permutation matrices within $U(R_{\Gamma})$; this claim is manifest upon considering where the subgroups may be mapped under the conjugation action. \ep

%, the normalizer \beq N_{U(R_{\Gamma})}(U(R_{\Gamma})^{\Gamma_p}) \simeq 0\,.\eeq \end{lem} \bp We noted previously that if $R_j$ denote the irreducible representations of $\Gamma_p$ and if $R_{\Gamma}$ decomposes, as a $\Gamma_p$-representation, as $$R_{\Gamma} \simeq \bigoplus_j V_j \otimes R_j\,,$$ then all the $V_j$ are nonzero by Frobenius reciprocity and $U(R_{\Gamma})^{\Gamma_p} \simeq \prod_j U(V_j)$ by Schur's lemma.  \ep

\begin{prop}\label{prop:localweyldef} Given $\Gamma\acts{M}$ isolated, and with $\tilde{\mc{G}}^s, N(\tilde{\mc{G}}^s)$ defined as above, we have a short exact sequence \beq 1 \to \tilde{\mc{G}}^s \to N(\tilde{\mc{G}}^s) \to \Big(\prod_{p \in F} W_{p}\Big)^{\Gamma} \to 1\,.\eeq\end{prop}

%We need to explain the notation $\Big(\prod_{p \in F} W_{p}\Big)^{\Gamma}$ in the above; here, $\Gamma$ naturally acts on the indices $i$, as if $p_i$ has stabilizer group $\Gamma_i \subset \Gamma$, then the $\Gamma$-orbit of $p_i$ produces a set of related indices with nontrivial stabilizers; for all these indices, the local Weyl groups $W_{p_i}$ are canonically identified. The $\Gamma$-invariant subgroup above, hence, is those collections of local Weyl transformations that are constant within $\Gamma$-orbits.

\bp A bump function argument immediately shows that for any $p \in M$, the evaluation map $\tilde{\mc{G}}^s \stackrel{\mathrm{ev}_p}{\twoheadrightarrow} U(R_{\Gamma})^{\Gamma_p}$ is surjective. Hence, we have an evaluation map $N(\tilde{\mc{G}}^s) \stackrel{\mathrm{ev}_p}{\to} N_{U(R_{\Gamma})}(U(R_{\Gamma})^{\Gamma_p})$. Once again, a bump function argument shows that this map is surjective: given an element of the target, extend it smoothly and $\Gamma_p$-equivariantly to a neighborhood of $p$ in such a way that it becomes identically equal to the identity outside a compact subset of said neighborhood. Extend by $\Gamma$-equivariance to produce an element of $N(\tilde{\mc{G}}^s)$. 

By dint of the above, we have a map $N(\tilde{\mc{G}}^s) \to \prod_{p \in F} W_{p}$; the above reasoning using bump functions (and the fact that the $p \in F$ are isolated!) naturally yields a surjection $N(\tilde{\mc{G}}^s) \twoheadrightarrow \Big( \prod_{p \in F} W_{p} \Big)^{\Gamma}$. It remains to note that an element in the kernel of this map must be $\Gamma$-equivariant, which once more follows from bump function arguments. \ep

\begin{defn} We term $N(\tilde{\mc{G}}^s)$ above the \emph{extended} gauge group of $\tilde{\mc{G}}^s$, and we define the quotient $W(\tilde\G^s) := N(\tilde\G^s) / \tilde\G^s \simeq (\prod_{p \in F} W_{p})^{\Gamma}$, which we call the \emph{Weyl} group of $\tilde\G^s$. One may similarly define the extended complexified gauge group $N(\tilde{\mc{G}}^s_{\mb{C}})$; the quotient $N(\tilde{\mc{G}}^s_{\mb{C}})/\tilde{\mc{G}}^s_{\mb{C}}$ is still $W(\tilde\G^s)$ via the natural isomorphism $N(\tilde{\mc{G}}^s)/\tilde{\mc{G}}^s \stackrel{\sim}{\to} N(\tilde{\mc{G}}^s_{\mb{C}})/\tilde{\mc{G}}^s_{\mb{C}}$. \end{defn}

\begin{rmk}\label{rmk:WeylPU} There are other ways to define an extended gauge group, and these other presentations often allow for nicer representatives of the $\tilde\G^s$-cosets. For example, one could study a normalizer of $H^{s,p\ell}_{c,\Gamma}(M,PU(R_{\Gamma}))$. Under mild hypotheses (and certainly for the cases of $M = T^{\vee}_{\Lambda}$), the coset space again identifies with the same $W(\tilde\G^s)$ Weyl group. For concreteness, consider the following example: consider $M = S^1 \simeq \mb{R} / 2\pi\mb{Z}$ with the action of $\Gamma = Z_2$ by $\theta \mapsto -\theta$. For simplicity, choose a basis of $R_{\Gamma}$ such that the $Z_2$-action is given by conjugation by $\sigma_z = \begin{pmatrix} 1 & 0 \\ 0 & -1 \end{pmatrix}$. Consider the function $g \in C^{\infty}(S^1, PU(R_{\Gamma}))$ given by 
\be g(\theta) = \begin{pmatrix} \cos(\theta/2) & -\sin(\theta/2) \\ \sin(\theta/2) & \cos(\theta/2)\end{pmatrix} \, . \label{eq:multiWeyl} \ee
It is in the normalizer of $C^{\infty}_{Z_2}(S^1,PU(R_{\Gamma}))$ but not \emph{a priori} in $N(\tilde{\mc{G}}^s)$ as we defined it above, as it has no lift to $C^{\infty}(S^1,U(R_{\Gamma}))$. However, it does lift to an element of $N(\tilde{\mc{G}}^s)$ up to multiplication by an element of the usual gauge group $\tilde{\mc{G}}^s$. We could formulate and state a full analog of Proposition~\ref{prop:Gdefn} for $N(\tilde{\mc{G}}^s)$ such that one presentation of $N(\tilde\G^s)$ would include this $g$, but we opt not to belabor this point further now. In future sections, we will freely refer to multi-valued gauge transformations such as~\eqref{eq:multiWeyl} as extended gauge transformations. (We will typically restrict to a fundamental domain for the $Z_2$ action, i.e. an interval $I$, at which point such $g$ will not look multi-valued, but they will still be in the normalizer of the gauge group inside of $C^\infty(I,U(R_\Gamma))$ instead of the gauge group itself.)
\end{rmk}

\begin{rmk}  Similarly, we may easily check by Proposition~\ref{prop:localweyldef} that for any $s' \ge s$, the natural morphism $W(\tilde{\mc{G}}^{s'}) \to W(\tilde{\mc{G}}^s)$ is an isomorphism, and so we will frequently simply write $W(\tilde{\mc{G}})$ without any further notational decoration. In fact, one could again define $\mc{G}^s := \tilde{\mc{G}}^s / U(1)$ and express $W(\tilde{\mc{G}})$ as a quotient by $\mc{G}^s$ of a similar normalizer of $\mc{G}^s$ (with suitable pointwise lifting conditions), and so we may frequently also use the notation $W(\mc{G})$ in place of $W(\tilde{\mc{G}})$.
\end{rmk}

We note another few actions. First, in full generality, we have the following:

\begin{prop}\label{prop:Weylact} The local Weyl actions $W_p\acts{\mf{z}_p}$ assemble to yield a global Weyl action $W(\tilde\G^s)\acts{\mf{z}}$. This action agrees with the canonical $N(\tilde{\mc{G}}^s)$-action on $(\mathrm{Lie}\ \tilde\G^s)^{\vee}$.\end{prop}

% In full generality, for any Lie group $G$, one has an action $\mathrm{Out}(G)\acts{((\mathrm{Lie}\ G)^{\vee})^G}$. I'm not sure in the present situation that $W(G) \hookrightarrow \mathrm{Out}(G)$ is an iso, but besides that, this fully general action agrees with that of Proposition~\ref{prop:Weylact} above.

We now specialize to the particular case of interest treated in~\S\ref{subsec:setup}.  We first note the following:
\begin{prop}\label{prop:Nact} 
The action $\tilde{\mc{G}}^{s+1}_{\hat{\Gamma}}\acts{\mc{A}^s_{\hat{\Gamma}}}$ naturally extends to an action $N(\tilde{\mc{G}}^{s+1}_{\hat{\Gamma}})\acts{\mc{A}^s_{\hat{\Gamma}}}$. The moment maps are $N(\tilde{\mc{G}}^{s+1}_{\hat{\Gamma}})$-equivariant.
\end{prop}

It is now more clear the role that $W(\mc{G})$ will play for the family of hyperkahler quotients which we would like to study. Namely, given an element $g \in N(\tilde{\mc{G}})$, one may attempt to act on the hyperkahler quotient $\mc{A}\hq\mc{G}$. Only the class of $g$ in the quotient group $W(\mc{G})$ is relevant, as the hyperkahler quotient has already quotiented by the action of $\mc{G}$. But, note that the action of $g$ should yield a morphism from $\mc{A}\hq\mc{G}$ not to itself, but to $\mc{A}\mathord{/ \!\! / \! \! /_{g(\xi)}}\mc{G}$. In other words, the total space of the family of hyperkahler manifolds should admit this $W(\mc{G})$-action equivariantly over the natural $W(\mc{G})$-action on the parameter space $\mf{z}^3$ of this family. Of course, there exist special loci within the parameter space $\mf{z}^3$ which are fixed by some subgroup of $W(\mc{G})$, and for these special parameter values, we will have actions by these subgroups. 

%\begin{defn}\label{defn:generic} We term $\vec{\xi} \in \mf{z}^3$ \emph{generic} if it has trivial stabilizer under the $W(\mc{G})$-action. \end{defn}

% Be careful. We have not yet extended to $W$ possibly being infinite and matching, e.g., the full $O(3,19,Z)$ of K3s. If $\vec{\xi}$ is special in the sense above, we certainly expect to have a singularity. But it may be generic above and still be singular. It is likely safer to not give this definition now to avoid saying something false or misleading in the `biholomorphic isomorphism' part of the conjecture below.

Note that the above discussion is very much a partial story at best. There is every reason to expect the natural group of symmetries to extend beyond the $W(\mc{G})$ that we defined above. For example, even in the original Kronheimer case of $r = 0$, if we take $\Gamma$ to be the binary tetrahedral group of order $24$, the group $W(\mc{G})$ we construct above is simply $S_3 \times S_3$. By contrast, the family $\mc{A}\hq\mc{G}$ of hyperkahler manifolds has a full group of symmetries given by $W(E_6)$, the Weyl group of $E_6$, as explained in~\cite[\S4]{kronheimer:construct}. This group is $1440$ times larger than its $S_3 \times S_3$ subgroup. In fact, one may ask for still more, such as a full extension of the McKay correspondence to the $r > 0$ cases: in particular, for each case of $\hat{\Gamma}\acts{Q}$ above, one could ask for an \emph{auxiliary} Lie algebra $L$, generalizing the example of $E_6$ above, such that the Cartan of $L$ identifies with $\mf{z}$ and the Weyl group $W(L)$ identifies with the full monodromy group of the (semi)universal deformation of $Q/\hat{\Gamma}$, again as in~\cite[\S4]{kronheimer:construct}. For the cases $r > 0$, this monodromy group will now be infinite and so $L$ should be an affine Kac-Moody, or still more general Lie algebra. See \cite{zwiebach:algebras} for work in this direction. These considerations, while thoroughly interesting, will not be pursued here; we will content ourselves with the relatively small but still useful $W(\mc{G})$ defined above.

\vspace{.2cm}

In any case, the above discussion illustrates the utility of considering more general variants of the usual $\mc{G}$-action on $\mc{A}$. To close this section, we will find it useful to define a still-larger variant as a `catch-all' term.

\begin{defn}\label{defn:for} 
Given $\hat{\Gamma}\acts{Q}$ as above, define the \emph{formal} (complexified) gauge group as \begin{align*}%\tilde{\mc{G}}_{\hat{\Gamma}}^{\mathrm{for}} &:=  C^0_{\Gamma}(T^{\vee}_{\Lambda}\setminus\{p_i\}, GL_{\mb{C}}(R_{\Gamma})) \\
\mc{G}_{\hat{\Gamma}}^{\mathrm{for}} &:=  C^0_{\Gamma}(T^{\vee}_{\Lambda}\setminus F, PGL_{\mb{C}}(R_{\Gamma}))\,.\end{align*} Here, $C^0$ denotes the space of continuous maps.
\end{defn}

In other words, the group of formal gauge transformations is the group of complexified gauge transformations but with no regularity assumption about any of the $p \in F$. Note that $N(\mc{G}^s_{\hat{\Gamma},\mb{C}}) \hookrightarrow \mc{G}^{\mathrm{for}}_{\hat{\Gamma}}$: 

\begin{prop} If $g \in H^s(T^{\vee}_{\Lambda},GL_{\mb{C}}(R_{\Gamma}))$ normalizes $\tilde{\mc{G}}^s_{\hat{\Gamma},\mb{C}}$, then $g$ is $\Gamma$-equivariant up to scalars. \end{prop}
\bp 
Given $g\in H^s(T^{\vee}_{\Lambda},GL_{\mb{C}}(R_{\Gamma}))$ and $\gamma \in \Gamma$, denote by $\gamma(g) \in H^s(T^{\vee}_{\Lambda},GL_{\mb{C}}(R_{\Gamma}))$ the map $\gamma(g)(p) = j(\gamma)^{-1} g(\gamma p) j(\gamma)$, where by abuse of notation we again use $j\colon \Gamma \to GL_{\mb{C}}(R_{\Gamma})$ to denote the natural permutation matrices on the regular representation. We have that $(g, \gamma(g)) \in (H^s(T^\vee_\Lambda, GL_\CC(R_\Gamma)))^2$ normalizes the diagonal subgroup of $(\tilde\G^s_{\hat\Gamma,\CC})^2$. Equivalently, $\gamma(g) g^{-1}$ commutes with $\widetilde{\mc{G}}^s_{\hat{\Gamma},\mb{C}}$. But then a bump function argument shows that $\gamma(g) g^{-1}$ is everywhere on $T^\vee_\Lambda\setminus F$ valued in the central $\mb{C}^{\times}$. Hence, by continuity, the same must be true on all of $T^{\vee}_{\Lambda}$; i.e., the image of $g$ in $H^s(T^{\vee}_{\Lambda}, PGL_{\mb{C}}(R_{\Gamma}))$ is $\Gamma$-equivariant.
\ep

% Proof: if h_1, h_2 \in PU(n) are NOT conjugate by some fixed g \in U(n), then the action of (h_1, h_2) by conjugation does not preserve the U(n) subgroup of U(n) \times U(n) of the form (u, gug^{-1})
%or, probably a better way to phrase it is the contrapositive:
%if (h_1, h_2) \in PU(n)^2 preserves the subgroup (u, gug^{-1}), then h_2 = g h_1 g^{-1}

\subsection{The flat orbifold quotient} \label{subsec:flatorb}

We now show that the hyperkahler quotient of Proposition~\ref{prop:hktime}, in the special case that $Q \simeq \mb{R}^4$ and $\vec{\xi} = 0$, returns the flat hyperkahler orbifold $\mb{R}^4/\hat{\Gamma} \simeq (T^r \times \mb{R}^{4-r})/\Gamma$. We will give two proofs of this result -- one in the present section and one in the subsequent one. That of the present section will be a faithful upgrade of \cite[Lemma 3.1]{kronheimer:construct} to infinite dimensions, but there will be interesting complications due to the existence of multiple (in fact, infinitely many) points with nontrivial stabilizers in $\hat\Gamma$. In the next section, we will explain another strategy which involves first reducing to a finite-dimensional problem. In the terminology of Constructions~\ref{constr:realKah} and \ref{constr:cplxKah}, we will focus on the real formulation and then derive the result in the complex formulation as a corollary.

\begin{thrm} \label{thm:noFI} In the case above of $\hat{\Gamma}\acts{Q}$ with $Q \simeq \mb{R}^4$ and $\Gamma \hookrightarrow Sp(Q)$, the hyperkahler quotient $\mc{A}^s_{\hat{\Gamma}}\hqnoxi \mc{G}^{s+1}_{\hat{\Gamma}}$ may be identified with the flat orbifold $Q/\hat{\Gamma} \simeq (T_{\Lambda}\times\Lambda_{\mb{R}}^{\perp}) / \Gamma$ for any $r/2 < s \le \infty$.
\end{thrm} 

\begin{rmk}\label{rmk:set} The identification $$\mc{A}^s_{\hat{\Gamma}}\hqnoxi\mc{G}^{s+1}_{\hat{\Gamma}} \simeq (T_{\Lambda}\times\Lambda_{\mb{R}}^{\perp})/\Gamma$$ could take place in several categories. Most naively, we could simply mean an isomorphism of sets. On the other extreme, we could attempt an isomorphism of hyperkahler orbifolds. Indeed, the right side visibly carries the structure of a (flat) hyperkahler orbifold. Since this bijection is obtained by finding a slice for the $\G^{s+1}_{\hat\Gamma}$-action, it immediately yields a smooth structure on the stabilizer-free locus on the left side, and the hyperkahler structure on $\A^s_{\hat\Gamma}$ then descends to one on this manifold. That these hyperkahler structures on the left and right are compatible is a \emph{verbatim} repetition of~\cite[Corollary 3.2]{kronheimer:construct}. We will hence show below the isomorphism of sets.

%So, too, does the left side for any $s < \infty$ by means of the Banach manifold structure on $\mc{A}^s_{\hat{\Gamma}}$ and use of the Banach manifold implicit function theorem to create Banach manifold (or orbifold, rather) slices to the $\mc{G}^{s+1}_{\hat{\Gamma}}$-action. We refer to, e.g.,~\cite[\S 4]{DK} and \S\ref{sec:hkq} for details on this procedure.

Note, however, that the above isomorphism does \emph{not} preserve stabilizer groups and so is not even an isomorphism of groupoids, let alone of (Artin-type) orbifolds.\footnote{i.e., an orbifold allowing positive-dimensional compact stabilizer groups} That the stabilizer groups are different is explained further in Remark~\ref{rmk:stab}. %Rather, here we simply mean an isomorphism of sets, or of hyperkahler manifolds away from points with nontrivial stabilizer.
\end{rmk}

We begin our proof of this theorem by studying some consequences of the vanishing of the moment maps. Recall that the moment maps applied to $A$ are all contractions of the two-form $\mathscr{F}(A)$, where $\mathscr{F}(A) \in H^{s-1}_{\Gamma}(T^{\vee}_{\Lambda},\mf{u}(R_{\Gamma})\underset{\mb{R}}{\otimes}\Wedge^2Q)$. Our first lemma will show that if all these contractions of $\mathscr{F}(A)$ vanish, then in fact $\mathscr{F}(A)$ vanishes identically. Explicitly, this flatness statement $\mathscr{F}(A) \equiv 0$ will imply that if we write $A = (\nabla,\phi) \in (\mc{A}^s_{\hat{\Gamma}})' \times (\mc{A}^s_{\hat{\Gamma}})''$, then $\nabla$ is a flat connection, $\phi$ is a flat section with respect to $\nabla$, and $[\phi,\phi]=0$.

We recapitulate some notation before stating this vanishing lemma. Identify $Q$ with the standard copy of the quaternions $\mb{R}\langle 1, i, j, k \rangle$, and let $$\Wedge^2 Q \simeq (\Wedge^2Q)^+\oplus(\Wedge^2Q)^-$$ be the standard decomposition of $\Wedge_{\mb{R}}^2Q$ into self-dual and anti-self-dual subspaces for the pairing
\beq\label{eq:Qbil}\Wedge^2Q\times\Wedge^2Q\to\Wedge^4Q\simeq\mb{R}\eeq
induced by the volume form on $Q$; i.e., \begin{align*} (\Wedge^2Q)^+ &= \mb{R} \langle 1 \wedge i + j \wedge k, 1 \wedge j - i \wedge k, 1 \wedge k + i \wedge j \rangle \\ (\Wedge^2Q)^- &= \mb{R} \langle 1 \wedge i - j \wedge k, 1 \wedge j + i \wedge k, 1 \wedge k - i \wedge j \rangle\,.\end{align*}
Writing $\mf{g}^s := H^s(T^\vee_\Lambda, \mf{u}(R_\Gamma))$,
\beq\label{eq:split}
\mf{g}^s\otimes\Wedge^2Q\simeq\Big(\mf{g}^s\otimes(\Wedge^2Q)^+\Big)\oplus\Big(\mf{g}^s\otimes(\Wedge^2Q)^-\Big)\eeq
is similarly a splitting of the indefinite quadratic form on $\mf{g}^s\otimes\Wedge^2Q$ into positive- and negative-definite subspaces. We can now explain the idea of the vanishing lemma: for `topological' reasons,\footnote{In the case $r = 4$, this statement is simply the vanishing of the second Chern class of the bundle on which $A$ is a connection. In the current context, we have restricted to only studying connections on topologically trivial bundles, \emph{ergo} the vanishing.} the norm-squared of the self-dual and anti-self-dual parts of $\mathscr{F}(A)$ agree. Since the vanishing of the moment maps is the vanishing of the self-dual part of $\mathscr{F}(A)$, $\mathscr{F}(A) \equiv 0$ follows.

The following lemma may be shown by direct computation:
\begin{lem}\label{lem:comm} Given $A \in \mc{A}^s$,
\beq \mathscr{F}(A) \cdot \mathscr{F}(A) = 0\,.\eeq
Here, the product $\cdot$ is the tensor product of the $L^2$ inner product on $\mf{g}^s$ and the pairing \eqref{eq:Qbil} on $\Wedge^2 Q$.
\end{lem}
\begin{cor}\label{cor:comm} If $A \in \mc{A}_{\hat\Gamma}^s$ satisfies $\mathscr{F}(A)^+=0$, then $\mathscr{F}(A)=0$. That is, if $A$ is regarded as a pair $(\nabla, \phi)\in \mathrm{Conn}^s_{\Gamma}(T_{\Lambda}^{\vee},R_{\Gamma}) \times H^s_{\Gamma}(T_{\Lambda}^{\vee}, \mf{u}(R_{\Gamma}) \otimes \Lambda_{\mb{R}}^{\perp})$, and if $A$ satisfies the moment map equations $\mathscr{F}(A)^+ = 0$, then 
\begin{itemize}
\item $\nabla$ is a flat connection on $T^\vee_\Lambda$, 
\item $\phi$ is a flat section for the flat connection $\nabla$, and \item $[\phi,\phi] \in H^s_\Gamma(T^\vee_\Lambda,\mf{u}(R_\Gamma) \otimes\Wedge^2 \Lambda_\RR^\perp)$ vanishes.
\end{itemize}
\end{cor}
Here $\mathscr{F}(A)^+$ denotes the projection of $\mathscr{F}(A)$ under $\Wedge^2Q \to (\Wedge^2Q)^+$, and the proof of Corollary~\ref{cor:comm} is immediate from Lemma~\ref{lem:comm} and the splitting~\eqref{eq:split} into positive- and negative-definite subspaces.

There are now two natural approaches to establishing Theorem~\ref{thm:noFI}. In the next section, we will illustrate an approach more familiar to finite-dimensional geometers and topologists, while in this section, we will proceed via infinite-dimensional functional analysis to the more immediate generalization of the argument of~\cite{kronheimer:construct}. %At this point, the obvious approach perhaps is to directly generalize the familiar fact that flat Hermitian connections on trivial $U(n)$-bundles over tori are gauge-equivalent to constant diagonal connections (i.e., connections $d+a$ with $a$ a constant diagonal 1-form) to equivariant connections. This will be the approach of the next section. Here, we will instead pursue an amusing functional analytic approach. Indirectly, this will also prove this result about flat equivariant connections.

For the sake of the following discussion, we work in the general setting of \S\ref{subsec:setup} (where our $\hat\Gamma$ action is still linear) -- i.e., we do not restrict $Q$ to be 4-dimensional and to have a hyperkahler structure. We then have
\begin{proposition} \label{prop:orbFlat}
For $s$ satisfying $r/2<s\le \infty$, there is a canonical bijection
\be \{A=(\nabla,\phi) \in \A^s_{\hat\Gamma} \bigm| \mathscr{F}(A)=0 \} / \G^{s+1}_{\hat\Gamma} \simeq Q/\hat\Gamma \ . \label{eq:flatOrbBij} \ee
The equivalence class on the left corresponding to the $\hat\Gamma$-equivalence class of $q\in Q$ is represented by a constant diagonal connection $d+a$, i.e., a connection such that $a$ is constant and diagonal, and a constant diagonal $\phi$ such that the top left entry of $-i(a,\phi)$, regarded as a point of $Q$, is $q$, and the remaining entries of these matrices are the orbit of $q$ under the $\Gamma\hookrightarrow O(Q)$ action.
\end{proposition}

In the remainder of this section, we will prove Proposition~\ref{prop:orbFlat} under the assumption that the nontrivially-stabilized locus for the $\hat{\Gamma}\acts{Q}$ action is \emph{isolated}. As we will see in the next section, this assumption is not necessary. %We also assume that $s=\infty$; again, this is merely for convenience. Indeed, as we will illustrate later in this paper, using the implicit function theorem one may prove that there exists a $g\in \G^{s+1}_{\hat\Gamma}$ such that $g\cdot \nabla$ is in Coulomb gauge with respect to a smooth connection, and overdetermined elliptic regularity applied to this condition together with $\F(g\cdot A)=0$ implies that $g\cdot A$ is smooth.}

In order to prove this result, we pick a basis for $Q$ and observe that this associates $\delta := \dim Q$ essentially skew-self-adjoint (possibly unbounded and only densely defined with domain $H^1(T^\vee_\Lambda, R_\Gamma)$)\footnote{These operators need not be skew-self-adjoint, as for $r>1$ we can have $\tilde\Phi_i \psi \in L^2(T^\vee_\Lambda, R_\Gamma)$ but $\psi\not\in H^1(T^\vee_\Lambda, R_\Gamma)$, roughly because $\psi$ is differentiable in one direction but not another.} operators $\tilde\Phi_i$ ($i=1,\ldots,\delta$) on $L^2(T^\vee_\Lambda, R_\Gamma)$ to each $A=(\nabla,\phi) \in \A^s_{\hat\Gamma}$, and that if $\mathscr{F}(A)=0$ then these operators pairwise commute when regarded as operators on $H^{-\infty}(T^\vee_\Lambda, R_\Gamma)$. Actually, as suggested by the above proposition, we prefer to work with the essentially self-adjoint relatives $\Phi_i = -i \tilde\Phi_i$. We then claim that the proposition follows from the following:

\begin{proposition} \label{prop:orbitbasisHilb}
For $s$ satisfying $r/2<s\le \infty$, given $A=(\nabla,\phi)\in \A^s_{\hat\Gamma}$ with $\mathscr{F}(A)=0$ and the associated operators $\Phi_i$, there exists a basis of $L^2(T^\vee_\Lambda,R_\Gamma)$ which is 
\begin{enumerate}[(i)]
\item orthonormal, 
\item an eigenbasis for the commuting operators $\Phi_i$,
\item a free $\hat\Gamma$-orbit, and
\item contained within $H^{s+1}(T^{\vee}_{\Lambda},R_{\Gamma})$.
\end{enumerate}
\end{proposition}

\begin{rmk}\label{rmk:autreg}
Of course, the above may be rephrased as saying that we have an $L^2$-orthonormal $\hat{\Gamma}$-orbit eigenbasis of $H^{s+1}(T^{\vee}_{\Lambda},R_{\Gamma})$. Moreover, it follows from overdetermined elliptic regularity that any eigenfunction of the $\Phi_i$ must automatically lie in $H^{s+1}(T^{\vee}_{\Lambda},R_{\Gamma})$, and so (iv) above is redundant once one has (ii).

To see the relationship between this proposition and Proposition \ref{prop:orbFlat}, note that our original basis for $L^2(T^\vee_\Lambda, R_\Gamma)$ is an orthonormal free $\hat\Gamma$-orbit contained within $H^{s+1}(T^\vee_\Lambda,R_\Gamma)$, and that choices of other such bases precisely correspond to the elements of $\tilde \G^{s+1}_{\hat\Gamma}$.
\end{rmk}

\begin{proof}[Proof of Proposition \ref{prop:orbitbasisHilb}]
Assume for the moment that $r\ge 2$. Let $\Delta = \nabla^* \nabla$ be the covariant Laplacian associated to $\nabla$. It is elliptic, up to the non-smoothness of its coefficients, but if $s > r/2$ then $H^1(T^\vee_\Lambda, R_\Gamma)$ is a Hilbert module over $H^s(T^\vee_\Lambda, \End(R_\Gamma))$, and so $\Delta\colon H^2(T^\vee_\Lambda, R_\Gamma)\to L^2(T^\vee_\Lambda, R_\Gamma)$ differs from the non-covariant Laplacian by the addition of an operator which factors through the compact inclusion $H^2(T^\vee_\Lambda, R_\Gamma) \hookrightarrow H^1(T^\vee_\Lambda, R_\Gamma)$. So, it defines a Fredholm operator $\Delta\colon H^2(T^\vee_\Lambda, R_\Gamma) \to L^2(T^\vee_\Lambda, R_\Gamma)$; a parametrix yields that, again despite the non-smoothness of the coefficients of $\Delta$, an elliptic estimate still holds which implies that $\psi\in L^2(T^\vee_\Lambda, R_\Gamma)$ has $\Delta\psi\in L^2(T^\vee_\Lambda,R_\Gamma)$ if and only if $\psi\in H^2(T^\vee_\Lambda,R_\Gamma)$. Therefore, as an unbounded operator with dense domain of definition given by $H^2(T^\vee_\Lambda, R_\Gamma) \subset L^2(T^\vee_\Lambda, R_\Gamma)$, $\Delta$ is self-adjoint. Fredholmness and self-adjointness imply that $\Delta$ admits a left inverse on the $L^2$-orthogonal complement of its kernel $K$. Compactness of $H^2(T^\vee_\Lambda, R_\Gamma) \hookrightarrow L^2(T^\vee_\Lambda, R_\Gamma)$ now implies that the direct sum of this left inverse and the identity on $K$ yields a compact self-adjoint operator $\kappa$ on $L^2(T^\vee_\Lambda, R_\Gamma)$. The spectral theorem for such operators then gives an orthonormal eigenbasis for $\Delta$. Furthermore, these eigenvectors are in $H^{s+1}(T^\vee_\Lambda, R_\Gamma)$. To see this, we observe that $\Delta\colon H^{s'+1}(T^\vee_\Lambda, R_\Gamma) \to H^{s'-1}(T^\vee_\Lambda, R_\Gamma)$ is Fredholm for all $1\le s'\le s$, and so for such $s'$ an elliptic estimate implies that eigenvectors of $\Delta$ which are in $H^{s'-1}(T^\vee_\Lambda, R_\Gamma)$ are in fact in $H^{s'+1}(T^\vee_\Lambda, R_\Gamma)$. Since $\kappa$ is injective, all of the eigenspaces are finite dimensional. Since our Hermitian operators $\Phi_i$ commute with $\Delta$ and with each other (and $H^{s+1}(T^\vee_\Lambda, R_\Gamma)$ is contained in their domains), it follows immediately from the finite-dimensional spectral theorem that there is a common orthonormal eigenbasis $\{e_a\}$ for the $\Phi_i$. Their eigenvalues $\mu_a$ are naturally valued in $Q$; we denote the set of these eigenvalues by $\Spec(-iA)\subset Q$. Note that this spectrum is discrete, since an accumulation point would yield an accumulation point in the spectrum of $\Delta$, which would in turn yield either a nonzero accumulation point or a sequence tending to infinity in the spectrum of $\kappa$.

The same arguments work when $r=1$ if $s \ge 1$; we can strengthen this assumption to $s > 1/2$ by repeating these arguments for $-i\nabla$ (which is elliptic in one dimension, up to the non-smoothness of its coefficients) instead of $\Delta$. (Here, we use the fact that $-i\nabla\colon H^{s'+1}(T^\vee_\Lambda, R_\Gamma) \to H^{s'}(T^\vee_\Lambda, R_\Gamma)$ is Fredholm for $0\le s' \le s$.) Finally, we can also allow for $r=0$, as the conclusions are immediate from the finite-dimensional spectral theorem in this case and there is no need for $\Delta$ or $\nabla$.

We now wish to show that this basis may be chosen to be a free $\hat\Gamma$-orbit. This is nontrivial when $r>0$ because there are multiple (indeed, infinitely many) points with nontrivial stabilizers. We consider the eigenspace decomposition 
\be \label{eq:eigs} L^2(T^\vee_\Lambda, R_\Gamma) \simeq \bigoplus_{\mu\in \Spec(-iA)} Y_\mu \ . \ee

We will now find the following lemmas useful:

\begin{lemma} \label{lem:freeBasis}
If $v_\mu \in L^2(T^\vee_\Lambda, R_\Gamma)$ has a trivial stabilizer in $\hat\Gamma$ and $\hat\Gamma\cdot v_\mu$ is an orthonormal set, then it is in fact a basis for $L^2(T^\vee_\Lambda, R_\Gamma)$.
\end{lemma}
\begin{proof}
Writing $j(\hat\gamma) = j(\tilde S(\hat\gamma)) j(S(\pi(\hat\gamma)))$, the fact that $\avg{j(\hat\gamma) v_\mu, v_\mu} = \delta_{\hat\gamma,1}$ for all $\hat\gamma\in \hat\Gamma$ may be interpreted as the fact that $j(S(\Gamma)) v_\mu$ is a.e. an orthonormal basis for $R_\Gamma$, and given any such basis it is clearly the case that its image under $j(\Lambda)$ is a basis for $L^2(T^\vee_\Lambda, R_\Gamma)$.
\end{proof}

This lemma makes it clear that if there is any $\mu\in \Spec(-iA)$ whose stabilizer in $\hat\Gamma$ is trivial, and if $v_\mu$ is an element of the $\{e_a\}$ basis contained in $Y_\mu$, then $\hat\Gamma\cdot v_\mu$ provides an example of the sort of basis we seek. In particular, this is the case if there is any $\mu\in \Spec(-iA)$ whose projection to $\Lambda_\RR^\perp$ is nonzero. So, we may henceforth assume that $\Spec(-iA)\subset \Lambda_\RR$, i.e., that $\phi$ vanishes identically.

\begin{lem}\label{lem:Fourbasisfun} There exists $g \in U(L^2(T^{\vee}_{\Lambda},R_{\Gamma}))^{j(\Lambda)} \simeq L^2(T^{\vee}_{\Lambda},U(R_{\Gamma}))$ such that for all $\hat{\gamma} \in \hat{\Gamma}$, $g\, \iota(\hat{\gamma}) = e_a$ for some $a$.\end{lem}

Note, as per Remark~\ref{rmk:autreg} above, that $g$ of this lemma automatically lies within $H^{s+1}(T^{\vee}_{\Lambda},U(R_{\Gamma}))$.

\bp
As we have just explained, this lemma is trivial -- indeed, there exists a $g\in U(L^2(T^\vee_\Lambda, R_\Gamma))^{j(\hat\Gamma)}$ -- if $\Spec(-iA)\not\subset \Lambda_\RR$. So, we assume that $\Spec(-iA) \subset \Lambda_\RR$.

For any $\lambda \in \Lambda$, we have a natural isomorphism of eigenspaces 
\begin{align*} Y_{\mu} &\stackrel{\sim}{\to} Y_{\mu + \lambda} \\ v &\mapsto v e^{i \lambda \theta}\,. \end{align*}
So, let $f \subset \Lambda_\RR$ be a fundamental domain for the $\Lambda$-translation action on $\Lambda_{\mb{R}}$ and denote $f\mathrm{Spec}(-iA) := \Spec(-iA) \cap f$ and $Y := \bigoplus_{\mu \in f\mathrm{Spec}(-iA)} Y_{\mu} \subset L^2(T^{\vee}_{\Lambda},R_{\Gamma})$. Discreteness of the spectrum yields that $f\mathrm{Spec}(-iA)$ is finite and so $Y$ is finite-dimensional as all the $Y_{\mu}$ are. Consider $\{e_{a_k}\}_{1 \le k \le n}$ as an orthonormal eigenbasis for $Y$, where $n:=\dim Y$ and we choose some ordering for the finitely many $e_a$ that appear in $Y$. By the isomorphisms above, it follows that $\{e_{a_k} e^{i \lambda \theta}\}_{1\le k\le n, \lambda\in \Lambda}$ is an orthonormal basis for $L^2(T^{\vee}_{\Lambda},R_{\Gamma})$. But then it is easy to see that $n = \dim R_{\Gamma} = |\Gamma|$ by noting that the values of the $e_{a_k} \in L^2(T^{\vee}_{\Lambda},R_{\Gamma})$ at almost any point of $T^{\vee}_{\Lambda}$ must yield a basis for $R_{\Gamma}$. Hence, we may construct $g$ by first choosing any bijection $\Gamma \stackrel{\sim}{\to} \{1, \cdots, n\}$ and demanding that $\iota(S(\gamma))$ map to the corresponding $e_{a_k}$ for $\gamma \in \Gamma$ and $1 \le k \le n$, and then extending the action of $g$ on the basis elements $\iota(\hat{\gamma})$ via $\Lambda$-invariance. \ep

We now need to understand the behavior of $Y_{\mu}$ as a representation of the finite stabilizer group $\hat{\Gamma}_{\mu} \subset \hat{\Gamma}$. Our main technical tool will be to compute a heat kernel-regularized character of $j\colon \hat\Gamma\to U(L^2(T^\vee_\Lambda, R_\Gamma))$ on both sides of \eqref{eq:eigs}. %Actually, we use two different regulators: on the left, we use $$\Tr_\tau j(\hat\gamma) := \sum_{\hat\gamma'\in \hat\Gamma} e^{-\tau \|\tilde S(\hat\gamma')\|^2} \avg{\iota(\hat\gamma'), \iota(\hat\gamma \hat\gamma')}\, ,$$
Indeed, consider $$\Tr e^{-\tau \Delta} j(\hat\gamma) = \sum_a e^{-\tau \|\mu_a\|^2} \avg{e_a, j(\hat\gamma) e_a}\,.$$
%The difference between these two regulators is due to the difference between $\nabla$ and $d$; explicitly, $\Tr_\tau j(\hat\gamma) = \Tr e^{-\tau d^* d} j(\hat\gamma)$ and 
%which one may also write as $\tilde\Tr_\tau j(\hat\gamma) = \Tr e^{-\tau \Delta} j(\hat\gamma)$.
Here and henceforth we assume that $\phi=0$, as explained above, and so $\Delta e_a = \|\mu_a\|^2 e_a$.

A useful observation for studying the heat operator $e^{-\tau\Delta}$ is that if $g\in H^{s+1}(T^\vee_\Lambda, U(R_\Gamma))$ is the unitary operator from Lemma \ref{lem:Fourbasisfun} that takes the $\iota(\hat\Gamma)$ basis to the $e_a$ one, then $g^{-1}\cdot A$ is associated to the Hermitian operators $g \Phi_i g^{-1}$ which are diagonal in the standard $\iota(\hat\Gamma)$ basis. That is, $g^{-1}\cdot A$ is constant and diagonal, and in particular it is smooth. The heat operator of this transformed $A$ is of course $g e^{-\tau \Delta} g^{-1}$, and this lets us translate results on heat operators of elliptic operators with smooth coefficients to the heat operator of interest. For example, we immediately see that $e^{-\tau \Delta}$ is trace class. Furthermore, the heat kernel for $g e^{-\tau \Delta} g^{-1}$ immediately yields the heat kernel for $e^{-\tau \Delta}$, since $g$ is a multiplication operator.\footnote{As an amusing aside, we observe that since $g e^{-\tau \Delta} g^{-1}$ is a bounded self-adjoint operator, the same is the case for $e^{-\tau \Delta}$ itself. It then follows from the Hille-Yosida theorem that $\Delta$ is itself self-adjoint on an appropriate dense domain. This is a priori surprising for $r=1$ and $s\in (1/2,1)$, since in this case $\Delta$ does not even map $C^\infty$ to $L^2$. However, $\Delta$ is self-adjoint on $g^{-1}\, H^2(T^\vee_\Lambda, R_\Gamma) \subset H^{s+1}(T^\vee_\Lambda, R_\Gamma)$, which is dense in $L^2(T^\vee_\Lambda, R_\Gamma)$.}

We would now like to apply this heat-kernel regularized trace to both sides of~\eqref{eq:eigs}, as advertised. On the left side, we have the following lemma:

\begin{lemma} \label{lem:limEq}
For all $\hat\gamma\in \hat\Gamma\setminus\{1\}$, $\Tr e^{-\tau \Delta} j(\hat\gamma)$ vanishes to all orders in $\tau$ as $\tau \to 0$.
\end{lemma}

\begin{proof}
%Indeed, we will prove the stronger result that this trace vanishes to all orders in $\tau$. 
As above, we employ the unitary operator $g$ to rewrite this trace as 
$$\Tr\brackets{g^{-1} (g e^{-\tau \Delta} g^{-1}) g j(\hat\gamma)} \ . $$
Since $g$ and $j(\hat\gamma)$ are multiplication operators, this trace only involves the values of the heat kernel along the diagonal. That is, regarding the heat kernel of $g e^{-\tau \Delta} g^{-1}$ as an element of $C^\infty(\RR^+ \times T^\vee_\Lambda \times T^\vee_\Lambda, \End(R_\Gamma))$, we need only consider the corresponding element of $C^\infty(\RR^+ \times T^\vee_\Lambda, \End(R_\Gamma))$ obtained by restricting to the diagonal of $T^\vee_\Lambda \times T^\vee_\Lambda$. The heat kernel of $g e^{-\tau \Delta} g^{-1}$ has a standard asymptotic expansion as $\tau\to 0$ (discussed, e.g., in \cite[ch. 2]{getzler:heatBook} and \cite[ch. 7]{melrose:aps}), and its coefficients along the diagonal are simple local gauge-covariant and diffeomorphism-invariant functions of the metric and $gAg^{-1}$. Once we conjugate by $g^{-1}$, we obtain local gauge-covariant and diffeomorphism-invariant functions of the metric and $A$. Denote such a function by $f$, and note that $f$ is a natural, regular, homogeneous invariant in the sense of~\cite{ABP} or~\cite{Gilkey}. We are interested in computing
\be \Tr f j(\hat\gamma) = \frac{1}{{\rm vol}(T^\vee_\Lambda)} \int_{T^\vee_\Lambda} \tra f(x) j(\hat\gamma)(x) \, dx \ , \ee
where we disambiguate here between the trace on $L^2$ and that on $R_\Gamma$ by denoting the latter by $\tra$. Recall (e.g., from \cite[\S3]{ABP} or~\cite[\S2.6]{Gilkey}) that $f$ must be formed from contractions of the Riemann curvature of the torus and of the curvature of the connection $A$. As the torus is flat and $\mathscr{F}(A) \equiv 0$, we must hence have that $f$ is a (translation-invariant) multiple of the identity. But then this trace vanishes unless $\tilde S(\hat\gamma) = 0$, or equivalently $\hat\gamma = S(\gamma)$ for some nontrivial $\gamma\in \Gamma$. Finally, in these cases we have $\tra f j(S(\gamma)) = f \tra j(S(\gamma)) = 0$.
\end{proof}

%\atodo{But the usual appeal to Weyl's invariant theory argument yields that $f$ must be formed from contractions of the Riemann curvature of the torus and of the curvature of the connection $A$, as per~\cite[\S3]{ABP} or~\cite[\S2.6]{Gilkey}; as the torus is flat and $\mathscr{F}(A) \equiv 0$, we must hence have that $f$ is simply a (translation-invariant) multiple of the identity.}

On the right side, we obtain the expression %The purpose of using these different regulators is that for all nontrivial $\hat\gamma\in\hat\Gamma$ and all $\tau>0$ we may immediately compute $\Tr_\tau j(\hat\gamma) = 0$ and 
\beq\label{eq:htkersum}\Tr e^{-\tau\Delta} j(\hat\gamma) = \sum_{\mbox{\footnotesize $\mu\in Q^{\hat\gamma} \cap \Spec(-iA)$}} e^{-\tau \|\mu\|^2} \Tr j(\hat\gamma)|_{Y_\mu}\,.\eeq Recall our assumption in this section that all nontrivially-stabilized points within $Q$ are isolated. Then the sum in~\eqref{eq:htkersum} has (at most) one term. Hence, applying $\mathrm{Tr}\,e^{-\tau \Delta}j(\hat{\gamma})$ to both sides of~\eqref{eq:eigs} and taking $\tau \to 0$ yields that for all $\mu \in \Spec(-iA)$ and for $\gamma \in \hat{\Gamma}_{\mu}\setminus\{1\}$, the finite-dimensional trace $\mathrm{Tr}\,j(\hat{\gamma})|_{Y_{\mu}} = 0$. In particular, for any $\mu \in \Spec\,(-iA)$, $Y_{\mu}$ as a representation of the finite group $\hat{\Gamma}_{\mu}$ has vanishing character for any nontrivial conjugacy class, and so must be a direct sum of finitely many copies of the regular representation. But we now have the following:

%More generally, if we do not assume isolatedness of the nontrivially-stabilized locus, let $\eta^j$ vary over all $\delta$-tuples of reals, and consider instead the heat-kernel regularized traces on both sides of~\eqref{eq:eigs} of the slightly more complicated operator $j(\hat{\gamma})\exp(2\pi i \Phi_j \eta^j)$. As $\eta^j$ varies over $\mb{R}^{\delta}$, we once again find that $\mathrm{Tr}\,j(\hat{\gamma})|_{Y_{\mu}} = 0$ for all $\gamma \in \hat{\Gamma}_{\mu}\setminus\{1\}$.}

\begin{lem} \label{lem:regRep}
Suppose $G$ is a finite group and $V$ a unitary $G$-representation. Then there exists $v \in V$ such that $G \cdot v$ is an orthonormal basis for $V$ if and only if $V$ is the regular representation of $G$. \end{lem}

\bp If any $G$-orbit basis exists, one may immediately compute that $V$ and the regular representation have the same character. The converse direction follows from the fact that unitary representations which are isomorphic as general linear representations are in fact also isomorphic as unitary representations. Indeed, if $T$ denotes a $G$-equivariant isomorphism, we may consider the unitary part of its polar decomposition $T (T^\dagger T)^{-1/2}$ as a $G$-equivariant isomorphism of unitary representations. So, we may assume that $V$ is the regular representation with its standard Hermitian inner product, whereupon we may simply take any of the standard basis vectors. \ep 

%\fstodo{if we can show that $\lim_{\tau \to 0} \Tr_\tau j(\hat\gamma) = \lim_{\tau\to 0} \tilde\Tr_\tau j(\hat\gamma)$ then it will follow that $\Tr j(\hat\gamma)|_{Y_{\mu_{\hat\gamma}}} = 0$, and therefore that $Y_\mu$ is the direct sum of finitely many copies of the regular representation of $\hat\Gamma_\mu$. We establish this equality of limits below as Lemma \ref{lem:limEq}; for now, we assume the result. \MZ{does this work for non-generic?}}

So, take any $\mu \in \Spec\,(-iA)$; we may find, by the above, some $v_{\mu} \in Y_{\mu}$ with trivial $\hat\Gamma_\mu$ stabilizer such that $\hat{\Gamma}_{\mu} \cdot v_{\mu}$ is an orthonormal set. It follows more generally that $v_\mu$ has a trivial $\hat\Gamma$ stabilizer and the orbit $\hat{\Gamma} \cdot v_{\mu}$ is an orthonormal set, and by Lemma \ref{lem:freeBasis} it is therefore a basis. (In particular, we have that each $Y_\mu$ is a single copy of the regular representation of $\hat\Gamma_\mu$.)
\end{proof}

%This concludes our proof of Proposition \ref{prop:orbitbasisHilb}, and therefore of Theorem \ref{thm:noFI}.

\begin{proof}[Proof of Proposition \ref{prop:orbFlat}]

Proposition~\ref{prop:orbitbasisHilb} shows that the discrete spectrum of commuting operators $\Spec\,(-iA) \subset Q$ is a full $\hat{\Gamma}$-orbit; i.e., we may consider this discrete spectrum as a point of $Q/\hat{\Gamma}$ in order to produce a map $$\{A \in \mc{A}^s_{\hat{\Gamma}} \bigm| \mathscr{F}(A) = 0\} / \mc{G}^{s+1}_{\hat{\Gamma}} \stackrel{\Spec}{\to} Q/\hat{\Gamma}\, .$$ The above map is well-defined as the spectrum is unaffected by the $\mc{G}^{s+1}_{\hat{\Gamma}}$-action. Surjectivity follows easily upon considering operators $\Phi_i$ which are already diagonal in the standard basis for $R^0_{\hat{\Gamma}} \simeq L^2(T^{\vee}_{\Lambda},R_{\Gamma})$, while injectivity follows from Proposition~\ref{prop:orbitbasisHilb} once again. Indeed, suppose $\mathscr{F}(A) = \mathscr{F}(A') = 0$. Then, as choices of other bases (as considered relative to the standard $\iota(\hat\Gamma)$ basis) satisfying the conditions of Proposition~\ref{prop:orbitbasisHilb} exactly correspond to gauge transformations in $\tilde{\mc{G}}^{s+1}_{\hat{\Gamma}}$, we may gauge-transform both $A, A'$ to be constant and diagonal, upon which if $\Spec(-iA) = \Spec(-iA')$ so that these diagonal entries are the same up to possible $\hat{\Gamma}$-reordering, it is immediate to apply a further equivariant gauge transformation to identify one operator with the other.

%Our original basis for $L^2(T^\vee_\Lambda, R_\Gamma)$ is, of course, an orthonormal free $\hat\Gamma$-orbit contained within $H^{s+1}(T^\vee_\Lambda,R_\Gamma)$. Choices of other such bases precisely correspond to the elements of $\tilde \G^{s+1}_{\hat\Gamma}$. So, Proposition \ref{prop:orbitbasisHilb} implies that there exists a gauge transformation that simultaneously diagonalizes the operators $\Phi_i$. The remaining gauge freedom is simply the $\hat\Gamma$-action that permutes the eigenvectors. \MZ{nope} So, we have demonstrated the bijection \eqref{eq:flatOrbBij}.

It remains only to more explicitly relate this functional analytic parametrization of the moduli space to the concrete connection and endomorphism $(\nabla,\phi)$. To do so, instead of regarding gauge transformations as describing the same operators $\Phi_i$ in a different basis we now regard them as describing conjugate operators in the original basis. Then, the eigenvalue of $\iota(1)$, regarded as an element of $Q$, is the top-left entry of $-i(a,\phi)$.
\end{proof}

\begin{rmk}
It is of interest now to revisit the behavior of the function $\mathrm{Tr}\,e^{-\tau\Delta}j(\hat{\gamma})$ studied in Lemma~\ref{lem:limEq}. On the one hand, we did not even need the full strength of this lemma; rather, we only used that $\lim_{\tau\to 0} e^{-\tau \Delta} j(\hat\gamma) = 0$ for all $\hat\gamma\in \hat\Gamma\setminus\{1\}$. On the other hand, it is now clear that a significantly stronger result is true: $\Tr e^{-\tau \Delta} j(\hat\gamma)$ not only vanishes to all orders, but in fact vanishes identically. Indeed, this is clearly the case when $\nabla$ is a constant diagonal connection, since $e^{-\tau\Delta}$ is diagonal in the original $\iota(\hat\Gamma)$ basis and $j(\hat\gamma)$ is purely off-diagonal; hence, now that we know that $g$ may be chosen to commute with all $j(\hat\gamma)$ we may directly compute
$$\Tr e^{-\tau\Delta} j(\hat\gamma) = \Tr (g e^{-\tau\Delta} g^{-1}) (g j(\hat\gamma) g^{-1}) = \Tr (g e^{-\tau\Delta} g^{-1}) j(\hat\gamma) = 0 \ .$$
Note in the proof of Lemma \ref{lem:limEq}, where we only knew \emph{a priori} that $g$ commuted with $j(\Lambda)$, that we could have already concluded $\mathrm{Tr}\,e^{-\tau\Delta}j(\lambda) = 0$ for $\lambda \in \Lambda\setminus\{0\}$. Alternatively, a second proof that $\mathrm{Tr}\,e^{-\tau\Delta}j(\hat{\gamma}) \equiv 0$ for all nontrivial $\hat{\gamma}$ is that we have shown that all of the terms in \eqref{eq:htkersum} vanish.\end{rmk}

\begin{rmk}\label{rmk:stab}
We may now expand on Remark~\ref{rmk:set}. The identification of Proposition \ref{prop:orbFlat} is an identification of quotient \emph{sets}. The identification does not hold as groupoids (or orbifolds) as the stabilizers of the quotient on the left of \eqref{eq:flatOrbBij} are much larger than the finite stabilizers of $Q/\hat{\Gamma}$. Indeed, given a point $\mu$ of $Q/\hat{\Gamma}$ with stabilizer $\hat\Gamma_\mu \subset \hat\Gamma$, the stabilizer in $\tilde\G_{\hat\Gamma}^{s+1}$ is the product of unitary groups
$$U(R_{\hat\Gamma_\mu})^{\hat\Gamma_\mu} \simeq \prod_{V \in\,\text{Irrep}\,\hat\Gamma_\mu} U(V)\,.$$ Note that this change-of-stabilizers is a relatively common phenomenon in GIT. For example, a familiar setting for this behavior may be the morphism $T\kqnoxi W \to G\kqnoxi G$, i.e., $\mathrm{Spec}(\mc{O}(T)^W) \to \mathrm{Spec}(\mc{O}(G)^G)$, where $G$ is a complex reductive group with maximal torus $T$ and Weyl group $W$, and $G\kqnoxi G$ is the affine GIT quotient by the conjugation action.
\end{rmk}

We now briefly indicate the analog of the above theorem in the complex formulation of Construction~\ref{constr:cplxKah}; after all, we have not yet proven a DUY isomorphism and so the below Theorem~\ref{thm:CnoFI} is \emph{a priori} independent of Theorem~\ref{thm:noFI}.

\begin{thrm}\label{thm:CnoFI} For any $r/2 < s \le \infty$, recall the complex moment map $$\mu_{\mb{C}}\colon \mc{A}^s_{\hat{\Gamma}} \to H^{s-1}_{\Gamma}(T^{\vee}_{\Lambda},\mf{gl}_{\mb{C}}(R_{\Gamma}))\,.$$ We have an isomorphism \beq (\mu_{\mb{C}}^{-1}(0))^{\mathrm{ps}} / (\mc{G}^{s+1}_{\hat{\Gamma}})_{\mb{C}} \stackrel{\sim}{\to} Q/\hat{\Gamma}\,.\eeq \end{thrm}

\begin{rmk} As in Remark~\ref{rmk:set}, we will only prove the above as an isomorphism of sets here, although it is again cheap to upgrade to an isomorphism of complex manifolds away from singular points (or complex-analytic spaces, say). One could also ask for more, like an intrinsically holomorphic characterization of the (poly)stable locus, or a holomorphic characterization of the polystable orbit in the orbit closure of a semistable orbit as a semisimplification under a Jordan-H\"older filtration. It is indeed only when one gives an alternative characterization of the stability condition that this theorem becomes truly \emph{interestingly} independent of Theorem~\ref{thm:noFI}, as we will see below in the $r = 1$ case.
\end{rmk}

\bp Given the definition above of the polystable locus and the already-proven Theorem~\ref{thm:noFI}, the claim here is simply that points of $Q/\hat{\Gamma}$ do not suffer any further identification when one considers the $(\mc{G}^{s+1}_{\hat{\Gamma}})_{\mb{C}}$-action rather than merely the $\mc{G}^{s+1}_{\hat{\Gamma}}$-action. But this is clear as the map $A \mapsto \Spec(-iA)$ used in the proof of Proposition~\ref{prop:orbFlat} is manifestly invariant under $(\mc{G}^{s+1}_{\hat{\Gamma}})_{\mb{C}}$; i.e., eigenvalues are manifestly invariant under conjugation even by possibly non-unitary operators. \ep %\MZ{from the fact that the $(\G^{s+1}_{\hat\Gamma})_\CC$-action does not affect the eigenvalues of our operators $\Phi_i$.}

%, as the eigenvalues map of Claim~\ref{claim:identify} is manifestly invariant even under a complexified group action. \ep

Finally, we note that both the proof of Proposition \ref{prop:orbFlat} (and therefore of Theorems \ref{thm:noFI} and \ref{thm:CnoFI}) in this section and that of the next section generalize immediately to the case of affine actions once one is armed with the ideas of \S\ref{subsec:asymorb}.

\subsection{The flat orbifold quotient reprise}\label{subsec:flatorb2}

In this section, we give a second proof of Proposition \ref{prop:orbFlat}. We first note that the usual correspondence between (gauge equivalence classes of) flat principal bundles and (conjugacy classes of) their monodromy representations holds equivariantly: 

\begin{lem}\label{lem:mon}
Suppose that $r/2<s\le \infty$. Let $G$ be a Lie group and let $M$ be an $r$-dimensional connected manifold, both with actions of a finite group $\Gamma$; let $P$ be a $\Gamma$-equivariant principal $G$-bundle. Suppose, moreover, that there exists some $m_0 \in M^{\Gamma}$, i.e., a point fixed under the full $\Gamma$-action.

Then, if $\nabla, \nabla'$ are two $\Gamma$-equivariant $H^s$-regular flat connections, they may be related by the action of $H^{s+1}_{\Gamma}(M,\Aut(P))$ if and only if their monodromy representations $\sigma,\sigma'\colon \pi_1(M,m_0) \to G$ are conjugate by an element of $G^{\Gamma}$.
\end{lem} 
\bp

Given a smooth path $f\colon I \to M$ and connection $\nabla$, let $PT_{\nabla,f}\colon P|_{f(0)} \to P|_{f(1)}$ be the ensuing parallel transport isomorphism. Then, suppose one has $\nabla, \nabla'$ as above with, on the nose, the same monodromy representation $\pi_1(M,m_0) \to G$, one may consider $g \in H^{s+1}_{\Gamma}(M,\mathrm{Aut}(P))$ as defined via the formula $g(m) \circ PT_{\nabla',f} = PT_{\nabla,f}$ for any path $f$ with $f(0) = m_0, f(1) = m$; the flatness of $\nabla$ and $\nabla'$ proves that $g$ is independent of local deformation of the path $f$ while the equality of monodromy representations shows the same for large deformations of the path (i.e., from traversing homotopically nontrivial loops). The fact that $g$ transforms $\nabla$ to $\nabla'$ may be rephrased as an overdetermined elliptic PDE, whereupon $H^s$-regularity of $\nabla, \nabla'$ yields $H^{s+1}$-regularity of $g$. It now remains to verify that $g$ is equivariant, but indeed, if one wishes to relate $g(m)$ and $g(\gamma m)$, one may simply apply $\gamma$ to any path $f$ from $m_0$ to $m$ in order to arrive at a path $\gamma f$ from $m_0$ to $\gamma m$. As the definition of $g$ above was path-independent and so we may, in particular, use $f$ and $\gamma f$ as our choices of path, the $\Gamma$-equivariance of the parallel transport maps then provides the desired equivariance of $g$. 

More generally, if the monodromy representations are conjugate by some element $\hat{g} \in G^{\Gamma}$, first apply a gauge transformation with $g(0) = \hat{g}$ to reduce to the prior case. Finally, the converse direction, i.e., that monodromy representations of gauge-equivalent flat connections are conjugate by $g(0) \in G^{\Gamma}$ if $g \in H^{s+1}_{\Gamma}(M,\mathrm{Aut}(P))$ is a gauge transformation that relates them, is standard.

\ep 

We now specialize to the case of interest, where $G=U(R_\Gamma)$ and $M$ is $T^\vee_\Lambda$ with its associated $\Gamma$-action and carrying the equivariant bundle which we shall here denote $$\mc{R}_{\Gamma} := R_{\Gamma} \times T^{\vee}_{\Lambda}\,.$$ We further have that $A=(\nabla,\phi)$ is $\Gamma$-equivariant, and we choose $0 \in T^\vee_\Lambda \simeq \mathrm{Hom}(\Lambda, U(1))$ to be our preferred $\Gamma$-fixed point $m_0$. Lemma~\ref{lem:mon} then shows that the information of $\nabla$ is equivalent to the monodromy information $\pi_1(T^\vee_\Lambda) \to U(R_\Gamma)$ even $\Gamma$-equivariantly. Furthermore, as $\phi$ is a flat section, it is completely determined by its value $\phi_0$ at $0$ (which must be preserved by the image of the monodromy representation). Hence, we may repackage the data of $A$ as an element of
$$\mathrm{Hom}_{\Gamma}(\Lambda^{\vee},U(R_\Gamma)) \times (\Lambda^{\perp}_{\mb{R}} \underset{\mb{R}}{\otimes} \mf{u}(R_{\Gamma}))^{\Gamma} $$
subject to the commutation conditions.\footnote{Not all such data may be realized as monodromy matrices, however; there is an additional condition. In fact, as we will see shortly, this extra condition is precisely that the monodromy data in $\mathrm{Hom}_{\Gamma}(\Lambda^{\vee},U(R_{\Gamma}))$ is contained within the identity connected component.}
We still have the freedom of a gauge redefinition at the fiber over $0$ -- i.e., the action of $U(R_\Gamma)^{\Gamma}$ on the above by conjugation.

Note that one could now feasibly attempt to recast the original infinite-dimensional hyperkahler quotient as something like a finite-dimensional hyperkahler quotient of $\mathrm{Hom}_{\Gamma}(\Lambda^{\vee},U(R_{\Gamma}))\times(\Lambda_{\mb{R}}^{\perp}\otimes_{\mb{R}}\mf{u}(R_{\Gamma}))^{\Gamma}$ by $U(R_{\Gamma})^{\Gamma}$, but now some of the moment maps will be naturally group- rather than Lie algebra-valued. Hence, one could attempt to make use of the theory of \emph{quasi}-Hamiltonian reduction. We will not pursue this here.

%\fatodo{It will be of use to find a particularly nice \emph{constant, flat, equivariant} $\nabla_0 = d + A_0$ with the same monodromy representation as a given flat, equivariant connection $\nabla$. Note that if $A_0 \in \mathrm{Hom}_{\Gamma}(\Lambda^{\vee},\mf{u}(R_{\Gamma}))$ is constant, its monodromy representation is its image under $\mathrm{Hom}_{\Gamma}(\Lambda^{\vee},\mf{u}(R_{\Gamma}))\stackrel{\exp}{\to}\mathrm{Hom}_{\Gamma}(\Lambda^{\vee},U(R_{\Gamma}))$. In other words, we need to show the monodromy representation of $\nabla$ lies in the image of the above exponential map. But $\Hom_{\Gamma}(\Lambda^{\vee},U(R_{\Gamma}))$ is a compact manifold\footnote{$\Hom_\Gamma(\Lambda^\vee, U(R_\Gamma))$ is clearly a closed sublocus of $\Hom(\Lambda^\vee, U(R_\Gamma))$, which is in turn a closed sublocus of $U(R_\Gamma)^{r}$ given by the values on basis lattice vectors.} with tangent space at the \MZ{trivial homomorphism?} identity \ftodo{it's a bit disingenuous to call an element of a non-group the identity but maybe everybody knows what I mean} $\Hom_{\Gamma}(\Lambda^{\vee},\mf{u}(R_{\Gamma}))$, and so the usual Hopf-Rinow surjectivity of the Riemannian exponential map for connected compact manifolds shows that it suffices to prove that the monodromy representation of $\nabla$ lies in the \ftodo{identity} connected component of $\mathrm{Hom}_{\Gamma}(\Lambda^{\vee},U(R_{\Gamma}))$. But if $\nabla = d + A$, let $\nabla_t = d + tA$. Then the monodromy representations of $\nabla_t$ as $0 \le t \le 1$ exhibit the desideratum.}

In any case, after performing the above reductions by means of Lemma~\ref{lem:mon}, we may restate Proposition~\ref{prop:orbFlat} as follows:

\begin{prop}\label{prop:orbFlat2} Given a linear action $\hat{\Gamma} \hookrightarrow \mathrm{Aff}(Q)$, and given $s \in (r/2, \infty]$, there is a canonical bijection \begin{align}\{(\sigma,\phi_0) &\bigm| \sigma\text{ is the monodromy representation of an equivariant $H^s$-regular flat connection on $\mc{R}_{\Gamma}$ }\nonumber\\&\text{and $\phi_0 \in (\mf{u}(R_{\Gamma}) \otimes \Lambda_{\mb{R}}^{\perp})^{\Gamma}$ is $\sigma$-invariant with $[\phi_0,\phi_0] = 0$}\} / U(R_{\Gamma})^{\Gamma} \simeq Q/\hat{\Gamma}\,.\label{eq:monbij}\end{align} \end{prop}

Technically, Proposition~\ref{prop:orbFlat2} as stated above leaves open the statements of Proposition~\ref{prop:orbFlat} concerning explicit representatives by constant diagonal connections, but it will become clear that the subsequent analysis in this section yields the full strength of these considerations as well.

So, the rest of this section is devoted to a proof of this proposition; we begin by a careful study of the monodromy representation $\sigma$ above. We now introduce notation for the eigenvalues of such a monodromy representation:

\begin{defn} Given $\sigma \in \Hom_{\Gamma}(\Lambda^{\vee},U(R_{\Gamma}))$, let \beq\label{eq:evals} R_{\Gamma} \simeq \bigoplus_{\mu \in T_{\Lambda}} V_{\mu} \eeq denote the eigenspace decomposition of $R_{\Gamma}$, where the eigenvalues of the family of commuting unitary operators $\sigma$ are naturally valued in $\mathrm{Hom}(\Lambda^{\vee},U(1)) \simeq T_{\Lambda}$. The $\Gamma$-equivariance of $\sigma$ implies that for $\gamma \in \Gamma, \mathrm{act}_{\gamma}\colon V_{\mu} \stackrel{\sim}{\to} V_{\gamma \cdot \mu}$. Hence, given $\nu \in T_{\Lambda}/\Gamma$, we may denote $$W_{\nu} := \bigoplus_{\mu \in \nu} V_{\mu}$$ to obtain a decomposition $$R_{\Gamma} \simeq \bigoplus_{\nu \in T_{\Lambda}/\Gamma} W_{\nu}$$ as $\Gamma$-representations. Given $\mu \in T_{\Lambda}$, denote by $\Gamma_{\mu} \subset \Gamma$ its stabilizer, so that $V_{\mu}$ is naturally a $\Gamma_{\mu}$-representation and $W_{[\mu]} \simeq \mathrm{Ind}^{\Gamma}_{\Gamma_{\mu}} V_{\mu}$. Finally, let $$\mathrm{ev}(\sigma) \in \mathrm{Sym}^{|\Gamma|}T_{\Lambda}$$ denote the set of eigenvalues $\mu \in T_{\Lambda}$ of $\sigma$ with multiplicity; this set is manifestly $\Gamma$-invariant, and we say that $\mathrm{ev}(\sigma)$ comprises a \emph{full $\Gamma$-orbit} if for some (or equivalently any) element $\mu$ in $\mathrm{ev}(\sigma)$, we have that $\mathrm{ev}(\sigma) = \Gamma \cdot \mu$. \end{defn}

Slightly more generally, we may include extra data $\phi_0$. As in the prior section, we will more naturally track the eigenvalues of the Hermitian operators $-i\phi_0$.

\begin{defn} Given $\sigma \in \Hom_{\Gamma}(\Lambda^{\vee},U(R_{\Gamma})), \phi_0 \in (\Lambda_{\mb{R}}^{\perp} \otimes \mf{u}(R_{\Gamma}))^{\Gamma}$ satisfying $[\phi_0,\phi_0] = 0$ and $\sigma$-invariance, one may simultaneously diagonalize all of these operators and find eigenvalues naturally valued in \beq\mathrm{ev}(\sigma,-i\phi_0) \in \mathrm{Sym}^{|\Gamma|}(Q/\Lambda)\,.\eeq \end{defn}

\begin{defn} In general, if $X$ is a set with the action of a finite group $\Gamma$, let $((\mathrm{Sym}^{|\Gamma|}(X))^{\Gamma})^0$ denote the connected component of $(\mathrm{Sym}^{|\Gamma|}(X))^{\Gamma}$ consisting of full $\Gamma$-orbits.\footnote{It may be of interest to compare to the definition of $G$-cluster in the study of the $G$-Hilbert scheme of~\cite{Nakamura}.} Note that there is an immediate bijection $((\mathrm{Sym}^{|\Gamma|}(X))^{\Gamma})^0 \simeq X/\Gamma$. \end{defn}

We now study some equivalent conditions on the connection $\nabla = d + a$ on $\mc{R}_{\Gamma}$ and on its ensuing monodromy representation $\sigma(\nabla)$. Let us further introduce the notation $$\Hom^0_{\Gamma}(\Lambda^{\vee},U(R_{\Gamma})) \subset \Hom_{\Gamma}(\Lambda^{\vee},U(R_{\Gamma}))$$
for the connected component that contains the constant map $\mathrm{Id}$. Similarly, if $$T \simeq \mathsf{Maps}(\Gamma,U(1)) \subset U(R_{\Gamma})$$ denotes the diagonal maximal torus, denote by $\mathrm{Hom}^0_{\Gamma}(\Lambda^{\vee},T)$ the component containing the constant map $\mathrm{Id}$.

\begin{prop}\label{prop:eqconnequiv} Given the above setup, the following are equivalent:\begin{enumerate}[(i)] \item The connection $a$ may be equivariantly gauge-transformed to a (flat) constant connection $a_0 \in \mathrm{Hom}_{\Gamma}(\Lambda_{\mb{R}}^*,\mf{u}(R_{\Gamma}))$, \item The connection $a$ may be equivariantly gauge-transformed to a constant connection $a_0 \in \mathrm{Hom}_{\Gamma}(\Lambda_{\mb{R}}^*,\mf{t})$ valued in the diagonal Cartan $\mf{t} \simeq \mathsf{Maps}(\Gamma,\mf{u}(1)) \subset \mf{u}(R_{\Gamma})$, \item $\sigma(a) \in \Hom^0_{\Gamma}(\Lambda^{\vee},U(R_{\Gamma}))$, \item $\sigma(a)$ may be $\Gamma$-equivariantly conjugated to lie within $\mathrm{Hom}_{\Gamma}^0(\Lambda^{\vee},T)$, and \item $\mathrm{ev}(\sigma)$ comprises a full $\Gamma$-orbit; i.e., $\mathrm{ev}(\sigma) \in ((\mathrm{Sym}^{|\Gamma|}(T_{\Lambda})^{\Gamma})^0$. \end{enumerate} \end{prop}

Note that condition (v) `almost' obviously holds: if $\mathrm{ev}(\sigma)$ has any support at a stabilizer-free point of $T_{\Lambda}$, then condition (v) will automatically be true. The more interesting case is hence when $\mathrm{ev}(\sigma)$ has support only at points with nontrivial stabilizer.

\bp We will show (iii) $\implies$ (v) $\implies$ (iv) $\implies$ (ii) $\implies$ (i) $\implies$ (iii). 

((iii) $\implies$ (v)) The eigenvalues map $\sigma \mapsto \mathrm{ev}(\sigma)$, as a map $\mathrm{Hom}_{\Gamma}(\Lambda^{\vee},U(R_{\Gamma})) \to (\mathrm{Sym}^{|\Gamma|}T_{\Lambda})^{\Gamma}$, is continuous. As the eigenvalues of the constant $\sigma \equiv \mathrm{Id}$ map manifestly comprise a full $\Gamma$-orbit, it follows that $\mathrm{Hom}_{\Gamma}^0(\Lambda^{\vee},U(R_{\Gamma})) \stackrel{\mathrm{ev}}{\to} ((\mathrm{Sym}^{|\Gamma|}(T_{\Lambda})^{\Gamma})^0$. 

((v) $\implies$ (iv)) Choose some $\mu \in \mathrm{ev}(\sigma)$, so that $\mathrm{ev}(\sigma) = \Gamma \cdot \mu$. By Frobenius reciprocity, $V_{\mu}$ carries the regular representation of $\Gamma_{\mu}$ and so, by Lemma \ref{lem:regRep}, we may choose some $v \in V_{\mu}$ such that $\Gamma_{\mu} \cdot v$ is an orthonormal eigenbasis for $V_{\mu}$. But then $\Gamma \cdot v$ is an orthonormal eigenbasis for $R_{\Gamma}$; i.e., we may conjugate by an element of $U(R_{\Gamma})^{\Gamma}$ such that $\sigma(a)$ lies in $\mathrm{Hom}_{\Gamma}(\Lambda^{\vee},T)$. It remains to show that it lies in the trivial component thereof. But now the eigenvalues map $\mathrm{Hom}_{\Gamma}(\Lambda^{\vee},T) \stackrel{\mathrm{ev}}{\to} (\mathrm{Sym}^{|\Gamma|}T_{\Lambda})^{|\Gamma|}$ is manifestly an isomorphism, so being in the trivial component of the target implies the same for the source.

((iv) $\implies$ (ii)) The map from constant, flat, equivariant connections to their monodromy representations is the exponential map
\beq\label{eq:explol}\Hom_{\Gamma}(\Lambda^{\vee},\mf{t}) \stackrel{\exp}{\to} \Hom_{\Gamma}(\Lambda^{\vee},T)\,,\eeq but the target of the above map is a compact abelian Lie group and~\eqref{eq:explol} may also be parsed as the exponential map from its Lie algebra, or as the universal cover map onto the connected component of the identity. In particular, this implication follows from the observation $$\Hom_{\Gamma}(\Lambda^{\vee},\mf{t}) \stackrel{\exp}{\twoheadrightarrow} \Hom_{\Gamma}^0(\Lambda^{\vee},T)$$ and an application of Lemma~\ref{lem:mon}.

%((ii) $\implies$ (iii)) The image of the connected (vector) space $\Hom_{\Gamma}(\Lambda^{\vee},\mf{t})$ under the continuous map given as the composition $\Hom_{\Gamma}(\Lambda^{\vee},\mf{t}) \stackrel{\exp}{\to} \Hom_{\Gamma}(\Lambda^{\vee},T) \hookrightarrow \Hom_{\Gamma}(\Lambda^{\vee},U(R_{\Gamma}))$ is indeed connected (and $0 \mapsto \mathrm{Id}$). 

((ii) $\implies$ (i)) Clear.

((i) $\implies$ (iii)) The moduli space of constant, equivariant connections is the vector space $(\Lambda_{\mb{R}} \otimes \mf{u}(R_{\Gamma}))^{\Gamma}$, and the sublocus therein of flat connections is the zero-locus of a homogeneous quadratic map. Hence, the moduli space of flat, constant, equivariant connections is a connected cone; the image under the continuous monodromy map thus lands in the trivial connected component. \ep 

\begin{prop}\label{prop:goodeq} If $\nabla = d + a \in \mathrm{Conn}^s_{\Gamma}(T^{\vee}_{\Lambda},R_{\Gamma})$ is a Hermitian, $\Gamma$-equivariant connection on $\mc{R}_{\Gamma}$, then $\sigma(\nabla)$ satisfies the equivalent conditions of Proposition~\ref{prop:eqconnequiv}. \end{prop}

\begin{rmk}\label{rmk:conninj} Let us give some slightly different context for this proposition. Namely, if $X$ is any finite-type space\footnote{For this purpose, it suffices to have $\pi_1(X)$ finitely-generated.} and $G$ a Lie group, we may consider the map $\mathrm{Hom}(\pi_1 X, G) \stackrel{B}{\to} \mathsf{Maps}_*(B\pi_1(X),BG) \to \mathsf{Maps}(B\pi_1(X),BG) \to \mathsf{Maps}(X,BG)$ which on connected components induces $$\pi_0 \mathrm{Hom}(\pi_1(X),G) \to [X,BG]$$ as the effect on $\pi_0$ of the inclusion of flat $G$-bundles into all $G$-bundles on $X$. We wish to know if this map is injective, or at least if the preimage of the trivial component has only one element; let us term \emph{principally injective} this particular partial notion of injectivity. Note that in general, even in $X$ is a $K(\pi,1)$, this map may fail to be (even principally!) injective; see, e.g.,~\cite[Corollary 2.9]{AdemCohen}. One may conduct a similar discussion when imposing additional $\Gamma$-equivariance by use of equivariant classifying spaces (although note that in the present setting, our equivariant bundles are \emph{intertwined} in the terminology of~\cite{LueckUribe}).\end{rmk}

We first give a concrete proof of Proposition~\ref{prop:goodeq} in the case of $\Gamma = Z_2$ acting on $T_{\Lambda}$ by negating all coordinates. The reason to focus first on this case is (i) it already illustrates in a particularly concrete example some of the general ideas and (ii) this case already encapsulates many of the cases of interest. In particular, in the list of examples where $\Gamma$ acts on $\mb{R}^4$ preserving its hyperkahler structure, the \emph{only} nontrivial possibilities with $r = 1$ or $r = 3$ have $\Gamma = Z_2$ while even for $r = 2$ or $r = 4$, the cases of $\Gamma = Z_2$ are the most studied examples.

\begin{proof}[Proof in the case $\Gamma = Z_2$]
In this case, it suffices (and in fact, is equivalent, by consideration of the possible eigenvalues of $\sigma$, and in particular the possibilities when these points are $2$-torsion points of $T_\Lambda$) to show that $\det \sigma \in \Hom_{\Gamma}(\Lambda^{\vee},U(1))$ is trivial. But now, given $\lambda^{\vee} \in \Lambda^{\vee}$, let $P$ denote the parallel transport via the connection $\nabla$ along the straight-line path from $0$ to $\frac{1}{2}\lambda^{\vee} \in \Lambda_{\mb{R}}^*$. The $Z_2$-equivariance condition then yields that $\sigma(\lambda^{\vee}) = P \cdot (\mathrm{Ad}(\sigma) P)^{-1}$, where $\mathrm{Ad}(\sigma)$ is the conjugation that implements the $Z_2$-equivariance. It is now immediate that $\det \sigma(\lambda^{\vee}) = 1$, as desired. \ep

\bp Consider the decomposition $R_{\Gamma} = \bigoplus_{\mu \in T_{\Lambda}} V_{\mu}$ as a decomposition of the fiber of $\mc{R}_{\Gamma}$ over $0 \in T_{\Lambda}^{\vee}$ and use the parallel transport provided by the flat, equivariant connection $\nabla$ to produce a decomposition 
\beq\label{eq:globaldecomp} \mc{R}_{\Gamma} = \bigoplus_{\mu \in T_{\Lambda}} \mc{V}_{\mu}\,.\eeq 
Note that the parallel transport isomorphisms provided by the flat connection $\nabla$ are ill-defined precisely due to the global monodromy transformations $\sigma(\nabla)$, but these ambiguities preserve the eigenspaces $V_{\mu}$ and so the above decomposition of $\mc{R}_{\Gamma}$ is well-defined. Now, for each $\gamma \in \Gamma$ and $p \in (T^{\vee}_{\Lambda})^{\gamma}$, consider the ``local character'' $\chi_{\gamma,p}$ which maps a $\Gamma$-equivariant vector bundle $W$ on $T^{\vee}_{\Lambda}$ to $\Tr(\gamma|_{W_p})$; if one likes, one may parse $\chi_{\gamma,p}$ as a functional $K_{\Gamma}(T^{\vee}_{\Lambda}) \stackrel{\chi_{\gamma,p}}{\to} \mb{C}$. We would like to apply $\chi_{\gamma,p}$ to~\eqref{eq:globaldecomp} for $\gamma$ nontrivial. The left side is easy to compute: the fiber at $p$ of the trivial bundle $\mc{R}_{\Gamma}$ is still $R_{\Gamma}$, albeit now considered only as a representation of the stabilizer group $\Gamma_p \subset \Gamma$; as such, it decomposes as $[\Gamma: \Gamma_p]$ many copies of the regular representation of $\Gamma_p$. Nonetheless, the character of any nontrivial element $\gamma$ vanishes. Applying $\chi_{\gamma,p}$ to the right side breaks up into two cases: (i) if $\mu \not\in (T_{\Lambda})^{\gamma}$, then recall that the action of $\gamma$ takes $\mc{V}_{\mu}$ to $\mc{V}_{\gamma \mu}$ so that the trace of $\gamma$ immediately vanishes via its interpretation as the sum of diagonal entries with respect to a basis, while (ii) if $\mu \in (T_{\Lambda})^{\gamma}$, we proceed as follows: 
\begin{lem}\label{lem:geneq} 
Suppose $M$ is a smooth manifold with the action of a finite group $\Gamma$ with $m_0 \in M^{\Gamma}$ some chosen fixed base-point, and suppose $\rho \in \mathrm{Hom}(\Gamma, U(1))$ is some one-dimensional representation of $\Gamma$. Let $\widetilde{M}$ denote the universal cover of $M$, and given $\mu \in \mathrm{Hom}_{\Gamma}(\pi_1(M,m_0),U(1)) \simeq \mathrm{H}^1(M,U(1))^{\Gamma}$, let $\mc{L}_{\mu,\rho}$ denote the equivariant flat line bundle $$\mc{L}_{\mu,\rho} := (\widetilde{M} \times \mb{C}_{\rho}) / \pi_1(M,m_0)$$ where the action of a loop $[\ell] \in \pi_1(M,m_0)$ is via $$(\widetilde{m},v) \stackrel{\ell}{\mapsto} (\ell(\widetilde{m}),\mu(\ell)v)$$ and the action of $\gamma \in \Gamma$ is given by $$(\widetilde{m},v) \stackrel{\gamma}{\mapsto} (\gamma(\widetilde{m}),\rho(\gamma)v)\,;$$ note that this action yields a well-defined equivariant structure on $\mc{L}_{\mu,\rho}$ by virtue of the $\Gamma$-invariance of $\mu$. Then given $p \in M^{\gamma}$ and $\chi_{\gamma,p}$ defined as above, we have $$\chi_{\gamma,p}(\mc{L}_{\mu,\rho}) = \mu(\ell_{\gamma,p})\rho(\gamma)\,,$$ where $\ell_{\gamma,p} \in \pi_1(M,m_0)$ is a loop such that for some lift $\widetilde{p} \in \widetilde{M}$ of $p$, we have $\ell_{\gamma,p}(\gamma(\widetilde{p})) = \widetilde{p}$. Note that it is more naturally $\pi_1(M,p)$ which acts on the lifts of $p$ to $\widetilde{M}$ and that transporting an element of $\pi_1(M,p)$ to $\pi_1(M,m_0)$ only naturally yields a conjugacy class; nonetheless, as $\mu$ is a homomorphism to the abelian group $U(1)$, the value of $\mu$ on this conjugacy class is well-defined. \end{lem} We now wish to specialize back to the case of $M = T^{\vee}_{\Lambda}$ with its usual $\Gamma$-action; $\mu(\ell_{\gamma,p})$ now may be written $\mu\Big((1 - \gamma)p\Big)$. In particular, applying $\chi_{\gamma,p}$ to~\eqref{eq:globaldecomp} for any nontrivial $\gamma$ yields 
\beq\label{eq:charfun} 0 = \sum_{\mu \in T_{\Lambda}^{\gamma}} \chi_{\gamma}(V_{\mu}) \mu\Big((1-\gamma)p\Big)\eeq 
for all $p \in (T_{\Lambda}^{\vee})^{\gamma}$, where here $\chi_{\gamma}(V_{\mu}) := \mathrm{Tr}(\gamma|_{V_{\mu}})$ is the usual character of $\gamma \in \Gamma_{\mu}\acts{V_{\mu}}$. Note in particular that each $\mc{V}_{\mu}$ is simply a direct sum of several lines $\mc{L}_{\mu,\rho}$ of the form in Lemma~\ref{lem:geneq} for varying characters $\rho$.  
\begin{lem}\label{lem:dft} 
Suppose a finite group $\Gamma$ acts on a torus $T_{\Lambda}$ such that all nontrivially-stabilized points are isolated. Then, for any $\gamma \in \Gamma\setminus\{1\}$, we have the following perfect pairing $\langle -, - \rangle_{\gamma}$ of finite abelian groups: \begin{align*} (T_{\Lambda})^{\gamma} \times (T_{\Lambda}^{\vee})^{\gamma} &\to U(1) \\ (\mu, p) &\mapsto \mu\Big((1 - \gamma)p\Big)\,. \end{align*} More generally, without assuming that nontrivially-stabilized points are isolated, the same result is true but as a perfect pairing of the finite abelian groups $\pi_0((T_{\Lambda})^{\gamma}) \times \pi_0((T^{\vee}_{\Lambda})^{\gamma}) \to U(1)$. 
\end{lem}
\bp It is an immediate check to verify that the above pairing is well-defined and bilinear as well as to reduce to case of isolated nontrivially-stabilized points by considering $\langle \gamma \rangle$ acting on $T_{\Lambda}$ modulo the identity component of $(T_{\Lambda})^{\gamma}$; the statement that we have a perfect pairing may then be restated as follows. Namely, if $L \supset \Lambda^{\vee}$ denotes the superlattice defined by $$L := \{\ell \in \Lambda_{\mb{R}}^* \bigm| \gamma \ell \equiv \ell\bmod{\Lambda^{\vee}}\}\,,$$ then we wish to show that the inclusion of lattices $(1-\gamma)L \subset \Lambda^{\vee}$ is in fact an equality. But one may immediately verify this claim by use of Smith normal form; for example, one may calculate that the index of both of the above lattices within $L$ is $\det((\gamma - \mathrm{Id})|_{\Lambda_{\mb{R}}^*})$. \ep

%But this follows from the observations that the natural pairing $\Lambda \times L \to \mb{Z}$ is indeed integrally-valued and that \MZ{needs normalization factors} \beq\label{eq:latticefun} [\Lambda : (\gamma - 1)\Lambda] = \det((\gamma - \mathrm{Id})|_{\Lambda_{\mb{R}}}) = \det((\gamma - \mathrm{Id})|_{\Lambda_{\mb{R}}^*}) = [L : \Lambda^{\vee}]\,.\eeq \ep

Now,~\eqref{eq:charfun} and Lemma~\ref{lem:dft} yield that for any $\gamma \in \Gamma_{\mu}\setminus\{1\}$, $\chi_{\gamma}(V_{\mu}) = 0$. Character theory then yields that $V_{\mu}$ must be some number of copies of the regular representation of $\Gamma_{\mu}$, at which point $\dim V_{\mu} \ge |\Gamma_{\mu}|$ immediately yields condition (v) of Proposition~\ref{prop:eqconnequiv} (and equality in the aforementioned inequality). \ep

\begin{proof}[Proof of Proposition~\ref{prop:orbFlat2}] Let us abbreviate $Y = \{(\sigma,\phi_0) \bigm| \sigma$ is the monodromy representation of an equivariant $H^s$-regular flat connection on $\mc{R}_{\Gamma}$ and $\phi_0 \in (\mf{u}(R_{\Gamma}) \otimes \Lambda_{\mb{R}}^{\perp})^{\Gamma}$ is $\sigma$-invariant with $[\phi_0,\phi_0] = 0\}$. We first claim that for any $(\sigma,\phi_0) \in Y$, $\mathrm{ev}(\sigma,-i\phi_0) \in ((\mathrm{Sym}^{|\Gamma|}(Q/\Lambda))^{\Gamma})^0$. Indeed, either all the $\phi_0$ identically vanish, or they do not. In the latter case, there will exist some eigenvalue $\mu \in Q/\Lambda \simeq T_{\Lambda} \times \Lambda_{\mb{R}}^{\perp}$ whose projection to $\Lambda_{\mb{R}}^{\perp}$ is nonzero, but then such $\mu$ has trivial $\Gamma$-stabilizer and hence immediately produces a full $\Gamma$-orbit. In the former case, $\mathrm{ev}(\sigma,-i\phi_0)$ simply collapses to $\mathrm{ev}(\sigma)$ and we conclude via the combination of Propositions~\ref{prop:eqconnequiv} and~\ref{prop:goodeq}.

So, given the above claim, we have a map $Y \stackrel{\mathrm{ev}}{\to} ((\mathrm{Sym}^{|\Gamma|}(Q/\Lambda))^{\Gamma})^0 \stackrel{\sim}{\to} Q/\hat{\Gamma}$ which is manifestly $U(R_{\Gamma})^{\Gamma}$-invariant, i.e., we have a map we continue to denote \beq\mathrm{ev}\colon Y/U(R_{\Gamma})^{\Gamma} \to Q/\hat{\Gamma}\,.\eeq
Surjectivity of this map is clear via the consideration of constant connections $a$ and functions $\phi$ valued in the diagonal Cartan $\mf{t}$. Injectivity also follows from a reduction to Cartan-valued operators: by Propositions~\ref{prop:eqconnequiv} and~\ref{prop:goodeq}, we may immediately apply equivariant gauge transformations to assume that our connections are equivariant, constant, and $\mf{t}$-valued. It is a similar, but easier, argument to then apply a further constant $U(R_{\Gamma})^{\Gamma}$ gauge transformation to assume that any $\phi_0$ is $\mf{t}$-valued. But now, given two such pieces of data with the same eigenvalues (with multiplicity), these data only differ by a permutation induced by the action of right-multiplication of some $\gamma \in \Gamma$ on the standard basis of $R_{\Gamma}$; in particular, they differ by the action of a further constant equivariant gauge transformation. \end{proof}

\begin{rmk} Readers familiar with either the math or physics literature on T-duality may recognize that a large part of the above proofs may be packaged as a computation of the effect of T-duality on the relevant K-theory lattices of dual torus \emph{orbifolds}. Indeed, we have a topological T-duality isomorphism $K_{\Gamma}(T_{\Lambda}) \simeq K^{\tau}_{\Gamma}(T^{\vee}_{\Lambda})$, where the $\tau$ on the right side indicates the possible twisting in the case of an affine action~\cite{GT}. If one complexifies the above isomorphism, there is a natural decomposition of both sides into twisted sectors labelled by the conjugacy classes of $\Gamma$ via the formalism of \emph{delocalized} equivariant K-theory (or Chen-Ruan orbifold cohomology, cohomology of the inertia stack, or so on), and one may ask for a more direct understanding of the topological T-duality isomorphism with respect to these decompositions. For example, the untwisted sectors $\mathrm{H}^{*}(T_{\Lambda};\mb{C})^{\Gamma}, \mathrm{H}^{*}(T^{\vee}_{\Lambda};\mb{C})^{\Gamma}$ corresponding to the identity conjugacy class are canonically isomorphic by restricting to the $\Gamma$-invariance of the usual (non-equivariant) T-duality isomorphism induced by the Poincar\'e line. One way to parse the above is then that the twisted sectors on either side (i) naturally organize into our spaces of FI parameters computed in~\S\ref{subsec:FIparam} and (ii) are mapped to each other by a \emph{discrete} Fourier transform induced by the perfect Weil-type pairing of Lemma~\ref{lem:dft}. This discrete Fourier transform will regularly appear in the geometric work to follow; for example, in the present paper, we will see that we naturally have two (triples of) FI parameters $\xi_0$ and $\xi_L$ associated to the two fixed points of $Z_2\acts{S^1}$, but that it is $\xi_0 \pm \xi_L$ which control whether the associated moduli space of instantons is singular (and more generally, yield the periods, after suitable normalization, of the resolved $2$-spheres of said singularities). 

One interpretation of the proof above then is that we are computing the equivariant Chern characters for the nontrivially twisted sectors in the delocalized decomposition; alternatively, one could pass through the T-duality to arrive at an equality in $K_{\Gamma}(T_{\Lambda})$ of the form \beq\label{eq:eqK}\iota_*[\mb{C}] = \sum_{[\mu] \in T_{\Lambda}/\Gamma} (\iota_{\mu})_* [V_{\mu}]\,,\eeq where $\iota_{\mu}$ denotes the inclusion of orbifolds $\{\mu\} / \Gamma_{\mu} \to T_{\Lambda}/\Gamma$ and $\iota$ the inclusion of any stabilizer-free point, which once again immediately yields that the $\mu$ comprise a full $\Gamma$-orbit. Note that $K_{\Gamma}(T_{\Lambda}) \simeq K_{\hat{\Gamma}}(Q)$, and one may interpret the analysis of the previous section as taking place in (an operator-theoretic model for) this $\hat{\Gamma}$-equivariant K-theory above. \end{rmk}

\begin{rmk}\label{rmk:fracD0}
For physically-inclined readers, the essence of condition (v) of Proposition~\ref{prop:eqconnequiv} and of the ensuing Proposition~\ref{prop:goodeq} is to confirm that we have a single integral D0 brane probing the orbifold $T_{\Lambda}/\Gamma$ rather than some more complicated configuration of fractional D0 branes frozen at nontrivially-stabilized subloci of $T_{\Lambda}$; the above is also a direct physical translation of~\eqref{eq:eqK}. This point of view also makes it clear why the subtleties encountered in these past two sections were absent in Kronheimer's original case of $r = 0$: there, $\Gamma\acts{\mb{C}^2}$ only had one fixed (or nontrivially stabilized) point, and so there was no possibility of fractional D0 branes frozen at different fixed points. Mathematically phrased, condition (v) of Proposition~\ref{prop:eqconnequiv} is automatic in the case $r = 0$.
\end{rmk}

It is reasonable in the present generality to consider the question of connectedness for moduli spaces of instantons on equivariant bundles $\mc{R}_{\Gamma}$ other than the trivial bundle considered here -- i.e., in the language of Remark~\ref{rmk:conninj} to ask whether the relevant map on $\pi_0$ is generally injective as opposed to merely principally injective. We will not pursue these questions further here. 

\subsection{The main conjecture}\label{subsec:mainconj}

We are finally in position to conjecture some benefits of our infinite-dimensional generalization of Kronheimer's construction. Namely, we now have a principled method by which we hope to systematically construct resolutions of flat hyperkahler orbifolds; i.e., as hyperkahler quotients of the datum of Proposition~\ref{prop:hktime} for the general moment map parameters $\vec{\xi} \in \mf{z}^3$ of Theorem~\ref{thm:FI} and as deforming the flat orbifold of Theorem~\ref{thm:noFI}. The constructions of the penultimate part of Conjecture~\ref{conj:mainconj} below are yet to be introduced in detail; we'll remark on the period maps after stating the conjecture.

\begin{conj}\label{conj:mainconj} Given $\hat{\Gamma}\acts{Q}$ for $Q \simeq \mb{R}^4$ and $\Gamma \hookrightarrow Sp(1)$ nontrivial as above, there exists the following: \begin{enumerate}[(i)] \item a Hilbert Lie group $\mc{G}_{\hat{\Gamma}}$ with $H^{s+1}_{c,\Gamma}(T^{\vee}_{\Lambda}\setminus F, PU(R_{\Gamma})) \subset \mc{G}_{\hat{\Gamma}} \subset H^{s+1}_{\mathrm{loc},\Gamma}(T^{\vee}_{\Lambda} \setminus F, PU(R_{\Gamma}))$, \item a Hilbert space $\mc{A}_{\hat{\Gamma}}$ with $\mathrm{Conn}^s_{c,\Gamma}(T^{\vee}_{\Lambda}\setminus F, R_{\Gamma}) \times H^s_{c,\Gamma}(T^{\vee}_{\Lambda}\setminus F, \mf{su}(R_{\Gamma})) \subset \mc{A}_{\hat{\Gamma}} \subset \mathrm{Conn}^s_{\mathrm{loc},\Gamma}(T^{\vee}_{\Lambda} \setminus F,R_{\Gamma}) \times H^s_{\mathrm{loc},\Gamma}(T^{\vee}_{\Lambda} \setminus F,\mf{su}(R_{\Gamma}))$, \item a family of flat, hyperkahler affine spaces $\mc{A}_{\hat{\Gamma}}(\vec{\xi})$ parametrized by $\vec{\xi} \in \mf{z}^3$ such that $\mc{A}_{\hat{\Gamma}}(\vec{\xi})$ is a torsor for $\mc{A}_{\hat{\Gamma}}$ for all $\xi$, and \item a family of tri-Hamiltonian actions $\mc{G}_{\hat{\Gamma}}\acts{\mc{A}_{\hat{\Gamma}}(\vec{\xi})}$ specializing at $\vec{\xi} = 0$ to an action compatible with the action $\mc{G}^{s+1}_{\hat{\Gamma}}\acts{\mc{A}^s_{\hat{\Gamma}}}$ of Proposition~\ref{prop:hktime}, i.e., as agreeing on the intersection $\mc{G}^{s+1}_{\hat{\Gamma}} \cap \mc{G}_{\hat{\Gamma}}$ acting on $\mc{A}^s_{\hat{\Gamma}} \cap \mc{A}_{\hat{\Gamma}}$ \end{enumerate} such that this data satisfies the following: \begin{enumerate}[(i)]\setcounter{enumi}{4}\item as $\vec{\xi} \in \mf{z}^3$ and $\Lambda \subset Q$ vary, the family of hyperkahler quotients $\mc{A}_{\hat{\Gamma}}(\vec{\xi})\hq\mc{G}_{\hat{\Gamma}}$ is a full-dimensional family of hyperkahler orbifolds within the deformation space of $Q/\hat{\Gamma}$ as a hyperkahler orbifold, \item the action of $\mc{G}_{\hat{\Gamma}}$ on $\mc{A}_{\hat{\Gamma}}(\vec{\xi})$ extends to a holomorphic action of a complexified Lie group $(\mc{G}_{\hat{\Gamma}})_{\mb{C}}$, \item there exists a DUY isomorphism $\mc{A}_{\hat{\Gamma}}(\vec{\xi})\hq \mc{G}_{\hat{\Gamma}} \stackrel{\sim}{\to} \mu_{\mb{C}}^{-1}(\xi_{\mb{C}})^{\mathrm{ps}}/(\mc{G}_{\hat{\Gamma}})_{\mb{C}}$, \item given any $\vec{\xi} = (\xi_{\mb{C}},\xi_{\mb{R}})$ and $\vec{\xi}' = (\xi_{\mb{C}}, 0)$, there exists a holomorphic map $\mc{A}_{\hat{\Gamma}}(\vec{\xi})\hq\mc{G}_{\hat{\Gamma}} \to \mc{A}_{\hat{\Gamma}}(\vec{\xi}')\mathord{/ \!\! / \! \! /_{\xi'}}\mc{G}_{\hat{\Gamma}}$ which is always a biholomorphic isomorphism when $\xi_{\mb{C}}$ is general\footnote{i.e., for an open dense locus in $\mf{z}_{\mb{C}}$}, \item for $r < 4$, the period maps of the K\"ahler forms for the family of hyperkahler manifolds $\mc{A}_{\hat{\Gamma}}(\vec{\xi})\hq \mc{G}_{\hat{\Gamma}}$ are linear in $\vec{\xi}$, and \item the normalizer of the gauge group $N(\mc{G}_{\hat{\Gamma}})$ acts on $\bigcup_{\vec{\xi} \in \mf{z}^3} \mc{A}_{\hat{\Gamma}}(\vec{\xi})$ over the diagonal action of $W(\mc{G})$ on $\mf{z}^3$, and so the total space of the hyperkahler family $\bigcup_{\vec{\xi}\in\mf{z}^3} \mc{A}_{\hat{\Gamma}}(\vec{\xi})\hq\mc{G}_{\hat{\Gamma}}$ admits a $W(\mc{G})$-action. \end{enumerate} 
\end{conj}

Let us first explain the main idea of the above conjecture, along with where the main difficulties lie as opposed to which statements should then follow formally. First, why introduce the family of spaces $\mc{A}_{\hat{\Gamma}}(\vec{\xi})$? Similarly, why allow for modifications in the regularity conditions about $p \in F$ for $\mc{G}_{\hat{\Gamma}}$, $\mc{A}_{\hat{\Gamma}}$?

While the deformed moment map equations $\mu_{\mb{C},\mb{R}} = \sum (\xi_i)_{\mb{C},\mb{R}} \delta(p_i)$ for $(\xi_i)_{\mb{C},\mb{R}} \in \mf{z}_{p_i}^3$ are perfectly well-formed equations, they admit no solutions within $\mc{A}^s_{\hat{\Gamma}}$ for any $s > r/2$. Indeed, a local analysis shows any solutions exhibit singular behavior at the points $p$, with this behavior becoming worse for higher $r$. Hence, it becomes a delicate analysis question to find the correct function space in which to solve this singular PDE. As we indicate in the $r = 1$ example of this paper, this function space will consist of $\vec{\xi}$-dependent singular terms plus some terms of higher regularity; hence, one arrives at the family of affine spaces $\mc{A}_{\hat{\Gamma}}(\vec{\xi})$ considered above. %Similarly, we will find it natural \MZ{it may be necessary? after all we didn't need that in this paper and we shouldn't for $r=2$ i believe, although not positive about that, but don't see why we should commit} to allow for modification in the regularity of the gauge transformations about $p \in F$.

%Note that the statement of the conjecture may be slightly loose in terms of the precise functional-analytic details. For example, it is of no great concern for reasons of simplicity, the spaces $\mc{A}_{\hat{\Gamma}}(\vec{\xi})$ are torsors over some space with Sobolev regularity modelled on $W^{s,p}$-regular equivariant connections and functions for some $p \ne 2$. Even in the $r = 1$ case we treat in the following sections of the present paper, we will find it prudent to work with slightly different $\mc{G}$ and $\mc{A}$ in terms of the number of conditions imposed at the $p_i$ as compared to the precise constructions of the above conjecture. 

Let us briefly indicate the main idea of (viii) above, as it will also elucidate the key dimension-dependence of this conjecture. The holomorphic map conjectured here generalizes the construction of~\cite[Proposition 3.10]{kronheimer:construct}, and one would like to construct it in general by simply using the DUY isomorphism of (vii) above to, schematically, perform the composition $$\mu^{-1}(\vec{\xi})/\mc{G}_{\hat{\Gamma}} \to \mu_{\mb{C}}^{-1}(\xi_{\mb{C}})/(\mc{G}_{\hat{\Gamma}})_{\mb{C}} \to \mu^{-1}(\vec{\xi}')/\mc{G}_{\hat{\Gamma}}\,.$$ If the above were valid, this construction would immediately yield (viii) modulo some care tracking the $\xi^{\mb{R}}$-dependent polystability condition. There is, of course, a wrinkle here: as just observed, the ambient spaces $\mc{A}_{\hat{\Gamma}}(\vec{\xi})$ themselves are $\xi$-dependent. One hence needs an additional step to map from $\mc{A}_{\hat{\Gamma}}(\vec{\xi})$ to $\mc{A}_{\hat{\Gamma}}(\vec{\xi}')$ in the first place if one wishes to carry out the above strategy. In good cases, such a map may be effected by a \emph{formal} gauge transformation that changes the local singularity behavior about the special locus $F$; we carry this out below for $r = 1$ and will do so in~\cite{mz:ALG} for $r = 2$, where the formal gauge transformation that does so is analytically well-behaved. The analogous formal gauge transformations for $r = 3, 4$ would have exponential blow-up, however. This is one indication of the increased analytic difficulty\footnote{For experts, another important difference between the low- and high-dimensional regimes runs as follows. As we will illustrate in later sections, given a solution of the deformed moment map equations, one can linearize the moment map equations about this solution in order to learn about the local structure of the moduli space. In tandem with an appropriate gauge fixing condition, this yields a linear PDE which is formally elliptic, but its coefficients are singular at the points of $F$. For $r\le 2$, there are existing formalisms that make manifest the existence of a parametrix for this linear operator, in spite of its singular coefficients; e.g., for $r=2$ this is Melrose's $b$-calculus \cite{melrose:b,melrose:b2}; see also \cite{LockhartMcOwen,MazzeoEdge}. For $r > 2$, however, the relevant technology is Melrose's scattering calculus \cite{melrose:scat0,melrose:scat}, and the operators one obtains are not fully elliptic in the appropriate sense for this calculus. Hence, the construction of a parametrix, an associated Fredholm theory, etc., require the development of new analytic tools.} of the \emph{high-dimensional regime} (namely, $r > 2$) of Conjecture~\ref{conj:mainconj}.

Finally, given a locally trivial family of manifolds, as is always the case if the manifolds are compact by Ehresmann and should be the case for the family $\mc{A}_{\hat{\Gamma}}(\vec{\xi})\hq \mc{G}_{\hat{\Gamma}}$ over the open, dense locus $U \subset \mf{z}^3$ of parameter space where the fibers are manifolds, one obtains a canonical flat connection on $\mathrm{H}_2(-; \mb{Z})$. Let us denote this lattice as $\mathrm{H}$. Integrating the forms $(\omega_\CC, \omega)$ hence yields, at least locally, well-defined maps $U \to \mathrm{Hom}(\mathrm{H},\mb{R})^{\oplus 3}$. It is these maps that we are asking to be linear in the $\vec{\xi}$; suitably interpreted (as $\Lambda \subset Q$ varies as well), the slope of this linear map will produce an identification between the relevant linear spaces. As motivation for this period linearity conjecture, we note that Duistermaat-Heckman~\cite{heckman:FIlinear} and Atiyah-Bott~\cite{atiyah:equiv} showed that periods are linear in moment map parameters for finite-dimensional (hyper)kahler quotients, and so in particular this is the case for Kronheimer's construction of ALE manifolds \cite[Prop. 4.1(i)]{kronheimer:construct}. We will also demonstrate it for the $D_2$ ALF manifold at the end of \S\ref{sec:complex}; as we will explain in \S\ref{sec:prelim}, this case can be reformulated as a finite-dimensional hyperkahler quotient of a non-affine space, and so the above general results apply, but our proof generalizes to the intrinsically infinite-dimensional $r=2$ case, as we will demonstrate in \cite{mz:rels}. This period linearity conjecture suggests that hyperkahler quotient constructions should yield a powerful means of proving surjectivity of period maps.

Note the caveat that $r < 4$ here; for $r = 4$, Hodge theory for the compact K3 manifold will imply that the triple of K\"ahler forms stays a (scaled) orthonormal basis of the self-dual part of $\mathrm{H}$ for all $\xi$, and linearity will fail if interpreted literally as above. We nevertheless expect an appropriate version of period-linearity will hold for $r = 4$. %A natural conjecture is that in any given complex structure, if one fixes $\xi^\CC$ and varies $\xi^\RR$, then the periods of the K\"ahler form still vary linearly and the complex structure does not change. (If one wishes to be able to hyperkahler rotate the K\"ahler forms into each other, then this conjecture means that one should simultaneously rescale the holomorphic 2-form so that it has the same norm as the K\"ahler form.) 
We leave fuller discussion of all these points to later work on these period maps~\cite{mz:HFA}.

\renewcommand{\thesubsection}{\thesection.\Alph{subsection}}
\setcounter{subsection}{0}
\subsection{Relation to the derived McKay correspondence}\label{subsec:sheaf}

The goal of Conjecture~\ref{conj:mainconj} is to produce a family of hyperkahler manifolds. While the complex formulation offers strictly less information, as it only produces a holomorphic symplectic manifold, it is useful for comparing to constructions in holomorphic (or algebraic) geometry. For example, see~\cite{CharbonneauHurtubise} for a pleasant (and not unrelated) overview of how the Nahm transform is related to the Fourier-Mukai transform of holomorphic geometry via DUY isomorphisms. We hence sketch in this appendix subsection some of the holomorphic aspects of Conjecture~\ref{conj:mainconj}. This will lead us to speculate on a number of relationships that this conjecture likely has with other constructions of hyperkahler manifolds.

Let us specialize, for sake of concreteness, to the case $r = 4$ with $\Gamma = Z_2$ acting by negation on the complex four-torus $T := \mb{C}^2/\Lambda$ (which has, until now, been denoted by $T_\Lambda$). We suppose that $T$ is equipped with a polarization, i.e., that it carries the structure of a polarized abelian surface. This restriction is for simplicity, as we will make reference to theorems in the algebraic geometry literature. The dual torus $\hat{T}$ (which we have elsewhere denoted by $T^\vee_\Lambda$) is now also an abelian surface; recall its construction as the moduli space of degree-zero line bundles on $T$. Let $S$ and $\hat{S}$ denote the Kummer K3 surfaces of $T$ and $\hat{T}$, respectively. In other words, denote by $S$ the canonical crepant resolution of $T/\Gamma$, and similarly for $\hat{S}$. Consider now the following square: $$\begin{tikzcd} S \ar[leftrightarrow, dashed, r] \ar[leftrightarrow, dashed, d] & \hat{S} \ar[leftrightarrow, dashed, d] \\ T/\Gamma \ar[leftrightarrow, dashed, r] & \hat{T}/\Gamma \end{tikzcd}$$ with the dashed arrows indicating \emph{some} sort of correspondence. To fix ideas, one could interpret each dashed arrow above as an equivalence of (bounded, derived) categories of coherent sheaves, where one takes the bottom row to consist of the $\Gamma$-equivariant category of sheaves on $T, \hat{T}$; then, the vertical arrows would be given by derived McKay correspondences~\cite{bridgeland:derivedMcKay} while the horizontal arrows would be given by Fourier-Mukai correspondences~\cite{MukaiFM}.\footnote{Note that the bottom horizontal arrow is in fact an \emph{equivariant} Fourier-Mukai correspondence~\cite{KrugSosna} induced by the natural $\Gamma$-equivariant structure on the Poincar\'e line bundle~\cite{MumfordAV} on $T \times \hat{T}$.}

We are here interested not just in equivalences of derived categories but rather in isomorphisms of moduli spaces of objects therein cut out by fixed numerical invariants and appropriate stability conditions. In the present context, we would particularly like to compare the following three moduli spaces: (i) $\Gamma$-equivariant rank-two sheaves on $\hat{T}$, (ii) $\Gamma$-equivariant length-two zero-dimensional subschemes on $T$, and (iii) structure sheaves of points on $S$. The correspondence between (ii) and (iii) is essentially tautological, as it is synonymous with the construction of $S$ as Nakamura's $\Gamma$-Hilbert scheme of points~\cite{Nakamura} on $T$ used in the work of~\cite{bridgeland:derivedMcKay}. The correspondence between (i) and (ii) is, of course, Fourier-Mukai, but it may also be worked out completely explicitly. Indeed, analogously to Atiyah's description of bundles on elliptic curves~\cite{atiyah1957vector}, the objects of (i) above are of the following form: (a) $\mc{L} \oplus \mc{L}^{-1}$ for $\mc{L} \in T \simeq \hat{\hat{T}}$ a degree-zero line on $\hat{T}$, or (b) a nontrivial extension of $\mc{L}$ by $\mc{L}^{-1} \simeq \mc{L}$ for $\mc{L}$ one of the 16 $2$-torsion points of $T$. The sub- and quotient-bundles above, as equivariant bundles, carry either trivial or sign representations at each of the 16 2-torsion points of $\hat T$. (The choice of these representations is constrained by the choice of `numerical invariants,' i.e., a $\Gamma$-equivariant K-theory class of $\hat T$. The relevant class may be determined via a variety of approaches \cite{kronheimer:instALE,waldram:lol,wendland:orb}.) Note that in contrast to Atiyah's case, if $\mc{L}$ is two-torsion then the relevant space of extensions \emph{ignoring equivariance data} is $\mathrm{H}^1(\mc{L}, \mc{L}) \simeq \mb{C}^2$ (by Serre duality and Riemann-Roch), so the isomorphism classes of nontrivial extensions are parametrized by a copy of $\mb{CP}^1$. As we will shortly note, choosing generic stability conditions will validate this computation which ignored the equivariance data. In this way, we find that the bundles in case (b) provide the resolved 2-spheres in $S$.

In fact, given the conjectural picture set out by Conjecture~\ref{conj:mainconj}, the full picture in which one may be interested is as follows: $$
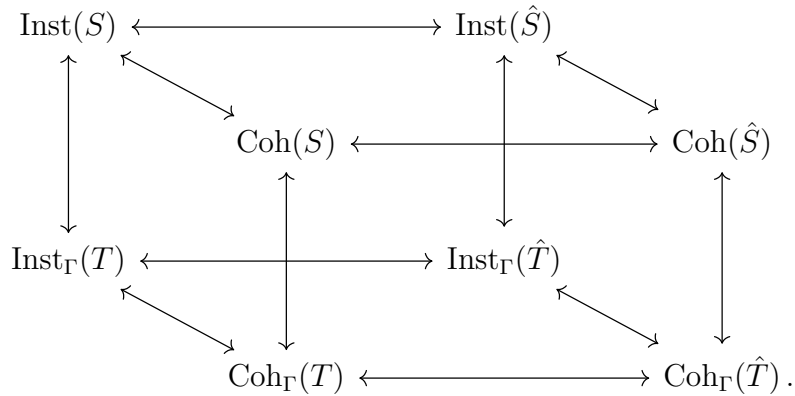
\begin{tikzcd} \mathrm{Inst}(S) \ar[leftrightarrow, rr] \ar[leftrightarrow, dr] \ar[leftrightarrow, dd] & & \mathrm{Inst}(\hat{S}) \ar[leftrightarrow, dr] \ar[leftrightarrow, dd] \\ & \mathrm{Coh}(S) \ar[leftrightarrow, rr] \ar[leftrightarrow, dd] & & \mathrm{Coh}(\hat{S}) \ar[leftrightarrow, dd] \\ \mathrm{Inst}_{\Gamma}(T) \ar[leftrightarrow, rr] \ar[leftrightarrow, dr] & & \mathrm{Inst}_{\Gamma}(\hat{T}) \ar[leftrightarrow, dr] \\ & \mathrm{Coh}_{\Gamma}(T) \ar[leftrightarrow, rr] & & \mathrm{Coh}_{\Gamma}(\hat{T})\makebox[0pt][l]{\,.} \end{tikzcd}$$

Schematically, in the above, we mean by $\mathrm{Coh}(-)$ \emph{some} moduli space of sheaves with fixed numerical invariants (Mukai vector) and stability condition, and by $\mathrm{Inst}(-)$ some space of instantons. The $\Gamma$ subscripts indicate equivariant sheaves or instantons, respectively. Then, one may now hope for correspondences between all the moduli spaces above where the correspondences on the front (coherent) face of the cube are those described previously, the diagonal arrows are DUY isomorphisms, and the horizontal arrows on the back face of the cube are Nahm transforms. The appropriate extension of Conjecture~\ref{conj:mainconj} to this setting now runs as follows: choose $\xi^{\mb{C}} \equiv 0$ for all $16$ points of $F \subset \hat{T}$, and choose $\xi^{\mb{R}}$ generic\footnote{As we are working in the complex formulation in this appendical subsection, the precise value of $\xi^{\mb{R}}$ is immaterial as long as it is generic -- although there is some chamber dependence that enters in the construction of $S$ as the moduli of $\Gamma$-clusters on $T$ as opposed to a slightly different choice of $\Gamma$-equivariant objects.} at all 16 points. Then the K3 manifold produced by the singular equivariant instantons of $\mathrm{Inst}_{\Gamma}(\hat{T})$ agrees with $S$ as an underlying holomorphic symplectic manifold, and this isomorphism is suitably compatible with the correspondences sketched above. In particular, there is a natural space of stability conditions on the Mukai lattice of $\hat{T}/\Gamma$, i.e., on $\mathrm{K}_{\Gamma}(\hat{T})$, given by pulling back the slope stability condition on $\hat{S}$ with varying K\"ahler moduli of the resolved 2-spheres; this space is in canonical isomorphism with the FI parameter space $\mf{z}$ with basis given by the $2$-torsion points of $\hat{T}$.\footnote{Indeed, varying these parameters will change whether it is the sub- or the quotient-line at each $2$-torsion point which carries the sign representation in the stable objects of type (b) above.} There are also degrees of freedom for the stability condition associated to varying the K\"ahler moduli of $\hat T$.

Note in the above cube of comparisons that the vertical arrows of the ``back face'' predict \emph{differential geometric} versions of the McKay correspondence; e.g., a correspondence between instantons on $\hat{S}$ and singular $\Gamma$-equivariant instantons on $\hat{T}$. This would even be of interest in the ALE case studied by Kronheimer and Nakajima~\cite{kronheimer:instALE}. We expect that this correspondence of instantons should generalize straightforwardly to the case of $\xi^\CC\not=0$.

%As the entirety of this section either (i) is already known holomorphic or algebraic geometry or (ii) rests on the already conjectural description of Conjecture~\ref{conj:mainconj}, we will refrain from offering further details in this or in the more general cases of $\hat{\Gamma}\acts{Q}$-actions, at least until one can more precisely describe the behavior of the singular equivariant instantons themselves.

% Baranovsky-Ginzburg-Kuznetsov, Crawley-Boevey-Holland

% Serre, prolongement -- http://www.numdam.org/article/AIF_1966__16_1_363_0.pdf -- or see Siu's papers in 69? Ok. But if only real parameters are on, Koszul-Malgrange seems to say we have an honest rank 2 holomorphic bundle. And ... why do we think that means there's a formal gauge transformation doing something, again? 

\subsection{Asymmetric orbifolds}\label{subsec:asymorb}

In this subsection, we explain the modifications to the general setup of \S\ref{subsec:setup} when the action of $\hat\Gamma$ on $Q$ is affine. The upshot will be that with only mild changes to the definitions, we again obtain a hyperkahler quotient construction of $T_\Lambda/\Gamma$ (and, conjecturally as per Conjecture~\ref{conj:mainconj}, of its deformations). In the case of an affine action, there is still a dual action of $\Gamma$ on $T^\vee_\Lambda$ as defined via the map $\Gamma \to O(Q)$, but this map forgets the translation data from the original action of $\Gamma$ on $T_\Lambda$. This additional data will instead be encoded via some interesting gauge-theoretic notions. The title of this section is taken from the physics literature, where one says that we are studying instantons on \emph{asymmetric orbifolds}; as we will see, mathematically, there is simply a more general context for studying instantons which need not have any relation to orbifold moduli spaces or Fourier duality. 

The key observation is that, while we no longer have the map $\tilde S\colon \hat\Gamma\to \Lambda$ introduced above Definition \ref{def:originalDefs}, we have an immediate analog in $f\colon \hat\Gamma\to \Lambda_\RR$ as defined by $f(\hat\gamma) := \hat\gamma(0)$, i.e., the result of acting on $0\in \Lambda_\RR$ with $\hat\gamma$.\footnote{Although not relevant for our current presentation, it is pleasant to note that if $\overline{f}\colon\Gamma\to T_{\Lambda}$ denotes the canonical descent of $f$, then $\overline{f}$ is an explicit crossed morphism realizing the class in $\mathrm{H^1}(\Gamma;T_{\Lambda}) \simeq \mathrm{H^2}(\Gamma;\Lambda)$ representing the original extension class of $\hat{\Gamma}$ as an extension of $\Gamma$ by $\Lambda$.} The $\hat\Gamma$ action on $\mf{u}^{-\infty}(R_{\hat{\Gamma}}) \otimes_{\mb{R}} Q$ from Definition \ref{def:originalDefs} then generalizes immediately, and therefore so do the definitions of $\A^s_{\hat\Gamma}$ and $\tilde\G^s_{\hat\Gamma}$. In particular, $d\in \mf{u}^{-\infty}(R_{\hat\Gamma})\underset{\mb{\RR}}{\otimes} Q$ maps $\iota(\tilde\gamma)$ to $i \, \iota(\tilde\gamma)\otimes f(\tilde\gamma)$. We next generalize \eqref{eq:Ldefn}:
\be \label{eq:twistL}
L(g)^{\gamma_1,\gamma_2} := \sum_{\hat{\gamma}_1 \in \pi^{-1}(\gamma_1)} \langle {\iota(\hat\gamma}_1), g \iota(\hat{\gamma}_2) \rangle \exp(i \brackets{f(\hat{\gamma}_1) - f(\hat{\gamma}_2) } \cdot \vec\theta) \ . \ee
Here, $\hat\gamma_2 \in \hat\Gamma$ is a fixed choice of lift of $\gamma_2\in \Gamma$; this choice is immaterial, thanks to the $\Lambda$-invariance constraint on $g$. Similarly, the $\Lambda$-invariance constraint on $A\in \A^{-\infty}_{\hat\Gamma}$ implies that we may write it as $A=d+a$, where $a$ satisfies $\avg{\iota(\lambda\hat\gamma_1), a\iota(\lambda\hat\gamma_2)} = \avg{\iota(\hat\gamma_1), a\iota(\hat\gamma_2)}$ for all $\lambda\in \Lambda$. Having done so, using
\be \label{eq:twistHom} f(\hat\gamma\hat\gamma_1) = \gamma f(\hat\gamma_1) + f(\hat\gamma) \ee
we find that $a$ satisfies $\avg{\iota(\hat\gamma\hat\gamma_1), a \iota(\hat\gamma \hat\gamma_2)} = \gamma \avg{\iota(\hat\gamma_1), a \iota(\hat\gamma_2)}$ for all $\hat\gamma\in \hat\Gamma$. This leads us to define $L(A)$ analogously to \eqref{eq:twistL}.

The first minor complication in the present setting, compared to the linear case, is that $L(g)$ is no longer a function from $T^{\vee}_{\Lambda}$ to $U(R_{\Gamma})$; instead, it now acquires phases upon traversing the cycles of $T^{\vee}_{\Lambda}$. This phase structure is not difficult to describe: $L(g)$ is a unitary automorphism of the Hermitian vector bundle $\R_\Gamma := \Lambda^*_\RR \times R_\Gamma / \Lambda^\vee$ over $T^\vee_\Lambda$, where $\lambda^\vee\in \Lambda^\vee$ acts by translation on $\Lambda_\RR^*$ and multiplies $\iota(\gamma)\in R_\Gamma$ by the phase $\exp(i f(\hat\gamma) \cdot \lambda^\vee)$; here, $\hat\gamma\in \hat\Gamma$ is an arbitrary lift of $\gamma\in \Gamma$, and this phase is independent of this choice. Similarly, a point $A\in \A^s_{\hat\Gamma}$ is naturally regarded as a tuple $L(A)$ consisting of a connection on this bundle and an element of $H^s(T^\vee_\Lambda, \End(\R_\Gamma)\otimes \Lambda_\RR^\perp)$. Before we address $\Gamma$-equivariance, we note that we may trivialize this bundle after picking a set-theoretic section $S\colon \Gamma\to \hat\Gamma$ of $\hat{\Gamma}\stackrel{\pi}{\twoheadrightarrow} \Gamma$. Indeed, the monodromy phase structure of the defining basis for $\R_\Gamma$ may be cancelled away by the similarly-monodromizing unitary change of basis which multiplies the section $\iota(\gamma)$ by $\exp(-i f(S(\gamma))\cdot \vec\theta)$. However, while it is pleasant to note that $\mc{R}_{\Gamma}$ is trivializable as a Hermitian vector bundle, we prefer to work below with the defining basis as it is both canonical (i.e., independent of a choice of $S$) and better adapted for the discussion of $\Gamma$-equivariance. We note that although the defining basis does not trivialize $\R_\Gamma$, we may still interpret $d$ as the exterior derivative in this basis, since pulling back $\R_\Gamma$ to $\Lambda_\RR^*$ via the cover $\Lambda_{\mb{R}}^* \twoheadrightarrow \Lambda_{\mb{R}}^*/\Lambda^{\vee}$ yields a trivial product bundle which carries a canonical connection $d$, and this connection is $\Lambda^{\vee}$-invariant and hence descends to a well-defined $d$ on $\mc{R}_{\Gamma}$.

We now turn to the subject of $\Gamma$-equivariance. We discover a second minor complication: na\"ively, we expect that the equivariant structure on the bundle relates $\iota(\gamma_1)$ in the fiber over $\vec\theta\in T^\vee_\Lambda$ to $\iota(\gamma\gamma_1)$ in the fiber over $\gamma \vec\theta$, but these vectors have different monodromies and so cannot be related so straightforwardly. Instead, we introduce a diffeomorphism of $\R_\Gamma$ that covers the action of $\gamma$ on $T^\vee_\Lambda$, but which relates fibers not just by $j(\gamma)$, but rather by the product $m(\gamma) j(\gamma)$ of the permutation matrix $j(\gamma)$ with a monodromizing function $m(\gamma) := \exp(-if(S(\gamma))\cdot\vec\theta)$, where, as above, $S$ is a set-theoretic section. The monodromy of $\exp(-if(S(\gamma))\cdot\vec\theta) \iota(\gamma\gamma_1)$ about a cycle $\gamma \lambda^\vee\in \Lambda^\vee$ is $\exp(i [-f(S(\gamma)) + f(\hat\gamma \hat\gamma_1)]\cdot \gamma \lambda^\vee)$, where $\hat\gamma,\hat\gamma_1$ are, respectively, arbitrary lifts of $\gamma,\gamma_1$. Using \eqref{eq:twistHom} and $\exp(-i f(S(\gamma))\cdot \gamma\lambda^\vee) = \exp(-i f(\hat\gamma)\cdot\gamma \lambda^\vee)$, this monodromy may be re-expressed as $\exp(i \gamma f(\hat\gamma_1)\cdot \gamma \lambda^\vee) = \exp(i f(\hat\gamma_1) \cdot \lambda^\vee)$, which coincides with the monodromy of $\iota(\gamma_1)$ about the cycle $\lambda^\vee$. So, we have a well-defined diffeomorphism of $\R_\Gamma$. However, this association $\gamma \mapsto m(\gamma) j(\gamma)$ does not define a $\Gamma$-action; rather, associativity fails by a $\Gamma$-cocycle and this structure is hence termed a twisted $\Gamma$-equivariant structure. We refer to~\cite{FreedMoore,GT,GKT} for detailed discussions of the twisted equivariance that arises here.\footnote{Similarly, if one wishes to study the holomorphic geometry of~\S\ref{subsec:sheaf} in this context, one finds that the Poincar\'e line bundle inducing the equivariant Fourier-Mukai isomorphism is a twisted equivariant bundle~\cite[\S4.2.1]{GT} yielding a twisted equivariant Fourier-Mukai isomorphism.} We note that the canonical connection $d$ on $\R_\Gamma$ is twisted equivariant, or alternatively that $d$ descends from the canonical connection on the trivial product bundle $\Lambda_{\mb{R}}^* \times R_{\Gamma}$ not just to a connection on the descended bundle over $\Lambda_{\mb{R}}^* / \Lambda^{\vee} \simeq T_{\Lambda}^{\vee}$ but in fact to a (twisted) connection (on the descended \emph{twisted} bundle) over $\Lambda_{\mb{R}}^* / (\Lambda^\vee \rtimes \Gamma) \simeq T_{\Lambda}^{\vee}/ \Gamma$, where the quotient notation above denotes orbifold quotients. This claim may, \emph{a priori}, seem surprising as $m(\gamma)$ is not constant, but recall that the gauge transformation law of a twisted equivariant connection involves a correction term, as per Definition 7.1 of~\cite{Karoubi}. In particular, the space of connections on a twisted bundle is still a torsor over adjoint-valued $1$-forms, and so upon subtracting the canonical twisted equivariant connection $d$ we may still expect $L(A)$ to simply be an equivariant monodromizing adjoint-valued $1$-form.

That being said, the twisted equivariance constraints on $L(A)$ and $L(g)$ are indistinguishable from their untwisted counterparts since $m(\gamma)$ cancels out of adjoint actions. We may verify, as in \eqref{eq:gEquiv}, that these constraints agree with the $\hat\Gamma$-invariance constraints on $A,g$. For $g$, this computation takes the form
\begin{align}
L(g)^{\gamma_1,\gamma_2}(\gamma\vec\theta) &= \sum_{\hat\gamma_1\in \pi^{-1}(\gamma_1)} \avg{\iota(\hat\gamma_1), g\iota(\hat\gamma_2)} \exp\parens{i\brackets{\gamma^{-1}(f(\hat\gamma_1)-f(\hat\gamma_2))}\cdot\vec\theta} \nonumber \\
&= \sum_{\hat\gamma_1\in \pi^{-1}(\gamma_1)} \avg{\iota(\hat\gamma_1), g \iota(\hat\gamma_2)} \exp\parens{i\brackets{f(\hat\gamma^{-1}\hat\gamma_1)-f(\hat\gamma^{-1}\hat\gamma_2)}\cdot\vec\theta} \nonumber \\
&= \sum_{\tilde\gamma_1\in \pi^{-1}(\gamma^{-1}\gamma_1)} \avg{\iota(\tilde\gamma_1), g \iota(\tilde\gamma_2)} \exp\parens{ i \brackets{f(\tilde\gamma_1) - f(\tilde\gamma_2)}\cdot\vec\theta} \nonumber \\
&= L(g)^{\gamma^{-1}\gamma_1,\gamma^{-1}\gamma_2}(\vec\theta) \ . \label{eq:gEquivTwist}
\end{align}
The second equality employed \eqref{eq:twistHom}, while the third equality introduced $\tilde\gamma_2 = \hat\gamma^{-1}\hat\gamma_2$ and made use of $\hat\gamma^{-1} g \hat\gamma = g$. The only caveat worth commenting on is that \eqref{eq:gEquivTwist} should be understood as being formulated in terms of the pullback of $L(g)$ to a gauge transformation of the trivialized bundle $\Lambda_\RR^* \times R_\Gamma$, so that we can make sense of comparing the matrix elements of $L(g)$ at $\vec\theta$ and $\gamma\vec\theta$.

We conclude this subsection by discussing the simple example introduced in Remark~\ref{rmk:crystal}, where $0 \to \Lambda \to \hat{\Gamma} \to \Gamma \to 0$ is given by $0 \to \mb{Z} \stackrel{2}{\to} \mb{Z} \to Z_2 \to 0$, and in particular the relationship with the alternative approach with $\Lambda' = \hat\Gamma$ and $\Gamma'$ trivial. Namely, we abbreviate by $\hat{T}$ and $\hat{T}'$ the dual circles $\Lambda_{\mb{R}}^* / \Lambda^{\vee}$ and $\Lambda_{\mb{R}}^* / (\Lambda')^{\vee}$, respectively, so that the natural map $\hat{T}' \to \hat{T}$ is a degree-$2$ cover. Then $\mc{R}_{\Gamma'}$ is simply the trivial product rank-one line bundle on $\hat{T}'$ while $\mc{R}_{\Gamma}$ is a rank-two bundle on $\hat{T}$ which may be described as the (sheaf) pushforward of $\mc{R}_{\Gamma'}$ under $\hat{T}' \twoheadrightarrow \hat{T}$. Elements of the gauge group in the `unprimed' description are maps from $\hat{T}$ to $U(2)$ whose off-diagonal entries pick up $-1$ signs as the circle is traversed once and which (the $\Gamma$-equivariance condition) commute with $\begin{pmatrix} 0 & 1 \\ 1 & 0 \end{pmatrix}$. It is enjoyable to see directly that this group agrees with $\mathsf{Maps}(\hat{T}', U(1))$; schematically, this isomorphism takes the form $g(\theta) \mapsto \frac{1}{2} \begin{pmatrix} g(\theta) + g(\theta+\pi) & g(\theta) - g(\theta+\pi) \\ g(\theta) - g(\theta+\pi) & g(\theta) + g(\theta+\pi) \end{pmatrix}$. 

\subsection{A remark on topology}\label{subsec:topology}
\renewcommand{\thesubsection}{\thesection.\arabic{subsection}}

Let us specialize in this section to $\Gamma \subset SU(2)$ nontrivial acting on $Q \simeq \mb{R}^4$; it is not difficult to see that in all cases, the corresponding minimal resolution of $Q/\hat{\Gamma}$ is simply-connected. Per the hyperkahler quotient construction of Conjecture~\ref{conj:mainconj}, one may then expect that the Lie group $\mc{G}_{\hat{\Gamma}}$ is connected. Indeed, \emph{a priori}, the simple-connectivity of the quotient only entails that $\pi_0 \mc{G}_{\hat{\Gamma}}$ acts simply-transitively on $\pi_0$ of the space of solutions of the moment map $\mu_{\xi}$ on $\mc{A}_{\hat{\Gamma}}(\xi)$, but it is reasonable to surmise that both these spaces are in fact connected. Note that this is very much in contrast to, e.g., the moduli space of flat $U(1)$ connections on a torus $T$, where the relevant gauge group $C^{\infty}(T,U(1))$ has a lattice of connected components isomorphic to the fundamental group of the flat connection moduli space $T^{\vee}$.

Of course, Conjecture~\ref{conj:mainconj} does not specify the regularity behavior of elements of $\mc{G}_{\hat{\Gamma}}$ about the nontrivially stabilized locus $F$, but topological properties of mapping spaces are insensitive to the regularity (at least for regularity sufficiently larger than the homological degree). Hence, we are led to expect the following:

\begin{conj}\label{conj:top} Given a linear\footnote{Of course, this hypothesis is for simplicity; in the case of an affine action, we mean the appropriate equivariant gauge group formulated in terms of $\mc{R}_{\Gamma}$.} action $\hat{\Gamma}\acts{Q}$ for $Q \simeq \mb{R}^4$ and $\Gamma \subset SU(2)$ nontrivial, the group $C_{\Gamma}(T^{\vee}_{\Lambda}, SU(R_{\Gamma})) / Z_{|\Gamma|}$ is connected. \end{conj}

Note that we use $C(X,Y)$ in the above to denote the space of continuous maps from $X$ to $Y$ as imposing additional regularity beyond that does not affect the topology of the mapping space.

In practice, the group above seems to be connected even without quotienting by the central $Z_{|\Gamma|}$. It is, in fact, a bit obtuse to call the above a conjecture; it is a straightforward exercise in many cases,\footnote{Indeed, one of the present authors set essentially the $r = 1, \Gamma = Z_2$ case as a qualifying exam problem, to appear at \url{https://www.math.harvard.edu/media/qualsF20\_no-solns.pdf}. }
especially for small values of $r$. Any particular example of a $\hat{\Gamma}\acts{Q}$ action may be computed in finite time, and so the above statement may be settled by simply progressing through all the finitely many cases of $\hat{\Gamma}\acts{Q}$ actions. Such a proof seems rather unilluminating though; what we would prefer is, of course, a simultaneous proof for all cases which may shed further light on how the topology of these hyperkahler quotients relates to their construction~\cite{atiyah:ym,McGN}.

\begin{prop} $C_{Z_2}(T^4,SU(2))$ is connected. \end{prop}
\bp Let us establish some notation. Given a space $X$, let $LX = C(S^1, X)$ denote the free loop space and $\Omega X = C((I, \partial I), (X, *))$ the based loop space, where we will freely use the notation above for mapping spaces of pairs and the notation $*$ for judiciously chosen base-points. More generally, denote $L^nX = C(T^n,X)$ the $n^{\text{th}}$ iterated free loop space and if $X$ has a $Z_2$-action, consider $L^nX$ as endowed with the $Z_2$-action which simultaneously acts on $T^n$ by negation and on $X$. Denote $L^n_{Z_2}X$ the invariant maps under this action, i.e., the invariant subspace $(L^nX)^{Z_2}$. Finally, we will freely use that if $X$ is a topological group, $LX \simeq X \times \Omega X$.

Imposing $Z_2$-invariance on the description $LX = C(S^1,X)$ immediately yields $L_{Z_2}X = C((I,\partial I),(X,X^{Z_2}))$. \begin{lem} Given data as above and $k \in \mb{Z}_{\ge 0}$, $X^{Z_2}$ $k$-connected and $X$ $(k+1)$-connected implies $L_{Z_2}X$ $k$-connected. \end{lem} \bp The evaluation maps at the boundary of the interval yield a fiber sequence \beq\label{eq:fiber}\Omega X \to C((I,\partial I),(X,X^{Z_2})) \to X^{Z_2} \times X^{Z_2}\,.\eeq \ep Iterating the above lemma yields \begin{lem} For any $k \ge 0$, $X^{Z_2}$ $k$-connected and $X$ $(n+k)$-connected implies $L^n_{Z_2}X$ $k$-connected. \end{lem} In the case at hand $X = SU(2)$ and $X^{Z_2} = U(1)$. We would like to use the above lemma for $k = 0$ and $n = 4$, but $SU(2)$ is not even $3$-connected; note, however, that the above lemma already serves to settle the $\Gamma = Z_2$ and $r = 1, 2$ cases of Conjecture~\ref{conj:top}. We now settle the $r = 3$ case; i.e., we now show $L^3_{Z_2}SU(2)$ is connected on the way to showing the same for $L^4_{Z_2}SU(2)$. 

The fiber sequence~\eqref{eq:fiber} yields that $\pi_0 L^3_{Z_2}SU(2) = *$ would follow from $\pi_1 L^2_{Z_2}SU(2) \twoheadrightarrow \pi_0 \Omega L^2SU(2)$, which would in turn follow by finding a $Z_2$-equivariant representative of a generator of $\pi_3SU(2)$, where here $Z_2$ acts on $S^3 = \{(x_0,x_1,x_2,x_3) \in \mb{R}^4 \bigm| \sum x_i^2 = 1\}$ by negating the sign of two of the four coordinates. But the identity map $S^3 \to SU(2)$ has exactly this equivariance structure for the usual $Z_2$-action of $\sigma_z$-conjugation on $SU(2)$.

Similarly, to show $\pi_0 L^4_{Z_2}SU(2) = *$, we wish to establish $\pi_1 L^3_{Z_2}SU(2) \twoheadrightarrow \pi_1 L^3 SU(2)$, where we may \emph{a priori} compute $\pi_1 L^3SU(2) \simeq \mb{Z}^3 \oplus Z_2$. Here, the $3$ copies of $\mb{Z}$ ultimately arise from $\pi_3SU(2)$ and admit equivariant representatives as in the prior paragraph; the last copy of $Z_2$ arises from $\pi_4 SU(2)$, and we wish to show an equivariant representative where $Z_2$ acts on $S^4 = \{(x_0,x_1,x_2,x_3,x_4) \bigm| \sum x_i^2 = 1\}$ by negating three of the five coordinates. But a representative for the nontrivial element of $\pi_4S^3$ is the suspension of the Hopf map $S^3 \to \mb{CP}^1$ given explicitly by $(x_0,x_1,x_2,x_3) \mapsto [z_0,z_1]$ for $z_0 = x_0 + i x_1, z_1 = x_2 + i x_3$; negating $x_1$ and $x_3$ enacts complex conjugation on $\mb{CP}^1$. Suspending this equivariant map and taking $Z_2$ to act nontrivially on the suspension coordinate yields a suitably equivariant map $S^4 \to S^3$.
\ep

We very briefly indicate how one would continue in other cases by, for example, continuing to compute fundamental domains for the $\Gamma$-action:
\begin{proof}[Proof sketch for $r = 2, \Gamma = Z_3$] Let us first use the rhombus with vertices $0,1,e^{2\pi i/3},1+e^{2\pi i/3}$ as a fundamental domain for the two-torus $T^{\vee}_{\Lambda}$ on which $Z_3$ acts; a further fundamental domain for this $Z_3$-action is given by the sub-rhombus with vertices $0, \frac{1}{3}(2+e^{2\pi i/3}),\frac{1}{3}(1+2e^{2\pi i/3}),1 \in \mb{C}$ and one may express the group $\mathsf{Maps}_{Z_3}(T^{\vee}_{\Lambda},SU(3))$ in terms of the map on this fundamental domain. Up to homeomorphism, we may treat this sub-rhombus as simply the standard $I^2 = [0,1]^2$, in which case this mapping space may be written $\mathsf{Maps}_r(I^2,SU(3))$, where the $r$ subscript refers to the restriction \begin{align*}\mathsf{Maps}_r&(I^2,SU(3)) := \\ &\hspace{0.75cm}\{f\colon I^2 \to SU(3) \bigm|\text{for }0 \le t \le 1, f(1,1-t) = c(f(t,0)), f(1-t,1)=c^2(f(0,t))\}\,,\end{align*} where $c\colon SU(3) \to SU(3)$ is the action of the standard generator of $Z_3$. As $\pi_1(SU(3)) = \pi_2(SU(3)) = *$, the restriction map $\mathsf{Maps}_r(I^2,SU(3)) \to \mathsf{Maps}_r(\partial I^2,SU(3))$ is an isomorphism on $\pi_0$, and by evaluation at the value of $f$ at $(0,0)$ and/or $(1,1)$, $\mathsf{Maps}_r(\partial I^2,SU(3))$ fibers over $SU(3)^{Z_3}$ with fiber $X^2$ for $X$ the space $$X:=\{f\colon I \to SU(3) \bigm| f(0) = \mathrm{Id},f(1)\in SU(3)^{Z_3}\}\,.$$ As $X$ itself fibers over $SU(3)^{Z_3}$ with fiber $\Omega SU(3)$, it suffices now to note that $SU(3)^{Z_3} \simeq U(1)^2$ is connected.\ep

The same strategy works \emph{mutatis mutandis} for the other $r = 2$ cases by using, e.g., the fundamental domain the isosceles triangle with vertices $0, \frac{1}{2}(1+i), 1$ for $\Gamma = Z_4$ or the equilateral triangle with vertices $0, \frac{1}{3}(1+2e^{2\pi i/3}), \frac{1}{3}(2 + e^{2\pi i/3})$ for $\Gamma = Z_6$; in these cases with $\Gamma = Z_n$ for $n$ composite, one now needs the more general observation that if $d|n$, $SU(n)^{Z_d} \simeq S(U(n/d)^{\times d})$ is still connected.

%\todo{I don't know if it's reasonable to conjecture this connectedness more generally than $Q \simeq \mb{R}^4$. For example, sticking with $\Gamma = Z_2$ but $r$ larger, we're effectively curious if $[S^{2+k,1},S^{2,1}] \to [S^{k+3},S^3]$ is surjective for all $k$, following the notation of Landweber's paper on involutions. Maybe Araki or Isaksen or co have helpful computations, or even Greenlees' spectral sequence, although these are all in the stable range.}

Similarly, one may pose the question of computing higher homotopy groups; e.g., one might ask whether $$\pi_1 L^n_{Z_2}SU(2) \stackrel{?}{\simeq} \mb{Z}^{\binom{n}{2}+2^n}\,.$$ This computation is not difficult to verify for $n = 1$; otherwise, we offer no further comment here.

%Implicitly, the question is something as follows: even in cases where we can't perform some full desingularization, is there still some ``stringy topology'' which we're probing? e.g. the cohomology thereof is the orbifold cohomology (or equivariant K-theory, or so on) ; the fundamental group we seem to expect to be trivial given any sort of reasonable genericity of the action ; etc.

\section{Nahm data}\label{sec:nahm}

\subsection{Definitions and statement of results}

We now specialize the general discussion of~\S\ref{sec:setup} to the case of $r = 1$ and $\Gamma = Z_2$, which acts on $\RR^3\times S^1$ by negating all four coordinates. The moment map equations are now ordinary differential equations on a circle, which we take to have circumference $2L$. They involve traceless $2\times 2$ skew-Hermitian matrices $A^\mu(t)$, where $\mu=0,\ldots,3$ and $t$ is a coordinate on the circle. We will often treat $A^0$ and $A^i$, $i=1,2,3$, on different footing as the former transforms as part of a connection $\partial_{A^0} := \partial + A^0$ while the latter transforms in the adjoint representation of the gauge group. We refer to this package $A^\mu$ as \emph{pre-Nahm data}. 

We choose a basis such that the $Z_2$-equivariance condition takes the form $A^\mu(-t) = -\sigma_z A^\mu(t) \sigma_z$. Here, and throughout this paper, we make use of the Pauli matrices:
\be \sigma_x = \twoMatrix{0}{1}{1}{0} \ ,\quad \sigma_y = \twoMatrix{0}{-i}{i}{0} \ , \quad \sigma_z = \twoMatrix{1}{0}{0}{-1} \ . \ee
The Pauli matrices, up to multiplication by $i$, give a basis for the Lie algebra $\mf{su}(2)$, and we will often make use of them by expanding $\mf{su}(2)$ matrices in said basis, writing
\beq A^{\mu}(t) = -i(A^{\mu}_x(t) \sigma_x + A^{\mu}_y(t) \sigma_y + A^{\mu}_z(t) \sigma_z)\,.\eeq
The points with nontrivial stabilizers are $0, L$, and we have triples of moment map parameters associated with each of them. Denoting them by $\xi^i_0$ and $\xi^i_L$, respectively, the deformed moment map equations take the form
\be \partial_{A^0} A^i + \frac{1}{2} \epsilon^{ijk} [A^j, A^k] = 2i (\xi^i_L \delta_L - \xi^i_0 \delta_0) \sigma_z  \ . \label{eq:circNahm} \ee
Of course, working $Z_2$-equivariantly over $S^1$ may be more easily accomplished by working over the interval $[0, L]$. As we will now explain, the right hand side of \eqref{eq:circNahm} affects the boundary conditions imposed on pre-Nahm data on this interval.

The delta functions introduce jump discontinuities in $A^i_z$ at $t=0,L$. In tandem with the equivariance condition $A^i_z(0^+) = - A^i_z(0^-)$, and the analogous condition at $L$, we find that $A_z^i(0^+) = \xi_0^i$ and $A_z^i(L^-) = \xi^i_L$. So, once we restrict to the interval $[0,L]$, the effects of the delta functions in \eqref{eq:circNahm} are to impose the boundary conditions
\be \label{eq:jumpCond}
A^i_z(0) = \xi^i_0\ ,\quad A^i_z(L) = \xi^i_L \ . \ee
With this imposition, we may omit the delta functions from the moment map equations. 

\emph{A priori}, the $Z_2$-equivariant formulation on $S^1$ may induce one to impose additional boundary conditions on derivatives of $A^\mu$. However, just as $A^i_z(0)$ and $A^i_z(L)$ are ill-defined on the circle, due to the jump discontinuities, the moment map equations similarly demand that derivatives at $0,L$ be ill-defined. Therefore, after restricting to the interval we will only impose the boundary condition \eqref{eq:jumpCond} (as well as an analogous condition on $A^0$ -- see footnote \ref{ft:gaugeCond}) on pre-Nahm data.\footnote{We note that for the undeformed moment map equations, with $\xi=0$, it is consistent to impose the boundary conditions $\partial^n A^\mu(t_0) = (-1)^{n+1} \sigma_z A^\mu(t_0) \sigma_z$ for all $n\ge 0$ at both $t_0=0,L$. However, if we exploit the gauge freedom to impose these conditions for $\mu=0$, then they follow for $\mu=1,2,3$ from the moment map equations, together with \eqref{eq:jumpCond}. So, in any case there is no need to impose them, and indeed it is preferable not to for the sake of a uniform treatment (i.e., one that works for both $\xi=0$ and $\xi\not=0$). 

We contrast these statements with the boundary conditions imposed in this setting by Gaiotto-Witten in~\cite[\S2.2]{w:boundaryConds}. Namely, in addition to \eqref{eq:jumpCond} they imposed $\partial_{A^0} A^i(t_0) = \sigma_z \partial_{A^0} A^i(t_0) \sigma_z$. This boundary condition would follow from the combination of $A^\mu(t_0) = - \sigma_z A^\mu(t_0) \sigma_z$ and $\partial A^i(t_0) = \sigma_z \partial A^i(t_0) \sigma_z$, but as we have explained generally neither of these conditions should be imposed. When the moment map parameters vanish, the first condition coincides with \eqref{eq:jumpCond}, and together with Nahm's equations (stated below \eqref{eq:Nahm}) it implies the second.

We also observe, for the benefit of physicists, that physically one should include Nahm's equations at the boundary as part of the boundary conditions of the theory, even though in the present context that seems unnecessary as there is no need to separate the imposition of Nahm's equations at the boundary and in the bulk. The reason is that in the present paper we are interested in supersymmetric vacua, where Nahm's equations hold everywhere, whereas boundary conditions are imposed even for excited states where Nahm's equations do not hold. For example, the boundary condition associated to a single D3-brane ending on a D5-brane that is often called `Neumann' in the physics literature was studied in \S2.1.1 of \cite{w:boundaryConds} under the moniker of `D5-like boundary condition,' and this `Neumann' boundary condition really involved the imposition of Nahm's equations.} We shall similarly only impose zeroth- and first-order boundary conditions in the definition of our gauge group, as these are all the conditions we need in order for the group action to preserve the boundary conditions on pre-Nahm data.\footnote{\label{ft:gaugeCond}The first order boundary condition is needed in order to impose the boundary condition $A^0_z(t_0)=0$, so that the boundary condition \eqref{eq:jumpCond} extends naturally to $\mu=0$. Alternatively, one could omit this boundary condition on $A^0$, so long as one also omitted the first order boundary condition on gauge transformations. In the other direction, one could also impose boundary conditions of the form $\partial^n A^0(t_0) = (-1)^{n+1} \sigma_z \partial^n A^0(t_0) \sigma_z$ for arbitrarily large values of $n$, so long as one simultaneously imposed boundary conditions of the form $\partial^{n+1} g(t_0) = (-1)^{n+1} \sigma_z \partial^{n+1} g(t_0) \sigma_z$ on gauge transformations. Similar alternatives hold in the complex formulation.} Again, from the point of view of gauge theory on $S^1$, this means that gauge transformations need not be smooth at $0,L$. In total, this is the first manifestation of both (i) the variation of the spaces $\mc{A}(\xi)$ with respect to $\xi$ and (ii) how one may modify the regularity conditions at $p \in F$ in the general formulation of Conjecture~\ref{conj:mainconj}. 

We make some further definitions now.

\begin{defn} We say the moment map parameters $(\vec{\xi}_0, \vec{\xi}_L) \in \mb{R}^3 \times \mb{R}^3$ are \emph{generic} if $\vec{\xi}_0 \ne \pm \vec{\xi}_L$. \end{defn}

We denote $I = [0, L]$ and $\partial I = \{0, L\} \subset I$. We will use the following abbreviation: if we have vector spaces $\mf{z}_0, \mf{z}_L$ depending on $0, L$, respectively, we will denote by $\mf{z}_{\partial I}$ their direct sum $\mf{z}_{\partial I} := \mf{z}_0 \oplus \mf{z}_L$. We will use similar notation for pairs of vectors therein.

Next, let $t$ denote the coordinate on $I$, and abbreviate $\partial_t = d/dt$ as simply $\partial$. Denote further
\begin{align*}
T &:= \{g \in SU(2) \bigm| g\text{ is pure diagonal}\} \ , \\
\mf{t} &:= \{X \in \mf{su}(2) \bigm| X\text{ is pure diagonal}\} \ ,\text{ and} \\
\mf{t}^{\perp} &:= \{X \in \mf{su}(2) \bigm| X\text{ is pure off-diagonal}\} \ ,
\end{align*}
so that we have the splitting\footnote{We adopt the notation $\mf{t}^{\perp}$ for the off-diagonal subspace of $\mf{su}(2)$ as it is indeed the orthogonal complement with respect to the canonical Killing form on $\mf{su}(2)$.} $\mf{su}(2) \simeq \mf{t} \oplus \mf{t}^{\perp}$. Similarly denote their complexifications $\mf{t}_{\mb{C}} = \mb{C}\langle \sigma_z \rangle, \mf{t}^{\perp}_{\mb{C}} = \mb{C} \langle \sigma_x, \sigma_y \rangle$, so that $\mf{sl}(2,\mb{C}) \simeq \mf{t}_{\mb{C}} \oplus \mf{t}^{\perp}_{\mb{C}}$. Finally, for $p \in \partial I$, we will find it convenient to abbreviate the length-four vector $(0, \xi^1_p, \xi^2_p, \xi^3_p)$ as $\xi^{\mu}_p$ and the pair $(\xi^{\mu}_0, \xi^{\mu}_L)$ as $\xi^{\mu}_{\partial I}$.

\begin{defn}\label{defn:maindefn} 
Given $\vec{\xi}_0, \vec{\xi}_L \in \mb{R}^3$, define
\begin{align}
\A_{\xi} &:= \{ A^{\mu} \in C^\infty(I, \mf{su}(2))^{\oplus 4} \ \bigm| \ A^{\mu}|_{\partial I} \in -i \xi^{\mu}_{\partial I} \sigma_z + \mf{t}^{\perp} \}  \ , \nonumber \\
\A_{\mathrm{all}} &:= \{ A^{\mu} \in C^\infty(I, \mf{su}(2))^{\oplus 4} \ \bigm| \ A^{0}|_{\partial I} \in \mf{t}^{\perp} \}  \ , \nonumber \\
\tilde \G &:= \{ g \in C^\infty(I, SU(2)) \bigm|  g|_{\partial I} \in T, g^{-1} \partial g|_{\partial I} \in \mf{t}^{\perp} \} \ , \nonumber \\
\G &:= \tilde \G / Z_2 \ , \nonumber \\
\mf{g} &:= \{ h \in C^\infty(I, \mf{su}(2)) \ \bigm| \ h|_{\partial I} \in \mf{t}, \partial h|_{\partial I} \in \mf{t}^{\perp}\} \ , \nonumber \\
\F &:= C^\infty(I, \mf{su}(2))^{\oplus 3} \ .
\end{align}

In the above, the $Z_2$ subgroup of $\tilde\G$ by which we quotient to obtain $\G$ is the center of $\G$ given by $\{\pm\mathrm{Id}\}$; this center acts trivially on pre-Nahm data. 

We have $\mu\colon\mc{A}_{\mathrm{all}} \to \mc{F}$ given explicitly by
\beq\label{eq:Nahm} (A^{\mu}) \mapsto \Big(\partial_{A^0}A^i + \frac{1}{2}\epsilon^{ijk}[A^j,A^k]\Big)\,. \eeq
Pre-Nahm data in $\mu^{-1}(0)$, i.e. satisfying \emph{Nahm's equations} $\mu^i(A) := \partial_{A^0}A^i + \frac{1}{2}\epsilon^{ijk}[A^j,A^k] = 0$, will be termed \emph{(real) Nahm data}. 
\end{defn}

% \partial_{A^0} A^i = \partial A^i + [A^0, A^i]

\begin{rmk}
Note that if we specialize $\xi$ to $0$ above, we obtain \begin{equation*} \A_0 := \{ A^{\mu} \in C^\infty(I, \mf{su}(2))^{\oplus 4} \ \bigm| \ A^{\mu}|_{\partial I} \in \mf{t}^{\perp} \}\,. \end{equation*} We hence have that $\mc{A}_0$ is a (vector) subspace of $\mc{A}_{\mathrm{all}}$ and that all the $\mc{A}_{\xi}$ are affine spaces modelled on (torsors for) $\mc{A}_0$. 
\end{rmk}
\begin{rmk}\label{rmk:so3}
Note that there is a natural action of $SO(3)$ on Nahm data via the standard $SO(3)\acts{\mb{R}^3}$ action on $(A^1, A^2, A^3)$, at least if one simultaneously acts in the same way on $(\xi^1_{\partial I}, \xi^2_{\partial I}, \xi^3_{\partial I})$.
\end{rmk}

Denote by $E = I \times \mb{C}^2$ the trivial rank 2 Hermitian vector bundle over $I$. Then, we may think of $A^0$ as a Hermitian connection on $E$ while the $A^i$ are Hermitian endomorphisms. 

We now make a definition in terms of the $\G$-action on $\A_{\rm all}$:

\begin{defn} We say pre-Nahm data is \emph{reducible} if it has a nontrivial stabilizer in $\G$. We denote the sublocus of \emph{irreducible} pre-Nahm data, with trivial stabilizer, as $\mc{A}^{\mathrm{irr}}_{\mathrm{all}} \subset \mc{A}_{\mathrm{all}}$ and $\mc{A}^{\mathrm{irr}}_{\xi} \subset \mc{A}_{\xi}$. \end{defn}

In the notation of the present section, the K\"ahler forms $\omega^i$, holomorphic symplectic form $\omega^\CC$ associated to the complex structure $J^1$, and metric on $\A_\xi$ are given by
\begin{align}
\omega^i(a,b) &= -\frac{1}{2L} \int_I \Tr(a^0 b^i - a^i b^0 + \epsilon^{ijk} a^j b^k)\, dt \nonumber \\
\omega^\CC(a,b) &= -\frac{1}{2L} \int_I \Tr( (a^0 + i a^1)(b^2+i b^3) - (a^2 + i a^3)(b^0 + i b^1) )\, dt \nonumber \\
g(a,b) &= - \frac{1}{2L} \int_I \Tr(a^0 b^0 + a^1 b^1 + a^2 b^2 + a^3 b^3)\, dt \ ,
\end{align}
where $a,b\in \A_0$.
%Compatible complex structures on torus: $J^1 = {{0,-1,0,0},{1,0,0,0},{0,0,0,-1},{0,0,1,0}}, J^2 = {{0,0,-1,0},{0,0,0,1},{1,0,0,0},{0,-1,0,0}}, J^3 = {{0,0,0,-1},{0,0,-1,0},{0,1,0,0},{1,0,0,0}}$

We now have the following proposition:

\begin{prop}\label{prop:Hamtrick} For any $\xi$, the $\mc{G}$-action on $\mc{A}_{\xi}$ is tri-Hamiltonian with moment maps $\mu^i$ given as in~\eqref{eq:Nahm}. \end{prop}
\bp Suppose $\delta A^{\mu}$ is a tangent vector field on $\mc{A}_{\xi}$ and $h \in \mf{g}$. Then, evaluated at $A^{\mu} \in \mc{A}_{\xi}$, we wish to verify that $- \frac{1}{2L} d\mu^i(h)(\delta A) = \iota_h \omega^i(\delta A)$. (See Remark \ref{rmk:mommap}.) We calculate
$$- \frac{1}{2L} d\mu^i(h)(\delta A) = \frac{1}{2L} \int_I \Tr\Bigg( \Big( \partial_{A^0} \delta A^i - [A^i, \delta A^0] + \epsilon^{ijk} [A^j, \delta A^k] \Big) h \Bigg) \, dt$$
while
$$\iota_h \omega^i(\delta A) = - \frac{1}{2L} \int_I \Tr \Bigg( \partial_{A^0} h\, \delta A^i - [A^i, h] \delta A^0 + \epsilon^{ijk} [A^j, h] \delta A^k \Bigg)\,dt \,.$$ An integration by parts shows that these expressions agree if and only if $$\Big[ \Tr (h \delta A^i)\Big]^L_0 = 0\,.$$ But this is precisely guaranteed by the assumption that $h|_{\partial I} \in \mf{t}, \delta A^i|_{\partial I} \in \mf{t}^{\perp}$. \ep

Let $\mf{z}^3_0, \mf{z}^3_L$ denote the copies of $\mb{R}^3$ in which $\vec{\xi}_0, \vec{\xi}_L$ lie, respectively, and let $\mf{z}^3_{\partial I} := \mf{z}^3_0 \times \mf{z}^3_L$ be the full six-dimensional parameter space. Note that $\mc{A}_{\mathrm{all}}$ is the disjoint union over all $\vec{\xi}_{\partial I} \in \mf{z}^3_{\partial I}$ of the $\mc{A}_{\xi}$; indeed, there is a natural map \beq \mc{A}_{\mathrm{all}} \stackrel{p}{\to} \mf{z}^3_{\partial I}\eeq whose fiber over $\xi \in \mf{z}^3_{\partial I}$ is $\mc{A}_{\xi}$.

\begin{defn} Denote \beq \M_{\mathrm{all}} := \mu^{-1}(0) / \mc{G}\,,\eeq where $\mu$ above is thought of as a map $\A_{\mathrm{all}} \stackrel{\mu}{\to} \mc{F}$. As $\G$ acts preserving the boundary values $A_z^{\mu}|_{\partial I}$, this quotient also admits a projection map
\beq \M_{\mathrm{all}} \stackrel{p}{\to} \mf{z}^3_{\partial I} \ . \eeq
We denote the fiber over $\xi \in \mf{z}^3_{\partial I}$ by $\M_{\xi}$; note that $\M_{\xi}$ is the hyperkahler quotient of $\mc{A}_{\xi}$ by $\mc{G}$. Similarly, denote by $\M^{\mathrm{irr}}_{\mathrm{all}} \subset \M_{\mathrm{all}}, \M^{\mathrm{irr}}_{\xi} \subset \M_{\xi}$ the irreducible subloci, i.e. the orbits of Nahm data with trivial stabilizers. \end{defn}

We can now state the main theorem of~\S\ref{sec:hkq}:

\begin{inttheorem}{thm:modHK}
The map $\M_{\mathrm{all}} \stackrel{p}{\to} \mf{z}^3_{\partial I}$ restricts to a smooth submersion of manifolds $\M^{\mathrm{irr}}_{\mathrm{all}} \stackrel{p}{\to} \mf{z}^3_{\partial I}$. Its fibers $\M^{\mathrm{irr}}_{\xi}$ over $\xi \in \mf{z}^3_{\partial I}$ are smooth (possibly empty) hyperkahler four-manifolds. 

If $\xi \in \mf{z}^3_{\partial I}$ is generic, all points of $\M_{\xi}$ are irreducible while if $\xi$ is nongeneric but not identically vanishing, there is a single point of $\M_{\xi} \setminus \M_{\xi}^{\mathrm{irr}}$. Finally, if $\xi$ identically vanishes, $\M_{\xi} \simeq (\mb{R}^3 \times S^1)/Z_2$ and the irreducible locus is precisely the complement of the two $A_1$ orbifold singularities.

The extended gauge transformations $g_1(t) = \exp(\pi i t \sigma_x / 2L)$ and $g_2(t) = \exp(\pi i (L - t) \sigma_x / 2L)$ act by commuting involutions on $\mc{M}_{\mathrm{all}}$. If $\mf{z}^3_{\partial I}$ is equipped with the commuting involutions $\xi_L \mapsto -\xi_L$ and $\xi_0 \mapsto -\xi_0$, the map $p$ is $(Z_2 \times Z_2)$-equivariant with respect to these involutions; the combined action of $g_1g_2 \equiv i\sigma_x$ is by complex conjugation and implements $\xi \mapsto -\xi$. The map $p$ is moreover $SO(3)$-equivariant for the actions of Remark~\ref{rmk:so3}; the ensuing diffeomorphisms between the corresponding $\M_{\xi}$ fibers may be interpreted as hyperkahler rotations.
\end{inttheorem}

\begin{rmk} Recall that an $A_1$ orbifold point has local model $\mb{R}^4/Z_2$, with the $Z_2$ acting by negation on all four coordinates. \end{rmk}

The statement of the last paragraph above is included for completeness and is an immediate computation once we define the terminology of \emph{complex conjugation}. Given a hyperkahler manifold, we refer to a complex conjugation thereon as an isometry of its Riemannian manifold structure under which the K\"ahler forms pull back to their negations (and so the complex structure endomorphisms transform to their conjugates). The action of $\sigma_x$-conjugation on $\mf{su}(2)$ is exactly such a complex-conjugation involution on the ambient space $\mc{A}_{\mathrm{all}}$, and hence the inherited action of $\sigma_x$-conjugation is precisely such an isometry $\M_{\xi} \to \M_{-\xi}$ under which the canonical K\"ahler forms pull back to their negations, as desired.

We will find it useful to use the following decomposition of our affine space $\mc{A}_{\xi}$:

\begin{defn} Given $\vec{\xi}_{\partial I}$ as above, define 
\begin{align}
\A' &:= \{A^0 \in C^\infty(I,\mf{su}(2)) \ | \ A^0|_{\partial I} \in \mf{t}^{\perp}\} \ , \nonumber \\
\A''_\xi &:= \{ A^i \in C^\infty(I,\mf{su}(2))^{\oplus 3} \, \bigm| \, A^i|_{\partial I} \in -i\vec{\xi}_{\partial I}\sigma_z + \mf{t}^{\perp}\} \ , \nonumber \\ \A''_{\mathrm{all}} &:= \{ A^i \in C^{\infty}(I,\mf{su}(2))^{\oplus 3}\} \ .
\end{align}
With these definitions we have
\beq \mc{A}_{\xi} \simeq \mc{A}' \times \mc{A}''_{\xi}\,.\eeq \end{defn}

We now state the complex formulation of the above, which is the subject of~\S\ref{sec:complex}. As outlined in Construction~\ref{constr:cplxKah}, we wish to (i) present our pre-Nahm data in a more manifestly complex form, (ii) give a characterization of the $\xi$-(poly)stable locus, and (iii) state a construction of $\M_{\xi}$ in the complex formulation. First, we assemble our pre-Nahm data into the more manifestly holomorphic package 
\be \alpha = A^0 + i A^1 \ , \quad \beta = A^2 + i A^3 \ . \ee
We may now parse $\alpha$ as a (non-Hermitian) connection on the rank-two complex vector bundle $E$ and $\beta$ as a general (traceless) complex endomorphism. Note, however, that the boundary conditions on $\alpha$ still make reference to the Hermitian structure on $E$; we will remedy this shortly.

We term such pairs $(\alpha, \beta)$ as \emph{complex} pre-Nahm data. Next, the complex moment map equation $\mu_{\mb{C}}(\alpha,\beta):= (\mu^2+i\mu^3)(\alpha,\beta)=0$ is
\be \partial_\alpha \beta := \partial \beta + [\alpha, \beta] = 0 \ . \ee
Complex pre-Nahm data $(\alpha, \beta)$ solving this equation $\partial_{\alpha} \beta = 0$ are termed \emph{complex Nahm data}. Note that the complex Nahm equation is invariant under the action of a complexified gauge group, i.e., $C^{\infty}(I, SL(2,\mb{C}))$.

It is prudent now to assemble $\xi^2$ and $\xi^3$ into
\beq \xi^{\mb{C}}_{\partial I} = \xi^2_{\partial I} + i \xi^3_{\partial I} \ . \eeq
We rename $\xi^1_{\partial I}$ to $\xi^{\mb{R}}_{\partial I}$.

The current boundary conditions on $(\alpha, \beta)$ depend somewhat awkwardly on $\xi^{\mb{R}}_{\partial I}$, as the boundary condition distinguishes between the Hermitian and skew-Hermitian parts of $\alpha$ and hence has a non-holomorphic flavor. We will hence perform a change of basis -- or a formal gauge transformation, per Definition~\ref{defn:for} -- following~\cite{simpson:parabolic}. Specifically, we act by a formal gauge transformation $g\in C^\infty(I,SL(2,\CC))$ with $g|_{\partial I}\in T_\CC$ and $g^{-1} \partial g|_{\partial I} \in -\xi^\RR_{\partial I} \sigma_z + \mf{t}_\CC^\perp$. Call the original frame the \emph{unitary frame} and this new one the \emph{standard frame}. The boundary conditions on $\alpha$, in standard frame, are simply $\alpha|_{\partial I} \in \mf{t}_{\mb{C}}^{\perp}$. The boundary conditions on $\beta$ and on complexified gauge transformations, in contrast, are unaffected by this change of basis.

Note that the information of $\xi^{\mb{R}}$ in the standard frame will now enter only via the stability condition. As long as we state our boundary conditions in this standard frame then we may henceforth neglect the Hermitian structure on $E$ and simply regard $E$ as a complex vector bundle. However, for later purposes it will often be convenient to equip it with a $\xi^\RR$-adapted Hermitian metric:

\begin{definition} \label{defn:hAdapt}
A $\xi^\RR$-adapted Hermitian metric is one in which $(g e_1, g e_2)$ is an orthonormal basis in a neighborhood of $t_0=0,L$, where $e_i$ are the standard basis vectors for $E = I\times \CC^2$ and $g$ is a formal gauge transformation which satisfies $g|_{\partial I}\in \mf{t}_\CC$ and $g^{-1}\partial g|_{\partial I} \in -\xi^\RR_{\partial I} \sigma_z + \mf{t}_\CC^\perp$.
\end{definition}

Given such a metric, one can undo the formal gauge transformation that we performed in passing from the unitary frame to the standard frame in order to obtain real pre-Nahm data $A^\mu$. We now introduce the main definitions for the complex formulation; as we have just explained, these notions are formulated in the standard frame, and we make no reference to a Hermitian metric.

\begin{defn}\label{defn:maincplxdefn} Given $\xi^{\mb{C}}_{\partial I} \in \mf{z}^{\mb{C}}_{\partial I}$, define \begin{align} \mc{B}_{\mathrm{all}} &:= \{(\alpha, \beta) \in C^{\infty}(I,\mf{sl}(2,\mb{C}))^{\oplus 2} \bigm| \alpha|_{\partial I} \in \mf{t}^{\perp}_{\mb{C}}\} \nonumber \\ \mc{B}_{\xi^{\mb{C}}} &:= \{(\alpha, \beta) \in C^{\infty}(I,\mf{sl}(2,\mb{C}))^{\oplus 2}  \bigm| \alpha|_{\partial I} \in \mf{t}^{\perp}_{\mb{C}}, \beta|_{\partial I} \in -i \xi^{\mb{C}}_{\partial I} \sigma_z + \mf{t}^{\perp}_{\mb{C}}\} \nonumber \\ \tilde{\mc{G}}_{\mb{C}} &:= \{g \in C^{\infty}(I, SL(2,\mb{C})) \bigm| g|_{\partial I} \in T_{\mb{C}}, g^{-1} \partial g|_{\partial I} \in \mf{t}^{\perp}_{\mb{C}} \} \nonumber \\ \mc{G}_{\mb{C}} &:= \tilde{\mc{G}}_{\mb{C}}/Z_2 \nonumber \\ \mf{g}_{\mb{C}} &:= \{h \in C^{\infty}(I, \mf{sl}(2,\mb{C})) \bigm| h|_{\partial I} \in \mf{t}_{\mb{C}}, \partial h|_{\partial I} \in \mf{t}^{\perp}_{\mb{C}}\} \nonumber \\ \mc{F}_{\mb{C}} &:= C^{\infty}(I, \mf{sl}(2,\mb{C})) \nonumber \\ \mc{D}_{\mathrm{all}} &:= \ker(\mc{B}_{\mathrm{all}} \stackrel{\mu_{\mb{C}}}{\to} \mc{F}_{\mb{C}}) = \{(\alpha, \beta) \in \mc{B}_{\mathrm{all}} \bigm| \partial_{\alpha} \beta = 0\}\nonumber \\ \mc{D}_{\xi^{\mb{C}}} &:= \ker(\mc{B}_{\xi^{\mb{C}}} \stackrel{\mu_{\mb{C}}}{\to} \mc{F}_{\mb{C}}) = \{(\alpha, \beta) \in \mc{B}_{\xi^{\mb{C}}} \bigm| \partial_{\alpha} \beta = 0\} \nonumber \\ \mc{B}' &:= \{\alpha \in C^{\infty}(I, \mf{sl}(2,\mb{C})) \bigm| \alpha|_{\partial I} \in \mf{t}^{\perp}_{\mb{C}}\} \nonumber \\ \mc{B}''_{\mathrm{all}} &:= \{\beta \in C^{\infty}(I,\mf{sl}(2,\mb{C}))\} \nonumber \\ \mc{B}''_{\xi^{\mb{C}}} &:= \{\beta \in C^{\infty}(I, \mf{sl}(2,\mb{C})) \bigm| \beta|_{\partial I} \in -i \xi^{\mb{C}}_{\partial I} \sigma_z + \mf{t}^{\perp}_{\mb{C}}\}\,. \end{align} \end{defn}

Stability will be phrased in terms of slopes of \emph{Nahm subdata}. As such, it will be useful to give a more general definition of pre-Nahm data (e.g., on bundles of general rank) and to note several natural constructions in this generality. 

\begin{defn} 
Define a \emph{marked vector bundle} on $I$ to be a tuple $\mc{E} = (E, E^{\pm}_{\partial I})$, where \begin{itemize} 
\item $E$ is a complex vector bundle on $I$, and 
\item $E^{\pm}_0 = (E_0^+, E_0^-)$ is a pair of subspaces of the fiber $E_0$ of $E$ over 0 providing a direct sum decomposition $E_0 = E_0^+ \oplus E_0^-$, and similarly for $E^{\pm}_L$.\footnote{The notation for these superscripts is due to the reformulation of pre-Nahm data on $I$ in terms of discontinuous $Z_2$-equivariant pre-Nahm data on $S^1$: in the latter conception, $E^\pm_{\partial I}$ are the eigenspaces of the $Z_2$ action on the fibers $E_{\partial I}$. However, for our purposes these subspaces $E^\pm_{0}$ will, for the most part, be treated identically (and similarly for $E^\pm_L$). The only way in which we will introduce a relative sign is that, in the cases where we wish to impose tracelessness of pre-Nahm data, we will demand $\xi^{+,\CC}_{\partial I} \dim E^+_{\partial I} = - \xi^{-,\CC}_{\partial I} \dim E^-_{\partial I}$ (and similarly for $\xi^{\pm,\RR}_{\partial I}$; these parameters will all be introduced shortly), which in the rank 2 case of interest happens to take the form $\xi^{+,\CC}_{\partial I} = - \xi^{-,\CC}_{\partial I}$.}
\end{itemize} 

%Note that, over $I$, the datum of a connection on $E$ is equivalent to providing a trivialization of $E$.

Next, given $(E, E^{\pm}_{\partial I})$, we have a direct sum decomposition $$\mathrm{End}(E_0) \simeq \Big( \mathrm{End}(E_0^+) \oplus \mathrm{End}(E_0^-) \Big) \oplus \Big(\mathrm{Hom}(E_0^+,E_0^-) \oplus \mathrm{Hom}(E_0^-,E_0^+)\Big)\,,$$ and we denote the above subspaces as $\mathrm{End}^{\mathrm{diag}}(E_0),\mathrm{End}^{\mathrm{off}}(E_0)$ with corresponding projections $\pi^{\mathrm{diag}}_0, \pi^{\mathrm{off}}_0$, with similar notation at $L$. 

Given tuples of complex numbers $\xi_{\partial I}^{\pm,\mb{C}}$, we now define \emph{pre-Nahm data} on a marked vector bundle $\mc{E} = (E, E^{\pm}_{\partial I})$ associated to complex parameters $\xi^{\pm,\mb{C}}_{\partial I}$ to be a tuple $N = (\mc{E}, \partial_\alpha, \beta)$, where 
\begin{itemize} 
\item $\partial_\alpha$ is a connection on $E$ and% written in terms of $\alpha \in C^{\infty}(I,\mathrm{End}(E))$ and satisfying the condition that $\alpha|_{\partial I} \in \mathrm{End}^{\mathrm{off}}(E|_{\partial I})$ and 
\item $\beta \in C^{\infty}(I, \mathrm{End}(E))$ such that $\pi^{\mathrm{diag}}_{\partial I} \beta|_{\partial I} = -i \xi^{\pm,\mb{C}}_{\partial I} \mathrm{Id}_{E^{\pm}_{\partial I}}$. \end{itemize} 

%We will typically say that $N$ as above is pre-Nahm data \emph{on the bundle $E$}, or rather \emph{on $(E, E^{\pm}_{\partial I})$} of the bundle $E$ together with its direct sum decompositions at $\partial I$. Given $(E, E^{\pm}_{\partial I})$, we have a direct sum decomposition $$\mathrm{End}(E_0) \simeq \Big( \mathrm{End}(E_0^+) \oplus \mathrm{End}(E_0^-) \Big) \oplus \Big(\mathrm{Hom}(E_0^+,E_0^-) \oplus \mathrm{Hom}(E_0^-,E_0^+)\Big)\,,$$ and we denote the above subspaces as $\mathrm{End}^{\mathrm{diag}}(E_0),\mathrm{End}^{\mathrm{off}}(E_0)$ with corresponding projections $\pi^{\mathrm{diag}}_0, \pi^{\mathrm{off}}_0$, with similar notation at $L$. 

%Then, given pre-Nahm data $N$ on $(E,E^{\pm}_{\partial I})$, we say that $N$ has associated complex parameters \MZ{I think we need to require in our definition of pre-Nahm data that the diagonal part of $\beta$ at the boundaries is a multiple of the identity. otherwise, if it's just some random endomorphism that's not a multiple of the identity, it won't be gauge-invariant. However, I don't think we need these multiples to be negatives of each other on $E^\pm_{\partial I}$ -- that comes from tracelessness.} $\xi^{\mb{C}}_{\partial I} \in \mathrm{End}^{\mathrm{off}}(E_{\partial I})$ defined via \beq \xi^{\mb{C}}_{\partial I} := i\pi^{\mathrm{off}}_{\partial I}(\beta|_{\partial I})\,.\eeq

We define the ranks of both $N$ and $\E$ to be the rank of $E$. When it is clear, in context, that a bundle $\E$ is equipped with particular pre-Nahm data, we may, via an abuse of notation, refer to $E$ or $(\partial_\alpha,\beta)$ as the pre-Nahm data.

Next, given two tuples of pre-Nahm data $N = (\mc{E}, \partial_{\alpha}, \beta), N' = (\mc{E}', \partial_{\alpha'}, \beta')$, we define a morphism from the former to the latter to be a smooth bundle homomorphism $f \in C^{\infty}(I,\mathrm{Hom}(E, E'))$ such that 
\begin{itemize} 
\item $f(E_{\partial I}^{\pm}) \subset E'^{\pm}_{\partial I}$ %and $\partial_{\mc{E}'} \circ f = f \circ \partial_{\mc{E}}$,
\item $\partial_{\alpha'} \circ f = f \circ \partial_\alpha$, and 
\item $\beta' \circ f = f \circ \beta$. 
\end{itemize}
We define pre-Nahm data $N$ and $N'$ to be isomorphic if they are related by a morphism which is a bundle isomorphism.

It is immediate to note that with the above definitions of objects and morphisms, pre-Nahm data form a $\mb{C}$-linear category; i.e., all hom-spaces $\mathsf{Hom}(N,N')$ are $\mb{C}$-vector spaces and composition maps are bilinear.

Finally, given pre-Nahm data $N = (\mc{E},\partial_{\alpha},\beta)$, we say that $s \in C^{\infty}(I,E)$ is in the \emph{kernel} of $N$ if $\partial_{\alpha} s = \beta s = 0$.
\end{defn}

\begin{definition} \label{def:genGp}
Given a marked bundle $\mc{E}$, the natural gauge group which acts on pre-Nahm data on $\mc{E}$ is 
\be \tilde\G^{\mc{E}}_\CC := \{g \in C^{\infty}(I, GL(E)) \bigm| \pi^{\mathrm{off}}_{\partial I} g|_{\partial I} = 0 \}\,. \ee
As usual, we define $\G^{\mc{E}}_\CC$ to be the quotient thereof by the trivially-acting central subgroup. Note that this action preserves the complex parameters $\xi^{\pm,\mb{C}}_{\partial I}$. We also note that, because of our convention that gauge transformations act on the right, a gauge transformation that maps $N_1$ to $N_2$ is a morphism from $N_2$ to $N_1$.
\end{definition}

\begin{rmk} 
Note that there is mild ambiguity in our language above now: we do not assume tracelessness for all pre-Nahm data. Nevertheless, in the particular case that our pre-Nahm data is on the original rank-two trivial bundle of interest, we will continue to assume that pre-Nahm data is traceless, as per Definition~\ref{defn:maincplxdefn}. More generally, we could define pre-Nahm data for any complex reductive structure group and functoriality in this generality; indeed, the reader anticipating Construction~\ref{constr:gen} may recognize that the definition we give above corresponds to $G = GL(n,\mb{C})$ and $G_0, G_L$ maximal Levi subgroups of the form $GL(k,\mb{C}) \times GL(n-k,\mb{C})$ for some $0 \le k \le n$. Certainly, there was no need for us above to impose these particular restrictions; rather, it is simply the most immediate way to construct a $\mb{C}$-linear category containing our case of interest. The tracelessness discussion above simply corresponds to the choice of $G = SL(2,\mb{C})$ versus $GL(2,\mb{C})$. Finally, note that in our particular case of interest, we impose the extra conditions that $\pi^{\mathrm{diag}}_{\partial I}g^{-1}\partial g|_{\partial I} = 0$ for elements of the gauge group, and that $\alpha|_{\partial I} \in \mathrm{End}^{\mathrm{off}}(E|_{\partial I})$. This latter condition only makes sense as we are working on a trivialized bundle on which we already have a canonical connection $\partial$ so that we may write $\partial_{\alpha} = \partial + \alpha$; as we do not wish to assume a canonical connection (or equivalently over $I$, a trivialization) of $E$ as part of the datum of a marked vector bundle $\mc{E}$ in general, we do not take this restriction as part of the general definition. It is nonetheless a useful restriction for our purposes when comparing between the real and complex formalisms.
\end{rmk}

As indicated above, we will now quickly note some basic constructions we may perform in this larger generality: 

\begin{lem} 
The following are natural constructions: 
\begin{itemize} 
\item (subdata) Given a marked vector bundle $\mc{E}$, a \emph{marked subbundle} $\mc{E}' \subset \mc{E}$ is a marked bundle $\E'=(E',E'^\pm_{\partial I})$ such that $E'$ is a subbundle of $E$ and $E'^\pm_{\partial I} = E^\pm_{\partial I}\cap E'_{\partial I}$. Given pre-Nahm data $N = (\mc{E},\partial_{\alpha},\beta)$, we define a good subbundle to be a marked subbundle $\E'=(E',E'^\pm_{\partial I}) \subset \E$ such that $E'$ is preserved by $\partial_\alpha$ and $\beta$. We then define \emph{pre-Nahm subdata} $N'\subset N$ on a good subbundle $\E'\subset \E$ to be $N' := (\mc{E}',\partial_{\alpha}|_{E'},\beta|_{E'})$.
\item (quotient data) Given a marked vector bundle $\mc{E}$ with marked subbundle $\mc{E}' \subset \mc{E}$, the \emph{marked quotient bundle} $\mc{E}''$ has underlying vector bundle $E/E'$ and $E''^{\pm}_{\partial I} := E^{\pm}_{\partial I}/E'^{\pm}_{\partial I}$; one still has $E''_{\partial I} = E''^+_{\partial I} \oplus E''^-_{\partial I}$. If $\E'$ is a good subbundle, then $\partial_{\alpha},\beta$ descend to pre-Nahm data on $\mc{E}''$, which we term $\emph{pre-Nahm quotient data}$ $N''$. 
\item (direct sums) It is clear how to define the direct sum $\mc{E}_1 \oplus \mc{E}_2$ of two marked vector bundles and $N_1 \oplus N_2$ of two pre-Nahm data thereon with the same values of $\xi^{\pm,\mb{C}}_{\partial I}$. 
\item (tensor products) Given marked bundles $\mc{E}_1, \mc{E}_2$, their tensor product $\mc{E}_1 \otimes \mc{E}_2$ has underlying bundle $E_1 \otimes E_2$ with $(E_1 \otimes E_2)^+_0 = (E_1)^+_0 \otimes (E_2)^+_0 \oplus (E_1)^-_0 \otimes (E_2)^-_0, (E_1 \otimes E_2)^-_0 = (E_1)^+_0 \otimes (E_2)^-_0 \oplus (E_1)^-_0 \otimes (E_2)^+_0$, and similarly at $L$. Given pre-Nahm data $N_i$ on $\mc{E}_i$ with $\xi^{+,\CC}_{\partial I,i} = \xi^{-,\CC}_{\partial I,i}$ for $i=1,2$, their tensor product is defined similarly in terms of the tensor product connection and $(\beta_1 \otimes \beta_2)(s_1 \otimes s_2) = \beta_1 s_1 \otimes s_2 + s_1 \otimes \beta_2 s_2$.
\item (determinants) For any pre-Nahm data $N$ on $\E$, one may define pre-Nahm data $\det N$ on the marked determinant bundle $\det \E$.
\item (dual data) Given pre-Nahm data $N=(\E,\partial_\alpha,\beta)$, dual pre-Nahm data $N^* = (\E^*, \partial_{\alpha^*}, \beta^*)$ is given by defining the dual marked bundle $\E^* := (E^*, (E^\pm_{\partial I})^*)$, letting $\partial_{\alpha^*}$ be the dual connection on $E^*$, and defining $\beta^*$ by $(\beta^* v)(w) := - v(\beta w)$ for $v\in C^{\infty}(I, E^*)$ and $w\in C^{\infty}(I,E)$. 
\item (Hom data) When $\xi^{+,\CC}_{\partial I,i} = - \xi^{-,\CC}_{\partial I,i}$, one may also define $\Hom(\mc{E}_1,\mc{E}_2) := \E_2\otimes \E_1^*$ and $\Hom(N_1,N_2) := N_2\otimes N_1^*$; in particular, if $\xi^{+,\CC}_{\partial I} = - \xi^{-,\CC}_{\partial I}$ one may define $\End(\E) = \Hom(\E,\E)$ and $\End(N) = \Hom(N,N)$.
\end{itemize}

Note that with sub- and quotient-objects defined as above, the $\mb{C}$-linear category of pre-Nahm data is not an abelian category as kernels and cokernels need not exist in general; however, if $f \in \mathsf{Hom}(N,N')$ has constant rank, one has natural $\ker f, \im f \simeq \mathrm{coim}\,f, \coker f$ as suitable sub- and quotient-objects, respectively. 
\end{lem}

\begin{prop} \label{prop:kerMorph}
Given two pre-Nahm data $N_1,N_2$ with $\xi^{+,\CC}_{\partial I,i} = -\xi^{-,\CC}_{\partial I,i}$ for $i=1,2$, a map $f \in C^{\infty}(I,\mathrm{Hom}(E_1, E_2))$ is a morphism of the corresponding pre-Nahm data if and only if it is in the kernel of the pre-Nahm data $\mathrm{Hom}(N_1,N_2)$.
\end{prop}

We may now make the following definition in this generality: \begin{defn} We say that pre-Nahm data $N = (\mc{E},\partial_{\alpha},\beta)$ is in fact \emph{Nahm data} if $\partial_{\alpha} \beta = 0$. Note that here $\partial_{\alpha}$ indicates the induced connection on $\mathrm{End}(E)$. \end{defn}

\begin{lem} Given Nahm data $N$, the dual pre-Nahm data $N^*$, the determinant pre-Nahm data $\det N$, and any pre-Nahm subdata $N'$ or quotient data $N''$ are automatically Nahm data. Similarly, if $N_1, N_2$ are Nahm, then so too are $N_1\oplus N_2$, $N_1 \otimes N_2$, and $\mathrm{Hom}(N_1, N_2)$ if the $\xi^{\pm,\mb{C}}_{\partial I,i}$ are such that these constructions exist.
\end{lem}

It follows from this lemma that there is another natural $\CC$-linear category of Nahm data; morphisms between Nahm data are the same as morphisms between them when regarded as pre-Nahm data. It may be helpful to note that the relationship between this category and that of pre-Nahm data is analogous to that between the category of holomorphic vector bundles on a complex manifold and the category of \emph{almost holomorphic vector bundles} on the same manifold, where the latter are defined to be complex vector bundles equipped with an almost-Dolbeault operator, i.e., a $(0,1)$-connection not assumed to be integrable (i.e., to square to zero).

We may now define what it means for pre-Nahm data to be stable with respect to some choice of real parameters $\xi^{\pm,\mb{R}}_{\partial I}$. 

\begin{defn} Given pre-Nahm data $N$ and $\xi^{\pm,\mb{R}}_{\partial I} \in \mb{R}$, define the \emph{degree} and \emph{slope} of $N$ with respect to $\xi^{\mb{R}} = (\xi^{\pm,\mb{R}}_{\partial I})$ to be 
\begin{align} \label{eq:slopeEqn}
\deg_{\xi^{\mb{R}}}N &:= \sum_{\dagger \in \{+,-\}} \parens{\xi^{\dagger,\mb{R}}_L \dim E^{\dagger}_L - \xi^{\dagger,\mb{R}}_0 \dim E^{\dagger}_0 } \nonumber \\ \slope_{\xi^{\mb{R}}}N &:= \frac{\deg_{\xi^{\mb{R}}}N}{\rk N}\text{ for $N$ of positive rank}\,.
\end{align}

Define pre-Nahm data $N$ to be
\begin{itemize} 
\item \emph{semistable} if $\slope N' \le \slope N$ for all pre-Nahm subdata $N' \subset N$ of positive rank, 
\item \emph{stable} if this inequality is strict for all positive rank proper $N' \subset N$, 
\item \emph{unstable} if it is not semistable, i.e., if there exists positive-rank pre-Nahm subdata $N' \subset N$ with $\slope N' > \slope N$, and
\item \emph{polystable} if it is a direct sum of stable\footnote{Note that rank 0 or $1$ pre-Nahm data is always automatically stable, and so in the case that interests us in this paper with $N$ of rank two, any proper subdata $N' \subset N$ will automatically be stable.} pre-Nahm subdata of the same slope, or equivalently if it is semistable and a direct sum of stable pre-Nahm data.
\end{itemize}
Finally, we say that polystable (semistable) pre-Nahm data is \emph{strictly} polystable (semistable) if it is not stable.
\end{defn}

We note that if $N'$ is pre-Nahm subdata of $N$ and $N''$ is the associated quotient pre-Nahm data then $\deg_{\xi^\RR} N'' = \deg_{\xi^\RR} N - \deg_{\xi^\RR} N'$.

\begin{rmk} We emphasize that complex Nahm data is $\xi^{\mb{R}}$-polystable (stable, etc.) if it is $\xi^{\mb{R}}$-polystable (stable, etc.) as pre-Nahm data. In particular, the definition of stability does not require the complex Nahm equation to be satisfied. We make this atypical choice in order to clarify that many results rely only on stability, and not in addition on the complex moment map equation. \end{rmk}

\begin{rmk}
Note that if $g \in \mc{G}^\E_{\mb{C}}$ acts on pre-Nahm data $(\alpha, \beta)$, $V \subset E$ is pre-Nahm subdata for $(\alpha, \beta)$ if and only if $g^{-1}V \subset E$ is pre-Nahm subdata for $g \cdot (\alpha, \beta)$. In particular, the above stability conditions are invariant under the action of $\mc{G}_{\mb{C}}$. \end{rmk}

We now specialize back to the choice of $\E$ where $E$ is our rank 2 bundle, $E^+_0 = E^+_L = \CC e_1$, and $E^-_0 = E^-_L = \CC e_2$, where we recall from Definition \ref{defn:hAdapt} that we denote the standard basis vectors of $\mb{C}^2$ by $e_i$. Given real FI parameters $\xi^\RR_{\partial I}$, we define the parameters that enter into \eqref{eq:slopeEqn} via $\xi^{\pm,\RR}_{\partial I} = \pm \xi^\RR_{\partial I}$. In particular, we note that we always have $\deg_{\xi^\RR} N = 0$ for pre-Nahm data $N$ on $\E$. We similarly define $\xi^{\pm,\CC}_{\partial I} = \pm \xi^{\CC}_{\partial I}$.

%Here, as $E$ is complex rank two, the only relevant subdata will be rank one, and we will give an equivalent definition later of a ``Nahm subline''. The definition we give here is modelled after the properties that the projection operator $\pi: E \to V$ should satisfy \fatodo{and are applicable in far more generality}.

%\fatodo{I now think there should just be a Sobolev regularity parameter $s$ throughout here, as opposed to noting slight extensions of the definition all again later.}

%\footnote{Technically, the connection $\partial_{\alpha}$ is more canonically a section of $\mathrm{Hom}(E, E \otimes T^*I)$, but we will freely use the obvious trivialization of $TI$ throughout.} 

We define the loci
\begin{align}
\mc{B}^{\mathrm{ps}}_{\xi} &:= \{(\alpha, \beta) \in \mc{B}_{\xi^{\mb{C}}} \bigm| (\alpha, \beta)\text{ is $\xi^{\mb{R}}$-polystable}\} \nonumber \\
\mc{B}^{\mathrm{ps}}_{\xi^{\mb{R}},\mathrm{all}} &:= \{(\alpha, \beta) \in \mc{B}_{\mathrm{all}} \bigm| (\alpha, \beta)\text{ is $\xi^{\mb{R}}$-polystable}\} \nonumber \\
\mc{D}_{\xi}^{\mathrm{ps}} &:= \mc{D}_{\xi^{\mb{C}}} \cap \mc{B}^{\mathrm{ps}}_{\xi} \ , \quad \mc{D}^{\mathrm{ps}}_{\xi^{\mb{R}},\mathrm{all}} = \mc{D}_{\mathrm{all}} \cap \mc{B}^{\mathrm{ps}}_{\xi^{\mb{R}},\mathrm{all}} \ ,
\end{align}
and similarly we denote the locus of $\xi^{\mb{R}}$-semistable pre-Nahm data by $\mc{B}^{\mathrm{ss}}_{\xi}$, that of $\xi^{\mb{R}}$-stable pre-Nahm data by $\mc{B}^{\mathrm{st}}_{\xi}$, etc.

It will often be convenient to introduce a $\xi^\RR$-adapted Hermitian metric and then characterize pre-Nahm subdata using orthogonal projections:

\begin{defn} \label{def:goodProj}
Given pre-Nahm data on $\E$, a \emph{good projection} is an orthogonal projection $\pi$ to a positive-rank subbundle $V \subset E$ satisfying the following conditions:
\begin{enumerate}[(i)]
\item $\pi|_{\partial I} \in \mf{t}_{\mb{C}},\partial \pi|_{\partial I} \in \mf{t}^{\perp}_{\mb{C}}\,,$ \label{item:piBoundary}
\item $\beta$ preserves $V$, i.e. $\beta\pi = \pi \beta \pi$\,,
\item $\partial_\alpha$ preserves $V$, i.e. $\partial_\alpha \circ \pi = \pi \circ \partial_\alpha\circ \pi$, which is equivalent to both $(1-\pi)(\partial_\alpha \pi) = 0$ and $(\partial_\alpha \pi)\pi = 0$ (where $\partial_\alpha \pi = \partial\pi + [\alpha, \pi]$)\,.
\end{enumerate}
\end{defn}
\begin{rmk} \label{rmk:redundant}
The equivalence of the last two conditions follows from
\be \partial_\alpha \pi = \partial_\alpha \pi^2 = \pi \partial_\alpha \pi + (\partial_\alpha \pi) \pi \ . \label{eq:pipi} \ee
We also note that this expression with $\alpha$ set to 0 implies that the condition on $\partial\pi$ in \eqref{item:piBoundary} above is redundant, as the diagonal part $(\partial \pi)_d$ of $\partial \pi$ satisfies $(\partial \pi)_d = 2 (\partial \pi)_d\, \pi$; multiplying both sides by $\pi$ gives $(\partial \pi)_d \, \pi = 0$, and so $(\partial \pi)_d = 0$.

We next note that the last two conditions for the goodness of good projections $\pi$ and $1-\pi$, whose existence is required for strict $\xi^\RR$-polystability, may be usefully recharacterized as the conditions $\partial_\alpha \pi = [\beta,\pi] = 0$.

Another useful observation, regarding good projections to rank 1 subbundles, is that they satisfy either $\pi(t_0) = \twoMatrix{1}{0}{0}{0}$ or $\pi(t_0) = \twoMatrix{0}{0}{0}{1}$ at each of $t_0=0,L$, where this choice can be the same or different at the two endpoints. That is, the traceless part of $\pi$ at each boundary is $\pm \half \sigma_z$. 
\end{rmk}

We finally have

\begin{inttheorem}{thm:modCpx} Given data as above, the quotient \beq \N_{\xi} := \D^{\mathrm{ps}}_{\xi} / \G_\CC\eeq is, if nonempty, a complex analytic surface with a holomorphic symplectic form on its smooth locus $\N^{\mathrm{st}}_{\xi} := \D^{\mathrm{st}}_{\xi} / \G_\CC$. If the $\xi$ are generic, $\N_{\xi}$ is a (possibly empty) smooth complex surface; else, it has a single $A_1$ orbifold singularity unless $\xi \equiv 0$, in which case it is the orbifold $(\mb{C} \times \mb{C}^{\times})/Z_2$, with $Z_2$ acting by negation and inversion, respectively, on the two factors. More generally, for any $\xi^{\mb{R}}_{\partial I}$, if we define $\N_{\xi^\RR,{\rm all}} := \D^{\rm ps}_{\xi^\RR,{\rm all}}/\G_\CC$, then the map
\beq \N_{\xi^\RR,{\rm all}} \stackrel{p^{\mb{C}}}{\to} \mf{z}_{\partial I}^{\mb{C}} \eeq is a submersion of complex orbifolds whose fibers are as described above.

The map $p^{\mb{C}}$ is $Z_2 \times Z_2$-equivariant as before. It is moreover $\mb{C}^{\times}$-equivariant for the simultaneous actions of $\mb{C}^{\times}$ by scaling of $\beta$ and of $\xi^{\mb{C}}_{\partial I}$. \end{inttheorem}

\begin{remark}\label{rmk:cTimes}
Note that the $U(1) \subset \mb{C}^{\times}$ which multiplies by a complex phase is the subgroup of $SO(3)$ preserving a fixed complex structure. Now, the entire $\CC^\times$ group preserves the complex structure, and indeed acts by holomorphic symplectomorphisms up to an overall rescaling of the holomorphic symplectic form.

There is a particularly interesting consequence of this $\CC^\times$ action. We begin by commenting on the role of $L$ in both the real and complex formulations. There is an isomorphism of $\A^{(1)}_{\rm all}$ and $\A_{\rm all}$, where the former indicates the usual space $\A_{\rm all}$ with $L=1$, given by $A^\mu_L(t) := L^{-1} A^\mu_1(t L^{-1})$, which maps real Nahm data to real Nahm data. Indeed, it gives a hyperkahler isomorphism $\M_{\xi} \simeq L^{-2} \M^{(1)}_{L\xi}$, where the overall factor of $L^{-2}$ indicates a rescaling of all K\"ahler forms by this factor. In the complex formulation, we similarly have $\N_\xi \simeq L^{-2} \N_{L\xi}^{(1)}$, where now the overall rescaling is of the holomorphic symplectic form. But, in the complex formulation we can now act via the $\CC^\times$ action to obtain $\N_\xi \simeq L^{-1} \N^{(1)}_{L\xi^\RR, \xi^\CC} \simeq L^{-1} \N^{(1)}_\xi$, where the last identification follows from considering the stability condition. So, $L$ does not affect the complex structure of $\N_\xi$, and it barely affects the holomorphic symplectic structure.\footnote{This is an example of what physicists call a \emph{non-renormalization theorem}.}
\end{remark}

The complex formulation above is perfectly suited to prove the next main theorem. Let us first introduce the following terminology: given two real numbers $x, y$, say that $x$ \emph{weakly has the same sign as}\footnote{The definition given here ultimately arises from the specialization structure of the non-Hausdorff quotient topological space $\mb{R} / \mb{R}^+$, where $\mb{R}^+$ acts multiplicatively by scaling.} $y$ if the signs of $x, y$ obey $y>0\Rightarrow x>0$ and $y<0\Rightarrow x<0$. Explicitly,
$$(\mathrm{sgn}\,x,\mathrm{sgn}\,y) \in \{(+,+),(+,0),(-,-),(-,0),(0,0)\}\,.$$ Given choices of real parameters $\xi^{\mb{R}}, \tilde{\xi}^{\mb{R}}$, say that $\xi^{\mb{R}}$ \emph{weakly has the same sign as}\footnote{This gives a natural poset structure on the parameter space.} $\tilde{\xi}^{\RR}$ if $\xi^{\mb{R}}_0 + \xi^{\mb{R}}_L$ weakly has the same sign as $\tilde{\xi}^{\mb{R}}_0 + \tilde{\xi}^{\mb{R}}_L$ and $\xi^{\mb{R}}_0 - \xi^{\mb{R}}_L$ weakly has the same sign as $\tilde{\xi}^{\mb{R}}_0 - \tilde{\xi}^{\mb{R}}_L$. We then have the following:

\begin{inttheorem}{thm:relXi} Suppose $\xi^{\mb{R}}$ weakly has the same sign as $\tilde{\xi}^{\mb{R}}$. Then for any $\xi^{\mb{C}}$, there exists a canonical holomorphic map $\mc{N}_{\xi^{\mb{R}},\xi^{\mb{C}}} \to \mc{N}_{\tilde{\xi}^{\mb{R}},\xi^{\mb{C}}}$. This map is a partial resolution of singularities, i.e., a proper birational morphism; moreover, on the dense open locus where the map is an isomorphism, it is a holomorphic symplectomorphism. In particular, if $(\tilde{\xi}^{\mb{R}},\xi^{\mb{C}})$ is already generic so that $\mc{N}_{\tilde{\xi}^{\mb{R}},\xi^{\mb{C}}}$ is smooth, the map $\N_{\xi^\RR,\xi^\RR} \to \N_{\tilde\xi^\RR,\xi^\CC}$ is a holomorphic symplectomorphism.
\end{inttheorem}

\begin{rmk} Once again, we have a families version of the above, i.e., if $\xi^{\mb{R}}$ weakly has the same sign as $\tilde{\xi}^{\mb{R}}$, we have a partial resolution of singularities $\mc{N}_{\xi^{\mb{R}},\mathrm{all}} \to \mc{N}_{\tilde{\xi}^{\mb{R}},\mathrm{all}}$ whose fibers over $\xi^{\mb{C}} \in \mf{z}^{\mb{C}}_{\partial I}$ are as described above. \end{rmk}

We also have a DUY isomorphism: 

\begin{inttheorem}{thm:duy}
There are canonical homeomorphisms \begin{align*} \{A\in \A_\xi\ \bigm| \ \mu_\RR(A) = 0\}/\G &\stackrel{\sim}{\to} \B^{\mathrm{ps}}_\xi/\G_\CC \\ \text{and }\M_{\xi} &\stackrel{\sim}{\to} \N_{\xi}\,. \end{align*} In particular, the identification $\M_{\xi} \stackrel{\sim}{\to} \N_{\xi}$ is an isomorphism of holomorphic symplectic orbifolds, and the restriction $\M^{\mathrm{irr}}_\xi \stackrel{\sim}{\to} \N^{\mathrm{st}}_\xi$ is an isomorphism of holomorphic symplectic manifolds.
\end{inttheorem}

\begin{rmk} Once again, there exists a version of the above isomorphism fibered over $\mf{z}^{\mb{C}}_{\partial I}$ if we fix some $\xi^{\mb{R}} \in \mf{z}_{\partial I}$ and let $\xi^{\mb{C}} \in \mf{z}^{\mb{C}}_{\partial I}$ vary. \end{rmk}

As an immediate corollary of the above statement, we now obtain the following:

\begin{inttheorem}{thm:res} For any $\xi \in \mf{z}^3_{\partial I}$ generic, the hyperkahler manifold $\M_{\xi}$ is diffeomorphic to the minimal resolution of $(\mb{R}^3 \times S^1)/Z_2$. \end{inttheorem}

Note that a minimal resolution of singularities is by no means guaranteed to exist in general; it, however, sensibly exists for the $A_1$ singularities of $(\mb{R}^3 \times S^1)/Z_2$. Moreover, any holomorphic symplectic resolution of singularities of a rational surface singularity such as the $A_1$ singularity is minimal.

It is less immediate that the following is a corollary, but it nevertheless follows from the (partial) resolution of Theorem~\ref{thm:relXi} and the DUY isomorphism Theorem~\ref{thm:duy}. Let $(\mf{z}^3_{\partial I})^{\circ} \subset \mf{z}^3_{\partial I}$ denote the open locus of generic parameters, and denote by $\M_{\mathrm{all}}^{\circ} \stackrel{p^{\circ}}{\to} (\mf{z}^3_{\partial I})^{\circ}$ the corresponding restriction of the family $p$. We then have the following:

\begin{inttheorem}{thm:conn}
The fibered manifold $p^\circ$ admits a canonical Ehresmann connection $\mathrm{Conn}_{\mc{M}}$ and is thus a fiber bundle.

This connection is uniquely characterized as follows: if one considers paths in $(\mf{z}^3_{\partial I})^{\circ}$ with only one of $\xi^i$ varying, then after hyperkahler rotating to a complex structure where the varying parameter is $\xi^\RR$, the corresponding parallel transport map is that given by the composition of the canonical morphisms of Theorems~\ref{thm:relXi} and~\ref{thm:duy}.

Furthermore, this same property is satisfied for variations valued in any space of the form $\mf{z}_{\partial I}\otimes \Xi$, where $\Xi\subset \RR^3$ is a dimension 1 subspace.

Finally, this connection is equivariant under the commuting actions of $Z_2\times Z_2$ and $SO(3)$ on $\M_{\rm all}^\circ$.
\end{inttheorem}

%\begin{inttheorem}{thm:conn} The restriction $p^\circ$ is a smooth fiber bundle admitting a unique connection $\mathrm{Conn}_{\mc{M}}$ characterized as follows: if one considers paths in $(\mf{z}^3_{\partial I})^{\circ}$ with fixed $\xi^{\mb{C}}$, the corresponding parallel transport map is precisely that given by the composition of the canonical morphisms of Theorems~\ref{thm:relXi} and~\ref{thm:duy}. More generally, the same is true after hyperkahler rotation mixing the $\xi^\mb{R}$ and $\xi^\mb{C}$ directions. %of the fiber bundle $\M_{\rm all} \to \mf{z}^3_{\partial I}$ to the open locus with $\xi$ generic admits a natural connection such that for any $\zeta\in \PP^1$, after further restricting to a locus with a fixed value of $(\xi^\CC)^\zeta$ \MZ{explain notation}, the parallel sections correspond to the maps between fibers provided by Theorem \ref{thm:relXi}.  \end{inttheorem}

Recall the notion of an Ehresmann connection for the fibered manifold $p^{\circ}$: for any path $f\colon [0,1] \to (\mf{z}^3_{\partial I})^{\circ}$, we have a privileged parallel-transport diffeomorphism $\M_{f(0)} \stackrel{\sim}{\to} \M_{f(1)}$. Alternatively, one could think of the datum of this connection infinitesimally, i.e., as a smooth horizontal subbundle of $T\M_{\rm all}^\circ$. One may then consider the curvature $F_{\M}$ as a horizontal two-form on $\mc{M}_{\mathrm{all}}^{\circ}$ valued in vertical tangent vectors. Note that by construction, this connection $\mathrm{Conn}_{\M}$ is flat when restricted to a sublocus of $(\mf{z}^3_{\partial I})^{\circ}$ with fixed $\xi^{\mb{C}}$ (or, more generally, after any hyperkahler rotation mixing the $\xi^\mb{R}$ and $\xi^\mb{C}$ directions), but this connection is not flat in general, as we demonstrate after Proposition~\ref{prop:curv}. Note finally that one may hope for such an Ehresmann connection to exist whenever one has a family of hyperkahler quotients parametrized by some space $\mf{z}^3$ of moment map parameters in much the same way that one may hope for DUY isomorphisms to exist. 

We finally pause to discuss the flow of ideas above. As is common in gauge-theoretic constructions, one of the (perhaps ironically) more difficult parts is to show that the resultant space $\M_{\xi}$ is in fact nonempty. Constructing the manifold structure and computing the expected dimension thereof are often rather formal from the (Banach) implicit function theorem and index theory applied to a linearized elliptic complex, respectively. On the other hand, nonemptiness, connectedness, and completeness are typically less formal. In the above, these statements are only finally obtained as a consequence of Theorem~\ref{thm:relXi} (coupled with Theorem~\ref{thm:duy} for the statement on $\M_{\xi}$ rather than $\N_{\xi}$). One did not have to wait so long, however. In fact, we can give a direct analysis of $\N_{\xi}$, as the complex moment map equation is relatively simple. We do so in the following subsection, yielding a simple and entirely different approach to, for example, showing the connectedness (and nonemptiness) of $\N_{\xi}$. Note that this approach still relies on Theorem~\ref{thm:duy} to conclude the analogous statement for $\M_{\xi}$.

%\ftodo{Finish -- claim this $D_2$ ALF. Whatever else might be going on in that last section -- e.g., circumference of $S^1$ fiber at $\infty$ is $\pi/L$.}

%\fatodo{Degenerations? e.g. $D_0$ and $D_1$ ALF, ALE. Prototype for later models of degenerations. (Don't really need to say it, but e.g. K3 to two 1/2 K3s, i.e. ALHs.)}

\vspace{0.2cm}

We end this section by standardizing (and reiterating) some notation of frequent use through the following sections.

\begin{defn} Given $1 \le s \le \infty$, $g \in \mc{G}^{s+1}$ a gauge transformation, and $A^{\mu} \in \mc{A}^s_{\mathrm{all}}$ real pre-Nahm data, denote the action of $g$ on $A$ by $g(A)$, whose components are given by the following formulae: \begin{align*} g(A^0) &= g^{-1}A^0g + g^{-1} \partial g \\ g(A^i) &= g^{-1}A^i g\,.
\end{align*}
We observe that $\partial_{g(A^0)} = g^{-1} \circ \partial_{A^0} \circ g$. Similarly, for $g \in \mc{G}^{s+1}_{\mb{C}}$ and $B = (\alpha, \beta) \in \mc{B}^s_{\mathrm{all}}$ complex pre-Nahm data, we denote the transformed pre-Nahm data by $g(B)$ and write
\begin{align*} g(\alpha) &= g^{-1}\alpha g + g^{-1} \partial g \\ g(\beta) &= g^{-1} \beta g\,. \end{align*} \end{defn}

Given two sets of pre-Nahm data, either real or complex, it will be convenient to produce new pre-Nahm data on $\mathrm{End}(E) \simeq E \otimes E^*$. %Gauge transformations mapping from one set to the other will then be invertible (or constant determinant one, with our normalizations) solutions of this pre-Nahm data on $\mathrm{End}(E)$. \MZ{i vote to delete this sentence bc we're about to say it in equations} We make this precise below:

\begin{defn}\label{defn:starconnection} Given $A^0_1, A^0_2 \in \mc{A}'^s$, denote by $A^0_1*A^0_2 \in H^s(I,\mathrm{End}(\mathrm{End}(E)))$ the operator $$(A^0_1*A^0_2)g = A^0_1g - gA^0_2\,.$$ As usual, we have the associated connection $\partial_{A^0_1*A^0_2}$. Note then that $g \in \tilde\G^{s+1}$ transforms $A^0_1$ to $A^0_2$ if and only if it is a flat section for the $\partial_{A^0_1*A^0_2}$ connection; that is, $\partial_{A^0_1*A^0_2} g = 0$ is equivalent to the equality $\partial_{A^0_2} = \partial_{g(A^0_1)}$ of covariant derivatives on $E$.

More generally, given two sets of real pre-Nahm data $(A^{\mu}_1), (A^{\mu}_2)$, we may define pre-Nahm data $(A^{\mu}_1 * A^{\mu}_2)$ on $\mathrm{End}(E)$ with components $A^0_1*A^0_2$ as defined above and $A^i_1*A^i_2$ defined by $$(A^i_1*A^i_2)g = A^i_1g - gA^i_2\,.$$ Then $g \in \tilde\G^{s+1}$ transforms $(A^{\mu}_1)$ to $(A^{\mu}_2)$ if and only if $\partial_{A^0_1*A^0_2}g = (A^i_1*A^i_2)g = 0$.

Similarly, for $\alpha_1,\alpha_2 \in \mc{B}'^s,$, denote by $\alpha_1*\alpha_2$ the operator $(\alpha_1*\alpha_2)g = \alpha_1g - g\alpha_2$, with $\partial_{\alpha_1*\alpha_2}$ the associated connection. We also define $(\beta_1*\beta_2)(g) = \beta_1g - g\beta_2$. This is a special case of the general definition $\Hom(N_2,N_1)$ -- that is, we have $\Hom(N_2,N_1) = N_1*N_2$. In accordance with Proposition \ref{prop:kerMorph} and a remark in Definition \ref{def:genGp}, $g \in \tilde\G^{s+1}_{\mb{C}}$ transforms $(\alpha_1, \beta_1)$ to $(\alpha_2,\beta_2)$ if and only if $\partial_{\alpha_1*\alpha_2}g = (\beta_1*\beta_2)g = 0$. \end{defn}

\subsection{Preliminary results} \label{sec:prelim}

There are many simplifications afforded by working in the low-dimensional regime in Conjecture~\ref{conj:mainconj} of $0 \le r \le 2$. One that we will not make especial use of, but that is nonetheless still gratifying to note, is the existence of a particularly simple choice of gauge -- i.e., of a canonical representative under the $\mc{G}$-action (in the real formulation) or $\mc{G}_{\mb{C}}$-action (in the complex formulation).\footnote{When we move to the $r = 2$ case in~\cite{mz:ALG}, only complex axial gauge will be available to us.} We term this gauge choice \emph{axial} gauge, in analogy to the terminology used in Yang-Mills theory; its characterizing feature is that the connection part of pre-Nahm data, i.e. $A^0$ (in the real formulation) or $\alpha$ (in the complex formulation), may be made \emph{constant},\footnote{Note that if we were doing gauge theory for an \emph{abelian} group -- as is relevant, e.g., for the $A_n$ ALF spaces -- then Nahm's equations would also imply that the $A^i$ are constant. Upon restricting attention to these constant Nahm data, one immediately finds a \emph{finite}-dimensional hyperkahler quotient construction~\cite{gibbons:hkMonopoles} of these spaces \cite{w:alfQuotient}.} up to a small subtlety in the complex formulation.

\begin{prop}[Real axial gauge]\label{prop:realAx} Any orbit in $\mc{A}'/\mc{G}$ has a unique representative of the form $A^0 = i c \sigma_x$ for $c$ a real constant valued in $[0, \pi/2L]$. This connection has a trivial stabilizer unless $c \in \{0, \pi/2L\}$, in which case its stabilizer is a $U(1)$ subgroup of $\mc{G}$. If $c = 0$, this stabilizer is the diagonal torus of $PSU(2)$, considered as constant gauge transformations; if $c = \pi/2L$, the stabilizer is this diagonal torus conjugated by the extended gauge transformation $\exp(\pi i t \sigma_x / 2L)$. \end{prop}

\begin{rmk} Note that the essence of this proposition harks back to Lemma~\ref{lem:mon}. \end{rmk}

We have an analogous statement in the complex formulation:

\begin{prop}[Complex axial gauge]\label{prop:compAx} 
Any orbit in $\mc{B}'/\mc{G}_{\mb{C}}$ has a unique representative of precisely one of the following forms:
\begin{enumerate}[(i)]
\item $(a + ic)\sigma_x$ with $c \in [0, \pi/2L]$ and $a \in \mb{R}$, with the further restriction $a \ge 0$ if $c \in \{0, \pi/2L\}$,
\item $\begin{pmatrix} 0 & 1/L \\ 0 & 0 \end{pmatrix}$,
\hspace{0.3cm} (iii) $\begin{pmatrix} 0 & 0 \\ 1/L & 0 \end{pmatrix}$,
\setcounter{enumi}{3}
\item $\frac{i}{2L} \twoMatrix{\sin(\pi t/L)}{\pi-2i\cos^2(\pi t/2L)}{\pi-2i\sin^2(\pi t/2L)}{-\sin(\pi t/L)}$, or 
\item $\frac{i}{2L} \twoMatrix{-\sin(\pi t/L)}{\pi-2i\sin^2(\pi t/2L)}{\pi-2i\cos^2(\pi t/2L)}{\sin(\pi t/L)}$. 
\end{enumerate} 
The stabilizer groups are trivial except for case (i) with $a = 0$ and $c \in \{0, \pi/2L\}$, in which case the stabilizer group is the diagonal $\mb{C}^{\times}$ torus of $PSL(2,\mb{C})$ or its conjugate by the extended gauge transformation $\exp(\pi i t \sigma_x / 2L)$, respectively. 
\end{prop}

\begin{rmk} Of course, the last two cases above can hardly be said to be in axial gauge per our original definition of ``constant connection.'' Indeed, these two orbits do not admit constant representatives; the representatives given above are simply cases (ii) and (iii) acted on by the usual extended gauge transformation $\exp(\pi i t \sigma_x / 2L)$. \end{rmk}

Making use of the canonical $\mc{G}_{\mb{C}}$-orbit representative afforded by the above allows us to give a complete description of $\mc{D}_{\mathrm{all}}/\mc{G}_{\mb{C}}$, thanks to the linearity of the complex Nahm equation when regarded as a differential equation only in $\beta$:

\begin{prop}\label{prop:compNahm} An orbit of $\mc{D}_{\mathrm{all}}/\mc{G}_{\mb{C}}$ has a unique representative of precisely one of the following forms: 
\begin{enumerate}[(i)]
\item $\alpha = \alpha_0 \sigma_x$, with $\alpha_0=a+ic$ a constant with $a\in \RR$, $c\in [0,\pi/2L]$, and $a>0$ if $c=0,\pi/2L$. Writing
\be \beta(t) = \beta_x(t) \sigma_x + \beta_y(t) \sigma_y + \beta_z(t) \sigma_z \ , \ee
we have
\begin{align}
\beta_y(t) &= (-\xi_0^\CC \cosh(2\alpha_0 (L-t)) + \xi_L^\CC \cosh(2\alpha_0 t)) \csch(2\alpha_0 L) \ , \nonumber \\
\beta_z(t) &= -i (\xi_0^\CC \sinh(2\alpha_0 (L-t)) + \xi_L^\CC \sinh(2\alpha_0 t)) \csch(2\alpha_0 L) \ . \label{eq:firstBeta}
\end{align}
Finally, $\beta_x$ is an arbitrary complex constant.

\item $\alpha = \twoMatrix{0}{1/L}{0}{0}$ and $\beta = -\frac{i}{L} \twoMatrix{\xi_0^\CC (L-t) + \xi_L^\CC t}{\xi_L^\CC (L + t^2/L) - \xi_0^\CC (L-t)^2/L + c}{(\xi_0^\CC - \xi_L^\CC)L}{-\xi_0^\CC (L-t) - \xi_L^\CC t}$ for some constant $c \in \CC$.

\item $\alpha = \twoMatrix{0}{0}{1/L}{0}$ and $\beta = -\frac{i}{L} \twoMatrix{\xi_0^\CC(L-t) + \xi_L^\CC t}{-(\xi_0^\CC - \xi_L^\CC) L}{\xi_0^\CC(L-t)^2/L - \xi_L^\CC(L+t^2/L) + c}{-\xi_0^\CC(L-t) - \xi_L^\CC t}$ for some constant $c\in \CC$.

\item\hspace{-0.2cm}, (v)\hspace{.2cm} The configurations related to the previous two by the extended gauge transformation $\exp(i \pi t \sigma_x / 2L)$ and with $\xi_L^\CC$ replaced with $-\xi_L^\CC$.

\setcounter{enumi}{5}\item If $\xi_0^\CC = \xi_L^\CC$, then we have $\alpha = 0$ and $\beta$ one of $\twoMatrix{-i \xi_0^\CC}{0}{0}{i\xi_0^\CC}$, $\twoMatrix{-i\xi_0^\CC}{1/L}{0}{i\xi_0^\CC}$, or $\twoMatrix{-i\xi_0^\CC}{c}{1/L}{i\xi_0^\CC}$, where $c$ is a complex constant.

\item If $\xi_0^\CC = - \xi_L^\CC$, then we have the configurations related to the previous ones by the extended gauge transformation $\exp(i \pi t \sigma_x / 2L)$.
\end{enumerate}
All of these have trivial stabilizers in $\G_\CC$, except for the first Nahm data in items (vi) and (vii), which have $\CC^\times$ stabilizers. \end{prop}

\bp This proposition follows immediately from the privileged gauge choices provided by Proposition~\ref{prop:compAx} and the simplicity of the complex Nahm equation $\partial \beta = -[\alpha, \beta]$. \ep

Up to the imposition of the stability condition, this gives a completely explicit description of the complex Nahm data in our moduli spaces. In particular, we see the necessity of $\xi_0^\CC = \pm \xi_L^\CC$ for the moduli space to possess points with non-trivial stabilizers. It also makes evident the importance of the stability condition, since for example without it the moduli spaces with $\xi_0^\CC = \xi_L^\CC$ would not be Hausdorff, as they would contain three $\CC^\times\subset \G_\CC$ orbits with $\alpha=0$ and $\beta$ constant such that any neighborhood of the first contains the latter two, namely the orbits with $\beta = \twoMatrix{-i\xi^\CC_0}{0}{0}{i\xi^\CC_0}$, $\beta = \twoMatrix{-i\xi^\CC_0}{t}{0}{i\xi^\CC_0}$, and $\beta = \twoMatrix{-i\xi^\CC_0}{0}{t}{i\xi^\CC_0}$, where $t\in \CC^\times$.

On the other hand, we may immediately see that most\footnote{in the technical sense of on a dense, open sublocus} orbits of complex Nahm data are unaffected by the stability condition -- i.e., that they are stable for any choice of $\xi^{\mb{R}}$. Indeed, we claim that all orbits of type (i) above admit no nontrivial subdata whatsoever (and hence are always automatically stable). This change of $\N_{\xi}$ only up to birational isomorphism as $\xi^{\mb{R}}$ varies is precisely the sort of variational (VGIT) result expounded upon in footnote~\ref{footnote:VGIT}.

\begin{prop}\label{prop:mostlystable}
Any complex pre-Nahm data $(\alpha,\beta)$ with $\alpha$ in an orbit of type (i) in the classification of Proposition \ref{prop:compAx} with $\alpha_0 := a+ic\not\in \{0,i\pi/2L\}$ admits no rank one subdata.
\end{prop} 
\bp Suppose we have a complex line subbundle $\mc{L} \subset E$, i.e., a map $\ell\colon I \to \mb{CP}^1$ tracking in which direction the complex line within $\mb{C}^2$ lies for every point $t \in I$. Then it is not difficult to translate the conditions that $\mc{L} \subset E$ provides pre-Nahm subdata as follows: (i) $\ell|_{\partial I} \in \{0, \infty\}$, (ii) $\partial \ell = \alpha_0 (\ell^2 - 1)$, and\footnote{Note, for those uncomfortable with differentiating a function near some point where its value is $\infty$, that one may perform the usual change of coordinates $\ell \mapsto \ell^{-1}$ in a neighborhood thereof.} (iii) $\ell$ is an eigenline for $\beta$ for all $t \in I$. But, even the first two conditions are impossible to satisfy simultaneously.
\ep
\begin{remark}
Proposition \ref{prop:compSings1} will provide another sense in which most orbits are unaffected by the stability condition: unless $\xi_0^\CC = \pm \xi_L^\CC$, we have $\D^{\rm st}_{\xi} = \D_{\xi^\CC}$ for all $\xi^\RR$.
\end{remark}

We now turn to the proof of our axial gauge statements. As in \S\ref{subsec:flatorb2}, the first observation is that the connection, up to gauge equivalence, is completely determined by its monodromy matrix. Indeed, we could essentially derive this statement from the argument of Lemma~\ref{lem:mon} by ``doubling'' the connection to be a connection over $S^1$, albeit with \emph{a priori} only mild regularity at the two fixed points of the $Z_2$-action. Alternatively, one may repeat the argument on the interval, and so we first give the general statement below.

%This used to say \emph{closed} subgroups below.
\begin{prop}\label{prop:mon} Suppose $G$ is a (finite-dimensional) connected Lie group endowed with a choice of invariant Hermitian pairing on its Lie algebra $\mf{g}$, and suppose $G$ has subgroups $G_0, G_L \subset G$. Denote the corresponding Lie subalgebras $\mf{g}_0, \mf{g}_L \subset \mf{g}$ and the orthogonal complements thereof as $\mf{g}_0^{\perp}, \mf{g}_L^{\perp}$. Then 
%MZ{semisimple? otherwise can't the pairing be degenerate, eg not hermitian? or is the point that there's still an invariant pairing on $\mf{u}(1)$, it's just not the Killing pairing? or is the point that by hermitian we're including degenerate?} Lie group with closed subgroups $G_0, G_L \subset G$ and corresponding Lie (sub)algebras $\mf{g}_0, \mf{g}_L \subset \mf{g}$ with corresponding orthogonal subspaces $\mf{g}_0^{\perp}, \mf{g}_L^{\perp}$ under an invariant Hermitian pairing on $\mf{g}$. Then 
the parallel transport map\footnote{We are using the notation $g^{-1}\partial g$ somewhat informally. More precisely, evaluated at any point $t \in I$, we have that $\partial g(t) \in T_{g(t)}G$, and letting $L_{g^{-1}}\colon G \to G$ denote the operation of left-multiplication by $g^{-1}$ so that $DL_{g^{-1}}$ maps $T_gG \to T_{\mathrm{id}}G \simeq \mf{g}$, we really mean $DL_{g^{-1}}(\partial g) \in \mf{g}$. Of course, our informal notation is meaningful as soon as one embeds $G$ in a matrix group.}
\begin{equation*}
\{ \alpha \in C^{\infty}(I, \mf{g}) \bigm| \alpha|_{\partial I} \in \mf{g}_{\partial I}^{\perp} \} / \{g \in C^{\infty}(I, G) \bigm| g|_{\partial I} \in G_{\partial I}, g^{-1}\partial g|_{\partial I} \in \mf{g}_{\partial I}^{\perp}\} \stackrel{PT}{\to} G_0 \backslash G / G_L \end{equation*} is an isomorphism of groupoids; i.e., it induces a bijection of quotient sets and isomorphisms of corresponding stabilizer groups. \end{prop}
\bp The parallel transport map is defined as usual as a map from the space of connections above to $G$; acting by gauge transformations with boundary values in $G_0, G_L$ as above may affect the parallel transport by left- and right-multiplication by elements of $G_0$ and $G_L$, as indicated. We hence have a well-defined $PT$ map. 

Next, the proof of injectivity is -- as in Lemma~\ref{lem:mon} -- straightforward. If $PT(\alpha_2) = g_0^{-1} PT(\alpha_1) g_L$, then consider the differential equation $g^{-1} \partial g + \Ad_{g^{-1}} \alpha_1 = \alpha_2$ with boundary condition $g(0) = g_0$. ODE theory provides a (unique) solution, and this solution moreover satisfies the equation $g(L) = PT(\alpha_1)^{-1} g_0 PT(\alpha_2) = g_L$. So, $g|_{\partial I} \in G_{\partial I}$, and the defining differential equation of $g$ then also implies that $g^{-1}\partial g|_{\partial I} \in \mf{g}^\perp_{\partial I}$. Therefore, the transformation $g$ lies in the desired gauge group.

%\MZ{; in this case, it even follows immediately from ODE theory that given two connections with the same parallel-transport matrix in $G_0 \backslash G / G_L$, one may construct a gauge-transformation satisfying the desired boundary conditions.} 

The stabilizer group of a connection $\alpha$ is the set of the gauge transformations $g$ that satisfy $g^{-1} \partial g + \Ad_{g^{-1}}(\alpha) = \alpha$, and so the uniqueness of solutions of ordinary differential equations implies that they are determined by their initial value $g(0) \in G_0$. Of course, solutions of this differential equation need not satisfy the boundary conditions at $L$ for all choices of $g(0)$; the condition on $g(0)$ is that the final value $g(L) = PT(\alpha)^{-1} g(0) PT(\alpha)$ lie in $G_L$. This characterization of the stabilizer of $\alpha$ makes it clear that it is isomorphic to the stabilizer of $PT(\alpha)$ in $G_0 \times G_L$.

Surjectivity may be argued as follows: consider a connection $\alpha$ which is compactly supported on the interior of $I$, and on its support is simply given by a scalar (i.e., valued in $\mb{R}$) bump function times a fixed element $X \in \mf{g}$. The transport automorphism $PT(\alpha)$ is simply $\exp(X)$, if the bump function has integral from 0 to $L$ equal to 1. Hence, if we knew $\mf{g} \stackrel{\exp}{\to} G$ were surjective, the result would follow (e.g., for compact groups). But in complete generality, $\exp(\mf{g})$ always generates $G$, and so we may establish surjectivity simply by cobbling together multiple bump-function profiles in the support of our connection $\alpha$. \ep

\bp[Proof of Proposition~\ref{prop:realAx}] Upon applying Proposition~\ref{prop:mon} for $G = SU(2), G_0 = G_L = T \subset SU(2)$, the majority of this statement follows from the observation that the parallel transport map $$\{i c \sigma_x \bigm| c \in [0, \pi/2L]\} \stackrel{\exp(L-)}{\to} T\backslash SU(2)/T$$ is a bijection of sets.

One may see this bijection in various ways; one geometric way to proceed is as follows. Note that $SU(2)/T \simeq \mb{CP}^1$, with the remaining $T$-action given by rotation of the sphere about one axis. The resultant coarse quotient space is the interval, and in fact one may explicitly describe this bijection $T \backslash SU(2) / T \stackrel{H}{\to} [0, 1]$ via $\begin{pmatrix} a & b \\ c & d \end{pmatrix} \mapsto ad$. It is then immediate to verify that $H(\exp(i c L \sigma_x))$ bijectively maps to $[0, 1]$ for $c \in [0, \pi/2L]$.

This presentation is also helpful for computing stabilizers, as now as an orbifold, $T \backslash SU(2) / T \simeq U(1) \backslash \mb{CP}^1$ is an interval whose endpoints have $U(1)$ stabilizers. 

Alternatively, we may simply study these stabilizers directly within $\A'/\G$ as follows. Suppose that $A^0 = ic\sigma_x$ and that it is stabilized by a non-trivial $g\in \G$. The latter condition is equivalent to $\partial_{A^0} g = 0$, where $g$ is regarded as an element of $\End(E)$. Uniqueness of solutions of ordinary differential equations implies that the solutions to this equation take the form $g(t) = \exp(-i c t \sigma_x) U \exp(ic t \sigma_x)$, where $U = g(0)$ is constant. Special unitarity of $g$ implies that $U$ should be special unitary. The boundary conditions on $g$ are satisfied at $t=0$ if $U$ is diagonal. Finally, the boundary conditions at $t=L$ require $c\in\{0, \pi/2L\}$, whence one may compute the stabilizer groups directly. \ep

\bp[Proof of Proposition~\ref{prop:compAx}] We now apply Proposition~\ref{prop:mon} to $G = SL(2, \mb{C}), G_0 = G_L = T_{\mb{C}} \subset SL(2,\mb{C})$. It is now a similar but slightly more involved task to verify that the five representatives listed map bijectively to $\mb{C}^{\times} \backslash SL(2,\mb{C}) / \mb{C}^{\times}$ and that there are still just two points with nontrivial stabilizers. One way to proceed is to again make use of the map $H\colon \twoMatrix{a}{b}{c}{d}\mapsto ad$, regarded now as a map $\CC^\times\backslash SL(2,\CC)/\CC^\times \to \CC$. It is no longer a bijection but rather merely the maximal Hausdorff quotient of the non-Hausdorff coarse quotient space $\mb{C}^{\times}\backslash SL(2,\mb{C}) / \mb{C}^{\times}$. Away from $0,1\in \CC$ it is a bijection, and the preimages of these two points are each homeomorphic to the quotient $\{(z_1,z_2)\in \CC^2 \ | \ z_1 z_2 = 0\}/\CC^\times$ which, for $\lambda\in \CC^\times$, identifies $(z_1,z_2)$ with $(\lambda z_1, \lambda^{-1} z_2)$. This quotient consists of the orbits of $(0,0)$, $(1,0)$, and $(0,1)$. The latter two orbits, over $1, 0 \in \mb{C}$, respectively, correspond to cases (ii),(iii) and (iv),(v) of Proposition \ref{prop:compAx}.
\ep

\bigskip

We next briefly mention a natural generalization of our constructions in the second part of this paper:
\begin{constr}\label{constr:gen} Given any connected, compact Lie group $G$ with closed subgroups $G_0, G_L \subset G$ and a choice of positive-definite invariant inner product on $\mf{g}$, we could consider the analogous gauge theory problem where the ambient space is the space of connections as stated in Proposition~\ref{prop:mon} times a copy of $$\mc{A}''_{\xi} := \{A^i \in C^\infty(I,\mf{g})^{\oplus 3} \, \bigm| \, A^i|_{\partial I} \in \vec{\xi}_{\partial I} + \mf{g}_{\partial I}^{\perp}\}\,.$$
The parameters $\xi$ are now naturally valued in $\mf{z}^3_{\partial I}$, where $\mf{z}_{\partial I} \simeq (\mf{g}_{\partial I})^{G_{\partial I}}$. All our results immediately generalize to this setting, and as usual, we focus attention on the $\mf{su}(2)$ case both due to historical tradition and due to its primacy as a special case of our main Conjecture~\ref{conj:mainconj}. We will, in future work, give a unified treatment of moduli spaces of Nahm data not only for these more general boundary conditions but also for still more general ``brane configurations'' ~\cite{mz:bigNahm}. \end{constr}

Before leaving the generality of Construction~\ref{constr:gen}, however, it may be of interest to note that %in this more general setting, the statement of real axial gauge holds if and only if the map $\mf{g}_0^{\perp} \cap \mf{g}_L^{\perp} \stackrel{\exp}{\to} G_L \backslash G / G_0$ is surjective. Does this hold in general? 
the complex moduli spaces $\mc{N}_{\xi}$ are deformations of partial resolutions of $T^*\Big((G_0)_{\mb{C}} \backslash G_{\mb{C}} / (G_L)_{\mb{C}}\Big)$, appropriately interpreted. Indeed, applying Proposition~\ref{prop:mon} for $(G_0)_{\mb{C}},(G_L)_{\mb{C}} \subset G_{\mb{C}}$ allows us to think of $\alpha$ in terms of its parallel transport automorphism $PT(\alpha)$, and then $\beta$ satisfying the complex Nahm equation is determined by $\beta(0) \in (\mf{g}_0^{\perp})_{\mb{C}}$ together with a condition that $\beta(L) = PT(\alpha)^{-1} \beta(0) PT(\alpha) \in (\mf{g}_L^{\perp})_{\mb{C}}$.

\bigskip

In fact, here is one way to contextualize the above construction. Kronheimer~\cite{kronheimer:cotangent} gives a construction of $T^*G_{\mb{C}}$, for any connected compact Lie group $G$, as the hyperkahler quotient of a space of Nahm data $\widehat{\mc{A}}$ by a gauge group $\mc{G}_0$, where $\widehat{\mc{A}}$ is even larger than $\mc{A}_{\mathrm{all}}$ in that \emph{no} boundary conditions are imposed, while $\mc{G}_0$ demands that gauge transformations be trivial at both boundaries. This hyperkahler quotient manifold $T^*G_{\mb{C}}$ continues to admit tri-Hamiltonian $G$-actions from both the left and right, and one could attempt to take a further hyperkahler quotient. Doing so by the subgroups $G_0, G_L \subset G$ on the right and left, respectively -- and re-parsing this two-step hyperkahler quotient construction\footnote{Note that this two-step procedure for constructing moduli spaces of Nahm data has long been the dominant strategy in the literature. See, e.g.,~\cite{dancer:alf,dancer:alf2,takayama:bow}.} as a one-step construction -- is the content of Construction~\ref{constr:gen} above. The particular specialization to the case of interest for us is the original construction by Dancer~\cite{dancer:alf2} of the $D_2$ ALF family of manifolds. For concreteness, we now explain this comparison in full detail.

The first step is Kronheimer's construction of cotangent bundles of complex reductive groups, which in this case we again state in full detail:

\begin{thrm}[Kronheimer~\cite{kronheimer:cotangent}]\label{thm:kroncot} There is a canonical isomorphism $$\tilde{\M}' := C^{\infty}(I,\mf{su}(2))^{\oplus 4}\hqnoxi \{g \in C^{\infty}(I,SU(2)) \bigm| g|_{\partial I} = \mathrm{Id}\} \simeq T^*SL(2,\mb{C})$$ of holomorphic symplectic manifolds. The remaining canonical action of $SU(2) \times SU(2)$ via its action on the boundary values on the LHS above is tri-Hamiltonian; note that on the RHS, this action is via left- and right-multiplication.
\end{thrm}

\begin{rmk} Kronheimer establishes this result for $C^2$ gauge transformations acting on $C^1$ pre-Nahm data, but analogous arguments to those of Propositions~\ref{prop:bij} and~\ref{prop:smoothReps} allow for a comparison with the $C^{\infty}$ case. The same is true for Dancer's construction. \end{rmk}

Dancer now proceeds by taking the hyperkahler quotient of the above by the $T \times T$ subgroup of $SU(2) \times SU(2)$. We will implicitly use the canonical isomorphism of the space of moment map parameters $((\mf{t} \times \mf{t})^{\vee})^{T \times T} \simeq \mf{t} \times \mf{t}$ with our $\mf{z}_{\partial I}$ in the following statement:

\begin{thrm}[Dancer~\cite{dancer:alf2}] For generic $\xi \in \mf{z}^3_{\partial I}$, the hyperkahler quotient $$\M'_{\xi} := \tilde{\M}'\hq(T \times T)$$ is a smooth, hyperkahler fourfold. This family of fourfolds is termed the \emph{$D_2$ ALF} family of manifolds. \end{thrm}

\begin{rmk} In fact, Dancer begins with the analogous Kronheimer construction $$C^{\infty}(I,\mf{u}(2))^{\oplus 4}\hqnoxi \{g \in C^{\infty}(I,U(2)) \bigm| g|_{\partial I} = \mathrm{Id}\} \simeq T^*GL(2,\mb{C})$$ before then further hyperkahler quotienting by the $\ker\Big(U(2) \times U(2) \stackrel{\det \times \det}{\to} U(1) \times U(1) \stackrel{\times}{\to} U(1)\Big)$ subgroup of $U(2) \times U(2)$, but simplifying to $SU(2)$ throughout is straightforward by judicious application of the following Fact~\ref{fact:third}. \end{rmk}

It will be convenient to explicitly note the following fact, whose proof is immediate.

\begin{fact}\label{fact:third} If $M$ is a hyperkahler manifold with tri-Hamiltonian action by $G$, and if $H \subset G$ is a closed normal subgroup, then $M\hqnoxi H$ admits a canonical tri-Hamiltonian $G/H$-action. Moreover, we have a canonical inclusion of the space of moment map parameters $$((\mathrm{Lie}\ G/H)^{\vee})^{G/H} \hookrightarrow ((\mathrm{Lie}\ G)^{\vee})^G\,,$$ and if $\xi$ is any triple of elements of the former, we have $(M \hqnoxi H)\hq (G/H) \simeq M\hq G$.\end{fact}

\begin{rmk} It is reasonable to ask in the above if one could also deform the first hyperkahler quotient to $M\hq H$ and similarly obtain a `third isomorphism theorem for hyperkahler quotients' statement; indeed, as moment maps are only well-defined up to this space of parameters, it is somewhat disingenuous to not do so -- although we note that in the case of interest here, the original construction of $T^*G_{\mb{C}}$ as a hyperkahler quotient does not admit moment map parameters and so there is no ambiguity. 

To make precise sense of the current desideratum requires some slight extra notation. For simplicity, we discuss the analogous theorem for symplectic quotients, i.e., $(M\kqnoxi H)\kqnoxi (G/H) \simeq M\kqnoxi G$, as the version for hyperkahler quotients is identical save for involving triples of moment map parameters everywhere. Essentially, then, we would like to relate moment map parameters for $G$ to those for $H$ and $G/H$. We first note the following diagram:

\begin{tikzcd} 0 \ar[r] & ((\mathrm{Lie}\ G/H)^{\vee})^{G/H} \ar[r] & ((\mathrm{Lie}\ G)^{\vee})^{G} \ar[r,"\pi"] \ar[dr] & ((\mathrm{Lie}\ H)^{\vee})^{G} \ar[r] \ar[d, hook] & 0 \\ & & & ((\mathrm{Lie}\ H)^{\vee})^{H} \makebox[0pt][l]{\,.}\end{tikzcd}

Now, abstractly, given a symplectic manifold $(M, \omega)$ with Hamiltonian action by a group $G$, let us denote by $\mf{z}_G(M)$ the space of moment map parameters for the $G$-action on $M$, i.e., the space of $G$-invariant maps $M \stackrel{\mu}{\to} \mf{g}^{\vee}$ such that $d\mu(X) = \iota_X \omega$. Then a slightly more precise statement to make is that $\mf{z}_G(M)$ is an affine space over the vector space $((\mathrm{Lie}\ G)^{\vee})^G$. In the situation at hand, we then claim that (i) we have a natural map $\mf{z}_G(M) \stackrel{\pi}{\twoheadrightarrow} (\mf{z}_H(M))^G$ as a torsor over the map denoted $\pi$ above, and (ii) given $\xi \in (\mf{z}_H(M))^G$, we have a canonical isomorphism $\pi^{-1}(\xi) \simeq \mf{z}_{G/H}(M\kq H)$ of affine spaces over $((\mathrm{Lie}\ G/H)^{\vee})^{G/H}$. Supposing that $\nu, \omicron$ denote identified elements on either side of the above isomorphism, we may hence make sense of the statement $(M\kq H)\mathord{/ \!\! /_{\omicron}}(G/H) \simeq M\mathord{/ \!\! /_{\nu}} G$. \end{rmk}

\begin{prop}\label{prop:comp} There is a natural isomorphism $\M_{\xi} \simeq \M'_{\xi}$ preserving the hyperkahler structure. \end{prop}
\bp All the isomorphisms we write of hyperkahler quotients below will naturally preserve the hyperkahler structure; i.e., the K\"ahler forms on the target will pull back to those on the source. It suffices to establish an isomorphism of sets (or groupoids, if we wish to treat nongeneric $\xi$).

We first note the following isomorphism on just the connection datum:
\begin{align*}&\{A^0 \in C^{\infty}(I,\mf{su}(2)) \bigm| A^0|_{\partial I} \in \mf{t}^{\perp}\} / \{g \in C^{\infty}(I,SU(2)) \bigm| g|_{\partial I} = \mathrm{Id}, g^{-1} \partial g|_{\partial I} \in \mf{t}^{\perp}\} \stackrel{\sim}{\to} \\ &\hspace{1.5cm} C^{\infty}(I,\mf{su}(2)) / \{g \in C^{\infty}(I,SU(2)) \bigm| g|_{\partial I} = \mathrm{Id}\} \,.\end{align*} As, for example, explained in~\cite{kronheimer:cotangent}, the moment map equations imposed in the hyperkahler quotient construction of $\tilde{\M}'$ are exactly Nahm's equations, but the above allows for the simplification:

\begin{align*} \tilde{\M}' &= \{ A^{\mu} \in C^{\infty}(I,\mf{su}(2))^{\oplus 4} \bigm| A^{\mu}\text{ satisfy Nahm's equations}\} / \{g \in C^{\infty}(I,SU(2)) \bigm| g|_{\partial I} = \mathrm{Id}\} \\ &\hspace{-.25cm}\stackrel{\sim}{\longleftarrow} \{A^{\mu} \in C^{\infty}(I,\mf{su}(2))^{\oplus 4} \bigm| A^{\mu}\text{ satisfy Nahm's equations}, A^0|_{\partial I} \in \mf{t}^{\perp}\} \ / \\ &\hspace{2.5cm}\{g \in C^{\infty}(I,SU(2)) \bigm| g|_{\partial I} = \mathrm{Id}, g^{-1} \partial g|_{\partial I} \in \mf{t}^{\perp}\}\end{align*}

We now wish to further impose a hyperkahler quotient by $T \times T$ on the above, but in terms of the latter presentation of $\tilde{\M}'$, the condition imposed by the moment maps is extremely easy to parse: it is simply the condition that $(A^i_z)|_{\partial I} = \xi^i_{\partial I}$. We hence immediately find the desideratum \begin{align*} \tilde{\M}'\hq(T \times T) \stackrel{\sim}{\leftarrow} \{A^{\mu} \in C^{\infty}(I,\mf{su}(2))^{\oplus 4} \bigm| A^{\mu}\text{ satisfy Nahm's equations}, A^{\mu}|_{\partial I} \in -i \xi^{\mu}_{\partial I} \sigma_z + \mf{t}^{\perp}\} / \widetilde{\mc{G}}\,.\end{align*} \ep

In some sense, having established the above comparison proposition, we could immediately deduce all the statements we wish for the $r = 1$ case of our main Conjecture~\ref{conj:mainconj}. For example, there is little discussion of stability and DUY isomorphisms in~\cite{dancer:alf2}, but one could essentially `import' the robust finite-dimensional Kempf-Ness theory and cheaply pull it back to deduce appropriate notions of stability, etc., on the infinite-dimensional space of complex (pre-)Nahm data. Indeed, we may immediately use this comparison to now list standard facts, such as asymptotically cubic volume growth, or even more sharply, that the asymptotic geometry is that of a circle of circumference $\pi/L$ fibered over $\mb{R}^3/Z_2$.

In general, however, we eschew this approach for multiple reasons. Most obviously, we would like to present an approach that most immediately generalizes to higher dimensions -- i.e., higher values of $r$ -- for the purposes of Conjecture~\ref{conj:mainconj}. For example, the study of stability given here should read as parallel to the sketch given in~\ref{subsec:sheaf}. Doing the infinite-dimensional analysis directly -- e.g., in the current $r = 1$ case, studying how boundary terms produced by integration by parts give rise to various subtleties -- is a clear model for the regularity conditions at the nontrivially-stabilized locus $F \subset T^{\vee}_{\Lambda}$ in higher dimensions. And on the negative side, the K\"{a}hler forms will not straightforwardly generalize from $\mc{A}_{\xi}$ to $\mc{A}_{\mathrm{all}}$ in higher dimensions as the variations incurred by changing the parameters $\xi$ will not be square-integrable. In particular, it seems unlikely for there to be an analog of any multi-step quotient procedure.\footnote{Similarly, for physically-inclined readers, one does not path-integrate over the boundary value of $A_z$ as it is fixed by boundary conditions. Relatedly, it is physically important that this D-brane moduli space is given by a single hyper-K\"ahler quotient of an affine space, as opposed to a multistep procedure. In particular, the physical gauge group is $\tilde\G$.}

Finally, we recall again that one of the original motivations for our work and for the main Conjecture~\ref{conj:mainconj} is to give explicit analytic constructions of the Ricci-flat hyperkahler metrics on manifolds about the flat orbifold limit. Our one-step hyperkahler quotient of a flat space gives an explicit paradigm for how to perform the appropriate Banach contraction construction of Taylor series expansion of solutions of the moment map (Nahm) equations in the $\xi$ deformation parameters. The multistep quotient approach, by comparison, is at best more opaque for this purpose.

Along these lines, it would also be desirable to prove the known results about the asymptotic geometry of our moduli spaces directly, as opposed to by appealing to alternative descriptions of the geometry. It should be possible to do so by following the approach of \cite{rafeLaura:Nf4} for studying the asymptotic geometry of ALG parabolic Higgs bundle moduli spaces. We leave this for future work.

We finally mention one possible puzzle. Given Fact~\ref{fact:third}, why is Dancer's construction intrinsically a multistep operation? It certainly seems that one should be able to perform a single-step hyperkahler quotient of $C^{\infty}(I,\mf{su}(2))^{\oplus 4}$ by a slightly larger group than in Kronheimer's construction, namely, $\{g \in C^{\infty}(SU(2)) \bigm| g|_{\partial I} \in T\}$. The difficulty is intrinsically analytic: the group above \emph{does not admit} a tri-Hamiltonian action on the ambient space, and so the first hypothesis of Fact~\ref{fact:third} fails from the outset. The reason for the failure of a moment map harks back to the proof of Proposition~\ref{prop:Hamtrick}, where we see precisely that Nahm's equations give good moment maps provided that certain boundary terms provided by integration by parts vanish. Informally, the variation of the gauge transformation at the boundary must be orthogonal to the allowed values of the pre-Nahm data there. This is the case in Kronheimer's construction of Theorem~\ref{thm:kroncot} and in our main construction per Definition~\ref{defn:maindefn}, but this vanishing fails for the proposed one-step construction above. We hence, once again, see the importance of careful treatment of regularity about the special points in our main Conjecture~\ref{conj:mainconj}.

\section{Moduli space of real Nahm data} \label{sec:hkq}

This section is devoted to the proof of the following theorem:
\begin{theorem} \label{thm:modHK}
Given definitions as in Definition~\ref{defn:maindefn}, the map $\M_{\mathrm{all}} \stackrel{p}{\to} \mf{z}^3_{\partial I}$ restricts to a smooth submersion of manifolds $\M^{\mathrm{irr}}_{\mathrm{all}} \stackrel{p}{\to} \mf{z}^3_{\partial I}$. Its fibers $\M^{\mathrm{irr}}_{\xi}$ over $\xi \in \mf{z}^3_{\partial I}$ are smooth (possibly empty) hyperkahler four-manifolds. 

If $\xi \in \mf{z}^3_{\partial I}$ is generic, all points of $\M_{\xi}$ are irreducible while if $\xi$ is nongeneric but not identically vanishing, there is a single point of $\M_{\xi} \setminus \M_{\xi}^{\mathrm{irr}}$. Finally, if $\xi$ identically vanishes, $\M_{\xi} \simeq (\mb{R}^3 \times S^1)/Z_2$ and the irreducible locus is precisely the complement of the two $A_1$ orbifold singularities. 
\end{theorem}

As usual in gauge theory, our approach to showing that $\M^{\mathrm{irr}}_{\mathrm{all}} := \mu^{-1}(0)/\mc{G}$ carries a manifold structure will occur via a comparison to an analogous quotient where everything is only Sobolev-regular. Recall our convention that we denote degrees of Sobolev regularity by a superscript. Hence, below, we will refer to the maps $\mu\colon\mc{A}_{\mathrm{all}} \to \mc{F}$ and $\mu^s\colon\mc{A}^s_{\mathrm{all}} \to \mc{F}^{s-1}$, respectively. Unless otherwise stated, we assume henceforth that $s\ge 1$. The comparison then takes place in the following three steps. See~\cite{DK} for an excellent textbook-level exposition of this strategy in a similar context.

\begin{enumerate}[(i)] \item First, show that the canonical continuous maps $\mc{M}_{\mathrm{all}} \to \mc{M}^s_{\mathrm{all}}$, i.e., $\mu^{-1}(0)/\mc{G} \to (\mu^s)^{-1}(0)/\mc{G}^{s+1}$, are isomorphisms of sets (or groupoids), and that the same is true for the irreducible subloci. \item Second, for any finite $s$, show by the Banach manifold implicit function theorem that $\M^{s,\mathrm{irr}}_{\mathrm{all}} := ((\mu^s)^{-1}(0))^{\mathrm{irr}}/\mc{G}^{s+1}$ canonically carries a manifold structure. \item Conclude that these smooth structures compare well under changing $s$ and hence that for any $s$ (including $s = \infty$), this parameterized hyperkahler quotient is indeed a smooth manifold. \end{enumerate}

We begin with part (i) of the above strategy.

\begin{lem}\label{lem:fredA0}
For any $s \ge 1$ and for any $A^0 \in \A'^s$, the operator $\partial_{A^0}\colon \mf{g}^{s+1}\to \A'^s$ is Fredholm.
\end{lem}
\begin{proof}
Recall that $\partial_{A^0}$ is the operator $\partial + [A^0, -]$. $\mf{g}^{s+1} \stackrel{\partial}{\to} \A'^s$ is Fredholm, as the kernel is the one-dimensional space of constant functions valued in $\mf{t}$ and integration over $I$ and then orthogonal projection to $\mf{t}^\perp$ yields an isomorphism $\coker(\partial) \to \mf{t}^\perp$. Adding the compact perturbation $$\mf{g}^{s+1} \hookrightarrow \{h \in H^s(I,\mf{su}(2)) \bigm| h|_{\partial I} \in \mf{t}\} \stackrel{[A^0, -]}{\to} \mc{A}'^s$$ preserves Fredholmness. \end{proof}

\begin{lem}
For any $s \ge 1$ and $A^0\in \A'^s$, the nonlinear map $F\colon \G^{s+1} \to \A'^s$ defined by $F(g) = g(A^0)$ is Fredholm.
\end{lem}
\begin{proof}
At $g\in \G^{s+1}$, the derivative of this map is $\mf{g}^{s+1}\ni h\mapsto \partial_{g(A^0)} h$. The previous lemma implies that this is Fredholm.
\end{proof}

\begin{prop}\label{prop:bij}
Given $s \ge 1$, we have the following: \begin{enumerate}[(i)] \item the natural map $\iota\colon \A'/\G \to \A'^s/\G^{s+1}$ is an isomorphism of groupoids, i.e., $\iota$ induces both a bijection of orbits and isomorphisms of stabilizer groups thereof, 
\item Nahm data $A \in \mc{A}' \times (\mc{A}''_{\mathrm{all}})^s$ is smooth, i.e., in $\mc{A}' \times \mc{A}''_{\mathrm{all}}$, 
\item $\mc{M}_{\mathrm{all}} \to \mc{M}^s_{\mathrm{all}}$ is an isomorphism of groupoids, and 
\item $\M^{\mathrm{irr}}_{\mathrm{all}} \to \M^{s,\mathrm{irr}}_{\mathrm{all}}$ is a bijection of sets. 
\end{enumerate} 
\end{prop}

\bp
Recall that two connections $A_1^0, A_2^0$ are related by a gauge transformation $g$ if and only if $g$ is a flat section of the bundle $\mathrm{End}(E) \simeq E \otimes E^{\vee}$ endowed with the connection $A_1^0 * A_2^0$. If $A_1^0, A_2^0$ are both smooth, then even if $g \in \mc{G}^{s+1}$ \emph{a priori}, elliptic regularity ensures that $g \in \mc{G}$. This both shows injectivity of the map in (i) and that stabilizer groups are identified.

We argue for surjectivity of the map in (i) following the argument of~\cite[Lemma 14.8]{atiyah:ym}. Given $A^0 \in \mc{A}'^s$, consider the Fredholm map $\mc{G}^{s+1} \stackrel{F}{\to} \mc{A}'^s$ given by $g \mapsto g(A^0)$. As the kernel and cokernel of the derivative of this group action have constant dimension, the constant rank theorem allows us to choose local coordinates that straighten the above map. Explicitly, we may find local neighborhoods $0 \in U' \subset \mf{g}^{s+1}$, $0 \in V' \subset \A'^s$, $\mathrm{Id} \in U \subset \G^{s+1}$, $A^0 \in V \subset \A'^s$ and local diffeomorphisms $U' \stackrel{u}{\to} U, V' \stackrel{v}{\to} V$, such that $F|_U = v \circ dF(\mathrm{Id})|_{U'} \circ u^{-1}$. We will find it convenient to pick $V$ to be a convex set; e.g., a ball. 

Denote now $r = \dim \coker dF(\mathrm{Id})$ and let $N$ be any complement to the image of $dF(\mathrm{Id})$; let $\pi$ denote the composition $V \stackrel{v^{-1}}{\to} V' \hookrightarrow \A'^s \stackrel{\mathrm{proj}}{\to} N$. Note that the intersection of the image of $F|_U$ with $V$ (i.e., the set of connections in $V$ which are gauge-equivalent to $A^0$ via gauge transformations in $U$) is $\pi^{-1}(0)$. Choose points $B_0, \cdots, B_r \in V$ with center of mass $A^0$ such that if one linearly embeds the $r$-dimensional simplex $\Delta^r$ in $V$ by taking the vertices to be the $B_i$, then $\Delta^r$ is injectively mapped to $N$ with $0$ in the interior of its image; this is possible because if the $B_i$ are close enough to $A^0$ then $\pi$ is approximated by its linearization at $A^0$, and its derivative at $A^0$ is surjective. But now, there exists some $\epsilon$ such that if $\|C_i - B_i\|_{H^s} < \epsilon$ for all $0 \le i \le r$, then $0$ will still be in the interior of the image of $\Delta^r$ if one modifies the vertices to lie at the $C_i$ rather than the $B_i$ -- for example, one may note that the degree of the map $\partial \Delta^r \to N \setminus \{0\}$ in $\mathrm{H}_{r-1}(N \setminus \{0\}; \mb{Z}) \simeq \mb{Z}$ is invariant under small perturbations of the $B_i$ as $\pi$ is continuous. 

Then, by density of $\mc{A}' \subset \A'^s$, we may pick the $C_i$ smooth and sufficiently close to the $B_i$ such that if $C \in V$ denotes an appropriate convex linear combination of the $C_i$ which is mapped to $0$ under $\pi$, then (i) $C$ is still smooth by virtue of being a linear combination of the $C_i$, but also (ii) $\pi(C) = 0$ implies that $C$ is gauge-equivalent to $A^0$, as desired.

Part (ii) is now immediate from elliptic regularity: if $A^0$ is smooth so that $\partial_{A^0}$ is an elliptic operator with smooth coefficients, Nahm's equations $\partial_{A^0} A^i = -\frac{1}{2}\epsilon^{ijk}[A^j,A^k]$ yield that $A^i$ is Sobolev $(s+1)$-regular, whereupon one bootstraps to $A^i$ smooth for all $i$.

Part (iii) is a consequence of the first two parts, and part (iv) is the specialization thereof to the stabilizer-free locus. \ep

We would now like to use the implicit function theorem to deduce a smooth manifold structure on $\mc{M}^{s,\mathrm{irr}}_{\mathrm{all}}$. First, it is useful to establish that it is second-countable and Hausdorff. The former, as noted in the introduction, is automatic from the separability of $\mc{A}^s_{\mathrm{all}}$; the latter is then equivalent to uniqueness of limits of sequences. We take this opportunity to show not just that $(\mu^s)^{-1}(0)/\mc{G}^{s+1}$ is Hausdorff but in fact that so is $\mc{A}^s_{\mathrm{all}}/\mc{G}^{s+1}$, and even that this larger ambient quotient space admits a natural metric induced from the $L^2$ metric on $\mc{A}^s_{\mathrm{all}}$.

\begin{proposition} \label{prop:gaugeConv}
Suppose that $B_n\to B$ and $g_n(B_n)\to A$ in $L^2(I,\mf{su}(2))^{\oplus 4}$, where $A,B\in \A^s_{\mathrm{all}}$, for some $g_n \in \{g\in H^1(I,SU(2)) \bigm| g|_{\partial I} \in T \}$. Then, there exists a $g\in \G^{s+1}$ such that $g(B)=A$. Furthermore, there is a subsequence of $g_n$ whose weak limit in $H^1(I,\End(E))$ and strong limit in $C^0(I,\End(E))$ is $g$. This also holds for $s=\infty$.

Analogously, if $B_n\to B$ and $g_n(B_n)\to A$ in $C^{s-1}$ for some $g_n\in C^{s}$ (with the usual boundary conditions imposed on $B_n,g_n(B_n),A,B,g_n$), then there exists a $g\in C^{s}$ such that $g(B)=A$ and a subsequence of $g_n$ converges to $g$. This too holds for $s=\infty$.

Finally, if $B_n\to B$ and $g_n(B_n)\to A$ in $\A_{\mathrm{all}}^s$ for some $g_n\in \G^{s+1}$ then there exists a $g\in \G^{s+1}$ such that $g(B)=A$ and a subsequence of $g_n$ converges to $g$ in $\G^{s+1}$.
\end{proposition}
\begin{proof}
For the first part of the proposition, we follow the proof of Lemma (4.2.4) of \cite{DK}. We focus on the case with $s$ finite, but the infinite $s$ case is handled identically. If we regard $g_n$ as elements of $\End(E)$ then we want to solve the differential equation
\be \partial g = g A^0 - B^0 g \ , \quad B^i g = g A^i \ . \label{eq:want} \ee
We start with the analogous equation
\be \partial g_n = g_n A^0_n - B^0_n g_n \ , \quad B^i_n g_n = g_n A^i_n \ , \label{eq:known} \ee
where $A_n = g_n(B_n)$. Uniform boundedness of $g_n$ (which follows from compactness of $SU(2)$) together with \eqref{eq:known} imply that the derivatives $\partial g_n$ reside in a bounded subset of $L^2(I,\End(E))$, and therefore that $g_n$ reside in a bounded subset of $H^1(I,\End(E))$. Weak compactness of this set implies that after passing to a subsequence, $g_n$ converges weakly to a limit $g\in H^1$. The compact embedding $H^1\hookrightarrow C^0$ implies that this is strongly convergent in $C^0$. Also, $\partial g_n$ converges weakly to $\partial g$ in $L^2$, since $g\mapsto (\partial g, \phi)_{L^2}$ is a continuous linear functional on $H^1$ for all $\phi\in L^2$. The first equation in \eqref{eq:want} then holds in $L^2$, since if $\phi\in C^0$ then
\be (\partial g, \phi)_{L^2} = \lim_{n\to\infty} (\partial g_n, \phi)_{L^2} = \lim_{n\to\infty} ( g_n A_n^0 - B^0_n g_n, \phi)_{L^2} = (g A^0 - B^0 g, \phi)_{L^2} \ , \ee
where the last equality holds because we have both $g_n A_n^0 \to g A^0$ and $B_n^0 g_n\to B^0 g$ in $L^1(I,\End(E))$, by the Cauchy-Schwarz and triangle inequalities. A similar argument shows that the second equation in \eqref{eq:want} holds. Elliptic regularity of the first equation implies that $g\in H^{s+1}(I,\End(E))$ -- that is, the right hand side is in $H^1$, so the left side is in $H^2$, and then we bootstrap up to $H^{s+1}$. $g$ also satisfies the boundary conditions to be in $\G^{s+1}$: the condition $g|_{\partial I} \in T$ is clear from the $C^0$ convergence $g_n\to g$, and the condition $g^{-1} \partial g|_{\partial I} \in \mf{t}^{\perp}$ follows from \eqref{eq:want}. Finally, $g$ is special unitary, thanks to the $C^0$ convergence $g_n\to g$ and the fact that $SU(2)$ is closed in $\CC^{2\times 2}$. 

We now address the second part of the proposition, following the proof of Proposition (2.3.15) in \cite{DK}. This is proved quite similarly: uniform boundedness of $g_n$, $C^0$ convergence of $A_n^0$ and $B_n^0$, and \eqref{eq:known} together imply that $\partial g_n$ is uniformly bounded, and bootstrapping this reasoning gives that the first $s$ derivatives of $g_n$ are uniformly bounded. The Arzela-Ascoli theorem then implies that a subsequence of $g_n$ converges to some $g\in C^{s-1}$. Furthermore, per the results of the previous paragraph we have that $g\in H^1$ and $g$ solves \eqref{eq:want}. \eqref{eq:want} then implies that this convergence holds in $C^{s}$. It is trivial that $g$ satisfies the appropriate boundary conditions and is special unitary.

The third part of the proposition follows from the second, thanks to Sobolev embedding. For, once we have $g\in H^1\subset C^1$, it follows from \eqref{eq:want} that it is in $\G^{s+1}$. Furthermore, if $g_n\to g$ in $H^{s'}$ with $1\le s' \le s$, then \eqref{eq:want} and \eqref{eq:known} imply that $g_n\to g$ in $H^{s'+1}$. So, the subsequence which converges in $C^1$ in fact converges in $H^{s+1}$.
\end{proof}

\begin{corollary} \label{cor:realMet}
$d^s([A],[B]) = \inf_{g\in \G^{s+1}} \|A-g(B)\|_{L^2}$ defines a metric on $\A^s_{\mathrm{all}}/\G^{s+1}$, as does $d([A],[B]) = \inf_{g\in \G} \|A-g(B)\|_{L^2}$ on $\A_{\mathrm{all}}/\G$.
\end{corollary}
\begin{proof}
We focus on $d^s$, as the proof that $d$ is a metric is identical. Symmetry is clear from $d^s([B],[A]) = \inf_{g\in \G^{s+1}} \|B-g(A)\|_{L^2} = \inf_{g\in \G^{s+1}} \|g^{-1}(B)-A\|_{L^2} = d^s([A],[B])$. The triangle inequality is also clear. It remains to show that if we have a sequence $g_n\in \G^{s+1}$ such that $g_n(B)\to A$ in $L^2$ then there is a $g\in \G^{s+1}$ such that $g(B)=A$. This follows from the proposition.
\end{proof}

It is now cheap to note that restricting the metric inherited from the $L^2$ metric above yields a metric on $\mc{M}^s_{\mathrm{all}}$ (although this metric certainly does \emph{not} agree with the hyperkahler metric to be constructed on $\mc{M}^s_{\mathrm{all}}$ over the course of this section):

\begin{corollary}
$\M_{\mathrm{all}}^s$ and $\M_{\mathrm{all}}$ admit metrics with respect to which the bijection between these spaces is an isometry.
\end{corollary}
\begin{proof}
$\M_{\mathrm{all}}^s$ is a subspace of $\A_{\mathrm{all}}^s/\G^{s+1}$ and $\M_{\mathrm{all}}$ is a subspace of $\A_{\mathrm{all}}/\G$. Suppose that $g'\in \G^{s+1}$ relates $A\in \A^s_{\mathrm{all}}$ to $A'\in \A_{\mathrm{all}}$ and $g''\in \G^{s+1}$ relates $B\in \A^s_{\mathrm{all}}$ to $B'\in \A_{\mathrm{all}}$. Then, $d^s([A],[B]) = \inf_{g\in \G^{s+1}} \|g'(A') - g(g''(B'))\|_{L^2} = \inf_{g\in \G^{s+1}} \|A' - g'^{-1}(g(g''(B'))) \|_{L^2} = \inf_{g\in \G^{s+1}} \|A' - g(B')\|_{L^2} \le d([A'],[B']).$ On the other hand, by density of $\G$ in $\G^{s+1}$ and continuity of the action $\G^{s+1}\times \A^s_{\mathrm{all}} \to \A^s_{\mathrm{all}}$, this inequality is actually an equality.
\end{proof}

In order to put a smooth structure on this space, we will, for each point $[A]\in \M_{\mathrm{all}}^s$ with representative\footnote{We are implicitly using Proposition~\ref{prop:bij} here to assume that we may choose a \emph{smooth} representative $A$; this choice is merely for the sake of simplicity when applying elliptic regularity arguments.} $A\in \A_{\mathrm{all}}$, give a manifestly finite-dimensional description of a neighborhood of $[A]$. To do so, we introduce the following two \emph{deformation complexes}, which we disambiguate as the \emph{special} and \emph{general} deformation complexes, respectively; these complexes respectively allow us to model the structure of $\M^s_{\xi}$ and $\M^s_{\mathrm{all}}$:
\begin{align} 0 &\to \mf{g}^{s+1} \xrightarrow{d_0} \A^s_{0} \xrightarrow{d_1} \F^{s-1} \to 0 \nonumber \\ 0 &\to \mf{g}^{s+1} \xrightarrow{d_0} \A^s_{\mathrm{all}} \xrightarrow{d_1} \F^{s-1} \to 0  \ . \label{eq:realDef} \end{align}
Here, $d_0$ describes the effect of infinitesimal gauge transformations: $$d_0 h := \partial_{A^0}h\oplus [A^i, h]\,.$$ Similarly, $d_1$ is the linearization of Nahm's equations about $A$: 
$$d_1 a := (\partial_{A^0} a^i - [A^i, a^0] + \epsilon^{ijk}[A^j, a^k])_{i=1,2,3}\,.$$

Note in the notation above that $d_0, d_1$ depend on the choice of $A \in \mc{A}_{\mathrm{all}}$, which is suppressed from the notation. That these are indeed complexes, i.e., that $d_1 d_0 = 0$, is precisely the statement that $A$ satisfies Nahm's equations. We can therefore define the cohomology groups ${\rm H}^{0,{\rm sp}}_A = {\rm H}^{0,{\rm gen}}_A = \ker d_0$, ${\rm H}^{1,{\rm sp}}_A = \ker d_1^{\rm sp}/\im d_0$, ${\rm H}^{1,{\rm gen}}_A = \ker d_1^{\rm gen}/\im d_0$, $\mathrm{H}^{2,{\rm sp}}_A = \coker d_1^{\rm sp}$, and $\mathrm{H}^{2,{\rm gen}}_A = \coker d_1^{\rm gen}$, where the superscripts on $d_1$ (which we will generally omit) specify the boundary conditions we impose. We also introduce formal adjoint differential operators $d_i^*$; these are defined in a formal $L^2$ sense, ignoring any boundary terms from integration by parts. Explicitly, we have $d_0^* a = -\partial_{A^0} a^0 - [A^i, a^i]$ and $d_1^* f = [A^i, f^i] \oplus (-\partial_{A^0} f^i + \epsilon^{ijk} [A^j, f^k])$, where repeated indices are summed from 1 to 3. These boundary terms require a careful functional analytic treatment of adjoints in order to get familiar results such as Hodge theorems for these complexes, as we will discuss below. However, it is the case that $\mathrm{H}^{0,{\rm sp}}_A = {\rm H}^{0,{\rm gen}}_A = \ker d_0 = \ker(\Delta_0\colon\mf{g}^{s+1}\to H^{s-1}(I,\mf{su}(2)))$, where the Laplacians of the complexes are defined to be $\Delta_i = d_i^* d_i + d_{i-1} d_{i-1}^*$; in particular, if $h\in \mf{g}$ satisfies $\Delta_0 h = 0$ then we have $0=(\Delta_0 h, h)_{L^2} = (d_0 h, d_0 h)_{L^2}$, thanks to the boundary conditions on $h\in \mf{g}$ and $d_0 h\in \A_0$.

The complexes are very closely related, as follows from investigating their cohomology groups. Clearly, they have the same $\mathrm{H}^0_A$, and so we will allow ourselves to omit the sp and gen superscripts for this group. Additionally, after we show that these complexes are Fredholm and hence have well-defined indices, we will show that the index of the general deformation complex is the index of the special deformation complex minus the dimension of $\mf{z}^3_{\partial I}$.

Roughly speaking, the tangent spaces to $[A]$ in $\M_{\rm all}$ and $\M_\xi$, respectively, are given by $\mathrm{H}^{1,{\rm gen}}_A$ and ${\rm H}^{1,{\rm sp}}_A$. There are two ways that this statement can fail. First, there may be an obstruction that prevents a solution of the linearized Nahm's equations from being perturbed to a genuine solution of Nahm's equations; such an obstruction would be valued in $\mathrm{H}^{2,{\rm gen}}_A$ or ${\rm H}^{2,{\rm sp}}_A$. Second, if $A$ has a nontrivial stabilizer then we must quotient our space of linearized perturbations by it. The Lie algebra of this stabilizer is $\mathrm{H}^0_A$, which is naturally regarded as a Lie subalgebra of $\mf{g}$. (The action of infinitesimal elements of the stabilizer on points in $\A_{\mathrm{all}}^s$ close to $A$ is trivial, but it is still necessary to consider the quotient by the stabilizer because non-infinitesimal elements keep one close to $A$.)

The main results about elliptic complexes such as these were proved by Atiyah and Bott in \cite[\S 6]{atiyah:complex}. For the reader's convenience, as well as to assuage any concerns that boundary conditions and/or Sobolev (as opposed to $C^\infty$) regularity might modify these conclusions, we will now re-prove them.

\begin{lemma}
Let $D\colon \mathscr{B}^{s+m} \to \mathscr{C}^{s}$ be an order $m$ elliptic linear operator acting between Sobolev spaces on $I$ with (finitely many homogeneous) boundary conditions, where $\mathscr{B}^{s+m}$ has Sobolev regularity $s+m$ and $\mathscr{C}^s$ has Sobolev regularity $s\ge 0$. This operator is Fredholm with a pseudodifferential parametrix $P\colon \mathscr{C}^s\to \mathscr{B}^{s+m}$ such that $DP\sim 1$ and $PD\sim 1$, where $\sim$ here means equality up to addition of a finite rank smoothing pseudodifferential operator. \label{lem:fred}
\end{lemma}
\begin{proof}
As an elliptic linear ordinary differential operator, $D$ has a finite-dimensional kernel. Let $\tilde{\mathscr{B}}^{s+m}, \tilde{\mathscr{C}}^s$ be the spaces obtained by dropping the boundary conditions from the definitions of $\mathscr{B}^{s+m}$ and $\mathscr{C}^s$, and denote the extension of $D$ to these spaces by $\tilde D: \tilde{\mathscr{B}}^{s+m} \to \tilde{\mathscr{C}}^s$. Since we are in one dimension, $\coker \tilde D$ is not only finite-dimensional, but trivial. So, for Fredholmness we only need to prove that the boundary conditions on $\mathscr{B}^{s+m}$ do not yield an infinite-dimensional cokernel for $D$.

Now, as the boundary conditions defining our Sobolev spaces are homogeneous, we know that $\mathscr{B}^{s+m}$ is the kernel of a bounded surjective map $\tilde{\mathscr{B}}^{s+m}\stackrel{\beta}{\to} V$. Note here that $V$ is finite-dimensional, as we assumed finitely many boundary conditions. We consider the linear map $\alpha: V \to \tilde{\mathscr{C}}^{s}/\im D$ which, given some $v\in V$, is defined by first finding a $\tilde b \in \tilde{\mathscr{B}}^{s+m}$ such that $\beta \tilde b = v$ and then defining $\alpha v = [\tilde D\tilde b]$. This is well-defined, since the difference in the result if we had chosen some other $\tilde b' \in \tilde{\mathscr{B}}^{s+m}$ is $[\tilde D(\tilde b'-\tilde b)] = [D(\tilde b'-\tilde b)] = 0$. $\alpha$ is also surjective, thanks to the surjectivity of $\tilde D$. So, $\tilde{\mathscr{C}}^s/\im D$ is a finite-dimensional vector space -- indeed, its dimension is at most $\dim V$, the number of boundary conditions. We thus find that $\im D$ is closed and has finite codimension in $\tilde{\mathscr{C}}^s$, and therefore also in $\mathscr{C}^s$.

Now, we observe that the cokernel of $D$ admits smooth representatives, since the $\tilde b$ appearing in the definition of $\alpha$ may be taken to be smooth, so that $\tilde D\tilde b$ is smooth. Letting $\tilde P$ be the inverse of $\tilde D$, $\pi_L$ be a projection onto $\ker D$, and $\pi_R$ be a projection onto these smooth cokernel representatives, we can then define a parametrix $P = (1-\pi_L) \tilde P (1-\pi_R)$. We then have $DP = D\tilde P(1-\pi_R) = 1-\pi_R$ and $PD = (1-\pi_L)\tilde P(1-\pi_R)D = 1-(\pi_L + \tilde P \pi_R D - \pi_L \tilde P \pi_R D)$, and both error terms are finite rank smoothing operators.
\end{proof}

We now recall the definition of a parametrix for an elliptic complex. This is a collection of bounded pseudodifferential operators $P_i$ mapping leftwards between spaces in the complex such that
\be d_{i-1} P_{i-1} + P_i d_i \sim 1 \ , \ee
where $\sim$ is defined as in Lemma \ref{lem:fred}. As we will see shortly, an elliptic complex always has such a parametrix. In the present instance, this means that $P_0 d_0 \sim 1$, $d_1 P_1 \sim 1$, and $d_0 P_0 + P_1 d_1 \sim 1$. The first two equations reflect the fact that $d_0$ and $d_1$ are semi-Fredholm -- the former because it is overdetermined elliptic (its symbol is injective away from the zero section of the cotangent bundle) and the latter because it is underdetermined elliptic (its symbol is surjective away from the zero section). The third equation states that $P_1$ roughly inverts $d_1$, up to a gauge transformation.
\begin{proposition}
Let
\be \mathscr{C}_0 \xrightarrow{d_0} \mathscr{C}_1 \xrightarrow{d_1} \cdots \xrightarrow{d_{n-1}} \mathscr{C}_n \ee
be an elliptic complex with first order linear operators $d_i$, where $\mathscr{C}_i$ is a Sobolev Hilbert space on $I$ with Sobolev regularity $s_i = s_0 - i$ and (finitely many homogeneous) boundary conditions $\beta_i$. It is a Fredholm complex. That is, its cohomology groups are finite-dimensional and $\im d_i$ is closed for each $i$. Furthermore, representatives for these cohomology groups may always be taken to be smooth, and so the cohomology groups $\mathrm{H}^i := \ker d_i/\im d_{i-1}$ of the complex are the same for all $s_0$ which are sufficiently large so that the boundary conditions $\beta_i$ may be imposed.
\end{proposition}
\begin{proof}
Except for the last result, our proof follows \cite{atiyah:complex}. To construct $P_i$, let $R_i$ be parametrices for the Laplacians $\Delta_i = d_i^* d_i + d_{i-1} d_{i-1}^*$ (which exist, by Lemma \ref{lem:fred}). Multiplying $d_i \Delta_i = \Delta_{i+1} d_i$ on the left by $R_{i+1}$ and on the right by $R_i$ and using $R_i \Delta_i \sim 1$ and $\Delta_i R_i \sim 1$ implies that $R_{i+1} d_i \sim d_i R_i$. Define
\be P_i = R_i d_i^* \ . \ee
This is easily seen to be a parametrix for the complex:
\be d_{i-1} R_{i-1} d_{i-1}^* + R_i d_i^* d_i \sim R_i \Delta_i \sim 1 \ . \ee
Note that the image of $P_i$ satisfies the boundary conditions $\beta_i$, since this is the case for the image of $R_i$.

Let $Z_i = \ker d_i$ and $\mathscr{B}_i = \im d_{i-1}$. The defining property of the parametrix implies that $d_{i-1} P_{i-1}|_{Z_i} \sim 1$. Note that the smoothing error restricts to an operator on $Z_i$, since $\im(d_{i-1}P_{i-1}|_{Z_i})\subset Z_i$. So, $d_{i-1}\colon \mathscr{C}_{i-1}\to Z_i$ is semi-Fredholm, i.e. has a closed image of finite codimension. To see this, write $d_{i-1} P_{i-1} |_{Z_i} = 1-S_i$ and note that $\mathscr{B}_i$ contains $(1-S_i) Z_i$. Since $1-S_i$ is Fredholm, $(1-S_i)Z_i$ is closed and of finite codimension in $Z_i$. This is then also true for $\mathscr{B}_i$. Finally, since $Z_i$ is closed, $\mathscr{B}_i$ is closed in $\mathscr{C}_i$.

We now turn to the final claim of the proposition. This follows from the fact that a subspace of $\im S_i$ provides a finite dimensional space of smooth representatives of $\coker(d_{i-1}\colon \mathscr{C}_{i-1}\to Z_i)$.
\end{proof}
\begin{remark}
Note that ellipticity of the Laplacians of each deformation complex is equivalent to ellipticity of the Dirac operator $d_0^*\oplus d_1$ of the complex. The latter can be interpreted as meaning that the generalized Coulomb gauge fixing condition $d_0^* a = 0$ is (nearly, if $\mathrm{H}^0_A\not=0$) transverse to the action of $\G^{s+1}$ on $\A^s_{\mathrm{all}}$ near $A$. This also follows from Fredholmness of $\Delta_0$.
\end{remark}

This Dirac operator is also useful for computing the index of the complex. To explain this, we first prove the following:
\begin{proposition} \label{prop:selfAdj}
Let $D\colon \mathscr{B}\to \mathscr{C}$ be an order $m$ linear differential operator acting between closed subspaces of Sobolev spaces $H^{m+s'}$ and $H^{s'}$ associated to $I$, where $s'\ge 0$. Let $\mathscr{B}^0$ be the completion of $\mathscr{B}$ with respect to the $L^2$ norm, and define $\mathscr{C}^0$ analogously. Let $\mathscr{B}_0\subset \mathscr{B}$ denote the closed subspace of functions which vanish at $\partial I$ along with their first $m-1$ derivatives, and suppose $\mathscr{B}_0$ is dense in $\mathscr{B}^0$. 

Then, regarding $D$ as an unbounded operator $\mathscr{B}^0 \to \mathscr{C}^0$ with dense domain of definition $\mathscr{B}\subset \mathscr{B}^0$, the adjoint operator $D^\dagger$ has domain of definition given by those $c\in \mathscr{C}^0$ such that $D^* c\in \mathscr{B}^0$ and $(D b,c)_{\mathscr{C}^0} = (b, D^* c)_{\mathscr{B}^0}$ for all $b\in \mathscr{B}$, and on this domain it coincides with $D^*$. The orthogonal complement of $\im D\subset \mathscr{C}^0$ coincides with $\ker D^\dagger$. 

There exists a canonical map $\coker D \stackrel{F}{\to} \ker D^{\dagger}$ of vector spaces. Assume now that $\im D$ is closed in $\mathscr{C}$ so that $\mathscr{C}/\im D$ inherits a Hilbert space structure; then $F$ is bounded.\footnote{Note that $\ker D^{\dagger} \subset \mathscr{C}^0$ is a closed subspace by virtue of being the orthogonal complement of $\im D$.} Assume furthermore that $\ker D^{\dagger} \subset \mathscr{C}$ and that $\mathscr{C} \cap \overline{\im D}_{\mathscr{C}^0} = \im D$, where $\overline{\im D}_{\mathscr{C}^0}$ denotes the closure of $\im D$ in $\mathscr{C}^0$. Then $F$ is an isomorphism of Hilbert spaces. 

%Then $F$ is surjective if $\ker D^{\dagger} \subset C$. Finally, if either (i) $D$ is overdetermined elliptic, $B^m \stackrel{D}{\to} C^0$ has closed image, where $B^m$ now denotes the closure of $B$ in $H^m$, and $B = B^m \cap H^{m+s'}$, or if (ii) $\im D \subset C$ has a complement contained within $\ker D^{\dagger} \cap C$, then $F$ is injective.  

%If $\im(D:B^m\to C^0)$ is closed \MZ{this needs to be assumed even for overdet ell bc $B$ could be tiny, right? eg don't know cokernel is finite-dimensional}, where $B^m$ is the completion of $B$ in the $H^m$ norm, and in addition either $\ker D^\dagger \subset C$ or $D$ is overdetermined elliptic then $F$ is injective.
\end{proposition}
\begin{remark}Note that if $\im D$ is closed in $\mathscr{C}$ and $\mathscr{C} = \mathscr{C}^0$ -- i.e., $s' = 0$ -- then $\im D = (\ker D^{\dagger})^{\perp}$. Note further that this proposition implies that $\Delta_0$, with domain of definition $\mf{g}^2\subset L^2(I,\mf{su}(2))$, is self-adjoint, since if $b\in \mf{g}^2$ then the boundary terms produced by integration by parts in relating $(\Delta_0 b,c)_{L^2}$ to $(b,\Delta_0 c)_{L^2}$ vanish if and only if $c\in \mf{g}^2$.
\end{remark}
\begin{proof}
We recall the definition of $D^\dagger$. One declares that the domain of $D^\dagger$ consists of those $c\in \mathscr{C}^0$ such that $b\mapsto (D b, c)_{\mathscr{C}^0}$ is a continuous functional on $\mathscr{B}\subset \mathscr{B}^0$. Equivalently (by the Hahn-Banach theorem), the domain of $D^\dagger$ consists of those $c\in \mathscr{C}^0$ for which there exists a (unique) $b'\in \mathscr{B}^0$ such that $(D b, c)_{\mathscr{C}^0} = (b, b')_{\mathscr{B}^0}$ for all $b\in \mathscr{B}$. One then of course defines $D^\dagger c = b'$. This definition makes manifest the fact that the orthogonal complement of the image of $D$ inside of $\mathscr{C}^0$ coincides with the kernel (and, in particular, is in the domain) of $D^\dagger$.

Next, we note that for $b\in \mathscr{B}_0$ we have $(D b, c)_{\mathscr{C}^0} = (b, D^* c)_{\mathscr{B}^0}$ for all $c\in \mathscr{C}^0$. Here, $D^* c$ is a priori regarded as a distribution in $(\mathscr{B}_0)'$, but if $c$ is in the domain of $D^\dagger$ then $D^* c$ extends to a (unique, by density of $\mathscr{B}_0$ in $\mathscr{B}^0$) continuous functional $D^* c\in \mathscr{B}^0$ on all of $\mathscr{B}^0$. Requiring that $(Db,c)_{\mathscr{C}^0} = (b, D^* c)_{\mathscr{B}^0}$ for all $b\in \mathscr{B}$, and not merely for $b\in \mathscr{B}_0$, then determines the boundary conditions on $D^\dagger$.

To construct the map $F$, we consider the continuous composition $\mathscr{C} \hookrightarrow \mathscr{C}^0 \twoheadrightarrow (\text{im }D)^{\perp} \simeq \ker D^{\dagger}$. The kernel of the projection, and hence the kernel of the composite map, contains $\text{im }D$. Assuming now that $\im D$ is closed, we hence obtain a continuous map $\coker D \stackrel{F}{\to} \ker D^{\dagger}$. It is immediate that the surjectivity of $F$ follows from $\ker D^{\dagger} \subset \mathscr{C}$. For injectivity, an element of $\mathscr{C}$ in the kernel of $F$ must lie in the kernel of $\mathscr{C}^0 \twoheadrightarrow (\im D)^{\perp}$, i.e., in $\overline{\im D}_{\mathscr{C}^0}$, and so the given hypothesis is exactly that necessary to guarantee injectivity. % \MZ{this argument seems overly complicated. isn't this just obvious since $F$ is obviously injective on $\ker D^\dagger$, is obviously vanishing on $\im D$, and $C$ decomposes as a sum of those two?} given $V \subset \ker D^{\dagger} \cap \mathscr{C}$ closed such that $\mathscr{C} \simeq V \oplus \im D$, we need to check $V \cap \overline{\im D} = \{0\}$. But $\ker D^{\dagger} \cap (\ker D^{\dagger})^{\perp} = \{0\}$, so this is automatic. 

\end{proof}

\begin{rmk}This proposition also gives another proof of the claim of Lemma \ref{lem:fred} that all classes in the cokernel admit smooth representatives. For, if $\ker D^\dagger = (\im D)^\perp$ denotes the $L^2$-orthogonal complement of $\im D$, then this space is finite dimensional and consists of smooth functions, since $D^\dagger$ on its domain coincides with the elliptic operator $D^*$. Therefore, the intersection of this space with $\mathscr{C}^s$ gives a canonical space of smooth representatives for $\coker D$.\end{rmk}

\begin{proposition} \label{prop:realInd}
The index of the special deformation complex coincides with that of
\be d_0^*\oplus d_1\colon \A_0^s \to H^{s-1}(I,\mf{su}(2)) \oplus H^{s-1}(I,\mf{su}(2))^{\oplus 3} \ . \ee
Indeed, 
\be \mathrm{H}^{1,{\rm sp}}_A \simeq \ker d_0^* \oplus d_1 = \ker(\Delta_1\colon \tilde\A_{\xi} \to C^\infty(I,\mf{su}(2))^{\oplus 4}) \ , \ee 
where
\be \tilde\A_{\xi} = \{ a\in \A_0 \bigm| d_0^* a|_{\partial I} \in \mf{t} , d_1 a|_{\partial I} \in \mf{t}^{\oplus 3} \} \ ,\footnote{Note that this space depends only on $\xi$, but not on $A$.} \ee 
and $\coker d_0^*\oplus d_1 \simeq \mathrm{H}^0_A\oplus \mathrm{H}^{2,{\rm sp}}_A$. Furthermore,
\be \mathrm{H}^0_A = \ker(\Delta_0\colon \mf{g}\to C^\infty(I,\mf{su}(2))) = \ker(d_0\colon \mf{g}\to \A_0) \ee
and
\be \mathrm{H}^{2,{\rm sp}}_A\simeq \ker(\Delta_2\colon \tilde\F\to \F) = \ker(d_1^*\colon \tilde\F\to \A_0) \ , \ee
where $$\tilde\F = \{f\in C^\infty(I,\mf{su}(2))^{\oplus 3} \bigm| f|_{\partial I} \in \mf{t}^{\oplus 3}, \partial f|_{\partial I} \in (\mf{t}^{\perp})^{\oplus 3} \} \,.$$
\end{proposition}
\begin{proof}

We first show that $\ker d_0^*\oplus d_1 \simeq \mathrm{H}^{1,{\rm sp}}_A$. This follows from the previous proposition with $s'=s$, $m=1$, $D=d_0$, $\mathscr{B} = \mf{g}^{s+1}$, and $\mathscr{C} = \ker d_1|_{\A^s_0}$. We observe that if $a\in L^2$, $d_1 a = 0$, and $d_0^* a \in L^2$, then $a\in H^1$ by elliptic regularity for $d_0^*\oplus d_1$. It follows that $D^\dagger = d_0^*$ with domain $\ker(d_1\colon\A_0^1 \to \F^0)$. So, the orthogonal complement to the image of $D$ inside of the $L^2$-completion of $\ker d_1$ is given by $\ker d_0^* \oplus d_1$ acting on $\A_0^1$; elliptic regularity allows us to instead take this to act on $\A_0$. So, $\ker D^\dagger \subset \mathscr{C}$. This subset is clearly closed, since it is finite-dimensional. Finally, the overdetermined elliptic operator $d_0\colon \{H^1(I,\mf{su}(2))\bigm| h|_{\partial I}\in \mf{t}\} \to L^2$ is semi-Fredholm, and so its image is closed and therefore contains the $L^2$-closure $\overline{\im D}$ of $\im D$; overdetermined elliptic regularity then implies that $\overline{\im D}\cap \mathscr{C} = \im D$. (In particular, if $h\in \{H^{s+1}(I,\mf{su}(2))\, \bigm|\, h|_{\partial I}\in \mf{t}\}$ satisfies $d_0 h\in \A_0^s$, then $\partial h|_{\partial I}\in \mf{t}^\perp$, i.e., $h\in \mf{g}^{s+1}$.) Proposition \ref{prop:selfAdj} then gives us a bijection $F\colon \mathrm{H}^{1,{\rm sp}}_A \to \ker d_0^* \oplus d_1$. To further identify this with $\ker(\Delta_1\colon \tilde\A_{\xi} \to C^\infty(I,\mf{su}(2))^{\oplus 4})$, we observe that the boundary conditions on $a\in \tilde\A_\xi$ are precisely those such that $d_1 a$ is in the domain of $d_1^\dagger$ and $d_0^* a$ is in the domain of $(d_0^*)^\dagger$. From this, one immediately concludes that $\ker \Delta_1 \subset \ker d_0^*\oplus d_1$; the reverse inclusion follows from observing that $\ker d_0^*\oplus d_1 \subset \tilde\A_\xi$.

Proposition \ref{prop:selfAdj} similarly allows us to identify $\coker d_0^*\oplus d_1$ with $\ker(d_0+d_1^*)$, where we impose boundary conditions requiring all elements of the domain to be diagonal, but we do not impose boundary conditions on derivatives. However, a glance at the structure of $d_0+d_1^*$ makes it clear that if $a = d_0 h + d_1^* f$ then the diagonal part of $a^0$ coincides with that of $\partial h$ and the diagonal part of $a^i$ coincides with that of $-\partial f^i$, so these zeroth order boundary conditions plus being in the kernel of $d_0+d_1^*$ together imply in addition that $h\in \mf{g}$ and $f\in \tilde\F$. (Note $h$ and $f$ are smooth by elliptic regularity.) It then follows that $\ker(d_0+d_1^*)=\ker d_0\oplus \ker d_1^*$, since with these boundary conditions we have $(d_0 h, d_1^* f) = (d_1 d_0 h, f) = 0$. The former summand is, by definition, $\mathrm{H}^0_A$. Next, Proposition \ref{prop:selfAdj} with $s'=0$ allows us to identify $\mathrm{H}^{2,{\rm sp}}_A = \coker d_1$ with $\ker d_1^\dagger$, since $d_1\colon \A_0^1\to \F^0$ is underdetermined elliptic and $d_1^*$ is overdetermined elliptic; this identification then holds for all $s'\ge 0$, since neither $\mathrm{H}^{2,{\rm sp}}_A$ nor $\ker d_1^*$ depend on the degree of Sobolev regularity. We thus have $\mathrm{H}^{2,{\rm sp}}_A \simeq \ker d_1^*$ with zeroth order boundary conditions. The same argument as above shows that this kernel satisfies the boundary conditions on $\tilde\F$. On $\tilde\F$, it is trivial to see that $\ker \Delta_2 = \ker d_1^*$, as $(\Delta_2 f, f)_{L^2} = (d_1^* f, d_1^* f)_{L^2}$.
\end{proof}
\begin{remark}
We note that an identical version of this proposition holds for the general deformation complex. The only change worth commenting on is that the boundary conditions relevant for ${\rm H}^{2,{\rm gen}}_A$ change to $f|_{\partial I} = 0$, $\partial f|_{\partial I} \in (\mf{t}^\perp)^{\oplus 3}$. However, this characterization immediately leads to the conclusion that ${\rm H}^{2,{\rm gen}}_A = 0$. For, uniqueness of solutions of overdetermined elliptic ordinary differential equations implies that solutions of $d_1^* f = 0$ are determined by their values at 0. So, the only solution with $f|_0 = 0$ is $f=0$ identically.
\end{remark}

The Hodge theory above motivates the following definition:

\begin{defn} Given the deformation complexes as above, we define the following harmonic spaces:
\begin{align} \label{eq:realHarmonic}
\Hh^0_A &:= \mathrm{H}^{0,{\rm sp}}_A = \mathrm{H}^{0,{\rm gen}}_A = \ker(\Delta_0\colon \mf{g}\to C^\infty(I,\mf{su}(2))) = \ker(d_0\colon \mf{g}\to \A_0) \ , \nonumber \\
\Hh^{1,{\rm sp}}_A &:= \ker(\Delta_1\colon \tilde\A_{\xi} \to C^\infty(I,\mf{su}(2))^{\oplus 4}) \nonumber \\
&= \ker(d_0^* \oplus d_1\colon \A_0 \to C^\infty(I,\mf{su}(2))\oplus C^{\infty}(I, \mf{su}(2))^{\oplus 3}) \ , \nonumber \\
\Hh^{1,{\rm gen}}_A &:= \ker(\Delta_1\colon \{a\in \A_{\rm all} \bigm| d_0^*a|_{\partial I} \in \mf{t},\ d_1 a|_{\partial I} = 0\} \to C^\infty(I,\mf{su}(2))^{\oplus 4}) \nonumber \\
&= \ker(d_0^* \oplus d_1\colon \A_{\rm all} \to C^\infty(I,\mf{su}(2))\oplus C^{\infty}(I, \mf{su}(2))^{\oplus 3}) \ , \nonumber \\
\Hh^{2,{\rm sp}}_A &:= \ker(\Delta_2\colon \tilde\F\to \F) = \ker(d_1^*\colon\tilde\F\to \A_0) \ , \nonumber \\
\Hh^{2,{\rm gen}}_A &:= 0 \ .
\end{align} \end{defn}

These cohomology groups are related in a number of ways:
\begin{proposition} \label{prop:hodge}
$\Hh^{2,{\rm sp}}_A \simeq \Hh^0_A\otimes \RR^3.$
\end{proposition}
\begin{proof}
We make use of the characterizations of $\Hh^0_A$ and $\Hh^{2,{\rm sp}}_A$ as kernels of Laplacians with diagonal zeroth-order boundary conditions and off-diagonal first-order ones. We regard $h\in \Hh^0_A$ ($f \in \Hh^{2,{\rm sp}}_A$, resp.) as an adjoint-valued function (self-dual 2-form $\sum_i f^i \omega_i$) on $[0,L]\times (S^1)^3$ which is invariant under translations of the last three coordinates. Similarly, $A$ defines a connection on a trivial bundle over this space. With these identifications, $\Delta_0$ and $\Delta_2$ become the covariant Laplacians acting on 0-forms and 2-forms, respectively. The desired isomorphism is then given either by wedging (to go from 0-forms to self-dual 2-forms) or contracting (to go backwards) with the three K\"ahler forms, as follows from the K\"ahler identities. For example, if $\Lambda_i$ denotes the operation of contracting with the K\"ahler form $\omega_i$ and $\bar\partial_i$ denotes the covariant Dolbeault operator in complex structure $i$ then the identities
\be [\Lambda_i, \partial_i] = i \bar\partial^*_i \ , \quad [\Lambda_i, \bar\partial_i]  = -i \partial^*_i \ , \quad [\Lambda_i, \partial_i^*] = [\Lambda_i, \bar\partial_i^*] = 0 \ee
imply that
\be [\Lambda_i, \Delta_{\partial_i}] = - [\Lambda_i, \Delta_{\bar\partial_i}] = i(\partial^*_i \bar\partial^*_i + \bar\partial^*_i \partial^*_i) \ , \ee
where $\Delta_{\partial_i} = \partial_i \partial_i^* + \partial_i^* \partial_i$ and $\Delta_{\bar\partial_i} = \bar\partial_i \bar\partial_i^* + \bar\partial_i^* \bar\partial_i$, and so $[\Lambda_i, \Delta] = 0$, where $\Delta = d_A d_A^* + d_A^* d_A = \Delta_{\partial_i} + \Delta_{\bar\partial_i}$. Similarly, the operations $L_i$ of wedging with $\omega_i$ also commute with $\Delta$.
\end{proof}

Now, it follows from a diagram chase that in full generality, we have the exact sequence \beq\label{eq:spgnes} 0 \to \mathrm{H}^{1,\mathrm{sp}}_A \to \mathrm{H}^{1,\mathrm{gen}}_A \to \mf{z}^3_{\partial I} \to \mathrm{H}^{2,\mathrm{sp}}_A \to \mathrm{H}^{2,\mathrm{gen}}_A \to 0\,,\eeq
which simplifies in the present case because ${\rm H}^{2,{\rm gen}}_A = 0$; in particular, we immediately obtain that the index of the general deformation complex is that of the special complex minus $\dim \mf{z}^3_{\partial I} = 6$. We now give an analytic derivation of this result which makes manifest its relationship with our harmonic representatives:

\begin{proposition}
For all Nahm data $A\in \A_{\rm all}$, we have the exact sequence $0 \to \Hh^{1,\mathrm{sp}}_A \to \mathcal{H}^{1,\mathrm{gen}}_A \to \mf{z}^3_{\partial I} \to \mathcal{H}^{2,\mathrm{sp}}_A \to 0$.
\end{proposition}
\begin{proof}
The difference between $\Hh^{1,{\rm sp}}_A$ and $\Hh^{1,{\rm gen}}_A$ owes to elements of $\A_{\rm all}$ in the kernel of $d_0^*\oplus d_1$. The difference of two such elements with boundary conditions $a^i|_{\partial I} \in i c_{\partial I}^i \sigma_z + \mf{t}^\perp$ is in $\Hh^{1,{\rm sp}}_A$, so the space of $c_{\partial I}^i$ for which a solution exists parametrizes the vector space $\Hh^{1,{\rm gen}}_A / \Hh^{1,{\rm sp}}_A$. To determine when such solutions exist, we fix some $b\in \A_{\rm all}$ with these boundary conditions and seek an $a\in \A_0$ such that $(d_0^*\oplus d_1) a = - (d_0^* \oplus d_1) b$. A solution will exist if and only if the right side is $L^2$-orthogonal to the $L^2$-orthogonal complement of $\im(d_0^*\oplus d_1|_{\A_0})$, i.e. if the right side is orthogonal to $\ker(d_0 + d_1^*)$ with the boundary conditions $h|_{\partial I}, f^i|_{\partial I} \in \mf{t}$. For such $h,f$ we compute $(d_0^* b, h)_{L^2} + (d_1 b, f)_{L^2} = -\sum_i (i c^i_L \Tr(\sigma_z f^i|_L) - i c^i_0 \Tr(\sigma_z f^i|_0))$. Recalling that $\ker(d_0 + d_1^*)$ with these boundary conditions splits as the direct sum $\ker d_0 \oplus \ker d_1^*$, we see that the obstruction to the existence of a solution with a given choice of $c_{\partial I}^i$ can be traced to an element of $\Hh^{2,{\rm spec}}_A$. Since $\mathrm{H}^{2,{\rm gen}}_A=0$, if the boundary values in $\mf{t}^3\oplus \mf{t}^3$ of $f_1,f_2\in \Hh^{2,{\rm sp}}_A$ satisfy a linear relation then so do $f_1$ and $f_2$ themselves. The proposition then follows immediately.
\end{proof}
%\MZ{\begin{remark}
%It follows from this proposition that the index of the general deformation complex is that of the special complex minus $\dim \mf{z}_{\partial I}^3 = 6$. Note also that it is clear, both from the proof of the proposition as well as from a diagram chase, that if we did not have ${\rm H}^{2,{\rm gen}}_A = 0$ then we would instead have the exact sequence $0 \to \mathrm{H}^{1,\mathrm{sp}}_A \to \mathrm{H}^{1,\mathrm{gen}}_A \to \mf{z}^3_{\partial I} \to \mathrm{H}^{2,\mathrm{sp}}_A \to \mathrm{H}^{2,\mathrm{gen}}_A \to 0$.
%\end{remark}}

We now reap a significant benefit from having related the indices of the deformation complexes to those of their Dirac operators: the complexes only exist when $A$ satisfies Nahm's equations, but the Dirac operator is elliptic for all $A\in \A_{\mathrm{all}}$. So, we can deform $A$ to 0 without changing the index. This yields the following corollary of Proposition \ref{prop:realInd}:
\begin{corollary}
For any Nahm data $A\in \A_{\mathrm{all}}$, the index of the special deformation complex is $-4$ and that of the general deformation complex is $-10$.
\end{corollary}
\begin{proof}
As we have just explained, we may set $A=0$. Then, we have $\dim \mathrm{H}^{1,{\rm sp}}_A=8$, $\dim {\rm H}^{1,{\rm gen}}_A = 11$, $\dim \mathrm{H}^0_A=1$, and $\dim \mathrm{H}^{2,{\rm sp}}_A = 3$, since $\Hh^{1,{\rm sp}}_A \simeq (\mf{t}^{\perp})^{\oplus 4}$, $\Hh^{1,{\rm gen}}_A \simeq \mf{t}^\perp \oplus \mf{su}(2)^{\oplus 3}$, $\Hh^0_A \simeq \mf{t}$, and $\Hh^{2,{\rm sp}}_A \simeq \mf{t}^{\oplus 3}$.
\end{proof}

Before putting a smooth structure on the irreducible loci of $\M^s_{\xi},\M^s_{\mathrm{all}}$, it is useful to first do so for $\A^s_\xi/\G^{s+1},\mc{A}^s_{\mathrm{all}}/\mc{G}^{s+1}$. We therefore begin with that. Note that if $g\in \G^{s+1}$ stabilizes $A^0\in \A'$ then $g\in \G$, thanks to elliptic regularity and $\partial_{A^0} g = 0$. We first strengthen Proposition \ref{prop:gaugeConv} by showing that it is not necessary to pass to subsequences, as long as we consider convergence in a quotient group: 

\begin{lemma} \label{lem:homeoOrb}
Consider some $B\in \A'\times \A''^s_{\mathrm{all}}$ with stabilizer $\hat\G\subset \G$. If $B_n\to B$ and $g_n(B_n)\to A$ in $\A_{\mathrm{all}}^s$ for some sequence $g_n\in \G^{s+1}$, then there exists $g\in \G^{s+1}$ such that $g(B)=A$ and $g_n \to g$ in $\hat\G \backslash \G^{s+1}$. 

%The bijection $\G^{s+1}/\hat\G \to \G^{s+1} \cdot A$ is a homeomorphism, where $\G^{s+1}/\hat\G$ is topologized as a quotient space of $\G^{s+1}$ and $\G^{s+1}\cdot A$ is topologized as a subspace of $\A^s_{\mathrm{all}}$.
\end{lemma}
\begin{proof}
Proposition \ref{prop:gaugeConv} implies that for any subsequence of $g_n$, a further subsequence converges in $\hat\G\backslash \G^{s+1}$ to the unique point $g$ satisfying $g(B)=A$. This implies that the sequence $g_n$ itself converges to $g$ in $\hat\G\backslash \G^{s+1}$. For, if otherwise then there would exist a neighborhood $U$ of $g$ in $\hat\G\backslash \G^{s+1}$ and a subsequence $g'_n$ of $g_n$ such that $g'_n \not\in U$ for all $n$.
\end{proof}

\begin{remark} \label{rmk:slightWeak}
If we assume from the outset that $A=g(B)$ for some $g\in\G^{s+1}$, then we may reformulate this lemma as follows: for any neighborhood $U$ of $g$ in $\hat\G\backslash \G^{s+1}$ there exist neighborhoods $V$ of $B$ and $V'$ of $g(B)$ in $\A^s_{\rm all}$ such that if $\tilde g\in \G^{s+1}$, $B'\in V$, and $\tilde g(B') \in V'$, then $\tilde g\in U$. This reformulation will make the upcoming application of this lemma more immediate, and it also facilitates comparison with Proposition \ref{prop:quanthomeo}.
\end{remark}

We are now ready to start producing manifold structures on our spaces of gauge orbits; we begin with $\mc{A}^s/\mc{G}^{s+1}$ before doing the same for $\mc{M}^s$. In both cases, we first give descriptions of local neighborhoods of (pre-)Nahm data in general before then specializing to the case of irreducible (pre-)Nahm data -- i.e., with trivial stabilizer -- to produce manifold structures. 

\begin{proposition}[Real Coulomb gauge]\label{prop:realcoul}
For $1 \le s < \infty$ and for $A\in \A' \times \A''^s_{\mathrm{all}}$ with stabilizer $\hat\G\subset \G$, the $\G^{s+1}$-equivariant map $F\colon (\G^{s+1} \times \ker d_0^*)/\hat{\mc{G}} \to \A^s_{\mathrm{all}}$ defined by $(g, a) \mapsto g(A+a)$, and with $\hat{g} \in \hat{\mc{G}}$ acting on $(g, a)$ via $(g, a) \mapsto (\hat{g}^{-1}g, \hat{g}^{-1}a\hat{g})$, yields a local diffeomorphism identifying a neighborhood of $(1,0)$ with a neighborhood of $A\in \A^s_{\rm all}$. It follows that a neighborhood of $[A] \in \mc{A}^{s}_{\mathrm{all}}/\mc{G}^{s+1}$ is homeomorphic to $T^s_{A,\epsilon,\mathrm{all}}/\hat{\mc{G}}$ for some $\epsilon>0$, where $T^s_{A,\epsilon,\mathrm{all}} := \{a\in \A^s_{\mathrm{all}} \bigm| d_0^* a = 0,\, \|a\|_{\A^s_{\mathrm{all}}}<\epsilon\}$ and $\hat\G$ acts via conjugation. Here, we endow the Sobolev space with boundary conditions $\A_{\mathrm{all}}^s$ with a norm that utilizes covariant derivatives with respect to $A^0$. 

This structure is compatible with the projection $p$; a neighborhood of $[A] \in \mc{A}^s_{\xi}/\mc{G}^{s+1}$ is then identified with $T^s_{A,\epsilon}/\hat{\mc{G}}$, where $T^s_{A,\epsilon} := \{a\in \A^s_{0} \bigm| d_0^* a = 0,\, \|a\|_{\A^s_0}<\epsilon\}$.
\end{proposition}

\bp 
We first explain how $\hat\G$ acts naturally on $T^s_{A,\epsilon}$ via conjugation; the argument for $T^s_{A,\epsilon,{\rm all}}$ is identical. Thinking of $a$ as a deviation from $A$, the appropriate action is $a\mapsto g(A+a)-A = g(A) + g^{-1} a g - A = g^{-1} a g$. Conjugation preserves both $d_0^* a = 0$ and $\|a\|_{\A_0^s}$: it preserves $\|a\|_{\A_0^s}$ because $\partial_{A^0} g = 0$, and we similarly have $d_0^*(g^{-1} a g)=0$ if $d_0^* a = 0$.

We now turn to the $\G^{s+1}$-equivariant map $F\colon (\G^{s+1} \times \ker d_0^*)/\hat{\mc{G}} \to \A^s_{\mathrm{all}}$. %, which maps a representative $(g,a)\in \G^{s+1}\times \ker d_0^*$ to $g(A+a)$. The action of $\hat g\in \hat\G$ on $\G^{s+1}\times \ker d_0^*$ is given by $(g,a)\mapsto (\hat g g, \hat g a \hat g^{-1})$, and with this action $F$ is well-defined.
The derivative $DF_{(1,0)}\colon \mf{g}^{s+1}/\hat{\mf{g}}\oplus \ker d_0^* \to \A_{\mathrm{all}}^s$ at $(1,0)$ is given by $DF_{(1,0)}(h,a) = d_0 h + a$. We wish to prove that $DF_{(1,0)}$ is an isomorphism. This is equivalent to asking if $\coker d_0 \simeq \ker d_0^*$. We first note that $(d_0 h, a)_{L^2} = (h, d_0^* a)_{L^2} = 0$ if $h\in \mf{g}^{s+1}$ and $a\in \ker d_0^*$. (This holds even for $a\in \A^s_{\rm all}$, as opposed to $\A^s_0$, as we still have $a^0|_{\partial I} \in \mf{t}^\perp$.) So, $DF_{(1,0)}(h,a) = 0$ requires $d_0 h=a=0$. But, $d_0 h=0$ is equivalent to $h\in \hat{\mf g}$, and so $DF_{(1,0)}$ is injective. For surjectivity, given $a'\in \A^s_{\mathrm{all}}$ we want to find an $h\in \mf{g}^{s+1}$ that solves $d_0^*(a'-d_0 h) = 0$, i.e. $\Delta_0 h = d_0^* a'$. By embedding $a'$ into $\A^1_{\mathrm{all}}$, we are able to apply Proposition \ref{prop:selfAdj}, and we can then solve for an $h\in \mf{g}^2$ because if $h'\in \mf{g}^2$ satisfies $\Delta_0 h' = 0$, and therefore $d_0 h'=0$, then $(d_0^* a', h')_{L^2} = (a', d_0 h')_{L^2} = 0$. Elliptic regularity then gives $h\in \mf{g}^{s+1}$. The inverse function theorem now implies that $F$ defines a diffeomorphism between a neighborhood $V$ of $A\in \A^s_{\mathrm{all}}$ and a neighborhood $U$ of $(1,0)$ in $(\G^{s+1}\times \ker d_0^*)/\hat\G$.

Finally, we consider the subset $p_1^{-1}(1)$ of $U$, where $p_1\colon (\G^{s+1}\times \ker d_0^*)/\hat{\mc{G}} \to \G^{s+1}/\hat{\mc{G}}$ is the projection onto the first factor and $p_2\colon (\G^{s+1}\times \ker d_0^*)/\hat{\G} \to \ker d_0^*/\hat{\G}$ is defined analogously. It follows from Lemma \ref{lem:homeoOrb} that by shrinking $p_2(p_1^{-1}(1))$ to a set $T^s_{A,\epsilon,{\rm all}}/\hat\G$ we may assume that the $\hat\G$-cosets of any elements of $\G^{s+1}$ which identify points in $F((\{1\}\times T^s_{A,\epsilon,\mathrm{all}})/\hat{\G})$ are contained in $p_1(U)$. But, the only element in $p_1(U)$ that maps $1\in \G^{s+1}/\hat\G$ to itself is the identity. So, a neighborhood of $[A]\in \A^s_{\mathrm{all}}/\G^{s+1}$ indeed is homeomorphic to $T^s_{A,\epsilon,\mathrm{all}}/\hat\G$.

Finally, the compatibility of all the above structure with the projection $p$ down to $\mf{z}^3_{\partial I}$ is clear. \ep

\begin{prop}\label{prop:preNahmmfld} For $1 \le s < \infty$, $\A^{s,\mathrm{irr}}_{\mathrm{all}}/\G^{s+1}$ is canonically a smooth Hilbert manifold. The map $\A^{s,\mathrm{irr}}_{\mathrm{all}}/\G^{s+1} \stackrel{p}{\to} \mf{z}^3_{\partial I}$ is a smooth submersion when the source is equipped with this manifold structure. 

Explicitly, charts for this manifold structure are given by the neighborhoods $T^s_{A,\epsilon,\mathrm{all}}$ of Proposition~\ref{prop:realcoul}, or by $T^s_{A,\epsilon}$ for the fibers $\A^{s,\mathrm{irr}}_{\xi}/\G^{s+1}$ of $p$.
\end{prop}

\begin{proof} We have already observed that $\mc{A}^s_{\mathrm{all}}/\mc{G}^{s+1}$ and $\mc{A}^s_{\xi}/\mc{G}^{s+1}$ are Hausdorff and second-countable. A manifold structure on the stabilizer-free locus then follows from the diffeomorphism (to a neighborhood of $A\in \A^s_{\rm all}$) and homeomorphism (to a neighborhood of $[A]\in \A^s_{\rm all}/\G^{s+1}$) produced in Proposition \ref{prop:realcoul} (both in the special case of trivial stabilizer).
\end{proof}

We now turn to producing the desired manifold structure on $\M^{s,\mathrm{irr}}_{\mathrm{all}} \to \mf{z}^3_{\partial I}$ and its fibers $\M^{s,\mathrm{irr}}_{\xi}$ as submanifolds of $\A^{s,\mathrm{irr}}_{\mathrm{all}}/\mc{G}^{s+1}$ cut out by the moment map equations. Once again, we first give a convenient characterization of neighborhoods of Nahm data in general before then specializing to the irreducible locus: 

\begin{proposition} \label{prop:realRG}
For any Nahm data $A\in \A_{\rm all}$ which is stabilized by $\hat\G\subset \G$ with Lie algebra $\hat{\mf{g}} \simeq \mathrm{H}^0_A$, a neighborhood of $[A]$ in $\M^s_{\rm all}$ is homeomorphic to a neighborhood of 0 in $\Hh^{1,{\rm gen}}_A/\hat\G$. Furthermore, if $A\in \A_\xi$ then a neighborhood of $[A]$ in $\M^s_\xi$ is homeomorphic to a neighborhood of 0 in $\hat\mu^{-1}(0)/\hat\G$, where $\hat\mu\colon \Hh^{1,{\rm sp}}_A\to (\hat{\mf{g}}^*)^{\oplus 3} \simeq \mathrm{H}^{2,{\rm sp}}_A$ is  a $\hat\G$-equivariant smooth map between finite-dimensional vector spaces.
\end{proposition}

\begin{proof}
We provide the proof of the claim concerning $\M_\xi^s$; the proof for $\M_{\rm all}^s$ is similar, but even easier, owing to the fact that ${\rm H}^{2,{\rm gen}}_A=0$.

We observe that the moment maps define a smooth $\hat\G$-equivariant non-linear map $\tilde\mu\colon T^s_{A,\epsilon} \to \F^{s-1}$, namely $\tilde\mu(a) := \mu(a+A)$. The derivative of this map is $d_1\colon \ker d_0^* \to \F^{s-1}$, which is Fredholm. Indeed, it is surjective onto the intersection, $V$, of the $L^2$-orthogonal complement of $\ker d_1^*: \tilde\F\to \A_0$ with $\F^{s-1}$, so the implicit function theorem implies that for any projection $\pi\colon \F^{s-1}\to V$ there is a diffeomorphism $\phi$ from a neighborhood $U$ of $0\in T^s_{A,\epsilon}$ to another such neighborhood such that $\pi\circ \tilde\mu \circ \phi = d_1$ on $U$. Note furthermore that if $\pi$ is $\hat{\mc{G}}$-equivariant, then $\phi$ may be chosen to be $\hat{\mc{G}}$-equivariant. Such a choice of $\pi$ is provided by the $L^2$-orthogonal projection onto $V=\im (d_1: \ker d_0^*\to \F^{s-1})$.

A neighborhood of 0 in $\tilde\mu^{-1}(0)$ is diffeomorphic to a neighborhood of 0 in the vanishing locus of $(1-\pi)\circ \tilde\mu \circ \phi$ acting on the finite-dimensional space $\ker d_0^*\oplus d_1$. Assuming now that $\pi$ and $\phi$ are $\hat\G$-equivariant, we have that $\hat\mu = (1-\pi)\circ\tilde\mu \circ \phi = \tilde\mu \circ \phi\colon \ker d_0^*\oplus d_1\to \ker d_1^*$ is a smooth $\hat\G$-equivariant map. Using Proposition \ref{prop:realcoul} to characterize a neighborhood of $[A]\in \A^s_\xi/\G^{s+1}$, it now follows that a neighborhood of $[A]\in \M_\xi^s$ is homeomorphic to a neighborhood of 0 in $\hat\mu^{-1}(0)/\hat\G$.
\end{proof}

\begin{rmk}\label{rmk:sadRG} 
When $A$ is reducible, the map $\hat\mu$ closely resembles the moment map that appeared in Kronheimer's construction of the $A_1$ (or $Z_2$) orbifold singularity; indeed, they coincide at the level of the quadratic (Hessian) part of $\hat\mu$, but the presence of the diffeomorphism $\phi$ leaves open the structure of higher-order corrections.\footnote{For physically-inclined readers, Proposition~\ref{prop:realRG} and the present remark pertain to the relationship between the 4d $\N=4$ gauge theory on $S^1/Z_2$ whose Higgs branch is $\M_\xi$ and the low-energy 3d $\N=4$ effective field theory whose Higgs branch approximates the geometry near a singular point of $\M_\xi$.} Given that we will later show that a neighborhood of $[A]\in \M^s_\xi$ is homeomorphic to this $A_1$ singularity one might expect that $\hat\mu$ may be improved to exactly agree with Kronheimer's moment maps by the use of further local diffeomorphisms. Similar considerations are exactly the concern of the recent work of~\cite{Fan,BuchdahlSchumacher} in an analogous setting. We will return to this line of thought in \S\S\ref{sec:complex} and \ref{sec:duy}.
\end{rmk}

The proof of the following proposition is then immediate:

\begin{prop}\label{prop:Nahmmfld}
$\M^{s,\mathrm{irr}}_{\mathrm{all}}$ is a smooth submanifold of $\mc{A}^{s,\mathrm{irr}}_{\mathrm{all}}/\mc{G}^{s+1}$ with a smooth submersion to $\mf{z}^3_{\partial I}$, the fibers $\M^{s,\mathrm{irr}}_{\xi}$ of which are smooth manifolds with dimension given by $$\dim \mathrm{H}^1_A - \dim \mathrm{H}^2_A - \dim \mathrm{H}^0_A = 4\,,$$ i.e., the negation of the index of the special deformation complex.

Explicitly, charts for this manifold structure on $\M^{s,\mathrm{irr}}_{\mathrm{all}}$ and $\M^{s,\mathrm{irr}}_{\xi}$ are given by $\epsilon$-balls in $\mc{H}^{1,{\rm gen}}_A$ and $\Hh^{1,{\rm sp}}_A$, respectively.
\end{prop}

We now note an interesting counterpart of Proposition~\ref{prop:bij}, which gives an explicit gauge choice in which Nahm data is smooth. (Of course, the gauge choice in Proposition~\ref{prop:realAx} gives another such choice, thanks to Proposition~\ref{prop:bij}.)
\begin{proposition} \label{prop:coulombSmooth}
If $A\in \A_\xi$, $a\in \A^s_0$, and $d^*_0 a = \mu(A+a) = 0$, then $A+a\in \A_\xi$, i.e. $a\in \A_0$.
\end{proposition}
\begin{proof}
Introduce the notation $\{a,a\}$ for the quadratic part of the moment map; i.e., $\mu(A+a) = \mu(A) + d_1 a + \{a,a\}$, and note that $a \in \mc{A}_0^s$ implies $\{a,a\} \in \mc{F}^s$. But then $d_1 a = -\mu(A) - \{a,a\} \in \F^s$ and elliptic regularity of $d_0^* \oplus d_1$ yields that $a\in \A^{s+1}_0$.
\end{proof}
%\noindent That is, the Nahm data studied in Proposition \ref{prop:realRG} was all smooth.

\begin{prop}\label{prop:smoothReps} 

Transport the smooth structures of Proposition~\ref{prop:Nahmmfld} under the bijections $\M^{s,{\rm irr}}_{\rm all}\simeq \M^{\rm irr}_{\rm all}$ and $\M^{s,{\rm irr}}_{\xi}\simeq \M^{\rm irr}_{\xi}$ of Proposition~\ref{prop:bij} to yield families of smooth structures on $\M^{\mathrm{irr}}_{\mathrm{all}}$ and $\M^{\mathrm{irr}}_{\xi}$. Then these smooth structures are independent of $s$. In particular, $\mc{M}^{\mathrm{irr}}_{\mathrm{all}}$ and $\mc{M}^{\mathrm{irr}}_{\xi}$ have canonical smooth structures. 
\end{prop}
\bp Recall that $\mc{M}^{s,\mathrm{irr}}_{\mathrm{all}}$ has explicit charts given by $\epsilon$-balls in $\mc{H}^{1,{\rm gen}}_A$. Together with the elliptic regularity observation of Proposition~\ref{prop:coulombSmooth}, we may immediately compare these charts for different values of $s$ and conclude $s$-independence, as desired. The same argument works for $\M^{\rm irr}_\xi$.
\ep 

\begin{rmk} It is pleasant now to compare back to Remark~\ref{rmk:canstr}.\end{rmk}

Finally, we analyze the reducible locus. We begin with

\begin{proposition} \label{prop:realStabs}
The following are equivalent for pre-Nahm data $A\in \A'\times\A''^s_{\mathrm{all}}$:
\begin{enumerate}[(i)]
\item $\ker d_0 \not= \{0\}$,
\item the stabilizer $\hat\G\subset \G$ of $A$ is nontrivial, and
\item there is a good projection $\pi_1$ to a rank one subbundle such that $\pi_2=1-\pi_1$ is also a good projection.
\end{enumerate}
\end{proposition}
\begin{proof}[Proof of Proposition~\ref{prop:realStabs}]
It is immediate that (i) implies (ii), as the exponential of a nonzero element of $\ker d_0$ yields a nontrivial element of the stabilizer. Moreover, (iii) implies (i), as the traceless combination $\hat\pi = \frac{i}{2}(\pi_1-\pi_2) \in \mf{g}$ is in $\ker d_0$. It remains to show that (ii) implies (iii).

Let $g\in \hat\G$ be a nontrivial element of the stabilizer. The fiber over 0 of the Hermitian bundle $E$ admits an orthogonal decomposition into rank 1 eigenspaces of $g$. Let $s_1(0),s_2(0)$ be orthonormal eigenvectors that span this space. We can then parallel transport them to find orthonormal sections $s_1,s_2$ which solve $\partial_{A^0} s_a = 0$, and these determine an orthogonal decomposition of $E$ into eigenbundles $P_1,P_2$ of $g$ with distinct constant eigenvalues $\lambda_1,\lambda_2$. For, if $g s_a = \lambda_a s_a$ at $t=0$ then $\partial_{A^0}(g s_a - \lambda_a s_a) = 0$ implies that $g s_a = \lambda_a s_a$ everywhere, thanks to the uniqueness of solutions of ordinary differential equations. Also, $\lambda_1=\lambda_2^*$ because $g$ is special unitary and $\lambda_1\not=\lambda_2$, since $g$ is not a multiple of the identity.

We now investigate the boundary values of $s_a$. $g s_a = \lambda_a s_a$ immediately implies that $\sigma_z s_a(t_0)$ is an eigenvector of $g(t_0)$ with eigenvalue $\lambda_a$ and is thus proportional to $s_a(t_0)$. We thus write $\sigma_z s_a(t_0) = c_{a,t_0} s_a(t_0)$ and note that $\sigma_z^2 = 1$ implies that $c_{a,t_0} = \pm 1$ (where these signs can differ for $t_0=0,L$, and we must have $c_{1,t_0} = - c_{2,t_0}$). For future use, let us also note now that $\partial s_a = \partial_{A^0} s_a - A^0 s_a = - A^0 s_a$ implies that $\sigma_z \partial s_a(t_0) = - c_{a,t_0} \partial s_a(t_0)$.

Let $\pi_a$ be the orthogonal projection onto $P_a$. The results of the previous paragraph imply that these endomorphisms satisfy the boundary conditions that are imposed on $\mf{g}$. For, they imply that at $t_0$, the eigenspaces of $g$ coincide with those of $\sigma_z$, so the $\pi_a$ are diagonal, and Remark~\ref{rmk:redundant} then handles the boundary conditions on $\partial \pi$. %To show that $\sigma_z \partial \pi_a(t_0) \sigma_z = - \partial\pi_a(t_0)$, we begin with $\partial(\pi_a s_b) = \delta_{a,b} \partial s_a = \partial \pi_a s_b + \pi_a \partial s_b$, so $\partial \pi_a s_b = (\delta_{a,b}-\pi_a) \partial s_b$. We then have $\sigma_z \partial \pi_a \sigma_z s_b(t_0) = c_{b,t_0} \sigma_z \partial \pi_a s_b(t_0) = c_{b,t_0} \sigma_z (\delta_{a,b}-\pi_a) \partial s_b(t_0) = - (\delta_{a,b}-\pi_a) \partial s_b(t_0) = - \partial \pi_a s_b(t_0)$.

Next, exponentiating $\hat\pi = \frac{i}{2}(\pi_1-\pi_2)$ gives a $U(1)$ subgroup of $\tilde\G$ containing $g$ that stabilizes $A$. To see this, we first note that $\lambda_1^{-2i\hat\pi} = \lambda_1^{\pi_1} \lambda_1^{-\pi_2} = (\pi_1 \lambda_1 + \pi_2) (\pi_1 + \pi_2 \lambda_2) = \pi_1 \lambda_1 + \pi_2 \lambda_2 = g$. Next, we note that $\pi_1$ and $\pi_2$, and therefore $\hat\pi$, are covariantly constant, since $s_1$ and $s_2$ are: $\partial_{A^0}(\pi_a s_b) = \partial_{A^0} \pi_a s_b = \delta_{a,b} \partial_{A^0} s_b = 0$. Finally, we show that each $\pi_a$, and therefore also $\hat\pi$, commutes with each $A^i$. This follows from $g A^i s_a = A^i g s_a = \lambda^a A^i s_a$, so that $A^i s_a$ is in the $\lambda_a$ eigenspace and $\pi_b A^i s_a = \delta_{a,b} A^i s_a = A^i \pi_b s_a$.
\end{proof}

\begin{proposition} \label{prop:realSings}
If $\xi_0^i = \xi_L^i = 0$ for all $i$ then $\M_\xi \simeq (\RR^3\times S^1)/Z_2$, with the reducible orbits corresponding to the two $A_1$ singularities. If this is not the case, but we do have $\xi_0^i = \pm \xi_L^i$ with the same sign for all $i$, then $\M_\xi \setminus \M_{\xi}^{\mathrm{irr}}$ is a single point. Finally, for generic $\xi$, all orbits are irreducible, i.e., $\M_{\xi} = \M_{\xi}^{\mathrm{irr}}$.
\end{proposition}
\begin{remark}
We will provide two different proofs of this proposition. The first generalizes more straightforwardly to the study of other moduli spaces; the first part of this proof is also an interesting infinite-dimensional analog of Kronheimer's Proposition 2.8 \cite{kronheimer:construct}. The second proof, instead, makes use of the gauge choice from Proposition \ref{prop:realAx}.
\end{remark}

\begin{proof}[First proof]
Suppose that $\M_\xi$ contains reducible Nahm data $A\in \A_\xi$. We will show that $\xi$ is not generic. Suppose that a nontrivial element $g'\in \G$ with representative $g\in \tilde\G$ stabilizes $A$. Let $\hat\pi$, $\pi_1$, and $\pi_2$ be as in the proof of Proposition \ref{prop:realStabs}. Then we have
\begin{align}
\Tr(\hat\pi A^i(L) &- \hat\pi A^i(0)) = \int_0^L \partial \Tr(\hat \pi A^i)\,dt = \int_0^L \Tr(\partial_{A^0}(\hat\pi A^i))\,dt = \int_0^L \Tr(\hat\pi \partial_{A^0} A^i)\,dt \nonumber \\
&= -\frac{1}{2} \epsilon^{ijk} \int_0^L \Tr(\hat\pi [A^j, A^k])\,dt = - \frac{i}{2} \epsilon^{ijk} \int_0^L \Tr(\pi_1[A^j,A^k])\,dt \nonumber \\
&= -\frac{i}{2} \epsilon^{ijk} \int_0^L \Tr [\pi_1 A^j, A^k]\,dt = 0 \ .
\end{align}
In passing from $\hat\pi$ to $\pi_1$ we used $\hat\pi = i(\pi_1-1/2)$, and we have repeatedly used the vanishing of traces of commutators. But, we know that $\hat\pi(t_0) = \pm \frac{i}{2}\sigma_z$, and accordingly $\Tr \hat\pi A^i(t_0) = \pm \xi^i_{t_0}$. We therefore have $\xi^i_{0} = \pm \xi^i_L$, with the same sign for all $i$.

We now classify reducible Nahm data. Define $g=(s_1\quad s_2)$, the matrix whose columns are $s_1$ and $s_2$ from Proposition \ref{prop:realStabs}. $g$ is everywhere unitary, since $s_a$ are orthonormal. Furthermore, $\det g$ is constant, since $g$ is covariantly constant, and so by multiplying $s_1$ by a constant phase we may assume that $\det g = 1$. Then, $g$ is an extended unitary gauge transformation which diagonalizes $\pi_1$ and $\pi_2$. By exchanging $s_1$ and $s_2$ if necessary, we may assume that $g$ satisfies the boundary conditions on $\G$ at either $t=0$ or $t=L$. If it satisfies the boundary conditions at both ends, and so is an honest element of $\G$, then we must have $\hat\pi(0) = \hat\pi(L)$. On the other hand, if it satisfies the boundary conditions at one end but not the other, then we must have $\hat\pi(0) = -\hat\pi(L)$.

After applying the extended gauge transformation $g$, we arrive at a basis where $\pi_1$ and $\pi_2$ are constant diagonal matrices $\twoMatrix{1}{0}{0}{0}$ and $\twoMatrix{0}{0}{0}{1}$, respectively. Furthermore, these commute with $A^\mu$ -- that is, $A^\mu$ are all diagonal. In addition, if $g$ was not an honest gauge transformation then it negated the diagonal part of $A^i(L)$, so $g$ gives a homeomorphism between moduli spaces $\M_\xi$ with opposite values of $\xi_L$. So, after applying $g$ the problem is reduced to the classification of diagonal reducible Nahm data with $\xi^i_0 = \xi^i_L$. Since $A^\mu$ all commute, Nahm's equations are simply $\partial A^i = 0$, so $A^i$ are all constant and equal to $-i\xi^i_0 \sigma_z$. Finally, we take advantage of the abelian subgroup $\T\subset \G$ consisting of gauge transformations that are everywhere diagonal. Elements $g'$ of this subgroup simply act on $A^0$ via $A^0\mapsto A^0 + \partial \log g'$. Since $A^0$ vanishes at the boundary (as it is diagonal, but the diagonal part of $A^0$ must vanish at the boundary), this is precisely enough freedom to set $A^0=0$ everywhere. The remaining gauge freedom that stabilizes $A=(0,-i\xi^i_0 \sigma_z)$ consists of the $U(1)$ subgroup of $\T$ consisting of constant gauge transformations.

So, each moduli space $\M_\xi$ has at most one singular point with $\hat\pi(0) = \hat\pi(L)$ and one with $\hat\pi(0) = - \hat\pi(L)$. The former possibility is realized if and only if $\xi^i_0 = \xi^i_L$ for all $i$, while the latter is realized if and only if $\xi^i_0 = - \xi^i_L$ for all $i$. We finally leave to the reader the verification that in the case that $\xi$ identically vanishes, $\M_{\xi} \simeq (\mb{R}^3 \times S^1)/Z_2$, parallel to the general treatment of Theorem~\ref{thm:noFI}.
\end{proof}

\begin{proof}[Second proof]
We will explicitly characterize all possible reducible Nahm data in the gauge choice from Proposition \ref{prop:realAx}. In the notation of that proposition, we clearly are already restricted to considering $c=0,\pi/2L$. 
Suppose first that $c = 0$. But then if Nahm data $A = (0, A^i)$ is stabilized by a transformation of the form $g(x) = e^{i\theta \sigma_z}$ with $\theta \in (0,\pi)$, we must have that $A^i_x = A^i_y = 0$ everywhere. Nahm's equations then imply that $A^i_z$ is constant. This is possible if and only if $\xi^i_0 = \xi^i_L$ for all $i$, in which case the stabilized Nahm data is $A = (0, -i \xi^i_0 \sigma_z)$.

Next, suppose that $c=\pi / 2L$. But then applying the extended gauge transformation $\exp(-i \pi t \sigma_x / 2L)$ reduces to the above case of $c = 0$ at the cost of flipping the signs of $\xi^i_L$. Consequently, we once again obtain a single reducible Nahm orbit corresponding to $c = \pi/2L$ if $\xi^i_0 = -\xi^i_L$ for all $i$. 

Finally, if $\xi$ identically vanishes, the constructions of the previous paragraphs give the reducible Nahm data $(0,0)$ and $(i\pi \sigma_x/2L, 0)$.
\end{proof}

This completes our proof of Theorem \ref{thm:modHK}.

\begin{rmk}\label{rmk:sadRGcont} As already discussed in~\S\ref{sec:nahm}, the statement of Theorem~\ref{thm:modHK} is certainly lacking, not least in that we do not yet know that $\M_{\xi}$ is \emph{nonempty} for most $\xi$. Another point of consternation rests with the reducible locus of $\M_{\xi}$; currently, we know for exactly which $\xi$ any reducible orbits exist and how many there are, but as yet we have not explained how they interact with the smooth manifold structure of $\M^{\mathrm{irr}}_{\xi}$. Remark~\ref{rmk:sadRG}, however, indicates that one should expect local neighborhoods of these orbits to have the structure of $A_1$ orbifold singularities. All of these results will follow from the main theorems of~\S\ref{sec:complex} and~\S\ref{sec:duy}.

\iffalse One could also give a direct proof at this stage, however. Looking ahead, the construction of the $RG_{\mb{C}}$ map in Definition~\ref{defn:RGC} applies equally well here to define a map
\be \label{eq:rgR} \{\text{real pre-Nahm data with constant $A^0$}\} \stackrel{RG_{\mb{R}}}{\to} \{\text{real \MZ{pre-}Kronheimer data}\} \ee
which one may use to show that local neighborhoods of reducible Nahm data are homeomorphic to $\mb{R}^4/Z_2$ (and diffeomorphic away from the singular point itself, or alternatively, diffeomorphic as orbifolds). \MZ{don't think this works}\fi
\end{rmk}

\section{Moduli space of complex Nahm data} \label{sec:complex}

We now wish to single out a preferred complex structure and give an essentially holomorphic construction of the moduli spaces. A usual motivation for this sort of construction in similar contexts is that it admits an algebro-geometric recharacterization. In particular, one may often then employ algebro-geometric techniques in order to prove that the moduli spaces are non-empty; this provides an interesting means of proving existence theorems for various non-linear differential equations. As mentioned in the introduction, the results of \cite{hitchin:monoCurve,donaldson:nahm,hurtubise:alg,hurtubise:monopoles,hurtubise:maps} lead us to suspect that similar statements should hold in the present context. However, while interesting, they are not our main motivation. Instead, the primary reason for our interest in this construction is that, following Kronheimer, it allows us to relate moduli spaces with different choices of $\xi$. This will allow us to prove not only that the moduli spaces are non-empty, but that they are (when smooth) diffeomorphic to the minimal resolution of $(\RR^3\times S^1)/Z_2$, and indeed biholomorphic when $\xi^\CC=0$. We prove these results in two stages: in the present section, we provide a holomorphic construction of new moduli spaces $\N_\xi$ and explain some relations between those with the same value of $\xi^\CC$, and in the following section we show that $\N_\xi \simeq \M_\xi$ as complex manifolds (possibly with orbifold singularities). Hyperkahler rotation then implies that there is an isomorphism $\N_{\xi^\zeta} \simeq \M_{\xi}^\zeta$ for any choice of complex structure $\zeta$ on $\M_\xi$, where $\xi^\zeta$ denotes the appropriate $SO(3)$-transformed $\xi$, and this allows us to relate moduli spaces with different values of all three moment map parameters.

Recall now the definitions of
\begin{align*}
& \mc{D}^{\mathrm{ps}}_{\xi} = \{(\alpha, \beta) \in C^{\infty}(I,\mf{sl}(2,\mb{C}))^{\oplus 2} \bigm| \alpha|_{\partial I} \in \mf{t}_{\mb{C}}^{\perp}, \beta|_{\partial I} \in -i \xi^{\mb{C}}_{\partial I} \sigma_z + \mf{t}_{\mb{C}}^{\perp}, \partial_{\alpha} \beta = 0, (\alpha, \beta)\ \xi^{\mb{R}}\text{-polystable}\}\,, \\
& \mc{D}^{\mathrm{ps}}_{\xi^{\mb{R}},\mathrm{all}} = \{(\alpha, \beta) \in C^{\infty}(I,\mf{sl}(2,\mb{C}))^{\oplus 2} \bigm| \alpha|_{\partial I} \in \mf{t}_{\mb{C}}^{\perp}, \partial_{\alpha} \beta = 0, (\alpha, \beta)\ \xi^{\mb{R}}\text{-polystable}\}\,,\\ & \mc{G}_{\mb{C}} = \{g \in C^{\infty}(I,SL(2,\mb{C})) \bigm| g|_{\partial I} \in T_{\mb{C}}, g^{-1} \partial g|_{\partial I} \in \mf{t}^{\perp}_{\mb{C}}\} / Z_2\,. \end{align*}

In parallel to the previous section, our main aim is to prove the following theorem:

\begin{theorem} \label{thm:modCpx} For any $\xi^{\mb{R}}_{\partial I}$, the map $p^\CC: \N_{\xi^\RR,{\rm all}} \to \mf{z}^{\mb{C}}_{\partial I}$ is a submersion of complex orbifolds whose fibers $\mc{N}_{\xi} := \mc{D}^{\mathrm{ps}}_{\xi}/\mc{G}_{\mb{C}}$ are (possibly empty) holomorphic symplectic surfaces with smooth loci $\N_{\xi}^{\mathrm{st}} := \D^{\mathrm{st}}_{\xi}/\G_\CC$. Moreover, if $\xi$ is generic, $\N_{\xi}$ is a (possibly empty) smooth surface; if $\xi$ is nongeneric but nonzero, it has a single $A_1$ orbifold point; and if $\xi$ identically vanishes, $\N_{\xi} \simeq (\mb{C} \times \mb{C}^{\times})/Z_2$. 
\end{theorem}

As always, we denote by $s$ superscripts the $H^s$-regular versions of the spaces we have defined; in particular, we have
\begin{align}
\B'^s &= \{\alpha \in H^s(I,\mf{sl}(2,\CC)) \bigm| \alpha|_{\partial I} \in \mf{t}^{\perp}_{\mb{C}}\} \ , \nonumber \\
\B''^s_{\xi^\CC} &= \{ \beta \in H^s(I,\mf{sl}(2,\mb{C})) \bigm| \beta|_{\partial I} \in -i \xi^{\mb{C}}_{\partial I} \sigma_z + \mf{t}^{\perp}_{\mb{C}}\} \ .
\end{align}
We also use such notation for the Sobolev $s$-regular (poly/semi)stable loci, e.g., $\B^{s,\mathrm{ss}}_\xi = \{b\in \B^s_{\xi^\CC} \ | \ \mbox{$b$ is $\xi^\RR$-semistable} \}$. Morphisms between pre-Nahm data in $\B^s_{\rm all}$ are required to be $H^{s+1}$-regular.

We accordingly modify the notion of good sub-bundle which enters into the stability condition in the Sobolev setting to allow for sub-bundles which are only $H^{s+1}$-regular (i.e., their transition functions are valued in $H^{s+1}$). In practice, we will find it convenient to reformulate this notion in a number of ways. We begin by once and for all choosing a $\xi^\RR$-adapted Hermitian metric, so that we can identify real and complex pre-Nahm data. Next, note that Definition~\ref{def:goodProj} makes sense in the Sobolev regular setting; i.e., an $H^{s+1}$-regular good projection is an $H^{s+1}$-regular orthogonal projection $\pi$ satisfying the conditions of Definition~\ref{def:goodProj}.

We then have the following:
\begin{proposition} \label{prop:manySub}
The following notions are equivalent for pre-Nahm data $(\alpha,\beta) \in \B^s_{\rm all}$:
\begin{enumerate}[(i)]
\item $H^{s+1}$-regular good subbundles $V\subset E$; \label{item:goodBund}
\item $H^1$-regular good subbundles $V\subset E$; \label{item:badBund}
\item $H^{s+1}$-regular good projections; \label{item:goodProj}
\item $H^1$-regular orthogonal projections to positive-rank subbundles $V\subset E$ satisfying $\pi|_{\partial I} \in \mf{t}_\CC$ and conditions (ii) and (iii) of Definition \ref{def:goodProj}. \label{item:badProj}
\end{enumerate}
Furthermore, when the subdata is rank 1 these are equivalent to
\begin{enumerate}[(i)]
\setcounter{enumi}{4}
\item \label{item:goodLine} $H^{s+1}$-regular maps $\ell\colon I\to \CC\PP^1$ satisfying $\ell|_{\partial I} \in \{0,\infty\}$,
\beq\label{eq:ellpreserved} \partial \ell = \alpha_{12} \ell^2 + (\alpha_{11}-\alpha_{22})\ell - \alpha_{21} \ , \eeq
and
\be \label{eq:ellPreserved2} 0 = \beta_{12} \ell^2 + (\beta_{11} - \beta_{22}) \ell - \beta_{21} \ . \ee
\item $H^1$-regular maps $\ell\colon I\to \CC\PP^1$ satisfying $\ell|_{\partial I} \in \{0,\infty\}$, \eqref{eq:ellpreserved}, and \eqref{eq:ellPreserved2}. \label{item:badLine}
\end{enumerate}
\end{proposition}
\begin{proof}
In light of Remark \ref{rmk:redundant}, the equivalence of \eqref{item:goodBund} and \eqref{item:goodProj} is clear. To demonstrate the equivalence with \eqref{item:badProj}, we take the fiberwise adjoint of $(\partial_\alpha \pi) \pi = 0$ in order to obtain $\pi \partial_{-\alpha^\dagger} \pi = 0$. We then add this to $(1-\pi) \partial_\alpha \pi$ in order to obtain
\be \partial_\alpha \pi = \pi [\alpha+\alpha^\dagger, \pi] \ . \label{eq:piEqn} \ee
Elliptic regularity and Sobolev embedding now imply that if $\pi$ is in $H^{s'}$ with $1\le s'\le s$ then the right hand side of this equation is in $H^{s'}$, and so $\pi$ is in $H^{s'+1}$. So, we bootstrap up from $\pi\in H^1$ to $\pi\in H^{s+1}$. Remark \ref{rmk:redundant} again enters to complete the demonstration that \eqref{item:badProj} is equivalent to \eqref{item:goodProj}. Finally, \eqref{item:badProj} is clearly equivalent to \eqref{item:badBund}.

Now, we restrict to rank 1 pre-Nahm subdata. The equivalence of \eqref{item:goodBund} and \eqref{item:goodLine} is immediate -- the subbundle $V$ is simply the span of $\column{a}{b}$, and we write $\ell = b/a$. The equivalence of \eqref{item:goodLine} and \eqref{item:badLine} follows from applying elliptic regularity to \eqref{eq:ellpreserved}.
\end{proof}

\begin{corollary} \label{cor:sameStabs}
If $B\in \B^s_{\rm all}$ with $1\le s\le \infty$ and $1\le s'\le s$, then $B\in \B^{s,{\rm ps}}_{\rm all}$ (resp., $\B^{s,{\rm st}}_{\rm all}$, $\B^{s,{\rm ss}}_{\rm all}$, $\B^{s,{\rm us}}_{\rm all}$) if and only if $B\in \B^{s',{\rm ps}}_{\rm all}$ ($\B^{s',{\rm st}}_{\rm all}$, $\B^{s',{\rm ss}}_{\rm all}$, $\B^{s',{\rm us}}_{\rm all}$).
\end{corollary}
\begin{proof}
The various notions of stability associated to $s,s'$ differ in the degree of Sobolev regularity demanded of good subbundles $V\subset E$, but the previous proposition implies that $V$ is automatically $H^{s+1}$-regular.
\end{proof}

In terms of the good projection perspective, we may note that $V$ is a good subbundle if and only if $gV$ is, for any complexified gauge transformation $g$, as follows: the orthogonal projection
\be \tilde\pi = g\pi(\pi g^\dagger g \pi + (1-\pi))^{-1} \pi g^\dagger \label{eq:newProj} \ee
onto $gV$ is a good projection for the transformed pre-Nahm data if that onto $V$ is for the original pre-Nahm data, i.e. $\tilde\pi$ satisfies the appropriate boundary conditions, $(g\circ\partial_\alpha\circ g^{-1})(gV) = g\partial_\alpha(V)\subset gV$, and $g\beta g^{-1}(gV) = g\beta V \subset gV$, thanks to the boundary conditions on $\pi$ and $g$.\footnote{A useful trick for working with the expression $\tilde\pi$ is to note that $\pi g^\dagger g \pi (\pi g^\dagger g \pi + (1-\pi))^{-1} = \pi (\pi g^\dagger g \pi + (1-\pi)) (\pi g^\dagger g \pi + (1-\pi))^{-1} = \pi$. For example, at the boundary, where $g$ and $\pi$ are diagonal, this implies that $\tilde\pi(t_0) = \pi(t_0)$. It also allows for easy proofs that $\tilde\pi^2 = \tilde \pi$ and $\pi g^\dagger (1-\tilde \pi) = 0$.}

In light of Remark \ref{rmk:redundant}, the following proposition is an immediate computation:
\begin{prop}
Given a good projection $\pi$ to a subbundle $V\subset E$, we have
\be \mathrm{deg}\,V = \Tr\Big(\sigma_z(\xi^{\mb{R}}_L\pi(L)-\xi^{\mb{R}}_0\pi(0))\Big) \ . \ee
\end{prop}
\noindent Using \eqref{eq:newProj}, one may confirm that $\deg V = \deg gV$.

This concludes our discussion of pre-Nahm subbundles. We now have the following propositions:
\begin{proposition} \label{prop:cplxModSpc}
Given $s \ge 1$, we have that $\mc{D}^{s,\mathrm{ps}}_{\xi^{\mb{R}},\mathrm{all}}/\mc{G}^{s+1}_{\mb{C}} \to \mf{z}^{\mb{C}}_{\partial I}$ is a submersion of complex orbifolds whose fibers $\N^s_{\xi} = \mc{D}^{s,\mathrm{ps}}_{\xi}/\mc{G}^{s+1}_{\mb{C}}$ are as described in Theorem~\ref{thm:modCpx}.
\end{proposition}

\begin{proposition} \label{prop:bijC}
For any $s\ge 1$, the map $\B'/\G_\CC \to \B'^s/\G^{s+1}_{\CC}$ is an isomorphism of groupoids, i.e., a bijection of sets which induces canonical identifications of stabilizer groups. 
\end{proposition}

\begin{corollary} \label{cor:complexSmooth}
Complex Nahm data $B \in \B'\times \B''^s_{\mathrm{all}}$ is in $\B_{\mathrm{all}}$. It follows that we have canonical isomorphisms of groupoids $\N_{\xi^{\mb{R}},\mathrm{all}} \stackrel{\sim}{\to} \N^s_{\xi^{\mb{R}},\mathrm{all}}$, $\N_{\xi} \stackrel{\sim}{\to} \N^s_{\xi}$.
\end{corollary}

The proofs of Proposition \ref{prop:bijC} and Corollary \ref{cor:complexSmooth} are identical to their counterparts in the previous section (once one makes use of Corollary \ref{cor:sameStabs}); we hence omit them. Most of this section is devoted to the proof of Proposition~\ref{prop:cplxModSpc}.

The main novelty in this section compared to~\S\ref{sec:hkq} is the role of the stability condition, so we now turn to it. In particular, we will show how it yields a Hausdorff moduli space; this is important, since the non-compactness of $SL(2,\CC)$ means that we do not have an immediate analog of Proposition \ref{prop:gaugeConv}.

We first provide a number of alternative characterizations of our notion of the degree of a good sub-bundle $V$, defined by a projection $\pi$. This will both give intuition for why our definition is the right one to employ and be useful below. Given pre-Nahm data $(\alpha,\beta)$ on $E$ and a good subbundle $V$ with an associated good projection $\pi$, we let $\partial_{\alpha'}$ and $\beta'$ denote the pre-Nahm subdata on $V$. The Hermitian metric allows us to define another set of pre-Nahm data on $V$, namely $\partial_{-\tilde \alpha} = \pi \partial_{-\alpha^\dagger}$, so $\tilde\alpha = \alpha^\dagger + (\partial_{-\alpha^\dagger} \pi)$, and $\tilde\beta = \pi \beta^\dagger$; however, note that this transforms under gauge transformations differently from $(\alpha',\beta')$. We observe that $\half(\partial_{\alpha'} + \partial_{-\tilde\alpha})$ is a Hermitian connection on $V$ (where the Hermitian metric is obtained via restriction). This follows from the fact that $\pi (\partial_\alpha \pi) \pi = 0$, which in turn follows from multiplying \eqref{eq:pipi} on the right by $\pi$. (Of course, this follows trivially from assumption (iii) in the definition of a good projection, but it holds even without that assumption.) Similarly, on $V$ we have $\tilde\beta^\dagger = \beta\pi = \beta'$.

We now have the following proposition, which accords with the usual definition of degree that involves integrating the analog of the real moment map for sub-bundles:
\begin{proposition} \label{prop:hermDeg}
If $V$ is a good sub-bundle defined by a projection $\pi$ and $\partial_{\alpha'} = \pi \partial_\alpha$, $\beta' = \pi \beta$, $\partial_{-\tilde\alpha} = \pi \partial_{-\alpha^\dagger}$, $\tilde\beta = \pi \beta^\dagger$, then
\begin{align}
\deg V &= \half \int_0^L \Tr_V\parens{ [\partial_{-\tilde\alpha}, \partial_{\alpha'}] - [\tilde\beta, \beta'] }\,dt \nonumber \\
&= \half \int_0^L \Tr_V\parens{\partial(\alpha' + \tilde\alpha) - [\tilde\alpha, \alpha'] - [\tilde\beta,\beta']}\,dt \nonumber \\
&= \half \int_0^L \Tr\parens{\pi \parens{ [\partial_{-\tilde\alpha}, \partial_{\alpha'}] - [\tilde\beta, \beta']} }\,dt \nonumber \\
&= \half \int_0^L \Tr\parens{\pi\brackets{\partial(\alpha' + \tilde\alpha) - [\tilde\alpha, \alpha'] - [\tilde\beta,\beta']}}\,dt \ . \label{eq:deg1}
\end{align}
\end{proposition}
\begin{proof}
Since $\pi$ is diagonal at the boundary,
\be \deg V = \Tr(\sigma_z(\xi^\RR_L \pi(L) - \xi^\RR_0 \pi(0))) = \half \Tr(\pi(\alpha+\alpha^\dagger)) |_0^L = \half \Tr_V(\alpha'+\tilde\alpha) |_0^L \ . \ee
This coincides with the second expression in the proposition, since traces of commutators vanish.
\end{proof}
\begin{remark} \label{rmk:parDeg}
Note that if we were studying discontinuous pre-Nahm data on the circle, instead of pre-Nahm data on the interval, then the integrand here would need an extra contribution from $\xi^\RR$, corresponding to the delta distributions in the real moment map equation. Such contributions are familiar in the study of parabolic Higgs bundles, where they constitute the difference between the usual degree and the parabolic degree.
\end{remark}

Another characterization of $\deg V$ will also prove useful:
\begin{proposition}
\begin{align}
\deg V &= \half \int_0^L \Tr\parens{\pi \brackets{\partial(\alpha+\alpha^\dagger) - [\alpha^\dagger,\alpha]} - (\partial_\alpha \pi)(\partial_{-\alpha^\dagger} \pi)}\,dt \nonumber\\
&= \half \int_0^L \Tr\parens{\pi \brackets{\partial(\alpha+\alpha^\dagger) - [\alpha^\dagger,\alpha] - [\beta^\dagger, \beta] } - (\partial_\alpha \pi)(\partial_{-\alpha^\dagger} \pi) - [\beta,\pi][\pi,\beta^\dagger] }\,dt \ . \label{eq:newDeg}
\end{align}
In particular,
\be \deg V \le \half \int_0^L \Tr\parens{\pi \brackets{\partial(\alpha+\alpha^\dagger) - [\alpha^\dagger,\alpha] - [\beta^\dagger, \beta] } }\,dt \ , \ee
with equality if and only if $\partial_\alpha \pi$ and $[\beta,\pi]$ both vanish identically. \label{prop:degIn}
\end{proposition}
\begin{proof}
First, note that the $\beta$ term in \eqref{eq:deg1} vanishes, since $\Tr(\pi[\pi \beta^\dagger,\beta]) = \Tr([\pi \beta^\dagger \pi, \pi \beta \pi]) = 0$. We now focus on the remaining terms:
\begin{align}
\half&\int_0^L \Tr\parens{\pi\brackets{\partial(\alpha' + \tilde\alpha) - [\tilde\alpha, \alpha'] }}\,dt \nonumber \\
&= \half \int_0^L \Tr\parens{\pi\brackets{\partial(\alpha+\alpha^\dagger) - [\alpha^\dagger,\alpha] + \partial_\alpha \partial_{-\alpha^\dagger} \pi}} \nonumber \\
&= \half \int_0^L \Tr(\pi\brackets{\partial(\alpha+\alpha^\dagger) - [\alpha^\dagger, \alpha] } - (\partial_\alpha \pi)( \partial_{-\alpha^\dagger} \pi) )\,dt \ .
\end{align}
This proves the first equality in \eqref{eq:newDeg}. A straightforward calculation proves that the terms involving $\beta$ cancel out from the following expression.
\end{proof}

We now work out some consequences of our notions of stability. We begin with an easy one which serves as an analog in this section of part of the first proof of Proposition \ref{prop:realSings}:
\begin{proposition} \label{prop:compSings1}
If some $B\in \D_{\xi^\CC}^s$ has rank 1 Nahm subdata, then $\xi_0^\CC = \pm \xi_L^\CC$; in particular, $\D^{s,{\rm st}}_\xi \not= \D^s_{\xi^\CC}$ implies that $\xi_0^\CC = \pm \xi_L^\CC$. If $\B^{s,\mathrm{ss}}_{\xi^\RR,{\rm all}}$ contains strictly $\xi^\RR$-semistable pre-Nahm data, then $\xi_0^\RR = \pm \xi_L^\RR$. Finally, if $\D_\xi^{s,{\rm ss}}$ contains strictly $\xi^\RR$-semistable Nahm data then these conditions on $\xi^\RR$ and $\xi^\CC$ are satisfied with the same sign.
\end{proposition}
\begin{proof}
Let $\pi$ be a good projection to a rank 1 sub-bundle, and recall that $\pi(t_0)$ is either $\twoMatrix{1}{0}{0}{0}$ or $\twoMatrix{0}{0}{0}{1}$ for $t_0=0,L$. Suppose that $B=(\alpha,\beta)$ solves the complex Nahm equation, $\partial_\alpha \beta = 0$. We then have
\begin{align}
-i \Tr(\sigma_z(\pi(L) \xi^\CC_L &- \pi(0) \xi^\CC_0)) = \Tr(\pi \beta)|_0^L = \int_0^L \Tr(\partial_\alpha(\pi \beta))\,dt = \int_0^L \Tr((\partial_\alpha \pi)\beta)\,dt \nonumber \\
&= \int_0^L \Tr(\pi (\partial_\alpha \pi) \beta)\,dt = \int_0^L \Tr((\partial_\alpha \pi) \pi \beta \pi)\,dt = 0 \ .
\end{align}
The equalities on the second row used the fact that $\pi$ is a good projection, as well as $\pi(\partial_\alpha \pi)\pi = 0$, which follows from \eqref{eq:pipi}. So, we have $\xi_0^\CC = \pm \xi_L^\CC$, where the relative sign is determined by whether $\pi(0) = \pi(L)$ or $\pi(0) = 1-\pi(L)$.

Now, let $B$ be strictly $\xi^\RR$-semistable pre-Nahm data. This means that $\xi_L^\RR \Tr(\sigma_z \pi(L)) = \xi_0^\RR \Tr(\sigma_z \pi(0))$ for some good projection $\pi$ to a rank 1 sub-bundle, i.e., $\xi_0^\RR = \pm \xi_L^\RR$, where the relative sign is determined as in the previous paragraph.
\end{proof}

Recall now the notation $\partial_{\alpha*\alpha'}$ from Definition~\ref{defn:starconnection}.

\begin{lemma} \label{lem:notZero}
Let $\alpha,\alpha'$ define two continuous connections on $E$. Then, any $g\in C^1(I,\End(E))$ with $\partial_{\alpha*\alpha'} g = 0$ has a covariantly constant determinant $\det g \in C^1(I,\End(\Wedge^2 E))$. So, $\det g$ either vanishes nowhere or everywhere, and if $\alpha,\alpha'$ are traceless then $\det g$ is constant. Similarly, $g$ itself either vanishes nowhere or everywhere.
\end{lemma}
\begin{proof}
Using the induced connection on $\End(\Wedge^2 E)$, we have $\partial_{\alpha*\alpha'} \det g = 0$. The remaining statements follow from uniqueness of solutions to ordinary differential equations.
\end{proof}

\begin{proposition} \label{prop:stabInv}
Suppose $(\alpha,\beta),(\alpha',\beta') \in \B^{s,\mathrm{ss}}_\xi$ are pre-Nahm data such that at least one is $\xi^\RR$-stable. Then, suppose $g\in H^{s+1}(I, \End(E))$ is nonzero, is diagonal at $0,L$, and relates the two pre-Nahm data via $g \circ \partial_{\alpha} = \partial_{\alpha'} \circ g$ and $g\beta = \beta' g$. That is, $g$ is a nonzero morphism from the former pre-Nahm data to the latter. Then $g$ is invertible. This also holds for $s=\infty$.
\end{proposition}
\begin{proof}
Suppose, toward a contradiction, that $g$ is not invertible somewhere. By the previous lemma, it is neither invertible nor zero anywhere. Hence, the rank of $g$ is constant and everywhere equal to one; in particular, $L_1 := \ker g$ and $L_2 := \im g$ are both line sub-bundles of $E$. Then, since $\partial_\alpha \ker g \subset \ker g$ and $\beta \ker g\subset \ker g$, and $\partial_{\alpha'} \im g \subset \im g$ and $\beta' \im g\subset \im g$, we conclude that $L_1$ is good with respect to $(\alpha,\beta)$ and $L_2$ is good with respect to $(\alpha',\beta')$. (Note here the importance of our relaxation of the regularity requirements of the projections defining these bundles so that they need only be in $H^{s+1}$: higher regularity is not guaranteed, since $g\in H^{s+1}$. We do at least have $H^{s+1}$ regularity: for $L_2$ the projector is $M (M^\dagger M)^{-1} M^\dagger$, where $M\in H^{s+1}$ is a matrix whose columns consist of a basis for $\im g$, and for $L_1$ the projector is $1-\tilde M(\tilde M^\dagger \tilde M)^{-1} \tilde M^\dagger$, where $\tilde M \in H^{s+1}$ is a matrix whose columns consist of a basis for $\im g^\dagger$.) The $\xi^\RR$-(semi)stability hypotheses then imply that $\deg L_1 \le \half \deg (\alpha,\beta)=0, \deg L_2 \le \half\deg (\alpha',\beta')=0$, where at least one of these inequalities is strict. On the other hand, $g$ defines an isomorphism from $E/L_1$ to $L_2$. We put the natural pre-Nahm data $(q\partial_\alpha, q\beta)$ on $E/L_1$, where $q : E\to E/L_1$ is the quotient map. Finally, letting $g^{-1}$ denote the inverse map from $L_2$ to $E/L_1$ and observing that $g q \partial_\alpha \circ g^{-1} = g\partial_\alpha \circ g^{-1} = \partial_{\alpha'}$, and similarly that $g q \beta g^{-1} = \beta'$, we find that the pre-Nahm data on $E/L_1$ and $L_2$ are isomorphic, and in particular have the same degree. This yields the contradiction $\deg E < \deg E$.
\end{proof}

\begin{corollary} \label{cor:freeAct}
For $1 \le s \le \infty$, if $B \in \mc{B}^s_{\mathrm{all}}$ is $\xi^{\mb{R}}$-stable for some $\xi^{\mb{R}}$ and $g$ is a morphism of pre-Nahm data from $B$ to itself, then $g$ is a constant multiple of the identity. In particular, $\G^{s+1}_\CC$ acts freely on $\B^{s,\mathrm{st}}_\xi$.
\end{corollary}
\begin{proof}
Let $g$ be a morphism from $(\alpha,\beta) \in \B^{s,\mathrm{st}}_\xi$ to itself. Choose some $t \in I$ and let $\lambda$ be one of the eigenvalues for the $g$ action on the fiber $E_t$. Then, $g-\lambda I$ defines an endomorphism of $E$ which is diagonal at $0,L$ and commutes with both $\beta$ and $\partial_\alpha$, and so it is either zero or an isomorphism. But, since it is not an isomorphism at $t$ it is zero everywhere. That is, $g$ is simply multiplication by a nonzero constant. Such gauge transformations with $\lambda = \pm 1$ are precisely the ones which we quotiented out in the definition of $\G^{s+1}_\CC$, so the latter acts freely on $\B^{s,\mathrm{st}}_\xi$.
\end{proof}

We now turn to a discussion of orbit closures. This discussion builds on the discussion at the end of \S6.2.2 of \cite{DK}. It gives a (weaker) analog of Proposition \ref{prop:gaugeConv}, where the role of compactness of $SU(2)$ is now played by polystability. We first introduce some terminology that will be of use.

\begin{defn} Suppose pre-Nahm data $B = (\alpha, \beta)$ admits a good subline $L \subset E$. Recall that both $L$ and $E/L$ inherit pre-Nahm data from $B$. Let $F$ denote the filtration $0 \subset L \subset E$ with associated exact sequence \beq\label{eq:JH}0 \to L \to E \to E/L \to 0\eeq when thought of as endowed with this pre-Nahm data; we term this filtration $F$ a \emph{good filtration}. In particular, $L \oplus E/L$ is a rank two bundle with pre-Nahm data. We refer to this pre-Nahm data on $L \oplus E/L$ as the associated graded pre-Nahm data of $F$ and denote it by $\mathrm{gr}\,B$, or  $\mathrm{gr}_{F}B$ for specificity.

Next, the fact that $E$ is naturally an $SL(2,\mb{C})$-bundle means that there is a natural isomorphism $\Wedge^2 E \simeq \underline{\mb{C}}$, where $\underline{\mb{C}}$ denotes the trivial line bundle (with nearly trivial pre-Nahm data, where $\underline{\mb{C}}^-_{\partial I} = \underline{\mb{C}}_{\partial I}$ and $\underline{\mb{C}}^+_{\partial I}$ is trivial\footnote{The fact that the pre-Nahm data $\det E$ is not entirely trivial, thanks to the fact that $\underline{\mb{C}}_{\partial I}$ equals $\underline{\mb{C}}^-_{\partial I}$, as opposed to $\underline{\mb{C}}^+_{\partial I}$, can be reinterpreted in the $Z_2$-equivariant construction on $S^1$ as the fact that the bundle is naturally only an equivariant $SL^\pm(2,\CC)$-bundle, but not an equivariant $SL(2,\CC)$-bundle, because $\det\sigma_z = -1$. Here, $SL^\pm(2,\CC)$ is defined as the kernel of the homomorphism $GL(2,\CC)\stackrel{\det^2}{\to} \CC^\times$.}); given a good subline $L \subset E$ as above, this isomorphism in turn induces an isomorphism $(L^*)^- \simeq E/L$, where the minus superscript indicates that we swap the superscripts on the spaces $(L^*)^\pm_{\partial I}$, which respects the natural pre-Nahm data on both sides.

Finally, given $\xi^{\mb{R}} \in \mf{z}_{\partial I}$, if $F$ is a good filtration such that $\deg_{\xi^{\mb{R}}}L = \deg_{\xi^{\mb{R}}}E = 0$, or equivalently $\deg_{\xi^{\mb{R}}}(E/L) = 0$, then we term $F$ a \emph{Jordan-H\"older filtration}.\footnote{Typically, we would furthermore demand that $L$ and $E/L$ be $\xi^{\mb{R}}$-stable, but these are automatic in the case here of $L, E/L$ both rank one.}
\end{defn}

Recall that the extended gauge group is the union of $\G^{s+1}_\CC$, $\G^{s+1}_\CC \cdot \exp(\pi i t \sigma_x / 2L)$, $\G^{s+1}_\CC \cdot i \sigma_x$, and $\G^{s+1}_\CC \cdot \exp(\pi i (L - t) \sigma_x / 2L)$, and that there is a map from this group to $Z_2\times Z_2$ which indicates in which of these four sets an extended gauge transformation is contained. We now establish the following technical lemma:

\begin{lem}\label{lem:guse} If the pre-Nahm data $B \in \mc{B}^s_{\mathrm{all}}$ admits a good subline $L \subset E$, we may find an extended gauge transformation $g$ such that $g^{-1} L$ is the subline $\mb{C} e_1 \subset E$. If moreover $B$ admits two good sublines $L_1, L_2$, then $g$ may be chosen such that $g^{-1} L_i$ is the subline $\mb{C}e_i \subset E$ for both $i \in \{1, 2\}$. \end{lem} 

\bp We may first apply an extended gauge transformation so as to assume wlog that $L$ is $\mb{C}e_1$ at both endpoints $\partial I$. Let us now write our putative gauge transformation $g \in \tilde{\mc{G}}^{s+1}_{\mb{C}}$ as comprised of columns $v(t), w(t)$. We wish to mandate that (i) $v(t)$ spans $L$ for all $t \in I$, (ii) $g$ everywhere has determinant $1$, and (iii) $g|_{\partial I}$ is diagonal and $g^{-1}\partial g|_{\partial I}$ is off-diagonal. To do so, construct $v$ by parallel transport of $v(0) = e_1$. By the assumption that $L$ is a good subline, we shall also have that $v$ at the right endpoint is $c e_1$ for some $c \in \mb{C}^{\times}$. Construct $w$ in local neighborhoods of the endpoints of $I$ by parallel transport of $e_2$ and $c^{-1}e_2$ at the left and right endpoints, respectively. As $\Tr \alpha = 0$, it follows that in these local neighborhoods, $g$ still has $\det 1$. Moreover, $g$ now manifestly satisfies the correct boundary conditions. Finally, in the interior of $I$, choose any continuation of $w$ that ensures $\det g = 1$ everywhere. % It's not immediately obvious one can do so. One approach is to interpret this as a problem where $v\colon I \to \mb{CP}^1$ is some smooth (or Sobolev-regular, or continuous) path, $w\colon \partial I \to \mb{CP}^1$ is a specification of boundary values, and one wishes to continue to $w\colon I \to \mb{CP}^1$ such that $w$ is everywhere disjoint from $v$ (after which one may scale to ensure $\det$ literally $1$ rather than merely nonzero). One approach is to force $w$ to be $Rv$ ``almost everywhere'' (in the non-technical sense), where $R$ is the antipodal map on $S^2 \simeq \mb{CP}^1$ and, using continuity (using the sufficient Sobolev regularity, in turn!) of $v$ near the endpoints, one may `rapidly' take paths for $w$ from its boundary values to $Rv(\partial I)$ without crossing over $v$ during those short time intervals.

Finally, if $B$ in fact admits two good sublines $L_1, L_2$, we may again apply an extended gauge transformation so as to wlog assume that $L_i$ agrees with $\mb{C}e_i$ at both endpoints $\partial I$ for $i \in \{1, 2\}$. It is now even easier to proceed: one may simply choose $w$ to be the parallel transport of $w(0) = e_2$ everywhere. By assumption that $L_2$ too is a good subline, it will automatically be true that the $g$ so constructed everywhere has $\det g = 1$ and that $w$ at the right endpoint of $I$ will be given by $c^{-1} e_2$. 
\ep

Recall now from Proposition~\ref{prop:realStabs} that pre-Nahm data $B$ is reducible if and only if 
$B$ admits two good sublines $L_1, L_2$ such that $E \simeq L_1 \oplus L_2$ as pre-Nahm data. Note in particular that the associated graded pre-Nahm data of a good filtration automatically satisfies this condition, as does $\xi^{\mb{R}}$-strictly polystable pre-Nahm data.

\begin{lem}\label{lem:psJH} Suppose $B$ is reducible pre-Nahm data. Then for any  good filtration $F$ on $B$, $B$ equals $\mathrm{gr}_F(B)$. In particular, this conclusion holds if $B$ is strictly $\xi^{\mb{R}}$-polystable for some $\xi^{\mb{R}}$. \end{lem} \bp If $E \simeq L_1 \oplus L_2$, any other good subline must, because of the boundary conditions, agree with either $L_1$ or $L_2$ at the left endpoint. But then, either in terms of the associated projector or the associated map $\ell$ in the terminology of Proposition~\ref{prop:manySub}, the good subline is completely determined by its value at the left endpoint and hence must agree with either $L_1$ or $L_2$; the conclusion is now immediate. \ep

The proof of this lemma also makes it clear that irreducible pre-Nahm data has at most one good subline, and so any irreducible pre-Nahm data is $\xi^\RR$-stable for some $\xi^\RR$. This strengthens the conclusion of Corollary \ref{cor:freeAct}. That irreducible pre-Nahm data has at most one good subline, while reducible pre-Nahm data has two good sublines, implies that pre-Nahm data admits a Jordan-H\"older filtration if and only if it is strictly $\xi^\RR$-semistable. We also note that all reducible pre-Nahm data is strictly $\xi^\RR$-polystable for some $\xi^\RR$, since this is the case when $\xi^\RR=0$.

\begin{prop}\label{prop:orbitclosure} Suppose $1 \le s \le \infty$, and suppose $B_n \in \mc{B}^s_{\mathrm{all}}, g_n \in \mc{G}^{s+1}_{\mb{C}}$ are sequences such that $B_n \to B, g_n(B_n) \to B'$ in $\mc{B}^s_{\mathrm{all}}$. Then, at least one of the following holds: either (i) there exists $g \in \mc{G}^{s+1}_{\mb{C}}$ such that $g(B) = B'$, or (ii) there exist good filtrations $0 \subset L \subset E$, $0 \subset L' \subset E$ of $B, B'$ respectively such that there exists a morphism $g$ from $B'$ to $B$ which is everywhere rank 1 and induces an isomorphism of pre-Nahm data $g'$ from $E/L'$ to $L$.

In the second case, we obtain a dual isomorphism $g'^{\vee}$ from $E/L$ to $L'$. These transformations $g'$ and $(g'^{\vee})^{-1}$ together yield an isomorphism relating the associated graded pre-Nahm data $L \oplus E/L$ and $E/L' \oplus L'$. There exist identifications of both $L\oplus E/L$ and $E/L'\oplus L'$ with the trivial rank-two (marked) bundle $E$ such that we may regard $g'\oplus (g'^\vee)^{-1}$ as a gauge transformation in $\G_\CC^{s+1}$.
\end{prop}

\bp
Write
\begin{align} H_1 &= \{g \in H^{s+1}(I,\mathrm{End}(E)) \bigm| g|_{\partial I}\text{ is diagonal}, \partial g|_{\partial I} \in \mf{t}_{\mb{C}}^{\perp}\} \nonumber \\ H_2 &= \{h \in H^s(I,\mathrm{End}(E))^{\oplus 2} \bigm| h|_{\partial I} \in (\mf{t}_{\mb{C}}^\perp)^{\oplus 2}\}\,.\label{eq:defnH1}\end{align}
($H_2$ differs from $\B_0^s$ in that elements are not required to be traceless away from $\partial I$.) Denote now $B_n = (\alpha_n, \beta_n), g_n(B_n) = (\alpha_n', \beta_n'), B = (\alpha, \beta), B' = (\alpha', \beta')$, and consider the sequence of operators $$D_n:=(\partial_{\alpha_n*\alpha_n'}) \oplus (\beta_n*\beta_n')\text{ as maps } H_1 \to H_2\,.$$ Note that $g_n$ lies in the kernel of the above operator. Similarly denote $D:= (\partial_{\alpha*\alpha'}) \oplus (\beta*\beta')$. For $s$ finite, the $D_n, D$ are all semi-Fredholm, and $D_n \to D$ as operators $H_1 \to H_2$. The upper semicontinuity of semi-Fredholm kernel dimension hence implies that $D$ has some nonzero element $g$ in its kernel. For infinite $s$, we may reach the same conclusion by observing that the same argument implies that $D\colon H_1^{s'+1}\to H_2^{s'}$ has a nontrivial kernel for any $1\le s' < \infty$, where the superscripts denote modifications of the Sobolev regularity of $H_1$ and $H_2$, and overdetermined elliptic regularity then implies that the kernel of $D$ consists of smooth functions.

By construction, $g$ is a gauge transformation transforming $B$ to $B'$ save for the fact that it may not have determinant one. By Lemma~\ref{lem:notZero}, $\det g$ is constant, so either (i) we may suppose, after an overall scaling, that $\det g$ is everywhere equal to $1$ or (ii) $\det g$ is everywhere equal to $0$. These correspond to the two cases in the statement of the proposition. More precisely, in case (ii), $g$ must everywhere have rank one, i.e., $g\colon E \to E$ factors as a composition $E \twoheadrightarrow E/\ker g \stackrel{\sim}{\to} \im g \hookrightarrow E$ for good filtrations $0 \subset \ker g \subset E, 0 \subset \im g \subset E$ of $B', B$ respectively. We let $L'=\ker g$ and $L = \im g$ and let $g'$ be the isomorphism between $E/L'$ and $L$. Constructing $g'^\vee$ is straightforward, using the triviality of both $E/L\otimes L \otimes \underline{\mb{C}}$ and $E/L'\otimes L' \otimes \underline{\mb{C}}$: we naturally obtain an isomorphism from $(E/L')^-$ to $L^-$, but tensoring both of these with $\underline{\mb{C}}$ yields the desired isomorphism $g'^\vee$. We now choose, by Lemma~\ref{lem:guse}, extended gauge transformations $g_1, g_2$ such that $g_1(B)$ is upper-triangular and $g_2(B')$ is lower-triangular. Note that $g_1, g_2$ have the same image in $Z_2 \times Z_2$, i.e., in the Weyl group of $\mc{G}^{s+1}_{\mb{C}}$. As a result, we have that $g_1^{-1} H_1 g_2 = H_1$. We then identify $g'$ with the top left (and only non-vanishing) entry of $g_1^{-1} g g_2$ and let $\tilde g := g'\oplus (g'^\vee)^{-1}$ be the element of $\tilde\G_\CC^{s+1}$ which, in this new basis, is block diagonal, with top-left entry given by $g'$ and bottom-right entry given by its inverse. This is everywhere special linear by construction, and the fact that $\tilde g^{-1}\partial \tilde g|_{\partial I} = 0$ follows from the boundary conditions satisfied by $g$. Finally, we may restore our usual marking on $E$ by undoing our change of basis, obtaining the element $g_1 \tilde g g_2^{-1} \in \tilde\G^{s+1}_\CC$.
\ep

We obtain the following results on orbit closures in $\mc{B}^s_{\mathrm{all}}$:

%\item if $B$ is $\xi^{\mb{R}}$-polystable, the $\mc{G}^{s+1}_{\mb{C}}$-orbit of $B$ is closed in $\mc{B}^{s,\mathrm{ss}}_{\xi^{\mb{R}},\mathrm{all}}$,
\begin{prop}\label{prop:orbitclosures} Suppose $1 \le s \le \infty$ and that we have FI parameters $\xi \in \mf{z}^3_{\partial I}$. Then we have the following: \begin{enumerate}[(i)] \item $\mc{B}^{s,\mathrm{ps}}_{\xi}/\mc{G}^{s+1}_{\mb{C}}$ is Hausdorff, \item if $B$ is $\xi^{\mb{R}}$-semistable and $B' \in \overline{B \cdot \mc{G}^{s+1}_{\mb{C}}} \setminus B \cdot \mc{G}^{s+1}_{\mb{C}}$ is in its orbit closure, then $B'$ is strictly $\xi^{\mb{R}}$-polystable, and
\item if $B$ is $\xi^{\mb{R}}$-semistable, then there is a unique $\xi^{\mb{R}}$-polystable orbit in the orbit closure of $B$. If $B$ is strictly $\xi^\RR$-semistable then this polystable orbit is strictly polystable and is given by the orbit of the associated graded pre-Nahm data of a Jordan-H\"older filtration of $B$. \end{enumerate}
In particular, a $\xi^{\mb{R}}$-semistable orbit is closed in $\mc{B}^{s,\mathrm{ss}}_{\xi}$ if and only if it is $\xi^{\mb{R}}$-polystable.
\end{prop}
\bp We first show a special case of (iii):
\begin{lem}
If $B \in \mc{B}^s_{\mathrm{all}}$ is $\xi^{\mb{R}}$-stable, the $\mc{G}^{s+1}_{\mb{C}}$-orbit of $B$ is closed in $\mc{B}^{s,\mathrm{ss}}_{\xi^{\mb{R}},\mathrm{all}}$.
\end{lem}
\bp
If $B'$ is in the orbit closure of $B$, Proposition~\ref{prop:stabInv} yields that we are in case (i) of Proposition~\ref{prop:orbitclosure} and so $B'$ is in the orbit of $B$.
\ep
Now, to show (i) of the proposition, we need to show that if $B_n \to B$ and $g_n(B_n) \to B'$ with $B, B'$ both $\xi^{\mb{R}}$-polystable, then $B, B'$ are in the same $\mc{G}^{s+1}_{\mb{C}}$-orbit. But by Proposition~\ref{prop:orbitclosure}, either we are in case (i), in which case we are done, or we are in case (ii), in which case $B$ and $B'$ have filtrations $0 \subset L \subset E, 0 \subset L' \subset E$ with $L \simeq E/L'$ and there is a gauge transformation relating $L\oplus E/L$ to $L'\oplus E/L'$; semistability then yields $0 \ge \deg L = \deg E/L' \ge 0$, whence $B, B'$ are both strictly polystable and we are done by Lemma~\ref{lem:psJH}.  

Similarly, for (ii), it remains to show that if a sequence $\{g_i(B)\}$ of points within a strictly $\xi^{\mb{R}}$-semistable orbit converges to a point $B'$ outside of the orbit of $B$, then that limit is strictly $\xi^\RR$-polystable. Because $B'$ is not gauge-equivalent to $B$, we are in case (ii) of Proposition \ref{prop:orbitclosure}, and so there is an isomorphism $g'$ relating pre-Nahm subdata $L$ of $B$ to pre-Nahm quotient data $E/L'$ for some good subline $L'$ of $B'$. We now make use of the following lemma:

\begin{lem}\label{lem:limsubs} 
Suppose that $B_i \to B$ in $\B^s_{\rm all}$ and that each $B_i$ has a good subline $L_i$ with an associated good projector $\pi_i$ such that $\Tr \pi_i \sigma_z|_{\partial I}$ is the same for all $i$. Then $\pi_i\to \pi$ in $H^{s+1}$, where $\pi$ is a good projector to a good subline of $B$.
\end{lem}
\bp This follows from the same reasoning as in Proposition \ref{prop:gaugeConv} and Lemma \ref{lem:homeoOrb}, using \eqref{eq:piEqn} and the analogous equation involving $\beta$ instead of \eqref{eq:want}, and using the uniform $C^0$-boundedness of orthogonal projectors instead of the uniform $C^0$-boundedness of unitary transformations. The role played by the uniqueness of $g$ in the proof of Lemma \ref{lem:homeoOrb} is now played by the uniqueness of $\pi$, which follows from the uniqueness of solutions of elliptic ordinary differential equations.
\ep

In the present context, the good projectors to the sublines $g_i^{-1}L$ of $g_i(B)$ converge to a good projector to a good subline $L''$ of $B'$; furthermore, as $L$ and $L'$ have opposite $\sigma_z$ eigenvalues at the two endpoints, it follows that $L'$ and $L''$ are distinct good sublines of $B'$. Since these have degree 0, $B'$ is polystable.

%Recall that for strictly polystable data, the associated good projectors are parallel for the $\alpha$ connection on $I$. As such, they are fully determined by the line sub-bundle at, say, the left endpoint -- where, recall, it lies in either the $e_1$ or $e_2$ direction. Hence, for a sequence $g_i(B)$ of strictly polystable points converging in $\mc{B}^{s,\mathrm{ps}}_{\xi}$, the associated sequence of good projectors $\pi_i$ projecting to, say, the $e_1$ direction at the left boundary (i) also converges in $H^{s+1}$ to a good projector on the limiting pre-Nahm data (by the same reasoning as in Proposition \ref{prop:gaugeConv} and Lemma \ref{lem:homeoOrb}, using the uniform $C^0$-boundedness of orthogonal projectors to get started) and (ii) has constant $\xi^{\mb{R}}$-degree, as the directions of the line subbundle at $\partial I$ are patently constant. As the same conclusion holds for the projectors which project to $e_2$ at the left boundary, the limiting pre-Nahm data is still $\xi^{\mb{R}}$-polystable and the prior statement (i) of this proposition applies.

Finally, for (iii), suppose that $B$ is strictly $\xi^{\mb{R}}$-semistable; i.e., suppose $L \subset E$ is a good degree 0 subline with degree 0 quotient bundle $E/L$. We now wish to show that the strictly $\xi^{\mb{R}}$-polystable pre-Nahm data on $L \oplus E/L$ is in the orbit closure of $B$. Lemma~\ref{lem:guse} provides an extended gauge transformation $g$ such that $g^{-1} L$ is the constant line $\mb{C}e_1$. 

After applying this $g$, we now have that $\alpha, \beta$ are both upper-triangular. We may now explicitly verify that conjugating $g_n = \begin{pmatrix} i n & 0 \\ 0 & -i/n \end{pmatrix}$ by $g$ and allowing $n \to \infty$ yields the limit $L \oplus E/L$ as within the orbit closure of our original pre-Nahm data. 

It remains to show that there is a \emph{unique} polystable orbit in the orbit closure of $B$, but this again follows from an immediate application of Proposition~\ref{prop:orbitclosure} and Lemma \ref{lem:psJH}. 
\ep

\begin{rmk} For $1 \le s \le \infty$ as above, one may similarly wonder if $\mc{B}^{s,\mathrm{ps}}_{\xi}/\mc{G}^{s+1}_{\mb{C}}$ is metrizable in addition to being Hausdorff. The $L^2$ metric analogs of Corollary \ref{cor:realMet}, i.e., $\inf_{g\in \G_\CC^{s+1}} \|A-g(B)\|_{L^2}$, certainly do not define metrics for these quotients. Nonetheless, these spaces are still metrizable via the DUY isomorphism, as will be remarked in Corollary \ref{cor:compMet}. \end{rmk}

The above Proposition~\ref{prop:orbitclosures} motivates the following definition:

\begin{defn} Two semistable orbits are said to be \emph{$S$-equivalent} if they have the same polystable orbit in their orbit closure. Note that this equivalence relation is trivial on the polystable locus; moreover, each $S$-equivalence class has a distinguished polystable representative by construction. \end{defn}

\begin{proposition} \label{prop:stabOpen}
$\xi^\RR$-stability and $\xi^\RR$-semistability are open conditions in $\B^s_{\xi^\CC}$ and $\B_{\xi^\CC}$.
\end{proposition}
\begin{proof}
We prove that $\xi^\RR$-semistability is an open condition; the proof for $\xi^\RR$-stability is identical. Suppose otherwise, so that there is a sequence of $\xi^\RR$-unstable points of $\B^s_{\xi^\CC}$ which converges to a $\xi^\RR$-semistable point $B$. Then, by Lemma~\ref{lem:limsubs}, we would have a destabilizing subbundle of $B$, contradicting its $\xi^{\mb{R}}$-semistability. (Note that the set of possible degrees of good sub-bundles, with $\xi^\RR$ fixed, is finite, and so a sequence of positive degrees cannot limit to 0.)
\end{proof}

\begin{rmk} Propositions~\ref{prop:orbitclosures} and \ref{prop:stabOpen} finally allow one to compare with the general formulation of (poly/semi)stability as outlined in Definition~\ref{defn:stab}. Although not logically necessary, it is pleasant to complete this comparison by noting that the orbit closure of an \emph{unstable} orbit may not intersect the polystable locus (or even the semistable locus), as follows immediately from Proposition \ref{prop:stabOpen}, or alternatively from another application of Proposition~\ref{prop:orbitclosure} and Lemma \ref{lem:psJH}. 

\end{rmk}

As in the previous section, we now wish to put a smooth structure on our moduli spaces; for simplicity, we here write out only the special deformation complex. So, consider the following complex associated to some complex Nahm data $B=(\alpha,\beta)\in \B^{\mathrm{ps}}_\xi$:
\be 0 \to \mf{g}^{s+1}_{\CC} \xrightarrow{\bar\partial_0} \B^s_0 \xrightarrow{\bar\partial_1} \F^{s-1}_{\CC} \to 0 \ . \ee
The notation is meant to be suggestive of a Dolbeault complex. Here, $\bar\partial_0 h = \partial_\alpha h \oplus [\beta, h]$ describes the effects of an infinitesimal gauge transformation, while $\bar\partial_1 (\alpha',\beta') = \partial_\alpha \beta' + [\alpha', \beta]$ describes the linearization of the complex Nahm equation about $B$. This is a complex, thanks to the fact that $B$ solves the complex Nahm equation. So, we can define the cohomology groups ${\rm H}^{0}_B = \ker \bar\partial_0$, ${\rm H}^{1,{\rm sp}}_B = \ker \bar\partial_1 / \im \bar\partial_0$, and ${\rm H}^{2,{\rm sp}}_B = \coker \bar\partial_1$. We also define the formal adjoints $\bar\partial_0^* (\alpha',\beta') = -\partial_{-\alpha^\dagger} \alpha' + [\beta^\dagger, \beta']$ and $\bar\partial_1^* f = (-[\beta^\dagger, f], -\partial_{-\alpha^\dagger} f)$, as well as the Laplacians $\bar\Delta_i = \bar\partial_i^* \bar\partial_i + \bar\partial_{i-1} \bar\partial_{i-1}^*$.

We now have an analog of Proposition \ref{prop:realInd}:
\begin{proposition} \label{prop:compInd}
The index of the special deformation complex coincides with that of
\be \bar\partial_0^*\oplus \bar\partial_1 \colon \B_0^s \to H^{s-1}(I,\mf{sl}(2,\CC)^{\oplus 2}) \ . \ee
Indeed,
\be \mathrm{H}^{1,{\rm sp}}_B \simeq \ker \bar\partial_0^* \oplus \bar\partial_1 = \ker(\bar\Delta_1\colon \tilde\B_{\xi^\CC} \to C^\infty(I,\mf{sl}(2,\CC))^{\oplus 2}) \ , \ee 
where
\be \tilde\B_{\xi^\CC} = \{ b\in \B_0 \bigm| \bar\partial_0^* b|_{\partial I} \in \mf{t}_\CC , \bar\partial_1 b|_{\partial I} \in \mf{t}_\CC \} \ ,\footnote{Note that this space depends only on $\xi^\CC$, but not on $B$.} \ee 
and $\coker \bar\partial_0^*\oplus \bar\partial_1 \simeq \mathrm{H}^{0}_B\oplus \mathrm{H}^{2,{\rm sp}}_B$. Furthermore,
\be \mathrm{H}^{0}_B = \ker(\bar\Delta_0\colon \mf{g}_\CC\to C^{\infty}(I,\mf{sl}(2,\CC))) = \ker(\bar\partial_0\colon \mf{g}_\CC\to \B_0) \ee
and
\be \mathrm{H}^{2,{\rm sp}}_B\simeq \ker(\bar\Delta_2\colon\tilde\F_\CC\to \F_\CC) = \ker(\bar\partial_1^*\colon\tilde\F_\CC\to \B_0) \ , \ee
where $\tilde\F_\CC = \{f\in C^{\infty}(I,\mf{sl}(2,\CC)) \bigm| f|_{\partial I} \in \mf{t}_{\mb{C}}, \partial f|_{\partial I} \in \mf{t}^{\perp}_{\mb{C}}\}$.
\end{proposition}
The proof is omitted, as it is the same as in the previous section. As in the previous section, analogous results hold for the general deformation complex, and in particular they imply that ${\rm H}^{2,{\rm gen}} = 0$. We define the `harmonic' spaces
\begin{align}
\Hh^0_B &:= \mathrm{H}^0_B = \ker(\bar\Delta_0\colon \mf{g}_\CC\to C^\infty(I;\mf{sl}(2,\CC))) = \ker(\bar\partial_0\colon \mf{g}_\CC\to \B_0) \nonumber \\
\Hh^{1,{\rm sp}}_B &:= \ker(\bar\Delta_1\colon \tilde\B_{\xi^\CC} \to C^\infty(I,\mf{sl}(2,\CC))^{\oplus 2}) \nonumber \\
&= \ker(\bar\partial_0^* \oplus \bar\partial_1\colon \B_0 \to C^\infty(I;\mf{sl}(2,\CC)^{\oplus 2})) \ , \nonumber \\
\Hh^{1,{\rm gen}}_B &:= \ker(\bar\Delta_1\colon \{b\in \B_{\rm all} \bigm| \bar\partial_0^*b|_{\partial I} \in \mf{t}_\CC,\ \bar\partial_1 b|_{\partial I} = 0\} \to C^\infty(I,\mf{sl}(2,\CC))^{\oplus 2}) \nonumber \\
&= \ker(\bar\partial_0^* \oplus \bar\partial_1\colon \B_{\rm all} \to C^\infty(I;\mf{sl}(2,\CC)^{\oplus 2})) \ , \nonumber \\
\Hh^{2,{\rm sp}}_B &:= \ker(\bar\Delta_2\colon \tilde\F_\CC\to \F_\CC) = \ker(\bar\partial_1^*\colon \tilde\F_\CC\to \B_0) \ , \nonumber \\
\Hh^{2,{\rm gen}}_B &:= 0 \ .
\end{align}

\begin{proposition}
$\mathrm{H}^{2,{\rm sp}}_B \simeq \mathrm{H}^0_B.$
\end{proposition}
\begin{proof}
The proof is identical to that of Proposition \ref{prop:hodge}, except instead of wedging and contracting with K\"ahler forms we wedge and contract with the anti-holomorphic 2-form.
\end{proof}

\begin{proposition}
For all complex Nahm data $B \in \D_{\rm all}$, we have the exact sequence $0 \to \Hh^{1,\mathrm{sp}}_B \to \mathcal{H}^{1,\mathrm{gen}}_B \to \mf{z}^{\mb{C}}_{\partial I} \to \mathcal{H}^{2,\mathrm{sp}}_B \to 0$.
\end{proposition}

The following is a corollary of Proposition \ref{prop:compInd}:

\begin{corollary}
For any complex Nahm data $B\in \B^{\mathrm{ps}}_\xi$, the index of the special deformation complex is $-2$ and that of the general deformation complex is $-4$.
\end{corollary}
\begin{proof}
Using deformation-invariance of the index, we compute the index of $\bar\partial_0^*\oplus \bar\partial_1$ with vanishing complex Nahm data $B$. Then, we have $\dim_\CC \mathrm{H}^{1,{\rm sp}}_B=4$, $\dim_\CC {\rm H}^{1,{\rm gen}}_B = 5$, $\dim_\CC \mathrm{H}^0_B=1$, and $\dim_\CC \mathrm{H}^{2,{\rm sp}}_B = 1$, since $\Hh^{1,{\rm sp}}_B \simeq (\mf{t}_\CC^{\perp})^{\oplus 2}$, $\Hh^{1,{\rm gen}}_B \simeq \mf{t}_\CC^\perp \oplus \mf{sl}(2,\CC)$, $\Hh^0_B \simeq \mf{t}_\CC$, and $\Hh^{2,{\rm sp}}_B \simeq \mf{t}_\CC$.
\end{proof}

We now have an analog of Proposition \ref{prop:realStabs}:
\begin{proposition} \label{prop:compStabs}
The following are equivalent for pre-Nahm data $B=(\alpha,\beta)\in \B'\times \B''^s_{\rm all}$:
\begin{enumerate}
\item $\ker\bar\partial_0 \not= \{0\}$;
\item the stabilizer $\hat\G_\CC\subset \G_\CC$ of $B$ is nontrivial;
\item $\ker d_0 \not= \{0\}$;
\item the stabilizer $\hat\G\subset \G$ of $B$ is nontrivial; and
\item there is a good projection $\pi_1$ to a rank one subbundle such that $\pi_2=1-\pi_1$ is also a good projection (and so, in particular, if $B$ is $\xi^\RR$-semistable then it is strictly $\xi^\RR$-polystable).
\end{enumerate}
\end{proposition}
\begin{proof}
As in Proposition \ref{prop:realStabs}, $5\Rightarrow 3\Rightarrow 1\Rightarrow 2$ and $5\Rightarrow 3\Rightarrow 4 \Rightarrow 2$ are immediate, so the only non-trivial statement is $2\Rightarrow 5$. The main novelty here is that given a non-trivial $g\in \hat\G$, if we do not account for the boundary conditions then we do not a priori know that $g(0)$ is diagonalizable, since $g$ is not unitary. However, the boundary conditions imply that $g(0)$ is diagonal, and $\det g(0)=1$ implies that it takes the form $\twoMatrix{c}{0}{0}{1/c}$ for some $c\in \CC^\times$. From here, the proof is nearly the same as that of Proposition \ref{prop:realStabs}; the only change is that the sections $s_a$ satisfying $\partial_\alpha s_a = 0$ are no longer required to be orthonormal, but they still determine a decomposition of $E$ into eigenbundles of $g$ with distinct constant eigenvalues.

So, we know that there are good projections $\pi$ and $1-\pi$ to proper subbundles. Since the degrees of these subbundles are negatives of each other, if $B$ is $\xi^\RR$-semistable then these degrees both vanish and $B$ is strictly $\xi^\RR$-polystable.
\end{proof}

We now give an analog of Proposition~\ref{prop:realSings}, i.e., a characterization of $\N_{\xi} \setminus \N_{\xi}^{\mathrm{st}}$. This statement is a mild strengthening of Proposition~\ref{prop:compSings1}.
\begin{proposition} \label{prop:compSings2}
If $\xi_0^\RR = \xi_L^\RR = \xi_0^\CC = \xi_L^\CC = 0$ then $\N_\xi \simeq (\CC\times \CC^\times)/Z_2$, with $\N_{\xi} \setminus \N_{\xi}^{\mathrm{st}}$ the two $A_1$ singularities of $(\mb{C} \times \mb{C}^{\times}) / Z_2$. If this is not the case, but we do have $\xi_0^\RR = \pm \xi_L^\RR$ and $\xi_0^\CC = \pm \xi_L^\CC$ with the same relative sign, then $\N_\xi \setminus \N_{\xi}^{\mathrm{st}}$ is a single point. Finally, for generic $\xi$, $\N_\xi = \N_{\xi}^{\mathrm{st}}$.
\end{proposition}
\begin{proof}
Both proofs of Proposition~\ref{prop:realSings} generalize immediately to the present context: the former generalization exploits Proposition~\ref{prop:compSings1} and Proposition~\ref{prop:compStabs}, while the latter uses Proposition~\ref{prop:compNahm}. The only change worth commenting on is that in the first proof, instead of defining $g = (s_1\quad s_2)$, we should define $g = c (s_1\quad s_2)$, where $c\in \CC^\times$ is chosen so that $\det g = 1$. (Note that $c$ is constant, since $\partial_\alpha (s_1\quad s_2) = 0$ and tracelessness of $\alpha$ together imply that $\partial \det(s_1\quad s_2) = 0$.)
\end{proof}

Once again, we would like to know more refined information in the second case above in order to characterize the local behavior at the strictly polystable points as $A_1$ singularities. We will do so below, in Proposition \ref{prop:RGitems} and Remark \ref{rmk:partialRes}.

We now put a smooth structure on our moduli space. We proceed as in the previous section, but in a slightly different order. We first have

\begin{lemma}[Complex Coulomb gauge] \label{lem:localDiff}
For $1 \le s < \infty$ and for $B\in \B' \times \B''^s_{\mathrm{all}}$ with stabilizer $\hat\G_\CC\subset \G_\CC$, the $\G_\CC^{s+1}$-equivariant map $F\colon (\G_\CC^{s+1} \times \ker \bar\partial_0^*)/\hat{\mc{G}}_{\mb{C}} \to \B^s_{\mathrm{all}}$ defined by $(g, b) \mapsto g(B + b)$, and with $\hat{g} \in \hat{\mc{G}}_{\mb{C}}$ acting on $(g, b)$ via $(g, b) \mapsto (\hat{g}^{-1}g, \hat{g}^{-1}b\hat{g})$, yields a local biholomorphism identifying a neighborhood of $(1,0)$ with a neighborhood of $B\in \B^s_{\rm all}$. %there is \MZ{we care what it is} a $\G_\CC^{s+1}$-equivariant map $F\colon \hat\G_\CC\backslash (\G_\CC^{s+1} \times \ker \bar\partial_0^*) \to \B^s_{\mathrm{all}}$ whose restriction to a neighborhood of $(1,0)$ yields a diffeomorphism with a neighborhood of $B\in \B^s_{\rm all}$.
\end{lemma}

\noindent The proof is as in Proposition \ref{prop:realcoul}, using the inverse function theorem.

We now wish to show that $\ker \bar \partial_0^* / \hat{\G}_\CC \to \B^s_{\mathrm{all}} / \G^{s+1}_{\CC}$ is a local homeomorphism. To do so, it would be convenient to have an analog of Lemma \ref{lem:homeoOrb}. However, the proof strategy adopted there does not work because the analog of Proposition \ref{prop:gaugeConv} in the present section, namely Proposition \ref{prop:orbitclosure}, is weaker: in case (ii) it need not be the case that $B$ and $B'$ are gauge-equivalent, and even in case (i) we do not have convergence of a subsequence in $\G_\CC^{s+1}$. In general, what we need to show is that there exists $\delta > 0$ such that if $B', B''$ are two points of $\ker\bar\partial_0^*$ in the same $\G^{s+1}_{\CC}$-orbit and within $\delta$ of $B$ in the $H^s$-norm, then there exists $g \in \hat{\G}_{\CC}$ such that $g(B') = B''$. Equivalently, by means of Lemma~\ref{lem:localDiff}, it suffices to show the \emph{a priori} weaker statement that for any $\epsilon > 0$, there exists $\delta > 0$ such that given $B', B''$ related as above, there exists $g \in \G^{s+1}_{\CC}$ and $\hat{g} \in \hat{\G}_{\CC}$ such that $g(B') = B''$ and $\| \hat{g} g - \mathrm{Id}\|_{H^{s+1}} < \epsilon$. In any case, this statement now immediately follows from the following two stronger statements: Lemma~\ref{lem:NKg} when $B$ is reducible and Proposition~\ref{prop:quanthomeo} when $B$ is irreducible. We note that Proposition \ref{prop:quanthomeo} implies a similar conclusion as in Lemma \ref{lem:homeoOrb} (see also Remark \ref{rmk:slightWeak}): in case (i) of Proposition \ref{prop:orbitclosure}, when $B$ is irreducible, we have $g_n \to g$ in $\hat\G_\CC \backslash \G^{s+1}_\CC$. When $B$ is irreducible, our proof follows a strategy employed in the proof of Theorem 6 in \cite{donaldson:top} and Theorem 7.3.17 in \cite{kobayashi:bundles}.

\begin{lem}\label{lem:NKg} Given $\alpha \in \mc{B}'$ reducible, let $\hat{\mc{G}}_{\alpha,\mb{C}} \subset \mc{G}^{s+1}_{\mb{C}}$ denote the stabilizer subgroup and $S \subset \mc{B}'^s$ be the slice $\{\alpha' \in \mc{B}'^s \bigm| \overline{\partial}^*_0 \alpha' = 0\}$. Then there exists $\delta > 0$ such that if $g \in \mc{G}^{s+1}_{\mb{C}}, \alpha' \in S$ satisfy $\|\alpha'\|_{H^s} < \delta, \|g(\alpha'+\alpha)-\alpha\|_{H^s} < \delta$, then $g \in \hat{\mc{G}}_{\alpha,\mb{C}}$. \end{lem}

\begin{rmk}
As introduced previously in Remark~\ref{rmk:sadRG} and as will continue in Remark~\ref{rmk:cplxRG}, one expects the local geometry of our hyperkahler quotient construction in a neighborhood of some singularity, i.e., of some strictly polystable point, to be well-approximated by the appropriate local geometry of the Kronheimer construction. Lemma~\ref{lem:NKg} speaks to this expectation at the level of the relevant gauge transformations: one may say that Nahm gauge transformations relating Coulomb gauge connections near a polystable point coincide with Kronheimer gauge transformations.
\end{rmk}

\bp The truth of this proposition is invariant under extended gauge transformation, so without loss of generality assume that $\alpha = 0$. Let us recall the description of Proposition~\ref{prop:mon} that two connections are gauge-equivalent if and only if their parallel transport maps agree in the relevant double-coset; it follows that we wish to show \beq\label{eq:exp}\mf{t}_{\mb{C}}^{\perp} / T_{\mb{C}} \stackrel{\exp}{\to} T_{\mb{C}} \backslash SL(2,\mb{C}) / T_{\mb{C}} \eeq is a local isomorphism in a neighborhood of $0 \in \mf{t}_{\mb{C}}^{\perp}$. Recall that $T_{\mb{C}}$ acts on $\mf{t}_{\mb{C}}^{\perp}$ by conjugation whereas the actions on $SL(2,\mb{C})$ are by left- and right-multiplication.

Introduce the analytic functions \begin{align*} \mathrm{ch}(z) &= \cosh z^{1/2} = \sum_{k=0}^{\infty} \frac{1}{(2k)!} z^k\,, \\ \mathrm{sh}(z) &= z^{-1/2} \sinh z^{1/2} = \sum_{k=0}^{\infty} \frac{1}{(2k+1)!} z^k\,.\end{align*} Then the map of~\eqref{eq:exp} above may be explicitly written as $\begin{pmatrix} 0 & a \\ b & 0 \end{pmatrix} \mapsto \begin{pmatrix} \mathrm{ch}(ab) & a \, \mathrm{sh}(ab) \\ b\, \mathrm{sh}(ab) & \mathrm{ch}(ab) \end{pmatrix}$. Let us now expand~\eqref{eq:exp} to the commutative square \beq\begin{tikzcd} \mf{t}_{\mb{C}}^{\perp}/T_{\mb{C}} \ar[r] \ar[d] & T_{\mb{C}} \backslash SL(2,\mb{C}) / T_{\mb{C}} \ar[d] \\ \mb{C} \ar[r] & \mb{C}\makebox[0pt][l]{\,.}\end{tikzcd}\label{eq:coolsq}\eeq where the left vertical arrow is $\begin{pmatrix} 0 & a \\ b & 0 \end{pmatrix} \mapsto ab$, the right vertical arrow is $\begin{pmatrix} A & B \\ C & D \end{pmatrix} \mapsto AD$, and the bottom horizontal arrow is $z \mapsto \mathrm{ch}^2(z)$. Recall that the vertical arrows above are biholomorphic isomorphisms on the complement of $0 \in \mb{C}$ on the left and $0, 1 \in \mb{C}$ on the right; all three of these special fibers are copies of the cardinality-three non-Hausdorff space $\mathrm{Thr} := \mb{C}^2 / \mb{C}^{\times}$, with action given by $(z_1, z_2) \mapsto (\lambda z_1, \lambda^{-1} z_2)$.

As $z \mapsto \mathrm{ch}^2(z)$ is locally injective about $z = 0$ by consideration of the nonzero derivative there, it now immediately follows that~\eqref{eq:exp} is locally isomorphic about $0 \in \mf{t}^{\perp}_{\mb{C}}$ except for, possibly, the fiber over $0 \in \mb{C}$, where we need to know if the $\mathrm{Thr}$ fiber over $0 \in \mb{C}$ -- i.e., whether it is $a, b$, or both that are zero -- on the left maps isomorphically to the $\mathrm{Thr}$ fiber over $1 \in \mb{C}$ on the right -- i.e., whether it is $B,C$, or both that are zero. But, if $ab = 0$, the off-diagonal elements of the exponential are $a$ and $b$ on the nose, and so this is immediate.  \ep

\begin{prop}\label{prop:quanthomeo} Suppose $B = (\alpha, \beta) \in \mc{B}' \times \mc{B}''^s_{\mathrm{all}}$ is irreducible. Then for any $\epsilon > 0$, there exists $\delta > 0$ such that if $g \in \tilde{\mc{G}}^{s+1}_{\mb{C}}$ and $B' \in \mc{B}^s_{\mathrm{all}}$ are such that $\|B' - B\|_{H^s} < \delta, \|g(B') - B\|_{H^s} < \delta$, then $\|\sigma g - {\rm Id}\|_{H^{s+1}} < \epsilon$ for some $\sigma \in \{\pm 1\}$. \end{prop}
\bp We define $b = B' - B$ and $a = g(B') - B$ and assume that $\|a\|_{H^s}, \|b\|_{H^s} < \delta$. With $H_1$ and $H_2$ as defined in \eqref{eq:defnH1}, we consider the differential operator $\bar\partial_0\colon H_1\to H_2$ (where $\bar\partial_0$ is defined with respect to $B$). That $g(B+b) = B+a$ is equivalent to $\bar\partial_0 g = g a - b g$. As in Lemma \ref{lem:fredA0}, we find that $\bar\partial_0$ is semi-Fredholm with a finite-dimensional kernel, i.e. has a left parametrix $P$, despite its coefficients not being smooth; we may assume that $P$ is a generalized left-inverse, i.e., that $P \bar\partial_0 = 1 - \pi$ where $\pi$ is the $L^2$-orthogonal projector onto $\ker\bar\partial_0$. Consider now the decomposition $g = g' + g''$ with $g' = \ker \bar\partial_0$ and $g''$ in the $L^2$-orthogonal complement thereof, where $g' \in H_1$ and so the same is true for $g''$. In fact, by assumption of $B$ irreducible, we have that $\ker \bar\partial_0 = \mb{C}\,\mathrm{Id}$, i.e., constant multiples of the identity matrix.

We now recall the usual overdetermined elliptic estimate:

\be \|g''\|_{H^{s+1}} = \|P \bar\partial_0 g'' \|_{H^{s+1}} \le \|P\|_{H^s\to H^{s+1}} \|\bar\partial_0 g''\|_{H^s} = \|P\| \|\bar\partial_0 g\|_{H^s} = \|P\| \|b g - g a\|_{H^s} \ . \ee
Boundedness of Sobolev multiplication allows us to further bound $\|b g\|_{H^s}, \|g a\|_{H^s} \le C' \|g\|_{H^{s+1}} \delta$ for some constant $C' > 0$, so the right hand side is at most $2\delta \|P\| C' \|g\|_{H^{s+1}} \le 2 \delta \|P\| C' (\|g'\|_{H^{s+1}} + \|g''\|_{H^{s+1}})$. Rearranging gives
\be \|g''\|_{H^{s+1}} \le \frac{2\delta \|P\| C' \|g'\|_{H^{s+1}}}{1-2\delta \|P\| C'} \le C \delta \|g'\|_{H^{s+1}} \label{eq:overdet} \ee
for some $C>0$ and for sufficiently small $\delta$. 

Let us now use that $g' = c\, {\rm Id}$ for some $c\in \CC$; note in particular that $c$ is nonzero, else~\eqref{eq:overdet} immediately yields $g'' = 0$ and hence $g = 0 \not\in \tilde{\mc{G}}^{s+1}_{\mb{C}}$. But then we have $\|g''\|_{C^0} \le C'' \delta c$ for some $C''>0$, and $g=c({\rm Id}+g''/c)$ and $\det g = 1$ imply that $c^2-1 = O(\delta)$, i.e., that $c = \pm 1 + O(\delta)$. Let $\sigma = \pm 1$ accordingly; we thus have $\|\sigma g - \mathrm{Id}\|_{H^{s+1}} = \|(\sigma c - 1) \mathrm{Id} + \sigma g''\|_{H^{s+1}} = O(\delta)$, as desired. \ep

We have now established the following: 
\begin{prop}\label{prop:cplxcoul} For $1 \le s < \infty$ and $B \in \mc{B}' \times \mc{B}''^s_{\mathrm{all}}$ with stabilizer $\hat{\mc{G}}_{\mb{C}} \subset \mc{G}^{s+1}_{\mb{C}}$, the natural map \beq\label{eq:cplxmap} \{b \in \mc{B}^s_{\mathrm{all}} \bigm| \bar\partial_0^* b = 0\} / \hat{\mc{G}}_{\mb{C}} \to \mc{B}^s_{\mathrm{all}} / \mc{G}^{s+1}_{\mb{C}}\eeq is a local homeomorphism identifying a neighborhood of $[0] \in \{b \in \mc{B}^s_{\mathrm{all}} \bigm| \bar\partial_0^* b = 0\} / \hat{\mc{G}}_{\mb{C}}$ with a neighborhood of $[B] \in \mc{B}^s_{\mathrm{all}}/\mc{G}^{s+1}_{\mb{C}}$. This structure is compatible with the natural projection $p^{\mb{C}}$ to $\mf{z}^{\mb{C}}_{\partial I}$, with a a neighborhood of $[0] \in \{b \in \mc{B}^s_{\xi^{\mb{C}}} \bigm| \bar\partial_0^* b = 0\} / \hat{\mc{G}}_{\mb{C}}$ identified with a neighborhood of $[B] \in \mc{B}^s_{\xi^{\mb{C}}} / \mc{G}^{s+1}_{\mb{C}}$. \end{prop}

If $B$ is stable, and in particular irreducible so that $\hat{\mc{G}}_{\mb{C}}$ is trivial, we may more explicitly write a neighborhood of $[0]$ on the left side of~\eqref{eq:cplxmap} as $$T^s_{B,\epsilon,\mathrm{all}} := \{b\in \B^s_{\rm all} \bigm| \bar\partial_0^* b = 0,\, \|b\|_{\B^s_{\rm all}} < \epsilon \}\,.$$ More generally, even if $B$ has stabilizer, a neighborhood of $[B]$ is given by $T^s_{B,\epsilon,\mathrm{all}}/\mbox{\!\,$\sim$}$ where the equivalence relation $\sim$ is that of the partially-defined conjugation action of $\hat{\mc{G}}_{\mb{C}}$. %; alternatively, using more standard notation, we may write this neighborhood as $T^s_{B,\epsilon,\mathrm{all}} \hat{\mc{G}}_{\mb{C}} / \hat{\mc{G}}_{\mb{C}}$.
Finally, similarly define $T^s_{B,\epsilon} := \{b\in \B^s_{0} \bigm| \bar\partial_0^* b = 0,\, \|b\|_{\B^s_0}<\epsilon\}$.

\begin{prop}\label{prop:cplxpreNahmmfld} For $1 \le s < \infty$, $\B^{s,\mathrm{st}}_{\xi^{\mb{R}},\mathrm{all}}/\G_{\mb{C}}^{s+1}$ is canonically a complex Hilbert manifold. The map $\B^{s,\mathrm{st}}_{\xi^{\mb{R}},\mathrm{all}}/\G_{\mb{C}}^{s+1} \stackrel{p^{\mb{C}}}{\to} \mf{z}^{\mb{C}}_{\partial I}$ is a holomorphic submersion when the source is equipped with this manifold structure. 

Explicitly, charts for this manifold structure are given by the neighborhoods $T^s_{B,\epsilon,\mathrm{all}}$ of Proposition~\ref{prop:cplxcoul}, or by $T^s_{B,\epsilon}$ for the fibers $\B^{s,\mathrm{st}}_{\xi}/\G_{\mb{C}}^{s+1}$ of $p^{\mb{C}}$.
\end{prop}

We now wish to cut out the moduli space of complex Nahm data within the space of gauge orbits of polystable pre-Nahm data; we first state an analog of Proposition~\ref{prop:realRG}. The argument is again an application of the implicit function theorem, although now with the mild upgrade that we are also interested in cutting out the polystable locus, which has an interpretation directly in terms of a stability condition on a finite-dimensional Stein space which we denote by $\hat\mu_\CC^{-1}(0) \subset \Hh^1_B$.

In order to distinguish between the notions of stability that we have employed throughout in our infinite-dimensional problem and the notions of stability appropriate for this finite-dimensional problem, we will refer to the latter as GIT-(poly/semi)stability. Although $\hat\mu_\CC^{-1}(0)$ will not be shown to be an affine scheme (because, as in Remark \ref{rmk:sadRG}, the function $\hat\mu_\CC$ can differ from a homogeneous quadratic polynomial by higher-order corrections), there is a theory of GIT quotients for such Stein spaces which closely parallels that of affine GIT quotients \cite{Heinzner}: one defines $\hat{\mu}_{\mb{C}}^{-1}(0)\kqnoxi\hat{\mc{G}}_{\mb{C}}$ to be the complex-analytic space whose holomorphic functions are the $\hat{\mc{G}}_{\mb{C}}$-invariant holomorphic functions on $\hat{\mu}_{\mb{C}}^{-1}(0)$. By results of~\cite{Heinzner}, the underlying quotient space $\hat{\mu}^{-1}_{\mb{C}}(0)\kqnoxi\hat{\mc{G}}_{\mb{C}}$ has underlying space given by the space of closed $\hat{\mc{G}}_{\mb{C}}$-orbits of $\hat{\mu}^{-1}_{\mb{C}}(0)$, which we refer to as GIT-polystable orbits. For our purposes, rather than relying on results of \cite{Heinzner}, we may simply take this as our definition of $\hat\mu^{-1}_\CC(0)\kqnoxi \hat\G_\CC$. We also define a GIT-polystable orbit to be GIT-stable if its points have finite stabilizers. We note that in this approach, it is natural to say that every orbit is GIT-semistable, as in affine GIT quotients, since GIT-polystable orbits are those which are closed in all of $\hat\mu_\CC^{-1}(0)$.

\begin{proposition} \label{prop:cplxRG}
For any $B\in \D_\xi^{\mathrm{ps}}$ which is stabilized by $\hat\G_{\mb{C}}\subset \G_\CC$ with Lie algebra $\hat{\mf{g}}_{\mb{C}} \simeq {\rm H}^0_B$, a neighborhood of $[B]$ in $\N^s_\xi$ is homeomorphic to a neighborhood of 0 in the finite-dimensional GIT quotient $\hat\mu_\CC^{-1}(0) \kqnoxi \hat\G_\CC$, where $\hat\mu_\CC\colon \Hh^1_B \to \hat{\mf{g}}_{\mb{C}}^* \simeq \mathrm{H}^2_B$ is a $\hat{\mc{G}}_{\mb{C}}$-equivariant holomorphic map.
\end{proposition}
\begin{proof}
%\MZ{maybe comment on coordinate dependence, ie B vs G picture}

The same argument as Proposition~\ref{prop:realRG} provides a holomorphic $\hat{\mc{G}}_{\mb{C}}$-equivariant map $\hat{\mu}_{\mb{C}}\colon \mc{H}^1_B \to \hat{\mf{g}}^{\vee}_{\mb{C}}$ such that a neighborhood $V$ of $[0]$ in $\hat{\mu}_{\mb{C}}^{-1}(0) /\hat{\mc{G}}_{\mb{C}}$ is identified with a neighborhood $U$ of $[B]$ in $\mc{D}^{s}_{\xi}/\mc{G}^{s+1}_{\mb{C}}$. The remaining content is hence to show that $\xi^{\mb{R}}$-polystability on $U$ precisely corresponds to GIT-polystability; we will simultaneously show as well that $\xi^{\mb{R}}$-stability corresponds to GIT stability.

Since $\xi^\RR$-semistability is an open condition, we may assume that all points in $U$ are $\xi^\RR$-semistable. 
We recall, e.g. from the proof of Proposition~\ref{prop:compSings2}, that strictly $\xi^\RR$-polystable points in $\N^{s}_\xi$ are isolated, so we may assume there are no strictly $\xi^{\mb{R}}$-polystable points in $U\setminus\{[B]\}$. %also assume that all strictly $\xi^\RR$-semistable points in $U\setminus\{[B]\}$ are not $\xi^\RR$-polystable.
We have shown above in Proposition~\ref{prop:orbitclosures} that if $p \in \mc{B}^{s,\mathrm{ss}}_{\xi}$ and $q \in \overline{p \cdot \mc{G}^{s+1}_{\mb{C}}} \setminus p \cdot \mc{G}^{s+1}_{\mb{C}}$ is a point in its orbit closure outside of the orbit itself, then $q$ must be strictly $\xi^\RR$-polystable. It hence follows from the above assumptions that if $\overline{p} \in U$ is a closed point of $U$, then it is a closed point\footnote{Note in general that if $U \subset X$ is an open subset of a non-Hausdorff space $X$ that it need not be true, \emph{a priori}, that closed points of $U$ are closed points of $X$.} of $\mc{D}^{s,\mathrm{ss}}_{\xi} / \mc{G}^{s+1}_{\mb{C}}$; indeed, any other $\xi^\RR$-semistable point in its closure must be a strictly $\xi^\RR$-polystable orbit and hence the point $[B]$ itself, in which case $\overline{p}$ would already fail to be closed in $U$. 

The same analysis is true, but more easily so, for $V$. Indeed, one may immediately see that the origin is the only nontrivially stabilized point for the $\hat{\mc{G}}_{\mb{C}}$-action on $\ker\bar\partial_0^*$ and hence \emph{a fortiori} the only strictly GIT-polystable point in $\hat{\mu}_{\mb{C}}^{-1}(0)$. But now, under the homeomorphism $V \stackrel{\sim}{\to} U$, closed points are identified with closed points, where the above discussion removes any ambiguity as to whether we mean closed in $U$ or closed in $\mc{D}^{s,\mathrm{ss}}_{\xi}/\mc{G}^{s+1}_{\mb{C}}$ (and similarly for closedness in $V$ versus closedness in $\hat{\mu}_{\mb{C}}^{-1}(0)/\hat{\mc{G}}_{\mb{C}}$); i.e., GIT-polystable orbits identify with $\xi^{\mb{R}}$-polystable orbits.

Finally, if a $\G_\CC^{s+1}$ orbit representing a point in $U$ is, in addition, $\xi^\RR$-stable, then its points have trivial $\G_\CC^{s+1}$ stabilizers, and therefore trivial $\hat\G_{\mb{C}}$ stabilizers, and so it is GIT-stable. Conversely, by choosing $U$ sufficiently small, we may suppose (either via the isolatedness of strictly $\xi^\RR$-polystable points or via Lemma \ref{lem:NKg}) that the stabilizers of points in the orbits parametrized by $U$ are subgroups of (a conjugate of) $\hat\G_{\mb{C}}$. But then GIT-stability implies that the points in the corresponding $\xi^\RR$-polystable $\G_\CC^{s+1}$ orbit have finite stabilizers, and so the orbit is in fact $\xi^\RR$-stable (and the stabilizers are trivial).
\end{proof}

We summarize the analogs of Propositions~\ref{prop:Nahmmfld},~\ref{prop:coulombSmooth},~\ref{prop:smoothReps}, whose proofs are all identical:
\begin{prop}\label{prop:cplxNahmmfld} $\N^{s,\mathrm{st}}_{\xi^{\mb{R}},\mathrm{all}}$ is a complex submanifold of $\B^{s,\mathrm{st}}_{\xi^{\mb{R}},\mathrm{all}}/\mc{G}^{s+1}_{\mb{C}}$ with a smooth submersion to $\mf{z}^{\mb{C}}_{\partial I}$, the fibers $\N^{s,\mathrm{st}}_{\xi}$ of which are complex surfaces with charts given by $\epsilon$-balls in $\mc{H}^1_B$. Complex manifold structures may be induced on $\N^{\rm st}_{\xi^\RR,{\rm all}}$ and $\N^{\rm st}_\xi$ via Corollary~\ref{cor:complexSmooth}, and the resulting structures are independent of the choice of $s$ used to define them. \end{prop}

We now turn to the study of the local behavior at the strictly polystable points. We first note the following construction, which makes more manifest the role that the semistable locus may play in the holomorphic story:
\begin{constr}\label{constr:sshol}
For any $\xi$ and $s\in [1,\infty]$, we have a natural map $\mc{D}^{s,\mathrm{ss}}_{\xi} \to \N^s_{\xi}$ given as the composition $\mc{D}^{s,\mathrm{ss}}_{\xi} \to \mc{D}^{s,\mathrm{ps}}_{\xi} \to \N^s_{\xi}$, where the first map ``collapses $S$-equivalence classes''; i.e., sends a semistable orbit to the unique polystable representative of its $S$-equivalence class, and the second map is the canonical quotient map. For $s<\infty$ we note that the locally closed sublocus $\mc{D}^{s,\mathrm{st}}_{\xi} \subset \mc{B}^s_{\xi^\CC}$ carries a canonical holomorphic structure by an implicit function theorem argument, since the derivative $\bar\partial_1$ of the holomorphic map $\mu_\CC$ is everywhere a surjective operator as ${\rm H}^2_B=0$ for $B\in \D^{s,{\rm st}}_\xi$;\footnote{We derived this for $B\in \D^{\rm st}_\xi$, but it is easy to see that ${\rm H}^2_B \simeq {\rm H}^2_{g(B)}$ for all $g\in \G^{s+1}_\CC$.} this holomorphic structure is also that induced by the presentation of $\mc{D}^{s,\mathrm{st}}_{\xi}$ as a $\mc{G}^{s+1}_{\mb{C}}$-bundle over $\mc{N}^{s,\mathrm{st}}_{\xi}$. Having introduced this structure, the map $\D^{s,\rm ss}_\xi \to \N^s_\xi$ restricts to a holomorphic submersion $\mc{D}^{s,\mathrm{st}}_{\xi} \to \N^{s,\mathrm{st}}_{\xi}$ on the stable locus; if so desired, one could also interpret the full map $\mc{D}^{s,\mathrm{ss}}_{\xi} \to \N^s_{\xi}$ as a holomorphic map of complex-analytic spaces (i.e., local ringed spaces with sheaves of holomorphic functions given by specified local models) by the construction of Proposition~\ref{prop:cplxRG}. 
\end{constr}

\begin{theorem}
Suppose $\xi, \tilde{\xi}$ are such that $\xi^{\mb{C}} = \tilde{\xi}^{\mb{C}}$ and $\xi^{\mb{R}}$ weakly has the same sign as  $\tilde{\xi}^{\mb{R}}$. Then there exists a canonical holomorphic map $r\colon\mc{N}_{\xi} \to \mc{N}_{\tilde{\xi}}$. This map is a partial resolution of singularities, i.e., a proper birational morphism; moreover, on the dense open locus where $r$ is an isomorphism, it is a holomorphic symplectomorphism. In particular, if $\tilde{\xi}$ is already generic so that $\mc{N}_{\tilde{\xi}}$ is smooth, the map $\N_{\xi} \to \N_{\tilde\xi}$ is a holomorphic symplectomorphism.\label{thm:relXi}
\end{theorem}

\bp The main content of the above statement is the following:

\begin{lem}\label{lem:genpoly} If $\xi^{\mb{R}}$ weakly has the same sign as $\tilde{\xi}^{\mb{R}}$, then any $\xi^{\mb{R}}$-semistable complex pre-Nahm data is $\tilde{\xi}^{\mb{R}}$-semistable.\end{lem}
\bp Suppose $B = (\alpha, \beta)$ is $\tilde{\xi}^{\mb{R}}$-unstable, i.e., $E$ admits a good subbundle $V \subset E$ such that $\deg_{\tilde{\xi}^{\mb{R}}}V > 0$. It is then immediate from the contemplation of $V|_{\partial I} \in \{\mb{C}e_1, \mb{C}e_2\}$ and the definition of ``weakly has the same sign'' that $\deg_{\xi^{\mb{R}}}V > 0$. \ep

We have the following converse of sorts to the above lemma:

\begin{lem}\label{lem:convhol} If $\xi^{\mb{R}}$ weakly has the same sign as $\tilde{\xi}^{\mb{R}}$, any $\tilde{\xi}^{\mb{R}}$-stable complex pre-Nahm data is $\xi^{\mb{R}}$-stable.
\end{lem} 
\bp The proof is similar: a good subbundle $V \subset E$ with $\deg_{\xi^{\mb{R}}}V \ge 0$ would have $\deg_{\tilde{\xi}^{\mb{R}}}V \ge 0$. 
\ep

Lemma~\ref{lem:genpoly} yields a morphism $\B^{s,\mathrm{ps}}_{\xi^{\mb{R}},\mathrm{all}} \to \B^{s,\mathrm{ss}}_{\tilde{\xi}^{\mb{R}},\mathrm{all}}$ over $\mf{z}^{\mb{C}}_{\partial I}$. Restricting to the subloci of complex Nahm data (i.e., solving the complex Nahm equation) and applying the ``collapse and quotient'' map of Construction~\ref{constr:sshol} yields a canonical map $\N_{\xi^{\mb{R}},\mathrm{all}} \to \N_{\tilde{\xi}^{\mb{R}},\mathrm{all}}$ over $\mf{z}^{\mb{C}}_{\partial I}$. Lemma~\ref{lem:convhol} shows that over the stable locus of $\N_{\tilde{\xi}^{\mb{R}},\mathrm{all}}$, this map is in fact a biholomorphic isomorphism; indeed, the holomorphic symplectic form pulls back compatibly over this stable locus as well. Hence, to show Theorem~\ref{thm:relXi}, it remains only to show that, for each $\xi^{\mb{C}}$, each map $r\colon \N_{\xi} \to \N_{\tilde{\xi}}$ so obtained is indeed a proper, birational map, i.e., an isomorphism over an open, dense locus such that preimages of compact sets are compact; by the above, birationality is immediate once we observe that $\mc{N}^{\mathrm{st}}_{\tilde{\xi}} \subset \mc{N}_{\tilde{\xi}}$ is dense, as follows.

Recall from the explicit description of all Nahm orbits in complex axial gauge via Proposition~\ref{prop:compNahm} that all orbits of type (i) --  which are always stable for any parameters by Proposition~\ref{prop:mostlystable} -- provide an open dense sublocus of $\mc{D}^{\mathrm{ps}}_{\xi}$ for any $\xi$. (One may also conclude that $\N^{\rm st}_{\tilde \xi}$ is dense in $\N_{\tilde \xi}$ via Proposition \ref{prop:compSings2}, although note that one still needs to rule out the possibility of an isolated strictly polystable point in the second listed case; this follows from Proposition \ref{prop:cplxRG}.) Hence, $r$ is birational.

It now remains only to show properness, which we separately establish as Proposition~\ref{prop:proper} below. \ep

\begin{rmk} The argument above took place in standard frame. It is useful to recall that in unitary frame -- which will be more relevant when there is no privileged choice of complex structure -- there is a nontrivial map involved in passing from complex pre-Nahm data for $\xi$ to that for $\tilde{\xi}$ in the first place. Namely, it is given by the action of a formal complexified gauge transformation $g\in C^\infty(I,SL(2,\CC))$ with $g|_{\partial I}\in T_\CC$ and $g^{-1} \partial g|_{\partial I} \in (\xi^\RR_{\partial I} - \tilde\xi^\RR_{\partial I}) \sigma_z + \mf{t}_\CC^\perp$. This reflects the discussion after the general statement of Conjecture~\ref{conj:mainconj}. \end{rmk}

\begin{corollary}
The moduli spaces $\N_\xi$ with $\xi^{\mb{C}}_{\partial I} = 0$ and $\xi^\RR_0\not=\pm \xi^\RR_L$ are minimal resolutions of $(\mb{C} \times \mb{C}^{\times})/Z_2$.
\end{corollary}

As promised, we now turn to properness. It may be of interest here to compare to the analogous argument used by Kronheimer in the ALE case~\cite[Lemma 3.11]{kronheimer:construct}, which ultimately relied on some boundedness of $\mu_{\mb{R}}$ -- as is reasonable, as the presentation there is given in the real formulation, using the terminology of Construction~\ref{constr:realKah}. We could similarly establish the requisite properness for Theorem~\ref{thm:relXi} by postponing to after our discussion of the DUY isomorphism in the next section and then using a bound on $\mu_{\mb{R}}$. It is pleasant, however, to show properness entirely within the framework of holomorphic geometry, and so we do so here. Our strategy will be to compare the local geometry of a strictly polystable point in the moduli space of complex Nahm data to the analogous local geometry in the finite-dimensional Kronheimer construction, whence properness follows from the properness of~\cite[Proposition 3.10]{kronheimer:construct}. As we have just noted, Kronheimer's proof uses a real moment map, but as we indicate in Construction \ref{constr:kroncplx}, there are purely holomorphic proofs of this properness as well. We will focus on the resolution of the singularity in moduli spaces with $\vec\xi_0 = \vec\xi_L$ corresponding to the strictly polystable Nahm data $B=(0, -i\xi_0^\CC \sigma_z)$, but as usual these results may be immediately translated to the other singularity in moduli spaces with $\vec\xi_0 = - \vec\xi_L$ by the extended gauge transformation $\exp(i \pi t \sigma_x / 2L)$.

We first summarize the data of Kronheimer's finite-dimensional construction in terms most manifestly similar to those we use here:

\begin{constr}\label{constr:kroncplx} We summarize the original construction of Kronheimer's resolution of the $A_1$ singularity, working entirely here in the complex formulation. As such, we are specializing~\cite{kronheimer:construct} after passing through the usual finite-dimensional DUY isomorphism comparing the real and complex (or GIT) formulations. As a result in algebraic geometry, however, this result is also an outgrowth of work of Brieskorn, Grothendieck, Slodowy, and others on the semi-universal resolution and deformation of Kleinian surface singularities~\cite{Brieskorn, Slodowy}; see~\cite{CassensSlodowy} for a unified treatment of these ideas. We specialize all of the above to the particularly simple case of interest here of the $A_1$ orbifold singularity; we also reformulate the data slightly to make the comparison to the Nahm data of interest in this paper more manifest.

We use the same notation as before for $T_{\mb{C}} \subset SL(2,\mb{C})$ the diagonal maximal torus, etc. We now introduce the following definitions: 
\begin{align}
\mc{B}^{K} &:= \{(\alpha, \beta) \in (\mf{t}^{\perp}_{\mb{C}})^{\oplus 2}\} \nonumber \\ 
\mc{D}^{K}_{\xi^{K,\mb{C}}} &:= \{(\alpha, \beta) \in (\mf{t}^{\perp}_{\mb{C}})^{\oplus 2} \bigm| [\alpha, \beta] = i \xi^{K,\mb{C}} \sigma_z \} \nonumber \\ 
\tilde\G^K_\CC &:= T_\CC \nonumber \\
\mc{G}^{K}_{\mb{C}} &:= \tilde\G^K_\CC / Z_2 \,, 
\end{align} 
where $\mc{G}^{K}_{\mb{C}}$ acts on $\mc{B}^{K}, \mc{D}^{K}_{\xi^{K,\mb{C}}}$ by simultaneous conjugation. Note that here $\xi^{K,\mb{C}}$ is a single complex scalar.

Denote by $E := \mb{C}^2$ now a two-dimensional complex vector space on which we think of $\alpha, \beta$ as acting. Term elements of $\mc{B}^{K}$ as \emph{pre-Kronheimer data} and elements of $\mc{D}^{K}_{\xi^{K,\mb{C}}}$ as \emph{Kronheimer data}. Define (pre-)Kronheimer \emph{subdata} to be a vector subspace $V \subset E$ preserved by $\alpha$ and $\beta$ and such that some subset of the standard basis $\{e_1, e_2\}$ of $E$ may be taken as a basis for $V$. Given $\xi^{K,\mb{R}} \in \mb{R}$, define $\deg_{\xi^{K,\mb{R}}}V$ in the usual way; in particular, in the interesting case that $V$ is one-dimensional, its degree is $+\xi^{K,\mb{R}}$ if it is $\mb{C}e_1$ and $-\xi^{K,\mb{R}}$ if it is $\mb{C}e_2$. Define $\xi^{K,\mb{R}}$-(semi/poly)stability in the usual way as well; e.g., $E$ is $\xi^{K,\mb{R}}$-polystable if either (i) it is stable, i.e., any Kronheimer subline $V \subset E$ has negative $\xi^{K,\mb{R}}$-degree, or (ii) it is strictly polystable, i.e., it is a direct sum of good sublines $\mb{C}e_1 \oplus \mb{C}e_2$, both with $\xi^{K,\mb{R}}$-degree zero. 

Denote $\N^{K}_{\xi^K} := \mc{D}^{K,\mathrm{ps}}_{\xi^K}/\mc{G}^{K}_{\mb{C}}$. Then the above references show  (and, in this very specialized case, it is an exercise to show directly) that if $\xi^{K,\mb{R}} = \xi^{K,\mb{C}} = 0$, we have $\N^{K}_0 \simeq \mb{C}^2/Z_2$, and that if $\xi^K = (\xi^{K,\mb{R}},\xi^{K,\mb{C}})$ has $\xi^{K,\mb{R}} \ne 0$ but $\xi^{K,\mb{C}} = 0$, the canonical map $\N^{K}_{\xi^K} \to \N^{K}_0$ is a minimal resolution of the $A_1$ singularity (i.e., the blow-up of $\mb{C}^2/Z_2$ at the origin).

Note that there is an (even simpler) analog of axial gauge: up to the $\mc{G}^{K}_{\mb{C}}$-action, $\alpha$ may be represented as either $\alpha_0 \sigma_x$ for $\alpha_0 \in \mb{C} / Z_2$, $\begin{pmatrix} 0 & 1/L \\ 0 & 0 \end{pmatrix}$, or $\begin{pmatrix} 0 & 0 \\ 1/L & 0 \end{pmatrix}$, with residual gauge symmetry (by the full $\G^K_{\mb{C}}$) only in the first case with $\alpha_0 = 0$.
\end{constr} %\fatodo{For the sake of self-containment, we quickly sketch the full classification of Kronheimer data in axial gauge and the ensuing minimal resolution of the $A_1$ singularity in an appendix subsection.}

We now introduce some definitions with a view towards mapping Nahm data to Kronheimer data. If there is danger of confusion, we will use $(\alpha^{N},\beta^{N})$ to refer to (pre-)Nahm data versus $(\alpha^{K},\beta^{K})$ for (pre-)Kronheimer data.

\begin{defn}\label{defn:RGC}
Denote by $\mathrm{proj}\colon \mf{sl}(2,\mb{C}) \to \mf{t}^{\perp}_{\mb{C}}$ the projection associated to the direct sum decomposition $\mf{sl}(2,\mb{C}) \simeq \mf{t}_{\mb{C}} \oplus \mf{t}^{\perp}_{\mb{C}}$. Next, suppose that $(\alpha^{N},\beta^{N})$ is pre-Nahm data such that $\alpha^{N}$ is \emph{constant}, i.e., so that $\alpha^{N}$ as a map $I \to \mf{sl}(2,\mb{C})$ is in fact identically some constant value $\alpha \in \mf{t}^{\perp}_{\mb{C}}$. Define associated pre-Kronheimer data $RG_{\mb{C}}(\alpha^{N},\beta^{N}) = (\alpha^{K},\beta^{K})$ as follows:
\begin{align}\label{eq:rgC}
\alpha^{K} &:= \alpha^N \nonumber \\ 
\beta^{K} &:= \mathrm{proj}\Big(\int_I \beta^{N}(t)\,dt\Big)\,. \end{align} 
\end{defn}

The following is an observation:

\begin{lem} Given $$RG_{\mb{C}}\colon\{\text{pre-Nahm data with constant $\alpha^{N}$}\} \to \{\text{pre-Kronheimer data}\}$$ defined as above, $RG_{\mb{C}}$ is naturally $\G^K_{\mb{C}}$-equivariant, where $\G^K_{\mb{C}}$ acts via constant gauge transformations on the left. \end{lem}

\begin{prop}\label{prop:RGitems}
Associate to Nahm FI parameters $\xi$ the Kronheimer FI parameters $\xi^{K,\CC} = \xi_L^\CC-\xi_0^\CC$ and $\xi^{K,\RR} = \xi_L^\RR - \xi_0^\RR$. Furthermore, assume that the complex Nahm data $B=(\alpha^N,\beta^N)$ has constant $\alpha^N$. Then, the $RG_{\mb{C}}$ map satisfies the following properties: 
\begin{enumerate}[(i)] 
\item If $B$ is in $\D_{\xi^\CC}$, then $RG_{\mb{C}}(B)$ is in $\D^K_{\xi^{K,\CC}}$.
\item The union of the $\G^K_\CC$-orbits of the Nahm data in $\D_{\xi^\CC}$ in items (i), (ii), (iii), and (vi) of the classification of Proposition \ref{prop:compNahm}, where the restriction on the data in (i) is modified to $c\ge 0$ and $a>0$ if $c=0$, corresponds bijectively, via $RG_\CC$, to $\D^K_{\xi^{K,\CC}}$. Let $\tilde \D_{\xi^\CC} \subset \D_{\xi^\CC}$ denote this union if we instead make the \emph{stronger} restriction on the data in (i) that $c < \pi/2L$, and let $\tilde \D^K_{\xi^{K,\CC}} \subset \D^K_{\xi^{K,\CC}}$ denote the image $RG_\CC(\tilde\D_{\xi^{\CC}})$.
\item For $B\in \tilde\D_{\xi^\CC}$, Nahm sublines of $B$ are in 1-1 correspondence with Kronheimer sublines of $RG_\CC(B)$; in particular, Kronheimer sublines only exist if $\xi^{K,\CC}=0$. This correspondence is simply that the Nahm subline is the product of the Kronheimer subline with $I$. The $\xi^\RR$-degree of a Nahm subline therefore coincides with the $\xi^{K,\RR}$-degree of its corresponding Kronheimer subline. It follows that $B$ is $\xi^{\mb{R}}$-(semi/poly)stable if and only if $RG_{\mb{C}}(B)$ is $\xi^{K,\mb{R}}$-(semi/poly)stable.
\item Define $\tilde \D^{K,{\rm ps}}_{\xi^K} = \D^{K,{\rm ps}}_{\xi^K} \cap \tilde\D^K_{\xi^{K,\CC}}$ and $\tilde\N^K_{\xi^K} = \tilde\D^{K,{\rm ps}}_{\xi^K} / \G^K_\CC$. Then, $\tilde \N^K_{\xi^K}$ embeds holomorphically symplectically as an open complex analytic subspace of $\N_{\xi}$ via the correspondence defined by (ii) and (iii). That is, the image of this embedding is $\tilde\N_{\xi} = \tilde\D^{\rm ps}_{\xi}/\G_\CC^K$, where $\tilde\D^{\rm ps}_\xi = \D^{\rm ps}_\xi \cap \tilde\D_{\xi^\CC}$.
\end{enumerate} 
\end{prop}
\bp 
Writing $(\alpha^K,\beta^K) = RG_\CC(B)$, and denoting by $(\cdot)^\parallel$ and $(\cdot)^\perp$, respectively, the projections to $\mf{t}_\CC$ and $\mf{t}_\CC^\perp$, we have
\be [\alpha^K, \beta^K] = \int_I [\alpha^N, (\beta^N)^\perp] \, dt = \int_I [\alpha^N, \beta^N]^\parallel \, dt = - \int_I \partial (\beta^N)^\parallel \, dt = i \sigma_z(\xi_L^\CC - \xi_0^\CC) \ . \ee
This proves (i).

The proofs of (ii)--(iv) are immediate, using Proposition \ref{prop:compNahm} and \eqref{eq:ellpreserved}. In particular, the Kronheimer data associated to the relevant Nahm data from this proposition are
\begin{enumerate}[(i)]
\item $\alpha^K = \alpha_0 \sigma_x$, $\beta^K = \frac{\xi^{K,\CC}}{2\alpha_0} \sigma_y + \beta_x L \sigma_x$
\item $\alpha^K = \twoMatrix{0}{1/L}{0}{0}$, $\beta^K = \twoMatrix{0}{-ic+\frac{iL}{3}(\xi_0^\CC - 4 \xi_L^\CC)}{iL\xi^{K,\CC}}{0}$
\item $\alpha^K = \twoMatrix{0}{0}{1/L}{0}$, $\beta^K = \twoMatrix{0}{-iL\xi^{K,\CC}}{-ic + \frac{iL}{3}(4\xi_L^\CC - \xi_0^\CC)}{0}$
\setcounter{enumi}{5}
\item If $\xi^{K,\CC}=0$, $\alpha^K=0$, $\beta^K = 0$, $\beta^K = \twoMatrix{0}{1}{0}{0}$, $\beta^K = \twoMatrix{0}{cL}{1}{0}$
\end{enumerate}
and the analogous classification of Kronheimer data in axial gauge is
\begin{enumerate}[(i)]
\item $\alpha^K = \alpha_0 \sigma_x$, $\beta^K = \beta_x \sigma_x + \frac{\xi^{K,\CC}}{2\alpha_0} \sigma_y$, with $\alpha_0 = a+ic$ satisfying $c\ge 0$ and $a>0$ if $c=0$, and with $\beta_x\in \CC$;
\item $\alpha^K = \twoMatrix{0}{1/L}{0}{0}$, $\beta^K = \twoMatrix{0}{c}{iL\xi^{K,\CC}}{0}$, $c\in \CC$;
\item $\alpha^K = \twoMatrix{0}{0}{1/L}{0}$, $\beta^K = \twoMatrix{0}{-iL\xi^{K,\CC}}{c}{0}$, $c\in \CC$;
\item if $\xi^{K,\CC}=0$, $\alpha^K = 0$, $\beta^K = 0$, $\beta^K = \twoMatrix{0}{1}{0}{0}$, $\beta^K = \twoMatrix{0}{cL}{1}{0}$, $c\in \CC$ \ .
\end{enumerate}
In this notation, the difference between $\N^K_{\xi^K}$ and $\tilde\N^K_{\xi^K}$ is that on the type (i) data in the latter we impose the additional restriction $c<\pi/2L$.

That the holomorphic symplectic forms on $\tilde \N^K_{\xi^K}$ and $\tilde\N_\xi$ agree follows immediately from the fact that $\alpha^N$ is constant and off-diagonal:
\be
\int_I \Tr \parens{\delta\alpha^N \delta\beta'^N - \delta\alpha'^N \delta\beta^N}\,dt = \Tr \Big(\delta\alpha^N \int_I \delta\beta'^N - \delta\alpha'^N \int_I \delta\beta^N\Big)\, dt = \Tr (\delta\alpha^K \delta\beta'^K - \delta\alpha'^K \delta\beta^K) \ . \ee
The only subtlety worth commenting on in this argument is that one must employ gauge transformations in order to smoothly patch together the various subsets of $\tilde\N^K_{\xi^K}$ and $\tilde\N_\xi$ described in the above lists. But, since gauge transformations in $\G^K_\CC$ suffice, this causes no difficulties. For example, to describe the approach of Kronheimer data of type (i) in the above list to data of type (ii), we act on the former with a gauge transformation of the form $\twoMatrix{(L\alpha_0)^{1/2}}{0}{0}{(L\alpha_0)^{-1/2}}$ (where the choice of square root is irrelevant) and make the change of variables $\beta_x = L\alpha_0 c + \frac{i\xi^{K,\CC}}{2\alpha_0}$ in order to obtain the Kronheimer data $\alpha^K = \twoMatrix{0}{1/L}{L\alpha_0^2}{0}$, $\beta^K = \twoMatrix{0}{c}{L^2\alpha_0^2 c + i \xi^{K,\CC} L}{0}$, which tends to the Kronheimer data of type (ii) described above as $\alpha_0\to 0$. Acting with the same gauge transformation on Nahm data of type (ii) and making the change of variables $\beta_x = -i\alpha_0 c + \frac{i \alpha_0 L}{3} (\xi^\CC_0 - 4 \xi^\CC_L) + \frac{i \xi^{K,\CC}}{2\alpha_0 L}$ yields the Nahm data of type (ii) from Proposition \ref{prop:compNahm}, plus $O(\alpha_0)$ corrections. That the same gauge transformation works in both the Nahm and Kronheimer settings -- that is, that applying this gauge transformation to Nahm/Kronheimer data of type (i) and taking the limit as $\alpha_0\to 0$ yields Nahm/Kronheimer data of type (ii) -- is no accident: the gauge transformation maps Kronheimer (resp., Nahm) data to Kronheimer (Nahm) data such that $\alpha^K$ ($\alpha^N$) tends to the appropriate limit as $\alpha_0\to 0$, and we have classified precisely the complex Kronheimer (Nahm) data with this limiting value of $\alpha^K$ ($\alpha^N$). A different argument is required when one is limiting to case (vi) Nahm data or case (iv) Kronheimer data, since $\alpha=0$ is stabilized by all of $\CC^\times$ and so Kronheimer or Nahm data can limit to $\alpha=0$ and still not fall in our classification. However, when $\alpha^N=0$, complex Nahm data is constant, and so its off-diagonal part is complex Kronheimer data, and this immediately implies that this case too poses no difficulties.
\ep

\begin{rmk}\label{rmk:cplxRG} 
In broad generality, one expects ALF manifolds to degenerate to ALE manifolds as one takes $L\to 0$ with $\xi$ fixed. This motivates the definition of $RG_\CC$: we think of $2\,\mathrm{proj}\Big(\int_I(-)\,dt\Big)$ as $\int_{S^1}(\text{the $Z_2$-equivariant \emph{doubled} Nahm data})\,dt$, so that $\beta^{K}$ only keeps the constant Fourier mode of $\beta^N$. The details of this degeneration depend on how one takes this limit; for example, for the $D_2$ ALF manifold, the two 2-spheres separate and one must `zoom in' on one of them as this limit is taken in order to obtain an $A_1$ ALE manifold. This process is evident in the presence of more and more of $\D^K_{\xi^{K,\CC}}$ inside of $\tilde\D^K_{\xi^{K,\CC}}$ as $L$ decreases toward 0. The fact that $\tilde\N^K_{\xi^K}$ embeds holomorphically, as opposed to, e.g., smoothly or topologically, into $\N_{\xi}$ may be intuited from the fact that this limiting procedure must be compatible with the fact, explained in Remark \ref{rmk:cTimes}, that the complex structure of $\N_\xi$ is independent of $L$. In particular, a toy model for this degeneration is the family of maps $G_{\mb{C}} \stackrel{L^{-1}\log(-)}{\longrightarrow} \mf{g}_{\mb{C}}$ for $G_{\mb{C}}$ a complex reductive group, or more specifically still, $\mb{C}^{\times} \stackrel{L^{-1} \log(-)}{\longrightarrow} \mb{C}$. Pondering how to make sense of this family of maps as $L \to 0$ leads one to statements as above.
\end{rmk}

\begin{rmk}\label{rmk:RGCmap} 

%\gtodo{you should throw out type (vii) orbits, i.e. remove a bunch of points in $\mf{t}_{\mb{C}}^{\perp}$ -- and the preimages thereof in $\N^{\mathrm{const}}$ -- given by the extended gauge transforms of $\alpha \equiv 0$ (keeping in mind we have our friendly $Z_2 \ltimes \mb{Z}$ large gauge transform action) in order to actually get a nice smooth map of manifolds (after imposing polystability and whatnot). Very cool. }

Here is a slightly different gloss on the construction above. Recall the map $\mf{t}_{\mb{C}}^{\perp}/T_{\mb{C}} \to T_{\mb{C}}\backslash SL(2,\mb{C})/T_{\mb{C}}$ relevant for the discussion of complex axial gauge in Proposition~\ref{prop:compAx}, and denote by $(\mf{t}_{\mb{C}}^{\perp}/T_{\mb{C}})^{\circ} \subset \mf{t}_{\mb{C}}^{\perp}/T_{\mb{C}}$ the open sublocus consisting of the complement of the equivalence classes of  $\braces{\frac{n \pi i}{2L} \sigma_x \bigm| n\text{ odd} }$; i.e., if we define $H\colon \mf{t}_{\mb{C}}^{\perp}/T_{\mb{C}} \to \mb{C}$ by
$\begin{pmatrix} 0 & a \\ b & 0 \end{pmatrix} \mapsto ab$ as in~\eqref{eq:coolsq}, we are imposing that $H$ not take the values $-(n\pi/2L)^2$. Define $\N_{\xi}^{\mathrm{const}}$ by the pullback diagram 
\begin{tikzcd} 
\N_{\xi}^{\mathrm{const}}\ar[r]\ar[d] & \N_{\xi} \ar[d] \\ (\mf{t}_{\mb{C}}^{\perp}/T_{\mb{C}})^{\circ} \ar[r] & T_{\mb{C}}\backslash SL(2,\mb{C})/T_{\mb{C}}
\end{tikzcd} 
so that explicitly,
\begin{align} \N_{\xi}^{\mathrm{const}} &:= \{(\alpha,\beta) \in \mf{t}^{\perp}_{\mb{C}} \times C^{\infty}(I,\mf{sl}(2,\mb{C})) \bigm| \beta|_{\partial I} \in -i\xi^{\mb{C}}_{\partial I}\sigma_z + \mf{t}^{\perp}_{\mb{C}},(\alpha,\beta)\text{ are $\xi^{\mb{R}}$-polystable},\nonumber \\ 
&\qquad\qquad H(\alpha)\ne -(n\pi/2L)^2\text{ for $n$ odd}, \partial_{\alpha}\beta = 0\}/\mc{G}^K_{\mb{C}} \,.
\end{align}
Then one may define a map 
\be \N_{\xi}^{\mathrm{const}} \stackrel{RG_{\mb{C}}}{\to} \N^K_{\xi^K} \ee
as above by simply integrating $\beta$ (and projecting to $\mf{t}^{\perp}_{\mb{C}}$); this map is a holomorphic symplectic map, as computed above. Restricting to open subloci on both sides then gives the holomorphic symplectic isomorphism $\tilde{\N}_{\xi} \stackrel{\sim}{\to} \tilde{\N}^K_{\xi^K}$. \end{rmk} 

\begin{prop}\label{prop:compRG} Suppose the Nahm FI parameters are as in Theorem~\ref{thm:relXi} and furthermore satisfy $\tilde\xi^\CC_0 = \tilde\xi^\CC_L$ and $\tilde\xi^\RR_0 = \tilde\xi^\RR_L$, so that $\N_{\tilde{\xi}}$ contains the strictly polystable $\G_\CC$-orbit of $(\alpha, \beta) = (0, -i \xi_0^{\mb{C}}\sigma_z)$. Let $\xi^{K}$ denote Kronheimer FI parameters with $\xi^{K,\mb{C}} = 0$ and $\xi^{K,\mb{R}} = \xi_L^{\mb{R}} - \xi_0^{\mb{R}}$ as above.

Then the map $RG_{\mb{C}}$ yields a commuting square as follows: \beq \begin{tikzcd} \tilde\N_{\xi}\ar[r]\ar[d] & \tilde\N^{K}_{\xi^{K}} \ar[d] \\ \tilde\N_{\tilde{\xi}}\ar[r] & \tilde\N^{K}_0 \end{tikzcd} \ . \eeq
The horizontal arrows are biholomorphic isomorphisms of complex analytic spaces, with the lower such isomorphism taking a neighborhood of $(0,-i\xi_0^\CC\sigma_z)$ to a neighborhood of 0.
\end{prop}

\bp 
This statement is essentially a translation of Proposition~\ref{prop:RGitems}.
\ep

We now complete the proof of Theorem~\ref{thm:relXi}:

\begin{prop}\label{prop:proper} 
The canonical holomorphic map of complex Nahm moduli spaces $\N_{\xi} \to \N_{\tilde{\xi}}$ is proper. \end{prop} 
\bp 
As always, this is immediate on the $\tilde{\xi}$-stable locus, where the map is an isomorphism. In a neighborhood of a strictly $\tilde{\xi}$-polystable point (or the preimage thereof), this claim now follows from Proposition~\ref{prop:compRG}, or the analogous result for strictly polystable Nahm data related to $(0,-i\xi_0^\CC\sigma_z)$ by the extended gauge transformation $\exp(i \pi t \sigma_x / 2L)$, and the analogous statement in the finite-dimensional Kronheimer case, as reviewed in Construction~\ref{constr:kroncplx}.
\ep

\begin{remark} \label{rmk:periods}
We observe that this local holomorphic symplectomorphism property of $RG_\CC$ immediately allows one to conclude that the periods of the holomorphic symplectic form on the two 2-spheres in $\N_\xi$ are given by $\xi^\CC_0\pm \xi^\CC_L$ (up to a possible overall normalization factor) thanks to the analogous result of Kronheimer \cite{kronheimer:construct}. Indeed, it is immediate from the classification of Kronheimer data above that $\N^K_{\xi^K}$ deformation retracts to $\tilde{\N}^K_{\xi^K}$, and so the nontrivial $S^2$ in $\mathrm{H}_2(\mc{N}^K_{\xi^K};\mb{Z})$ already exists within $\mathrm{H}_2(\tilde{\mc{N}}^K_{\xi^K};\mb{Z})$. So, the class of the holomorphic symplectic form in $\mathrm{H}^2(\tilde\N^K_{\xi^K};\CC)$ -- i.e., the period of $\omega^\CC$ on the homology class of the nontrivial $S^2$ -- may be computed by embedding $\tilde\N^K_{\xi^K}$ into $\N^K_{\xi^K}$. Our $RG_{\mb{C}}$ argument above then applies. Moreover, the DUY-type theorem we are about to prove in the next section will allow one to reach the same conclusion in all complex structures on $\M_\xi$. Hence, the periods of the K\"ahler forms $\omega^i$ on the 2-spheres in $\M_\xi$ are $\xi^i_0 \pm \xi^i_L$.
\end{remark}

\section{Equivalence of the moduli spaces} \label{sec:duy}

In this section, we relate the constructions of the previous two sections. Throughout this section, it will be convenient to take $\mu_\RR$ to be a map to $i\,\mf{su}(2)$, as opposed to $\mf{su}(2)$, and so we employ the following definition:
\be \mu_\RR(\alpha,\beta) = \partial(\alpha+\alpha^\dagger)+[\alpha,\alpha^\dagger]+[\beta,\beta^\dagger] = -[\partial_\alpha,\partial_{-\alpha^\dagger}] + [\beta,\beta^\dagger] \ . \label{eq:muRDef} \ee

We begin with

\begin{proposition}
Solutions to the real Nahm equation in $\A^s_\xi$ are $\xi^\RR$-polystable, and they are irreducible if and only if they are $\xi^\RR$-stable. This also holds for $s=\infty$. \label{prop:realStable}
\end{proposition}
\begin{proof}
These statements follow from Lemma \ref{prop:degIn}. In particular, strict $\xi^\RR$-semistability implies that there is a good projection $\pi$ to a proper good subbundle with $\partial_\alpha \pi = [\beta,\pi] = 0$. This implies that $1-\pi$ is another good projection, so the pre-Nahm data is strictly $\xi^\RR$-polystable. Furthermore, because the traceless part $\hat\pi$ of $\pi$ is Hermitian, $i\hat\pi$ is in the kernel of $d_0$ and generates a $U(1)$ group of gauge transformations that stabilize $A$. Strict $\xi^\RR$-semistability thus implies both strict $\xi^\RR$-polystability and reducibility. Conversely, that reducibility and $\xi^\RR$-semistability together imply strict $\xi^\RR$-polystability follows from Proposition \ref{prop:realStabs}.
\end{proof}

This proposition implies that there are maps $j\colon \{A\in \A_\xi \ | \ \mu_\RR(A) = 0\}/\G \to \B^{\mathrm{ps}}_\xi/\G_\CC$ and $j\colon \{A\in \A^s_\xi \ | \ \mu_\RR(A) = 0\}/\G^{s+1} \to \B^{s,\mathrm{ps}}_\xi/\G^{s+1}_\CC$, which restrict to maps we also denote $j\colon \M_\xi \to \N_\xi$ and $j\colon \M^s_\xi \to \N^s_\xi$. We will state the following theorems for fixed values of $\xi$, although as in prior sections, there is no loss in allowing $\xi^{\mb{C}}$ to vary to obtain parametrized versions of these statements.

\begin{theorem}
These maps are homeomorphisms, and the restrictions $\M^{\mathrm{irr}}_\xi \to \N^{\mathrm{st}}_\xi$ and $\M^{s,\mathrm{irr}}_\xi \to \N^{s,\mathrm{st}}_\xi$ are holomorphic symplectomorphisms. Similarly, the map $\{A\in A^{s,\mathrm{irr}}_\xi \bigm| \mu_\RR(A) = 0\}/\G^{s+1} \to \B^{s,\mathrm{st}}_\xi/\G^{s+1}_\CC$ is a holomorphic symplectomorphism. \label{thm:duy}
\end{theorem}
\noindent Note in the last statement above that the result in the Sobolev setting is stronger, as we put a manifold structure on $\A^{s,\mathrm{irr}}_\xi/\G^{s+1}$ and $\B^{s,\mathrm{st}}_\xi/\G^{s+1}_\CC$, but not $\A^{\mathrm{irr}}_\xi/\G$ and $\B^{\mathrm{st}}_\xi/\G_\CC$.

Before proving this theorem, we state an important corollary thereof, following Kronheimer~\cite{kronheimer:construct}.
\begin{theorem}\label{thm:res}
For any generic $\xi$, $\M_\xi$ is diffeomorphic to the minimal resolution of $(\RR^3\times S^1)/Z_2$.
\end{theorem}

\begin{proof}
We first pick a complex structure $K$ on $\M_\xi$ such that $\tilde\xi$, defined by $\tilde\xi^\CC = 0$ and $\tilde\xi^\RR = \xi^\RR$, is generic. We then introduce two other complex structures, $I$ and $J$, which together satisfy the unit quaternion algebra.

In complex structure $I$, we employ the map $j$ to learn that $\M_\xi$ coincides with $\N_\xi$ as a complex manifold. Theorem \ref{thm:relXi} then shows that there is a minimal resolution of singularities $\M_\xi \to \M_{\xi^\RR=0,\xi^\CC}$. Next, we hyperkahler rotate, i.e. regard $\M_{\xi^\RR=0,\xi^\CC}$ as a complex manifold in the $J$ complex structure. Continuing as above for the $J$ complex structure yields a diffeomorphism with $\M_{\tilde\xi}$, which by Theorem \ref{thm:relXi} is the minimal resolution of $(\RR^3\times S^1)/Z_2$.
\end{proof}

\begin{remark} \label{rmk:partialRes}
For $\xi$ nonzero, but non-generic, this same proof shows that $\M_\xi$ is diffeomorphic in an orbifold sense to a partial resolution of $(\RR^3\times S^1)/Z_2$. Indeed, since there are only two collapsed 2-cycles in $(\RR^3\times S^1)/Z_2$, it is actually the case that if $\xi$ is nonzero but non-generic then we may assume that $\xi^\CC=0$ after a hyperkahler rotation. (For instance, if $\xi^i_0 - \xi^i_L = 0$ for all $L$, then we simply need to rotate the vector $\xi^i_0 + \xi^i_L$ to point in the $\xi^\RR$ direction.) Theorem \ref{thm:relXi} then implies that $\M_\xi$ is a (holomorphic symplectic) partial resolution of $(\CC\times \CC^\times)/Z_2$ in this complex structure, where one of the singularities is unresolved and the other is resolved (minimally).
\end{remark}

Recalling the metrics defined in Corollary \ref{cor:realMet}, we also have the following corollary, which further informs the construction of \S\ref{sec:complex}:
\begin{corollary} \label{cor:compMet}
$d^s(j^{-1}([A]),j^{-1}([B]))$ defines a metric on $\B^{s,\mathrm{ps}}_\xi/\G_\CC^{s+1}$, as does $d(j^{-1}([A]),j^{-1}([B]))$ on $\B^{\mathrm{ps}}_\xi/\G_\CC$.
\end{corollary}

We begin the proof of Theorem~\ref{thm:duy} by proving the easy direction, namely injectivity, following the proof of Proposition~6.1.10 in~\cite{DK}. We first state a lemma which follows from dimensional reduction of Weitzenb\"ock formulae on $T^4$, and which can also be easily verified by direct calculation. For the purposes of this lemma, we generalize our operators $d_0$ and $\bar\partial_0$ from earlier sections in two ways. First, we do not require $A=(\alpha,\beta)$ to satisfy any of Nahm's equations. And, second, we allow for the operators $\partial_\alpha,\beta,\partial_{A^0},A^i$ to be connections or endomorphisms acting on an arbitrary $SU(2)$-bundle over $I$; $\sigma_z$ acts on the fibers of such a bundle as a $U(2)$ transformation (which conjugates $SU(2)$ to itself). In addition to the usual operators $\Delta_0 = d_0^* d_0 = - \partial_{A^0}^2 - \sum_i (A^i)^2$ and $\bar\Delta_0 = \bar\partial_0^* \bar\partial_0 = - \partial_{-\alpha^\dagger} \partial_\alpha + \beta^\dagger \beta$, we also define $\tilde\Delta_0 = -\partial_\alpha \partial_{-\alpha^\dagger} + \beta \beta^\dagger$.

\begin{lemma} \label{lem:weitz}
For any pre-Nahm data $A=(\alpha,\beta)$, we have
\be \Delta_0 = \bar\Delta_0 + \half \mu_\RR \ee
and
\be \Delta_0 = \tilde\Delta_0 - \half \mu_\RR \ , \ee
where $\mu_\RR$ here is regarded as a Hermitian endomorphism. Now, let $s$ be a section which satisfies $\partial^n s(t_0) = (-1)^n \sigma_z\cdot \partial^n s(t_0)$ for $n=0,1$ and $t_0=0,L$. It follows from the above results that if $\mu_\RR(A)=0$ (or, more generally, if $\mu_\RR(A)$ is a negative semi-definite endomorphism) and $\bar\partial_0 s := \partial_\alpha s \oplus (\beta s) = 0$ then $d_0 s := \partial_{A^0} s \oplus (A^i s) = 0$. Similarly, if $\mu_\RR(A)=0$ (or, more generally, if $\mu_\RR(A)$ is a positive semi-definite endomorphism) and $\tilde\partial_0 s:= \partial_{-\alpha^\dagger} s \oplus (\beta^\dagger s) = 0$ then $d_0 s = 0$.
\end{lemma}

\begin{proposition} \label{prop:uniqueSol}
$j$ is injective.
\end{proposition}
\begin{proof}
We want to show that if two solutions $a_1,a_2\in \A^s_\xi$ of the real Nahm equation are related by an element $g\in \G^{s+1}_\CC$, then they are also related by a $g'\in \G^{s+1}$. Such a $g'$ is given by $g'=(g g^\dagger)^{-1/2} g$. To demonstrate this, we first reformulate the equation $a_2 = g(a_1)$ in the form $\partial_{\alpha_1*\alpha_2} g \oplus (\beta_1*\beta_2)g = 0$. Next, we observe that $(\alpha_1*\alpha_2,\beta_1*\beta_2)$ satisfies the real Nahm equation. The previous lemma now implies that $\partial_{a_1^0*a_2^0} g \oplus (a_1^i*a_2^i)g = 0$. That is, $a_2^0 = g^{-1} \partial g + g^{-1} a_1^0 g$ and $a_2^i = g^{-1} a_1^i g$, so $a_2$ and $a_1$ are related by $g$ as if the latter were a unitary transformation. This implies that $\partial_{a_1^0} (g g^\dagger) \oplus [a_1^i, g g^\dagger] = 0$, and in turn that $\partial_{a_1^0 * a_2^0} g' \oplus (a_1^i*a_2^i)g' = 0$.
\end{proof}

The rest of this section is devoted to the proof that $j$ is surjective. To do so, for any $B\in \B^{s,\mathrm{ps}}_\xi$, where $1\le s\le \infty$, we will find a Hermitian positive-definite $g\in \G^{s+1}_\CC$ such that $g(B)$ satisfies $\mu_\RR(g(B)) = 0$. Equivalently (thanks to the polar decomposition), we will find a $\G^{s+1}$-orbit $[g]$ in $\G^{s+1}_\CC$ (with $\G^{s+1}$ acting on the right) such that $\mu_\RR(g(B)) = 0$ for any element of this orbit. We introduce the notation
\be \G^{s+1}_+ := \exp(i \mf{g}^{s+1}) \ee
for the space of positive-definite elements of $\tilde\G^{s+1}_\CC$. We make this into a quasigroup (i.e., a non-associative group -- or, more precisely, a loop because we will have an identity) by observing that if $g=g_1 g_2$ is a composition of positive-definite gauge transformations then the positive-definite representative of its $\G^{s+1}$-orbit is $g_1\times g_2 := (g_1 g_2^2 g_1)^{1/2}$ (as is clear from the polar decomposition of $g_1 g_2$). Inverses and the identity are the same as in $\tilde\G^{s+1}_\CC$; in particular, left and right inverses coincide. Despite the general lack of associativity, we note that the identity $g_1^{-1}\times (g_1\times g_2) = g_2$ does hold. For positive definite $g$, we write $h = \log g$, where $\log g$ is the Hermitian logarithm of $g$; this gives a bijection between $\G_+^{s+1}$ and $i\mf{g}^{s+1}$.

We begin with the strictly $\xi^\RR$-polystable case, as this is particularly simple. We let $\pi,1-\pi$ denote good projections to proper subbundles. Then, we claim that there exists a $g\in \G^{s+1}_+$ which is diagonal in the basis associated to $\pi,1-\pi$ and which satisfies $\mu_\RR(g(B))=0$. As usual, we recall that a unitary gauge transformation in $\tilde\G^{s+1}$ that relates the usual unitary frame to this basis does not exist if $\xi_0^\RR = - \xi_L^\RR$, but that a unitary extended gauge transformation does exist; for both signs in $\xi^i_0 = \pm \xi^i_L$, we have $A^i_z(0) = A^i_z(L)$ after we act with this extended gauge transformation. Indeed, since $A$ is diagonal in this basis we have $A^i(0) = A^i(L)$. Since extended gauge transformations conjugate $\tilde\G^{s+1}$ to itself, there is no harm with working in this basis. In this basis, we write $g=\twoMatrix{g_1}{}{}{g_1^{-1}}$ and $A^1 = \twoMatrix{A^1_1}{}{}{-A^1_1}$. We also write $g_1 = e^{h_1}$, where $h_1 \in H^{s+1}(I,\RR)$ and its derivative vanishes at both endpoints. For this abelian problem, the real moment map equation may be written in the form $\Delta_0 h_1 = i \partial A^1_1$, where $\Delta_0 = -\partial^2$ is defined using the rank 1 Nahm data associated to the subbundle $\pi E$. So, we simply need to show that this equation admits a solution with the prescribed boundary conditions. But, $-\partial^2$ with the boundary conditions $\partial h_1|_{\partial I}=0$ is self-adjoint, so this will follow if we can show that $i\partial A^1_1$ is orthogonal (in the $L^2$ sense) to the kernel of $-\partial^2$ with these boundary conditions -- i.e., is orthogonal to constants. But, $\int \partial A^1_1\,dt = A^1_1|_{\partial I}=0$. So, a solution exists. Note that the solution is not unique, since $-\partial^2$ with these boundary conditions has a kernel consisting of constants, but that adding such a constant to $h_1$ simply multiplies $g$ by an element of the stabilizer of $B$, and so the resulting solution of the real Nahm equation is unchanged, in accordance with Proposition \ref{prop:uniqueSol}. With this, we have reduced to the stable case.

Before turning to this, we state a useful lemma, analogous to Simpson's Lemma 3.1 \cite{simpson:higgs}:
\begin{lemma} \label{lem:muR}
For $g\in \G_+^{2}$, we have $\partial_{g(\alpha)} = g^{-1} \circ \partial_\alpha \circ g$, $\partial_{-g(\alpha)^\dagger} = g\circ \partial_{-\alpha^\dagger}\circ g^{-1}$, and
\begin{align}
\mu_\RR(B) - g \mu_\RR(g(B)) g^{-1} &= -\partial_\alpha((\partial_{-\alpha^\dagger} H) H^{-1}) + [\beta, [\beta^\dagger, H] H^{-1}] \nonumber \\
&= (\tilde \Delta_0 H + (\partial_{-\alpha^\dagger} H) H^{-1} (\partial_\alpha H) - [\beta^\dagger,H] H^{-1} [\beta,H]) H^{-1} \ , \label{eq:niceMuR}
\end{align}
where $H=g^2$. We also have
\be -\partial^2\log \Tr H \le \sqrt{2} (|\mu_\RR(B)| + |\mu_\RR(g(B))|) \ , \ee
where $|\cdot|$ is the Frobenius norm.
\end{lemma}
\begin{proof}
The first results are easy to verify, and \eqref{eq:niceMuR} is a straightforward consequence of them if one employs the final expression in \eqref{eq:muRDef}. Multiplying \eqref{eq:niceMuR} on the right by $H$ and then taking the trace, and using the positive-definiteness of $H$ and the Cauchy-Schwarz inequality, gives
\be -\partial^2 \Tr H + |H^{-1/2} \partial_\alpha H|^2 = \Tr((\mu_\RR(B)-\mu_\RR(g(B))) H) \le \sqrt{2} \Tr(H) (|\mu_\RR(B)| + |\mu_\RR(g(B))|) \ . \ee
Explicitly, if $\mu_i$ are the diagonal values of $\mu_\RR(B)$ in a basis in which $H=\diag(H_i)$, then $\sum_i H_i |\mu_i| \le (\Tr H) \sum_i |\mu_i| \le \sqrt{2} (\Tr H) |\mu_\RR(B)|$, and a similar inequality holds for the other term. The Cauchy-Schwarz inequality also yields
\be (\partial \Tr H)^2 = (\Tr\partial_\alpha H)^2 = (\Tr (H^{1/2} H^{-1/2}(\partial_\alpha H)))^2 \le (\Tr H) |H^{-1/2} \partial_\alpha H|^2 \ . \ee
Putting these results together, we have
\be -(\Tr H) \partial^2 \log \Tr H = -\partial^2 \Tr H + \frac{(\partial \Tr H)^2}{\Tr H} \le \sqrt{2} \Tr(H) (|\mu_\RR(B)| + |\mu_\RR(g(B))|) \ . \ee
\end{proof}

Following Donaldson \cite{donaldson:duy,donaldson:duy2}, we will prove our DUY theorem by extremizing a functional of $h$.\footnote{We note that there is another functional due to Donaldson \cite{donaldson:nahm} which was useful in a similar context which specifically involved Nahm data. However, the Euler-Lagrange equation of this latter functional is not $\mu_\RR(g(B))=0$ in the present context, thanks to boundary terms due to integration by parts. At stationary points $g$, these boundary terms enforce $(\tilde h,\mu_\RR(g(B)))_{L^2}\not=0$ for most $\tilde h\in i\mf{g}^2$, and so stationary points do not satisfy $\mu_\RR(g(B))=0$. We therefore work with the Nahm data analog of the functional from \cite{donaldson:duy,donaldson:duy2} because, as we will see, such boundary terms are not present. Finally, a third candidate functional is the dimensional reduction of the Yang-Mills functional on $I\times T^3$, or equivalently $\|\mu_\RR(g(B))\|_{L^2}^2$ (since $\|\mu_\CC(g(B))\| = \|\mu_\CC(B)\|$ and the norm-squared of the anti-self-dual part of the curvature is determined by that of the self-dual part because the second Chern class is fixed). This would presumably suffice for most of our purposes, but the analysis would be a bit more unpleasant because of the extra derivatives present in this functional. We note that the condition for this functional to be stationary is $(\mu_\RR(g(B)), \Delta_0' \tilde h)_{L^2}=0$ for all $\tilde h\in i\mf{g}^2$, where the prime on $\Delta_0'$ indicates that the pre-Nahm data entering into its definition is $g(B)$, and generically integration by parts of this inner product in order to find the Euler-Lagrange equation will yield boundary terms in addition to $(\Delta_0' \mu_\RR(g(B)), \tilde h)_{L^2}$. Nevertheless, choosing $\tilde h$ to solve $\Delta_0' \tilde h = \mu_\RR(g(B))$, as in Prop. \ref{prop:minSolves}, makes it clear that a stationary point of this functional satisfies $\mu_\RR(g(B))=0$. This is compatible with the existence of the aforementioned boundary terms because for this functional, the latter vanish when $\mu_\RR(g(B))=0$.} We begin by describing two constructions involving Hermitian matrices. Let $h$ be a Hermitian matrix with an orthonormal basis of eigenvectors $w_i$ (thought of as column vectors) and corresponding eigenvalues $\lambda_i$, and let $A = \sum_{i,j} a_{ij} w_i w_j^\dagger$ be an arbitrary square matrix. Let $F\colon\RR\times \RR\to \RR$ be a smooth function. Then, define $F(h)(A) := \sum_{i,j} a_{ij} F(\lambda_j,\lambda_i) w_i w_j^\dagger$. Similarly, given a smooth function $f\colon \RR\to \RR$, define $f(h) := \sum_i f(\lambda_i) w_i w_i^\dagger$. We also define $df\colon \RR\times\RR\to \RR$ by
\be df(x,y) = \piecewise{\frac{f(x)-f(y)}{x-y}}{x\not=y}{f'(x)}{x=y} \ , \ee
so that $df$ can be used for the first construction. We will now explain how these constructions work when these matrices are upgraded to functions.

For any constant $c\ge 0$, define $\tilde P_c^{0} = \{h\in C^0(I, i\, \mf{su}(2)) \bigm| \|h\|_{C^0} \le c \}$ and $\tilde P_c^{1} = \{h\in H^1(I, i\, \mf{su}(2))\bigm|\|h\|_{C^0} \le c \}$. We then have the following lemma, which is similar to Simpson's Proposition 4.1 \cite{simpson:higgs}:
\begin{lemma} \label{lem:psiCont}
The construction $F$ defines a Lipschitz continuous and uniformly bounded map $F\colon \tilde P^{0}_c \to \End(L^2(\End(E)))$, where the codomain is endowed with the operator norm topology. By uniform boundedness, we mean that there exists a $C_1>0$ such that $\|F(h)\|_{L^2\to L^2} \le C_1$. We have $(F^2(h)(A),A)_{L^2} = \|F(h)(A)\|_{L^2}^2$ for all $A\in L^2(\End(E))$ and $h\in \tilde P_c^{0}$. Also, if $\im F\subset \RR^+$ then there is a $C_2>0$ such that $(F(h)(A), A)_{L^2} = \|F^{1/2}(h)(A)\|^2_{L^2} \ge C_2 \|A\|_{L^2}^2$.

Similarly, the construction $f$ defines a continuous map $f\colon \tilde P^1_c\to 
\tilde P^1_{c'}$ for some $c'\ge 0$. Also, we have $\partial_\alpha f(h) = df(h)(\partial_\alpha h)$ and $[\beta, f(h)] = df(h)([\beta,h])$.
\end{lemma}
\begin{proof}
The supremum bound on elements of the domain gives a pointwise (almost everywhere) bound $|F(h)(A)| < C_1 |A|$, so that $\|F(h)\|_{L^2\to L^2} < C_1$, for all $h\in \tilde P_c^{0}$. Next, we have
\be \Tr(F^2(h)(A) \, A^\dagger) = \sum_{i,j} F^2(\lambda_j,\lambda_i) |a_{ij}|^2 = \Tr(F(h)(A) \, (F(h)(A))^\dagger) \ , \ee
and so $(F^2(h)(A),A)_{L^2} = \|F(h)(A)\|^2_{L^2}$. If $F$ is a positive function then we may replace $F^2$ by $F$ and $F$ by the non-negative square root $F^{1/2}$ in this equation. The supremum bound on $h$ gives a uniform bound on the eigenvalues, and so $F(\lambda_j,\lambda_i)>C_2$ for all $i,j$, everywhere on $I$, for some $C_2>0$. We thus have $(F(h)(A), A)_{L^2} = \|F^{1/2}(h)(A)\|^2_{L^2} \ge C_2 \|A\|^2_{L^2}$.

Now consider $h_1,h_2\in \tilde P^{0}_c$. Because of the uniform bound, $F(\lambda_1,\lambda_2)$ is effectively a smooth function on a compact subset of $\RR^2$, and is therefore Lipschitz continuous on this compact space. We therefore have an (almost everywhere) pointwise bound $|F(h_1)(A)-F(h_2)(A)| \le C_3 |h_1-h_2| |A|$ for any $A\in L^2$, and so $\|F(h_1)(A)-F(h_2)(A)\|_{L^2} \le C_3 \|h_1-h_2\|_{C^0} \|A\|_{L^2}$. That is, $\|F(h_1)-F(h_2)\|_{L^2\to L^2} \le C_3 \|h_1-h_2\|_{C^0}$.

Suppose now that $h\in \tilde P^1_c$. We first prove that $\partial_\alpha f(h) = df(h)(\partial_\alpha h)$ when $f(\lambda) = \lambda^n$ is a monomial (which immediately implies that it holds for polynomials). This follows from comparing
\be df(x,y) = \sum_{m=0}^{n-1} y^{m} x^{n-m-1} \ee
and
\be \partial_\alpha h^n = \sum_{m=0}^{n-1} h^m (\partial_\alpha v) h^{n-m-1} \ . \ee
This suffices to prove the general case, where $f(\lambda)$ is an arbitrary smooth function, since $\lambda$ is effectively restricted to a compact set by the sup bound on $h$ and smooth functions on a closed interval $J\subset \RR$ may be approximated in the $C^1$ norm by polynomials. For, given such an approximating sequence $p_n\to f$, if we assume that $\|f'-p_n'\|_{C^0(J)} \le \epsilon$ then the mean value theorem implies that $|f(x)-p_n(x)-f(y)+p_n(y)| \le \epsilon |x-y|$ for all $x,y$ in the interval, and so we see that $dp_n \to df$ in $C^0(J^2)$. Therefore, $dp_n(h) \to df(h)$ in $\End(L^2(\End(E)))$, and we easily have that $\partial_\alpha p_n(h)\to \partial_\alpha f(h)$ in $L^2(\End(E))$. The proof that $[\beta,f(h)] = df(h)([\beta,h])$ is identical.

Finally, we prove that $f\colon \tilde P^1_c\to \tilde P^1_{c'}$ is continuous for some $c'\ge 0$. This follows from the above results, since
\begin{align}
\|\partial_\alpha f(h_1) - \partial_\alpha f(h_2)\|_{L^2} &= \| df(h_1)(\partial_\alpha(h_1-h_2)) + (df(h_1)-df(h_2))(\partial_\alpha h_2) \|_{L^2} \nonumber \\
&\le C_4 \|h_1-h_2\|_{H^1} + C_5 \|h_1-h_2\|_{C^0} \|h_2\|_{H^1} \\
&\le C_4 \|h_1-h_2\|_{H^1} + C_6 \|h_1-h_2\|_{H^1} \|h_2\|_{H^1} \ .
\end{align}
We also have a bound $|f(h_1)-f(h_2)| \le C_7 |h_1-h_2|$, so that $\|f(h_1)-f(h_2)\|_{L^2} \le C_7 \|h_1-h_2\|_{L^2} \le C_8 \|h_1-h_2\|_{H^1}$.
\end{proof}

For our purposes, the most important instance of these constructions will be associated to the positive function
\be \Psi(x,y) := \frac{e^{2(y-x)} - 2(y - x) - 1}{2 (y-x)^2} = \sum_{n\ge 0} \sum_{m=0}^n \frac{2^{n+1}}{(n+2)!} \column{n}{m} y^{n-m} (-x)^{m} \ . \ee
For this function, we have
\be \Psi(h)(A) = \sum_{n\ge 0} \sum_{m=0}^{n} \frac{2^{n+1}}{(n+2)!} \column{n}{m} h^{n-m} A (-h)^m \ . \label{eq:psiDef} \ee
With this, we can define Donaldson's functional:\footnote{Because the notation is quite different, we briefly explain the equivalence with the formulation of \cite{donaldson:duy,donaldson:duy2,simpson:higgs}. There, one regards $M$ as a functional of two Hermitian metrics, $K$ and $H=Ke^s$, where $s$ is Hermitian with respect to $K$. That is, $(v,w)_H = (e^s v, w)_K$. In our formulation, $K$ is the single fixed metric and $s=-2h=-2\log g$. The analogs in these references of our Proposition \ref{prop:fakeHomo} also involve a third metric, $J=H \tilde g^{-2}$, where $\tilde g$ is positive-definite with respect to $H$. That is, $g^{-1} \tilde g g \in \G_+^{2}$. Writing $\tilde g = g g' g^{-1}$, we then have $J = K g^{-2} \tilde g^{-2} = K (g\times g')^{-2}$.}
\be M(h) = \int \Tr(- h \mu_\RR(\alpha,\beta) + \Psi(-h) (\partial_\alpha h)\, \partial_{-\alpha^\dagger} h + \Psi(-h)([\beta,h])\, [h,\beta^\dagger] ) \, dt \ . \ee
When we wish to make the dependence on $\alpha,\beta$ explicit, we will write this as $M(h;\alpha,\beta)$. We note that this functional is invariant under gauge transformations $(\alpha,\beta)\mapsto u(\alpha,\beta)$, $h\mapsto u^{-1} h u$ with $u\in \G^{2}$; the action on $h$ is equivalent to the action $e^h\mapsto u^{-1} e^h u = e^{u^{-1} h u}$ of $\G^{2}$ on $\G_+^{2}$.
\begin{proposition}
$M\colon H^1(I, i\, \mf{su}(2)) \to \RR$ is a continuous functional.
\end{proposition}
\begin{proof}
Note that we do not require a sup bound on the domain. However, for the sake of the present result we are only interested in continuity at a point $h_1$, and this certainly has a finite $C^0$ norm, and any nearby point $h_2$ satisfies $\|h_1-h_2\|_{C^0} \le C \|h_1-h_2\|_{H^1}$. So, the lemma still yields useful inequalities:
\begin{align}
&\abs{ (\Psi(-h_1)(\partial_\alpha h_1), \partial_\alpha h_1)_{L^2} - (\Psi(-h_2)(\partial_\alpha h_2), \partial_\alpha h_2)_{L^2} } \nonumber \\
&= \abs{ (\Psi(-h_1)(\partial_\alpha h_1), \partial_\alpha(h_1-h_2))_{L^2} + (\Psi(-h_1)(\partial_\alpha(h_1-h_2)), \partial_\alpha h_2)_{L^2} \right. \nonumber \\
&\quad \left.+ (\Psi(-h_1)(\partial_\alpha h_2) - \Psi(-h_2)(\partial_\alpha h_2), \partial_\alpha h_2)_{L^2} } \nonumber \\
&\le C \|h_1 - h_2\|_{H^1} (\|\Psi(-h_1)(\partial_\alpha h_1)\|_{L^2} + \|h_2\|_{H^1} + \|h_2\|^2_{H^1}) \ .
\end{align}
A similar inequality holds for the term in $M(h)$ involving $[\beta,h]$. Finally, the first term is handled by
\be |(h_1-h_2, \mu_\RR(\alpha,\beta))_{L^2}| \le \|h_1-h_2\|_{L^2} \|\mu_\RR(\alpha,\beta)\|_{L^2}  \ . \ee
Reality of the first term follows from Hermiticity of $h$ and $\mu_\RR$, while reality of the last terms follows from the stronger result $(\Psi(-h)(A),A)_{L^2} \ge 0$.
\end{proof}

The next proposition shows that $g\mapsto M(\log g)$ is a sort of homomorphism. Its proof involves elements of those of Proposition 6 in \cite{donaldson:duy} and Proposition 5.1 in \cite{simpson:higgs}; we fill in some gaps in the latter.

\begin{proposition} \label{prop:fakeHomo}
$M(h;\alpha,\beta)$ is uniquely characterized by the following two properties:
\begin{enumerate}[(i)]
\item  For any smooth map $h\colon [a,b]\to i\mf{g}^2$, we have
\be \partial_\tau M(h(\tau);\alpha,\beta) = \half(e^h (\partial_\tau e^{-2h}) e^{h}, \mu_\RR(e^h(\alpha,\beta)))_{L^2} \ . \label{eq:anyPath} \ee
\item $M(0;\alpha,\beta)=0$.
\end{enumerate}
It follows that for all $g_1, g_2 \in \G_+^{2}$,
\be M(\log(g_1\times g_2); \alpha,\beta) = M(\log g_2; g_1(\alpha,\beta)) + M(\log g_1; \alpha,\beta) \ . \label{eq:fakeHomo} \ee
\end{proposition}
\begin{proof}
We first prove that the first part of the proposition implies \eqref{eq:fakeHomo}. We consider two particular paths, $h: [0,2]\to i\mf{g}^2$ and $\tilde h: [1,2]\to i\mf{g}^2$, which satisfy $h(0)=0$, $h(1) = \log g_1$, $h(2) = \log(g_1\times g_2)$, and $\tilde h = \log(g_1^{-1} \times e^h)$ on $[1,2]$. In particular, we have $\tilde h(1) = 0$ and $\tilde h(2) = \log g_2$. We define the unitary gauge transformation $U(\tau) = e^{-h(\tau)} g_1 e^{\tilde h(\tau)}$; taking the adjoint and then the inverse of this equation shows that we also have $U = e^{h} g_1^{-1} e^{-\tilde h}$. Then, keeping in mind that gauge transformations act from the right, we have
\begin{align}
\partial_\tau M(\tilde h(\tau); g_1(\alpha,\beta)) &= \half(e^{\tilde h} (\partial_\tau e^{-2\tilde h}) e^{\tilde h}, \mu_\RR(e^{\tilde h}(g_1(\alpha,\beta))))_{L^2} \nonumber \\
&= \half(e^{\tilde h} g_1 (\partial_\tau e^{-2 h}) g_1 e^{\tilde h}, \mu_\RR(U(e^h(\alpha,\beta))))_{L^2} \nonumber \\
&= \half(e^{h} (\partial_\tau e^{-2 h}) e^{h}, \mu_\RR(e^h(\alpha,\beta)))_{L^2} \nonumber \\
&= \partial_\tau M(h(\tau); \alpha,\beta) \ . \label{eq:secondPath}
\end{align}
Since \eqref{eq:anyPath} and \eqref{eq:secondPath} agree for $\tau\in [1,2]$, but $M(h(1);\alpha,\beta) = M(\tilde h(1);g_1(\alpha,\beta)) + M(\log g_1;\alpha,\beta)$, \eqref{eq:fakeHomo} follows.

We now prove \eqref{eq:anyPath}. To prove this for a fixed value of $\tau$, we will first prove it for a special path $[0,1]\ni u \mapsto u h(\tau)$. That is, we claim that
\be \partial_u M(uh; \alpha,\beta) = - (h, \mu_\RR(B'))_{L^2} \ , \label{eq:specialPath} \ee
where $B' = e^{uh}(B) = (\alpha', \beta')$. We prove this by observing that it holds at $u=0$ and then showing that the derivative of this equation holds for all $u\in [0,1]$. That is, we prove that
\be \partial^2_u M(uh;\alpha,\beta) = -(h, \partial_{u'} \mu_\RR(e^{u' h}(B')) |_{u'=0})_{L^2} = (h, [h, \mu_\RR(B')] + 2\tilde\Delta_0' h)_{L^2} = 2(h, \Delta'_0 h)_{L^2} \ . \label{eq:muDeriv} \ee
The last two equalities employed Lemmas \ref{lem:weitz} and \ref{lem:muR}, and the primes on $\tilde\Delta_0'$ and $\Delta_0'$ indicate that the pre-Nahm data that enters into their definitions is $B'$. (Since $(h, [h,\mu_\RR(B')])_{L^2}=0$, it actually does not matter which of $\Delta_0'$, $\tilde\Delta_0'$, or $\bar\Delta_0'$ appears in the last expression.) This claim is demonstrated as follows:
\begin{align}
\partial^2_u M(uh; \alpha,\beta) &= 
\sum_{n\ge 0} \sum_{m=0}^n \frac{2^{n+1}}{m!(n-m)!} \brackets{ ((-uh)^{n-m} (\partial_\alpha h) (uh)^m, \partial_\alpha h)_{L^2} + ((-uh)^{n-m} [\beta,h] (uh)^m, [\beta,h])_{L^2} } \nonumber \\
&= 2 \sum_{n',m\ge 0} \frac{1}{m! n'!} \brackets{ ((-2uh)^{n'} (\partial_\alpha h) (2uh)^m, \partial_\alpha h)_{L^2} + ((-2uh)^{n'} [\beta,h] (2uh)^m, [\beta,h])_{L^2} } \nonumber \\
&= 2 \brackets{ (e^{-2uh} (\partial_\alpha h) e^{2uh}, \partial_\alpha h)_{L^2} + (e^{-2uh} [\beta,h] e^{2uh}, [\beta,h])_{L^2} } \nonumber \\
&= 2 \brackets{ \|e^{-uh} (\partial_\alpha h) e^{uh}\|^2_{L^2} + \|e^{-uh} [\beta,h] e^{uh}\|^2_{L^2} } \nonumber \\
&= 2 \brackets{ \|\partial_{\alpha'} h\|^2_{L^2} + \|[\beta', h]\|^2_{L^2} } \nonumber \\
&= 2 (h, \Delta_0' h)_{L^2} \ . \label{eq:twoUDeriv}
\end{align}
We now differentiate \eqref{eq:specialPath} with respect to $\tau$ and compare with \eqref{eq:anyPath}. We introduce the notation $\dot h := \partial_\tau h$. \eqref{eq:anyPath} clearly holds at $u=0$  -- i.e., $\partial_\tau M(0;\alpha,\beta) = 0$ -- and so it remains only to show that
\be \partial_u (e^{uh}(\partial_\tau e^{-2uh}) e^{uh}, \mu_\RR(B'))_{L^2} = -2 \partial_\tau (h, \mu_\RR(B'))_{L^2}
%= -2\partial_\tau (e^{-uh} h e^{uh}, e^{uh} \mu_\RR(B') e^{-uh})
= -2(\dot h, \mu_\RR(B'))_{L^2} - 2(h, \partial_\tau \mu_\RR(B'))_{L^2} 
\ . \ee
We manipulate the left side by noting that
\begin{align}
\partial_u(e^{uh} (\partial_\tau e^{-2uh}) e^{uh}) &= h e^{uh} (\partial_\tau e^{-2uh}) e^{uh} - e^{uh}(\partial_\tau(h e^{-2uh})) e^{uh} - e^{uh}(\partial_\tau(e^{-2uh} h)) e^{uh} + e^{uh} (\partial_\tau e^{-2uh}) e^{uh} h \nonumber \\
&= - e^{uh} \dot h e^{-uh} - e^{-uh} \dot h e^{uh}
\end{align}
and, as in \eqref{eq:muDeriv}, $\partial_u \mu_\RR(B') = -2 \Delta_0' h$. So, the left side equals
\be - (e^{uh} \dot h e^{-uh} + e^{-uh} \dot h e^{uh}, \mu_\RR(B'))_{L^2} - 2 (e^{uh} (\partial_\tau e^{-2uh}) e^{uh}, \Delta_0' h)_{L^2} \ . \label{eq:niceLeft} \ee
Turning to the right side, we first use $\partial_{\alpha'} = e^{-uh} \circ \partial_\alpha \circ e^{uh}$ and $\partial_\tau(e^{-uh} e^{uh})=0$ to find that
\be
\partial_\tau \partial_{\alpha'} s = (\partial_\tau e^{-uh}) e^{uh} \partial_{\alpha'} s + \partial_{\alpha'}(e^{-uh} (\partial_\tau e^{uh}) s) = \partial_{\alpha'}(e^{-uh}(\partial_\tau e^{uh})) s \ ,
\ee
where $s$ is a section of $E$ (which does not depend on $\tau$). We similarly have
\be \partial_\tau \partial_{-(\alpha')^\dagger} s = \partial_{-\alpha^\dagger}(e^{uh}(\partial_\tau e^{-uh})) s \ . \ee
Therefore, we have
\be \partial_\tau [\partial_{-(\alpha')^\dagger}, \partial_{\alpha'}] s = \brackets{ \partial_{-(\alpha')^\dagger} \partial_{\alpha'}(e^{-uh}(\partial_\tau e^{uh})) - \partial_{\alpha'} \partial_{-(\alpha')^\dagger} (e^{uh}(\partial_\tau e^{-uh})) } s \ . \ee
Combined with an analogous computation involving $\beta'$, this yields
\begin{align}
\partial_\tau \mu_\RR(B') &= - \bar\Delta_0' (e^{-uh} \partial_\tau e^{uh}) + \tilde\Delta_0' (e^{uh} \partial_\tau e^{-uh}) \nonumber \\
&= \bar\Delta_0' ((\partial_\tau e^{-uh}) e^{uh}) + \tilde\Delta_0'(e^{uh} \partial_\tau e^{-uh}) \nonumber \\
&= \Delta_0'((\partial_\tau e^{-uh}) e^{uh} + e^{uh} \partial_\tau e^{-uh}) - \half [\mu_\RR(B'), [\partial_\tau e^{-uh}, e^{uh}]] \nonumber \\
&= \Delta_0'(e^{uh} (\partial_\tau e^{-2uh}) e^{uh}) - \half [\mu_\RR(B'), [\partial_\tau e^{-uh}, e^{uh}] ] \ .
\end{align}
The right hand side then becomes
\begin{align}
&-2(\dot h, \mu_\RR(B'))_{L^2} - 2(\Delta_0' h, e^{uh}(\partial_\tau e^{-2uh}) e^{uh})_{L^2} + (h, [\mu_\RR(B'), [\partial_\tau e^{-uh}, e^{uh}]])_{L^2} \nonumber \\
&= -2(e^{uh}(\partial_\tau e^{-2uh}) e^{uh}, \Delta_0' h)_{L^2} + (-\partial_\tau(e^{uh} h e^{-uh} + e^{-uh} h e^{uh}) + [h, [e^{uh}, \partial_\tau e^{-uh}]], \mu_\RR(B'))_{L^2} \nonumber \\
&= -2(e^{uh}(\partial_\tau e^{-2uh})e^{uh}, \Delta_0' h)_{L^2} - (e^{uh} \dot h e^{-uh} + e^{-uh} \dot h e^{uh}, \mu_\RR(B'))_{L^2} \ ,
\end{align}
which agrees with \eqref{eq:niceLeft}.
\end{proof}

\begin{corollary} \label{cor:EL}
The Euler-Lagrange equation for $M$ on $i\mf{g}^{2}$ is $\mu_\RR(e^{h}(\alpha,\beta))=0$.
\end{corollary}
\begin{proof}
$M(\log(e^h\times e^{\delta h});\alpha,\beta) = M(h;\alpha,\beta) + M(\delta h; e^{h}(\alpha,\beta))$ implies that the variation is
\be \delta M = - \int \Tr( \delta h\, \mu_\RR(e^{h}(\alpha,\beta)) ) \, dt \ . \ee
(Note that any $g\in \G_+^2$ may be written as $e^h\times e^{\delta h}$ for some unique $\delta h$, namely $\delta h = \log(e^{-h}\times g)$.)
\end{proof}

We thus find that minimizing this functional will solve our problem. However, it is not yet clear that $M$ is bounded below. We will shortly prove an estimate that, among other things, implies this under an appropriate hypothesis. We begin with an analog of Simpson's Proposition 2.1 \cite{simpson:higgs}:
\begin{lemma} \label{lem:subHarm}
There are constants $C_1,C_2>0$ such that if $f\in H^2(I,\RR)$ is a non-negative function whose derivative vanishes at the boundaries, $g\in L^2(I,\RR)$, and $-\partial^2 f \le g$ almost everywhere, then $\|f\|_{C^0} \le C_1 \|g\|_{L^2} + C_2 \|f\|_{L^2}$.
\end{lemma}
\begin{proof}
The boundary condition on the derivative of $f$ implies that $\tilde f: [-L,L]\to \RR$, $\tilde f(t) = \piecewise{f(t)}{t\ge 0}{f(-t)}{t\le 0}$ is in $H^2([-L,L],\RR)$. We then extend this by periodicity to a non-negative function $\tilde f \in H^2(\RR/2L\ZZ, \RR)$. We similarly define $\tilde g\in L^2(\RR/2L\ZZ,\RR)$ and have $-\partial^2 \tilde f \le \tilde g$.

Let $w\in H^2([-L/2,3L/2],\RR)$ be the solution of $-\partial^2 w = \tilde g$ with $\int_{-L/2}^{3L/2} w\,dt = \int_{-L/2}^{3L/2} t w\,dt = 0$. Then, for every $t\in [a,b]\subset [-L/2,3L/2]$ we have $\tilde f(t) \le w(t) + (\tilde f(a)-w(a)) + \frac{\tilde f(b)-w(b) - \tilde f(a) + w(a)}{b-a} (t-a)$, since a convex function such as $\tilde f - w$ is bounded above by the line between its endpoints. But, the terms involving $w$ are in turn bounded above by $2\|w\|_{C^0} + 2L \|\partial w\|_{C^0} \le C \|w\|_{H^2} \le C' \|\tilde g\|_{L^2} \le C'' \|g\|_{L^2}$, by the mean value theorem, Sobolev embedding, and elliptic regularity. Note that the elliptic estimate does not need a term involving $\|w\|_{L^2}$, since $w$ is orthogonal to $\ker \partial^2$. So, we have $\tilde f(t) \le C'' \|g\|_{L^2} + \frac{1}{b-a}((t-a) \tilde f(b) + (b-t) \tilde f(a))$ for all $a \le b$ such that $t\in [a,b]\subset [-L/2,3L/2]$. Furthermore, for some $a\in [-L/2,0]$ we must have $\tilde f(a) \le \sqrt{\frac{2}{L}} \|f\|_{L^2}$, and for some $b\in [L,3L/2]$ we must have $\tilde f(b) \le \sqrt{\frac{2}{L}} \|f\|_{L^2}$. So, for any $t\in [0,L]$ we have $f(t) \le C'' \|g\|_{L^2} + \sqrt{\frac{2}{L}} \|f\|_{L^2}\cdot \frac{1}{L}(3L/2+3L/2) = C'' \|g\|_{L^2} + 3 \sqrt{\frac{2}{L}} \|f\|_{L^2}$.
\end{proof}

Now, closely following Simpson (Proposition 5.3) \cite{simpson:higgs}, who was himself following Uhlenbeck-Yau \cite{yau:duy}, we will prove the following proposition which provides the link between stability and being able to solve the real moment map equation.
\begin{proposition} \label{prop:stabBound}
Suppose that $\| \mu_\RR(\alpha,\beta) \|_{L^2} \le C$ and $\alpha,\beta$ is $\xi^\RR$-stable. Then there exist $C_1,C_2>0$ such that
\be \|h\|_{C^0} \le C_1 + C_2 M(h; \alpha,\beta) \label{eq:stabBound} \ee
for any $h\in i\mf{g}^{2}$ such that $\|\mu_\RR(e^{h}(\alpha,\beta))\|_{L^2} \le C$.
\end{proposition}
\begin{proof}
First, we observe that $h$ is traceless and Hermitian, so its largest eigenvalue is always non-negative, and it is positive except where $h=0$. So, $\log \Tr e^{2h} \ge 0$, and again this is positive where $h\not=0$. Furthermore, at the boundaries of $I$ we have $\partial \log \Tr e^{2h} = \frac{2}{\Tr e^{2h}} \Tr \partial h \, e^{2h} = 0$. Next, we recall from Lemma \ref{lem:muR} that
\be -\partial^2 \log \Tr e^{2h} \le \sqrt{2} (|\mu_\RR(\alpha,\beta)| + |\mu_\RR(e^{h}(\alpha,\beta))|) \ . \ee
Lemma \ref{lem:subHarm} now implies that $\|\log \Tr e^{2h}\|_{C^0} \le C_1 + C_2 \|\log \Tr e^{2h}\|_{L^2}$. We then have inequalities of the form
\be \|h\|_{C^0} \le C_3 \| \log \Tr e^{2h} \|_{C^0} \le C'_1 + C'_2 \|\log \Tr e^{2h}\|_{L^2} \le C''_1 + C''_2 \|h\|_{L^2} \ . \label{eq:supBound} \ee
So, if the desired estimate \eqref{eq:stabBound} does not hold then we must be able to violate it with $\|h\|_{L^2}$ arbitrarily large, since otherwise the estimate would hold for sufficiently large $C_1$. Take a sequence $C_n$ tending to infinity and a sequence $h_n \in i \mf{g}^{2}$ with $l_n := \|h_n\|_{L^2}$ also tending to infinity and $\|h_n\|_{L^2} \ge C_n M(h_n)$. Define $u_n = - h_n / l_n$, which are uniformly bounded by \eqref{eq:supBound}.

We now show that after passing to a subsequence we have a bound on $\|u_n\|_{H^1}$. We first observe that the last condition on $h_n$, when written in terms of the $u_n$, takes the form
\be (u_n,\mu_\RR(\alpha,\beta))_{L^2} + l_n (\Psi(l_n u_n)(\partial_\alpha u_n), \partial_\alpha u_n)_{L^2} + l_n (\Psi(l_n u_n)([\beta,u_n]), [\beta, u_n])_{L^2} \le \frac{1}{C_n} \ , \label{eq:uIneq} \ee
at least for sufficiently large $n$ so that $l_n\not=0$. Now, let $\Phi\colon \RR\times\RR\to \RR^{+}$ be a positive smooth function such that $\Phi(x,y) \le (x-y)^{-1}$ whenever $x>y$. We observe that as $l\to\infty$, $l\Psi(l x, ly)$ increases monotonically to $(x-y)^{-1}$ if $x>y$ and $\infty$ if $x\le y$. Furthermore, since $\|u_n\|_{C^0}$ is bounded above, the eigenvalues $\{\lambda_i\}$ of $\{u_n\}$ are contained in a compact subset of $\RR$. It follows that for any $\epsilon>0$ we have $\Phi(\lambda_i,\lambda_j) < l \Psi(l \lambda_i, l \lambda_j)+\epsilon$ for sufficiently large $l$. For sufficiently large $n$, we then have
\begin{align}
C (\|\partial_\alpha u_n\|^2_{L^2} + \|[\beta,u_n]\|_{L^2}^2) &\le (\Phi(u_n)(\partial_\alpha u_n),\partial_\alpha u_n)_{L^2} + (\Phi(u_n)([\beta,u_n]),[\beta,u_n])_{L^2} \nonumber \\
&< l_n (\Psi(l_n u_n)(\partial_\alpha u_n), \partial_\alpha u_n)_{L^2} + l_n (\Psi(l_n u_n)([\beta,u_n]), [\beta, u_n])_{L^2} \nonumber \\
&\quad + \epsilon (\|\partial_\alpha u_n\|^2_{L^2} + \|[\beta,u_n]\|_{L^2}^2) \nonumber \\
&\le C_n^{-1} - (u_n, \mu_\RR(\alpha,\beta))_{L^2} + \epsilon (\|\partial_\alpha u_n\|^2_{L^2} + \|[\beta,u_n]\|_{L^2}^2) \ , \label{eq:ineqs}
\end{align}
where $C>0$. After choosing $\epsilon < C$, the sup bound on $u_n$ and the fact that $C_n\to \infty$ give an upper bound on $\|\partial_\alpha u_n\|^2_{L^2}$, which together with the sup bound gives a bound on $\|u_n\|_{H^1}$. For the purposes of the next paragraph, we now allow $\Phi$ to be non-negative, as opposed to positive, and observe that the second two inequalities in \eqref{eq:ineqs} still hold. Indeed, because of the bound on $\|u_n\|_{H^1}$, for any $\epsilon>0$ we have
\be (u_n,\mu_\RR(\alpha,\beta))_{L^2} + (\Phi(u_n)(\partial_\alpha u_n), \partial_\alpha u_n)_{L^2} + (\Phi(u_n)([\beta,u_n]),[\beta,u_n])_{L^2} < \epsilon \ee
for sufficiently large $n$.

Weak compactness of a ball in $H^1$ centered at 0 now lets us pass to a subsequence such that $u_n\rightharpoonup u_\infty$ in $\{h\in H^1(I,i \mf{g}) \bigm| h|_{\partial I} \in i\mf{t} \}$. The compact embedding $H^1\hookrightarrow L^2$ implies that $u_n\to u_\infty$ in $L^2$, and so $\|u_\infty\|_{L^2} = 1$ and $u_\infty$ is not the zero function. This also implies that $(u_n,\mu_\RR(\alpha,\beta))_{L^2} \to (u_\infty, \mu_\RR(\alpha,\beta))_{L^2}$. Similarly, the compact embedding $H^1\hookrightarrow C^0$ implies that $u_n\to u_\infty$ in $C^0$, and so Lemma \ref{lem:psiCont} implies that $\Phi(u_n) \to \Phi(u_\infty)$ in $\End(L^2(\End(E)))$. Together with the upper bounds on $\|\partial_\alpha u_n\|_{L^2}$ and $\|[\beta,u_n]\|_{L^2}$, this implies that $(\Phi(u_n)(\partial_\alpha u_n) - \Phi(u_\infty)(\partial_\alpha u_n),\partial_\alpha u_n)_{L^2} \to 0$ and $(\Phi(u_n)([\beta,u_n]) - \Phi(u_\infty)([\beta,u_n]),[\beta,u_n])_{L^2} \to 0$. Therefore, for any $\epsilon>0$ we have
\be (u_\infty,\mu_\RR(\alpha,\beta))_{L^2} + (\Phi(u_\infty)(\partial_\alpha u_n), \partial_\alpha u_n)_{L^2} + (\Phi(u_\infty)([\beta,u_n]),[\beta,u_n])_{L^2} < \epsilon \ee
for all sufficiently large $n$. Of course, $[\beta,u_n]\to [\beta,u_\infty]$ in $L^2$, so we may replace $[\beta,u_n]$ by $[\beta,u_\infty]$ in this inequality. Finally, the weak convergence $\Phi^{1/2}(u_\infty)(\partial_\alpha u_n) \rightharpoonup \Phi^{1/2}(u_\infty)(\partial_\alpha u_\infty)$ in $L^2$ implies that $(\Phi(u_\infty)(\partial_\alpha u_\infty),\partial_\alpha u_\infty)_{L^2} = \|\Phi^{1/2}(u_\infty)(\partial_\alpha u_\infty)\|_{L^2}^2 \le \liminf_i \|\Phi^{1/2}(u_\infty)(\partial_\alpha u_i)\|_{L^2}^2$, and so we finally find that
\be (u_\infty,\mu_\RR(\alpha,\beta))_{L^2} + (\Phi(u_\infty)(\partial_\alpha u_\infty), \partial_\alpha u_\infty)_{L^2} + (\Phi(u_\infty)([\beta,u_\infty]),[\beta,u_\infty])_{L^2} \le 0 \ , \label{eq:phiSmall} \ee
since the left hand side is less than $\epsilon$ for all $\epsilon>0$.

Because we will repeatedly use this conclusion, we state it as a lemma:
\begin{lemma} \label{lem:phiSmall}
For any smooth $\Phi\colon \RR\times \RR\to \RR^{\ge 0}$ such that $\Phi(x,y)\le (x-y)^{-1}$ whenever $x>y$, \eqref{eq:phiSmall} holds.
\end{lemma}

We now use this inequality to show that the eigenvalues of $u_\infty$ are constant. We do so by showing that $\partial \Tr f(u_\infty)=0$ for all smooth $f\colon\RR\to \RR$. We begin by observing that $\partial \Tr f(u_\infty) = \Tr \partial_\alpha f(u_\infty) = \Tr df(u_\infty)(\partial_\alpha u_\infty)$. But, $\Tr df(u_\infty)(\partial_\alpha u_\infty) = \Tr \Phi(u_\infty)(\partial_\alpha u_\infty)$ for any smooth $F\colon \RR\times\RR\to \RR$ such that $F(x,x) = df(x,x)$ for all $x\in \RR$. For any $N>1$, we can find such a $F$ which, in addition, satisfies $F^2(x,y) < \frac{1}{N(x-y)}$ when $x>y$. Lemma \ref{lem:phiSmall} (with $\Phi = N F^2$) now gives an estimate $\|\partial \Tr f(u_\infty)\|_{L^2}^2 \le C (\|F(u_\infty)(\partial_\alpha u_\infty)\|_{L^2}^2 + \|F(u_\infty)([\beta,u_\infty])\|_{L^2}^2) \le C'/N$, where $C$ and $C'$ are independent of the choices of $N$ and $F$. Since $N>1$ was arbitrary, $\partial \Tr f(u_\infty)=0$. So, the eigenvalues of $u_\infty$ are constant. Note that they are distinct -- in fact, negatives of each other -- since $\Tr(u_\infty)=0$ but $u_\infty\not=0$.

We will now construct a good projector using the eigenspaces of $u_\infty$. Let the eigenvalues be $\lambda_1,\lambda_2$, with $\lambda_1<0$ and $\lambda_2=-\lambda_1 > 0$. We let $\pi$ be the orthogonal projection onto the $\lambda_1$ eigenspace. We can characterize $\pi$ as $p(u_\infty)$, where $p(\lambda)$ is a smooth function with $p(\lambda) = 1$ if $\lambda \le \lambda_1$ and $p(\lambda)=0$ if $\lambda \ge \lambda_2$, and this makes it clear that $\pi \in \{h\in H^1(I,i \mf{u}(2)) \bigm| h|_{\partial I} \in i\mf{t}\oplus i\mf{u}(1) \}$. Furthermore, $\pi$ satisfies $(1-\pi)\partial_\alpha \pi = 0$ and $\beta\pi = \pi\beta\pi$. To see this, we first define $F_0: \RR\times\RR\to \RR$ by $F_0(x,y) = (1-p)(y) \, dp(x,y)$. Since $\partial_\alpha \pi = \partial_\alpha p(u_\infty) = dp(u_\infty)(\partial_\alpha u_\infty)$, we have $F_0(u_\infty)(\partial_\alpha u_\infty) = (1-\pi)\partial_\alpha \pi$. Similarly, we have $F_0(u_\infty)([\beta,u_\infty]) = (1-\pi) \beta \pi$. We observe that $F_0(\lambda_2, \lambda_1) = 0$; we will now show that this implies that $F_0(u_\infty)(\partial_\alpha u_\infty) = F_0(u_\infty)([\beta,u_\infty]) = 0$. Since $F_0(u_\infty)$ only depends on $F_0(\lambda_i,\lambda_j)$, for any $N>1$ we can replace $F_0$ by a function $F$ satisfying $F(\lambda_i,\lambda_i) = F_0(\lambda_i,\lambda_j)$ and $F^2(x,y) < \frac{1}{N(x-y)}$ when $x>y$. As above, we then have $\|F_0(u_\infty)(\partial_\alpha u_\infty)\|^2_{L^2} + \|F_0(u_\infty)([\beta, u_\infty])\|^2_{L^2} = \|F(u_\infty)(\partial_\alpha u_\infty)\|^2_{L^2} + \|F(u_\infty)([\beta, u_\infty])\|^2_{L^2} \le C/N$ for all $N$, and so $F_0(u_\infty)(\partial_\alpha u_\infty) = F_0(u_\infty)([\beta,u_\infty]) = 0$. Proposition \ref{prop:manySub} now shows that $\pi$ is a good projection.

Finally, we show that $\deg \pi = \slope \pi \ge \slope E = 0$, so that $(\alpha,\beta)$ is not stable. We define
\be W = -4\lambda_2 \deg \pi = 2\lambda_2 \deg(E) - 2 (\lambda_2-\lambda_1) \deg \pi \ . \ee
Using Proposition \ref{prop:degIn}, we have
\begin{align}
W &= (u_\infty, \mu_\RR(\alpha,\beta))_{L^2} + (\lambda_2-\lambda_1) (\|\partial_\alpha \pi\|_{L^2}^2 + \|[\beta,\pi]\|^2_{L^2}) \nonumber \\
&= (u_\infty, \mu_\RR(\alpha,\beta))_{L^2} + (\lambda_2-\lambda_1) (\|dp(u_\infty)(\partial_\alpha u_\infty)\|_{L^2}^2 + \|dp(u_\infty)([\beta,u_\infty])\|^2_{L^2}) \nonumber \\
&= (u_\infty, \mu_\RR(\alpha,\beta))_{L^2} + (\lambda_2-\lambda_1) \brackets{ ((dp)^2(u_\infty)(\partial_\alpha u_\infty), \partial_\alpha u_\infty)_{L^2} + ((dp)^2(u_\infty)([\beta,u_\infty]), [\beta,u_\infty])_{L^2} } \ .
\end{align}
We observe that this expression is unchanged if we replace $(\lambda_2-\lambda_1)(dp)^2$ by a smooth non-negative function $\Phi$ which satisfies $\Phi(\lambda_i,\lambda_j) = (\lambda_2-\lambda_1)(dp)^2(\lambda_i,\lambda_j)$ for all $i,j$, but also $\Phi(x,y) \le (x-y)^{-1}$ whenever $x>y$. (Such a function exists, since $(\lambda_2-\lambda_1)(dp)^2(\lambda_2,\lambda_1) = (\lambda_2-\lambda_1)^{-1}$.) Lemma \ref{lem:phiSmall} then implies that $W\le 0$, and therefore $\deg \pi \ge 0$.
\end{proof}
\begin{remark}
Since this is the only place where stability will enter into our proof of the DUY theorem, and we know that non-polystable $\G_\CC$ orbits do not contain solutions of the real moment map equation, it follows that this proposition is false if $\alpha,\beta$ is not $\xi^\RR$-polystable. \end{remark}

Motivated by this proposition, for any $c \ge \| \mu_\RR(\alpha,\beta) \|_{L^2}$ we define
\be P_c^{2} = \{h\in i\mf{g}^{2} \bigm| \|\mu_\RR(e^{h}(\alpha,\beta))\|_{L^2} \le c \} \ . \ee
Given any such $c$, we will now, following Konno (Lemma 4.11) \cite{konno:higgs}, use the so-called `direct method' for solving variational problems to extremize the functional $M$ on $P_c^{2}$. To begin with, we observe the following corollary of the previous proposition:
\begin{corollary}
$M(h;\alpha,\beta)$ is bounded below on $P^{2}_c$. \label{cor:stabBounds}
\end{corollary}

We next have
\begin{proposition}
The infimum of $M(h;\alpha,\beta)$ is attained on $P^{2}_c$.
\end{proposition}
\begin{proof}
Let $h_i\in P^{2}_c$ be a minimizing sequence. Proposition~\ref{prop:stabBound} implies that there is a $C_1>0$ such that $\|h_i\|_{C^0} < C_1$ for all $i$. This uniform bound and the fact that we are minimizing $M$ together yield a bound $(\Psi(-h_i)(\partial_\alpha h_i), \partial_{\alpha} h_i)_{L^2} < C_2$ for all $i$. Lemma \ref{lem:psiCont} then yields an inequality $\|\partial_\alpha h_i\|_{L^2} < C_3$, i.e. $\|h_i\|_{H^1} < C_4$. Next, using Lemma \ref{lem:muR}, we combine the $C^0$ bound on $h_i$ with the bound $\|\mu_\RR(e^{h_i}(\alpha,\beta))\|_{L^2} \le c$ to obtain a bound $\|\partial_\alpha(H_i \partial_{-\alpha^\dagger} H_i^{-1})\|_{L^2} \le C_5$, where $H_i = e^{2 h_i}$. Elliptic regularity then gives $\|H_i \partial_{-\alpha^\dagger} H_i^{-1}\|_{H^1} \le C_6$. Combined with the $H^1$ bound on $H_i^{-1}$, this gives $\|\partial_{-\alpha^\dagger} H_i^{-1}\|_{H^1} \le C_7$, and so elliptic regularity gives $\|H_i\|_{H^2} \le C_8$.

Weak compactness of the ball in $\{H\in H^2(I, i\, \mf{u}(2)) \bigm| H|_{\partial I} \in i\mf{t}\oplus i\mf{u}(1) \ , \  \partial H|_{\partial I} \in i\mf{t}^\perp \}$ of radius $C_8$ implies that we may pass to a weakly convergent subsequence: $H_i \rightharpoonup H_\infty$. The compact embedding $H^2\hookrightarrow H^1$ implies that it is strongly convergent in $H^1$, and therefore in $C^0$. This gives $\det H_\infty=1$. So, we can write $H_\infty=e^{2 h_\infty}$ with $h_\infty\in i\, \mf{g}^2$. Also, convergence in $H^1$ implies that $M(h_i;\alpha,\beta)\to M(h_\infty; \alpha,\beta)$.

It remains only to show that $h_\infty \in P^2_c$. We do so by showing that $\mu_\RR(e^{h_i}(B))\rightharpoonup \mu_\RR(e^{h_\infty}(B))$ in $L^2$; this suffices because we will then have
\be \|\mu_\RR(e^{h_\infty}(B))\|_{L^2} \le \liminf_i \|\mu_\RR(e^{h_i}(B))\|_{L^2} \le c \ . \ee
This weak convergence follows from the formulae in Lemma \ref{lem:muR}: every term is strongly convergent in $L^2$ except for $\partial^2 H_i$, which is weakly convergent. (In particular, the term of the schematic form $(\partial H_i)^2$ converges strongly because of the $C^0$ bound on $\partial H_i$ together with strong convergence of $\partial H_i$ in $L^2$.)
\end{proof}

We have now minimized $M(h)$ subject to the constraint $\|\mu_\RR(e^{h}(B))\|_{L^2} \le c$. It remains to show that the that the constraint is weak enough that a minimizer $h_\infty \in P^2_c$ must satisfy the Euler-Lagrange equation $\mu_\RR(e^{h_\infty}(B))=0$. That is, the computation of Corollary \ref{cor:EL} shows that $(\delta h, \mu_\RR(e^{h_\infty}(B)))_{L^2} \le 0$ for any variation $\delta h \in i \mf{g}^2$ which exponentiates to a 1-parameter family of positive-definite gauge transformations $g_\tau = e^{\delta h \, \tau}$ such that $\|\mu_\RR(g_\tau(e^{h_\infty}(B)))\|_{L^2} \le c$ for small $\tau\ge 0$, and we must show that this implies that $\mu_\RR(e^{h_\infty}(B))=0$. However, it could be the case that $\|\mu_\RR(e^{h_\infty}(B))\|_{L^2} = c$, in which case this restriction on $\delta h$ could be severe. Fortunately, following Konno (Lemma 4.13) \cite{konno:higgs}, we will now see that there is always a sufficiently useful choice of $\delta h$:
\begin{proposition} \label{prop:minSolves}
$\mu_\RR(e^{h_\infty}(B))=0$, and $h_\infty \in i \mf{g}^{s+1}$.
\end{proposition}
\begin{proof}
Per the preceding discussion, we will explore the consequences of Corollary \ref{cor:EL} for a useful choice of $\delta h$. Specifically, we let $\delta h \in i \mf{g}^2$ be the solution to
\be \tilde\Delta'_0 \delta h = \mu_\RR(e^{h_\infty}(B)) \ , \label{eq:varChoice} \ee
where the prime on $\tilde\Delta_0'$ (and on $\bar\partial'_0$ below) indicates that the pre-Nahm data that enters into the definition of this operator is $B'=e^{h_\infty}(B)$. Such a solution exists and is unique, since $\tilde\Delta_0'$ is self-adjoint and has no kernel (thanks to the stability of $B'$). That such a choice might be useful is clear from $(\delta h, \mu_\RR(B'))_{L^2} = \|\bar\partial'_0 \delta h\|^2_{L^2}$, as if $g_\tau$ satisfies the constraint for small $\tau\ge 0$ then we must have $\bar\partial_0' \delta h = 0$, and therefore $\delta h=0$ because of the stability of $B'$. Finally, \eqref{eq:varChoice} would then imply that $\mu_\RR(B') = 0$. So, it remains only to test if $g_\tau$ satisfies the constraint for small $\tau\ge 0$. Using Lemma \ref{lem:muR}, we compute
\be \left. \frac{d}{d\tau} \|\mu_\RR(g_\tau(B'))\|_{L^2}^2 \right|_{\tau=0} = -4 \|\mu_\RR(B')\|_{L^2}^2 \ . \ee
So, either $\mu_\RR(B')=0$ (in which case $\delta h = 0$ and the constraint is satisfied for all $\tau\ge 0$ because $g_\tau$ is the identity) or the constraint is satisfied for small $\tau\ge 0$, which implies that $\mu_\RR(B')=0$. Finally, that $h_\infty$ is in $i \mf{g}^{s+1}$ follows from applying elliptic regularity to the real moment map equation, regarded as a differential equation in $h_\infty$.
\end{proof}

\begin{remark}
It follows from this proposition that $h_\infty$ -- which a priori is only the minimizer of $M$ on $P^2_c$ for a particular choice of $c$ -- is actually a minimizer for $M$ on all of $i\mf{g}^2$. For, any other minimizer on $P^2_{c'}$ with $c'>c$ would also be in $P^2_c$, by this proposition. In particular, it follows that $M$ is bounded below on all of $i\mf{g}^2$. Indeed, by density of $i \mf{g}^2$ in $\{h\in H^1(I, i\,\mf{su}(2)) \bigm| h|_{\partial I} \in i \mf{t}\}$, $h_\infty$ is even a global minimizer on this larger space.
\end{remark}

As a final observation about Donaldson's functional, we observe that it is convex in an appropriate sense, and that this gives a second proof of the injectivity of $j$.
\begin{proposition}
For any $h\in H^1(I,i\, \mf{su}(2))$, $\tau\in \RR$, and $B''=(\alpha'',\beta'') \in \B^{s}_{\rm all}$,
\be \partial^2_\tau M(\tau h;B'') = 2 \|\bar\partial'''_0 h\|_{L^2}^2 \ , \label{eq:convex} \ee
where the pre-Nahm data entering into the definition of $\bar\partial_0'''$ is $B'''=e^{\tau h}(B'')$ (which is $L^2$-regular). It follows that $h_\infty$ is the unique global minimum of $M(h;\alpha,\beta)$ on $i\mf{g}^2$.
\end{proposition}
\begin{proof}
This second derivative was computed in \eqref{eq:twoUDeriv}; note that integration by parts is not needed to arrive at \eqref{eq:convex}. Now, for any nonzero $h\in i\mf{g}^2$ we study the function $f\colon[0,1]\to \RR$ given by $f(\tau) = M(\log(e^{h_\infty} \times e^{\tau h});\alpha,\beta)$, which by Proposition \ref{prop:fakeHomo} coincides with $M(\tau h; e^{h_\infty}(\alpha,\beta)) + M(h_\infty;\alpha,\beta)$. Per the previous computation, this is convex -- indeed, strictly so, thanks to the stability of $B$. Since $h_\infty$ satisfies the Euler-Lagrange equation, $f'(0) = 0$, and so $f(1) > f(0)$. Since any $h'\in i \mf{g}^2$ can be written as $\log(e^{h_\infty} \times e^{h})$ for some $h\in i\mf{g}^2$ (specifically, $h = \log(e^{-h_\infty} \times e^{h'})$), we find that $h_\infty$ is the unique global minimum of $M(h';\alpha,\beta)$ on $i\mf{g}^2$.
\end{proof}
\begin{remark}
From Proposition \ref{prop:uniqueSol} we know that any other minimizer $h'\in i\mf{g}^{s+1}$ must satisfy $g'(e^{h'}(B))=e^{h_\infty}(B)$ for some $g'\in \G^{s+1}$, and the present proposition seems to strengthen this by adding that $h'=h_\infty$. However, we already knew that $e^{h'} g' = e^{h_\infty}$ by the stability of $B$, and the uniqueness of the polar decomposition then gives $g'=1$ and $h'=h_\infty$.
\end{remark}

We conclude by expanding on Theorem~\ref{thm:res}.

\begin{theorem} \label{thm:conn}
Let $(\mf{z}^3_{\partial I})^{\circ} \subset \mf{z}^3_{\partial I}$ denote the open sublocus of generic parameter values, and denote by $p^{\circ}\colon \mc{M}_{\mathrm{all}}^{\circ} \to (\mf{z}^3_{\partial I})^{\circ}$ the pullback of the projection map $p$. The fibered manifold $p^\circ$ admits a canonical Ehresmann connection $\mathrm{Conn}_{\mc{M}}$ and is thus a fiber bundle.

This connection is uniquely characterized as follows: if one considers paths in $(\mf{z}^3_{\partial I})^{\circ}$ with only one of $\xi^i$ varying, then after hyperkahler rotating to a complex structure where the varying parameter is $\xi^\RR$, the corresponding parallel transport map is that given by the composition of the canonical morphisms of Theorems~\ref{thm:relXi} and~\ref{thm:duy}.

Furthermore, this same property is satisfied for variations valued in any space of the form $\mf{z}_{\partial I}\otimes \Xi$, where $\Xi\subset \RR^3$ is a dimension 1 subspace.

Finally, this connection is equivariant under the commuting actions of $Z_2\times Z_2$ and $SO(3)$ on $\M_{\rm all}^\circ$.
\end{theorem}
\begin{proof}
We will establish the existence of this Ehresmann connection $\mathrm{Conn}_{\mc{M}}$, which in particular will provide local trivializations of $p^{\circ}$ to prove that we have a fiber bundle.

So, consider some Nahm data $A$ representing a point in $\M_{\rm all}^\circ$. Let $p_i^\circ \colon (\M_{\rm all}^\circ)_i\to ((\mf{z}_{\partial I}^3)^\circ)_i$ be the restrictions to the subspaces containing $A$ with only $\xi^i$ varying, which include into $p^\circ$. There are canonical horizontal subspaces $H_{i,A}$ through $A$ associated to each $p_i^\circ$. Pushing these forward to $T_A \M^\circ_{\rm all}$ and taking the direct sum yields a horizontal subspace $H_A \subset T_A \M^\circ_{\rm all}$.

In order to describe $H_A$ explicitly, we recall the description of the tangent space given in Propositions \ref{prop:realcoul} and \ref{prop:realRG}, namely $T_A \M^\circ_{\rm all} \simeq \ker d_0^*\oplus d_1$ (since $A$ is irreducible); note that there are no boundary conditions on $d_0^*\oplus d_1$, as is appropriate for the first harmonic space of the general deformation complex. We now combine this with our description of the morphisms of Theorems~\ref{thm:relXi} and~\ref{thm:duy}. Together, these results imply that the tangent vector at $A$ associated to a variation $\delta\xi^\RR$ may be described uniquely as $a=\bar\partial_0 h$, where $h\in C^{\infty}(I,\mf{sl}(2,\CC))$, $h|_{\partial I} \in \mf{t}_\CC$, $\partial h|_{\partial I} \in \mf{t}_\CC^\perp + \delta\xi^\RR_{\partial I} \sigma_z$, and $d_0^* a = d_1 a = 0$. (Note that we are treating $\bar\partial_0$ as an $\RR$-linear map from $C^\infty(I,\mf{sl}(2,\CC))$ to $C^\infty(I,\mf{su}(2)^{\oplus 4})$, as opposed to a $\CC$-linear map to $C^\infty(I,\mf{sl}(2,\CC)^{\oplus 2})$. This is in contrast to operators such as $d_0$ and $\Delta_0$, which extend naturally to be $\CC$-linear.) We compute that
\be d_0^*a = \half \Delta_0(h-h^\dagger) - \frac{1}{4} [\mu_\RR(A), h+h^\dagger] \ , \quad (d_1 a)^1 = \frac{i}{2} \Delta_0(h+h^\dagger) - \frac{i}{4} [\mu_\RR(A), h-h^\dagger] \ , \ee
so that once we use $\mu_\RR(A)=0$ the vanishing of these quantities becomes equivalent to $\Delta_0 h = 0$. (Of course, it is automatic that $(d_1 a)^2 = (d_1 a)^3 = 0$, using $\mu_\CC(A)=0$, since an infinitesimal formal complexified gauge transformation preserves the vanishing of the complex moment map.) We note that \emph{a priori} this gauge differs from the `positive-definite gauge' that we employed throughout this section, in which $h$ is Hermitian, but that the equation $\Delta_0(h-h^\dagger)=0$, with the above boundary conditions that imply that $h-h^\dagger \in \mf{g}$, actually implies that $h=h^\dagger$ (since $A$ is irreducible). Nevertheless, it is more useful to think of our gauge choice as $d_0^* a = 0$ as this is independent of the choice of complex structure, which we will now vary. Specifically, if instead of varying $\xi^\RR$ we vary either $\xi^2$ or $\xi^3$, the same description of the resulting tangent vector applies after we hyperkahler rotate.

This explicit description of the tangent vector $a$ immediately implies that $H_A$ varies smoothly as we vary $A$, so that we have indeed defined a smooth connection. This follows from the smooth dependence of solutions of elliptic ordinary differential equations, such as $\Delta_0 h = 0$, on their parameters and boundary conditions. This is clearly the unique connection with the correct parallel transport properties in the three primary directions.

We next study the remaining directions. We compare two different ways of defining an infinitesimal variation of $A$ as we vary $\xi$ to $\xi+\delta\xi$ with $\delta\xi\in \mf{z}_{\partial I}\otimes \Xi$ for some line through the origin $\Xi\subset \RR^3$. The first uses $\mathrm{Conn}_{\mc{M}}$, while the latter follows from the DUY morphism associated to a complex structure in which the variation $\delta \xi$ leaves $\xi^\CC$ invariant. Our above explicit discussion enables us to verify that these prescriptions agree. We begin by letting $h_0,h_L \in C^\infty(I, i\, \mf{su}(2))$ be the solutions of $\Delta_0 h = 0$ with the respective boundary conditions $h_0,h_L|_{\partial I}\in i\mf{t}$, $\partial h_0|_0, \partial h_L|_L \in i\mf{t}^\perp + \sigma_z$, and $\partial h_0|_L, \partial h_L|_0 \in i\mf{t}^\perp$. We also let $\zeta$ be a unit vector in $\Xi$. Then, the former prescription yields $a_1 = \sum_{t_0\in \{0,L\}} \sum_i \delta\xi^i_{t_0} \bar\partial^{(i)}_0 h_{t_0}$, while the latter yields $a_2 = \sum_{t_0\in \{0,L\}} \delta\xi^{(\zeta)}_{t_0} \bar\partial_0^{(\zeta)} h_{t_0}$, where $\delta\xi_{t_0}^{(\zeta)}$ and these modified $\bar\partial_0$ operators are defined via hyperkahler rotation, as we now explain. The desired result then follows from the fact that $\sum_i \delta\xi^i_{t_0} \bar\partial_0^{(i)} h = \delta\xi_{t_0}^{(\zeta)} \bar\partial_0^{(\zeta)} h$ for all $h\in C^\infty(I,i\, \mf{su}(2))$.

Choose some $\zeta\in S^2$, where $S^2$ is the unit sphere in $\RR^3$, and let $R\in SO(3)$ be a rotation matrix which takes $\zeta$ to the first standard basis vector $e_1$ for $\RR^3$. We use this matrix $R$ to define the operator $\bar\partial_0^{(\zeta)}$:
\be \bar\partial_0^{(\zeta)} h := R^{-1} \bar\partial_{0,R(A)} h \ . \ee
Here, $R^{-1}$ acts in the usual way, fixing the first component and rotating the last three components. The presence of $R(A)$ in the subscript on the right denotes the fact that the Nahm data entering into the definition of $\bar\partial_{0,R(A)}$ is $R(A)$, the result of hyperkahler rotating $A$ by $R$. There is a $U(1)$ freedom in the choice of $R$, given by $R\sim R' R$, where $R'\in SO(3)$ preserves $e_1$. However, this freedom does not affect the definition of $\bar\partial_0^{(\zeta)}$. To see this, we observe that in the standard complex structure associated to $e_1$, such rotations simply multiply $\beta$ by a phase, and this phase is cancelled by the $R^{-1}$ action. That is,
\be (R' R)^{-1} \bar\partial_{0, R' R(A)} h = R^{-1} \bar\partial_{0, R(A)} h \ee
for all $R\in SO(3)$ and $R'\in U(1)$. A short computation yields the explicit formula
\be \bar\partial_0^{(\zeta)} h = i R^{1j} ([A^j, h], (- \delta^{j\ell} \partial_{A^0} h + \epsilon^{j\ell i} [A^i, h])_{\ell=1,2,3}) \ . \label{eq:genRot0} \ee
Here, the repeated $i,j$ indices are summed from 1 to 3. The fact that only the first row of $R$ enters into this result is another manifestation of the $U(1)$-invariance. Indeed, if we write the components of $\zeta\in \RR^3$ as $\zeta^j$ then we have
\be \bar\partial_0^{(\zeta)} h = i \zeta^j ([A^j, h], (- \delta^{j\ell} \partial_{A^0} h + \epsilon^{j\ell i} [A^i, h])_{\ell=1,2,3}) \ . \label{eq:genRot} \ee
In particular, the operators $\bar\partial^{(k)}$ associated to the standard basis vectors for $\RR^3$ are given by
\be
\bar\partial_0^{(k)} h = i ([A^k, h], (-\delta^{k\ell} \partial_{A^0} h + \epsilon^{k\ell i} [A^i, h])_{\ell=1,2,3}) \ , \label{eq:specRot} \ee
and so $\bar\partial_0^{(\zeta)} h = \zeta^j \bar\partial_0^{(j)} h$. We note that it was quite important for this result that $h$ was Hermitian: otherwise, on the right hand side of \eqref{eq:genRot0} $h$ would have been replaced by $\frac{h+h^\dagger}{2}$ and there would have been an additional term $d_0 \parens{\frac{h-h^\dagger}{2}}$, which does not depend on $\zeta$.

If we now let $\delta\xi^{(\zeta)}$ be the first (and only nonvanishing) component of the vector $R \, \vec\delta\xi$, then it is clear that the above formulae for $a_1$ and $a_2$ do, indeed, give the correct variations of $A$. So, it remains only to verify that $\sum_i \delta\xi^i_{t_0} \bar\partial_0^{(i)} h = \delta\xi_{t_0}^{(\zeta)} \bar\partial_0^{(\zeta)} h$ for all $h\in C^\infty(I,i\, \mf{su}(2))$. This follows from
\be \delta\xi^{(\zeta)} \bar\partial_0^{(\zeta)} h = \delta\xi^{(\zeta)} \zeta^j \bar\partial_0^{(j)} h = \delta \xi^j \bar\partial_0^{(j)} h \ . \ee

Finally, we prove that the connection is $Z_2\times Z_2\times SO(3)$-equivariant. First, we study the $SO(3)$ action. This leaves $\Delta_0$, and therefore $h_{t_0}$, invariant. So, we want to show that $\sum_{t_0\in\{0,L\}} \sum_{i,j} (R^{ij} \delta\xi^j_{t_0}) \bar\partial_{0,R(A)}^{(i)} h_{t_0} = \sum_{t_0\in\{0,L\}} \sum_i \delta\xi^i_{t_0} R(\bar\partial_0^{(i)} h_{t_0})$ for any $R\in SO(3)$. Let $\zeta_i\in S^2$ be the $i$-th row of $R$ (transposed into a column vector), so that $R$ maps $\zeta_i$ to the $i$-th standard basis vector $e_i$ of $\RR^3$, and let $R_i\in SO(3)$ rotate $e_i$ to $e_1$. Then, we easily have that
\be \bar\partial_{0,R(A)}^{(i)} h_{t_0} = R (R_i R)^{-1} \bar\partial_{0, R_i R(A)} h_{t_0} = R \bar\partial_0^{(\zeta_i)} h_{t_0} = \sum_j \zeta_i^j R \bar\partial_0^{(j)} h_{t_0} = \sum_j R^{ij} R \bar\partial_0^{(j)} h_{t_0} \ . \ee
This gives the desired equivariance. Next, we turn to the $Z_2\times Z_2$ equivariance. Let $g$ be any of the extended gauge transformations $g_1$, $g_2$, or $g_1 g_2$ defined in Theorem \ref{thm:modHK}. We write the $g$-transformed $\delta\xi$ as $(g(\delta\xi))^i_{t_0} = c_{t_0} \delta\xi^i_{t_0}$, where $c_0,c_L\in \{-1,1\}$. We have $\Delta_{0,g(A)}= g^{-1}\circ \Delta_0 \circ g$, so if we define $h'_{t_0} = c_{t_0} g^{-1}h_{t_0} g$ then $\Delta_{0,g(A)}h'_{t_0} = 0$ (since $h_{t_0}$ transforms in the complexified adjoint representation). Furthermore, $h'_{t_0}$ satisfies the same boundary conditions as $h_{t_0}$. We thus want to show that $\sum_{t_0\in \{0,L\}} \sum_i c_{t_0} \delta\xi_{t_0}^i \bar\partial_{0,g(A)}^{(i)} h'_{t_0} = \sum_{t_0\in\{0,L\}} \sum_i \delta\xi_{t_0}^i g^{-1}(\bar\partial_0^{(i)} h_{t_0}) g$, but this follows immediately from $\bar\partial^{(i)}_{0,g(A)} = g^{-1}\circ \bar\partial^{(i)}_0 \circ g$.

%\be \delta \xi^i \bar\partial_0^{(i)} h = (R^{ij} \delta\xi^j) (R^{ik} \bar\partial_0^{(k)} h) = \delta\xi^{(\zeta)} R^{1k} \bar\partial_0^{(k)} h = \delta\xi^{(\zeta)} \bar\partial_0^{(\zeta)} h \ . \ee
%Here, the second equality used $R^{ij} \delta \xi^j = \delta \xi^{(\zeta)} \delta^{i1}$ and the third equality used \eqref{eq:genRot} and \eqref{eq:specRot}
\end{proof}

\begin{remark}
For a fixed $\xi^\RR$, let $(\mf{z}^\CC_{\partial I})^\circ$ be the locus in $\mf{z}^\CC_{\partial I}$ consisting of those $\xi^\CC$ such that $(\xi^\RR,\xi^\CC)\in (\mf{z}_{\partial I}^3)^\circ$. Then, let $p^{\CC,\circ}: \N^\circ_{\xi^\RR,{\rm all}} \to (\mf{z}^\CC_{\partial I})^\circ$ denote the pullback of $p^\CC$. Since this is a complex fibered manifold, it is natural to ask how ${\rm Conn}_\M$ interacts with the complex structures on $\N^\circ_{\xi^\RR,{\rm all}}$ and $(\mf{z}^\CC_{\partial I})^\circ$. Using
\begin{align}
\bar\partial_0^{(2)} h &= i([A^2, h], - [A^3, h], -\partial_{A^0} h, [A^1, h]) \ , \nonumber \\
\bar\partial_0^{(3)} h &= i([A^3, h], [A^2, h], - [A^1, h], -\partial_{A^0} h) \ ,
\end{align}
one easily finds that the horizontal tangent vector associated to a variation $\delta\xi^\CC$ is $(\delta \alpha, \delta \beta) = -i \sum_{t_0\in \{0,L\}} \delta\xi^\CC_{t_0} ([\beta^\dagger, h_{t_0}], \partial_{-\alpha^\dagger} h_{t_0})$. So, $(1,0)$ vectors in $T(\mf{z}^\CC_{\partial I})^\circ \otimes \CC$ lift to $(1,0)$ vectors in $T\N^\circ_{\xi^\RR,{\rm all}} \otimes \CC$, but the dependence of the horizontal spaces on $(\alpha,\beta)$ is almost certainly not holomorphic (due both to the explicit appearance of $\alpha^\dagger$ and $\beta^\dagger$, as well as to their implicit presence via the definition of $h_{t_0}$).
\end{remark}

We next compute the curvature of this connection; we will shortly thereafter demonstrate that it does not vanish identically.

\begin{prop}\label{prop:curv}
The curvature of $\mathrm{Conn}_{\mc{M}}$ at $A$ may be described as follows: given infinitesimal variations $\delta \xi^i_{t_0}$ and $\delta\xi'^j_{t_1}$, with $t_0,t_1\in \{0,L\}$ and $i,j\in \{1,2,3\}$, the commutator of the associated infinitesimal variations of $A$ is, to quadratic order in the variations and in Coulomb gauge with respect to $A$, given by
\begin{align} \label{eq:curvature}
\delta\xi_{t_0}^i \delta&\xi'^j_{t_1} \brackets{ i ([a_1^i, h_{t_0}], (- \delta^{i\ell} [a_1^0, h_{t_0}] + \sum_m \epsilon^{i\ell m} [a_1^m, h_{t_0}])_{\ell=1,2,3}) \right. \nonumber \\
&\left. \ \  - i ([a_0^j, h_{t_1}], (- \delta^{j\ell} [a_0^0, h_{t_1}] + \sum_m \epsilon^{j\ell m} [a_0^m, h_{t_1}])_{\ell=1,2,3}) \right. \nonumber \\
&\left.\quad + \bar\partial_{0}^{(i)} h^H_{0} + d_0 h^A_0 - \bar\partial_{0}^{(j)} h^H_{1} - d_0 h^A_1 } \ ,
\end{align}
where there is no implicit sum over repeated indices, $h_0,h_L$ are as in the proof of Theorem \ref{thm:conn}, $a_0 = \bar\partial^{(i)}_0 h_{t_0}$, $a_1 = \bar\partial^{(j)}_0 h_{t_1}$, and $h_0^H,h_1^H \in i \mf{g}$ and $h_0^A,h_1^A \in \mf{g}$ are the unique solutions of
\begin{align} \label{eq:hChange}
\Delta_0 h^H_{0} &= 2 \sum_\mu [a_1^\mu, (d_0 h_{t_0})^\mu] \ , \quad
\Delta_0 h^H_{1} = 2 \sum_\mu [a_0^\mu, (d_0 h_{t_1})^\mu] \ , \nonumber \\
\Delta_0 h^A_0 &= - \sum_\mu [a_1^\mu, a_0^\mu] \ , \quad \Delta_0 h_1^A = - \sum_\mu [a_0^\mu, a_1^\mu] \ .
\end{align}
\end{prop}
\begin{proof}
The commutator is given by
\be
\delta\xi^i_{t_0} (\bar\partial^{(i)}_{0, A + \delta\xi'^j_{t_1} a_1} h_{t_0} + \bar\partial^{(i)}_{0,A} \delta h_{t_0}) + \delta\xi'^j_{t_1} a_1 -
\delta\xi'^j_{t_1} (\bar\partial^{(j)}_{0, A + \delta\xi^i_{t_0} a_0} h_{t_1} + \bar\partial^{(j)}_{0,A} \delta h_{t_1}) - \delta\xi^i_{t_0} a_0 \ , \label{eq:commPaths} \ee
where $\delta h_{t_0},\delta h_{t_1}$ denote the first order (in $\delta \xi'$ and $\delta\xi$, respectively) changes in $h_{t_0},h_{t_1}$ due to moving, respectively, from $A$ to $A+\delta \xi'^{j}_{t_1} a_1$ or $A+\delta \xi^{i}_{t_0} a_0$. So, our main task is to compute these changes. We cannot use the formulae from the proof of Theorem \ref{thm:conn} to do so, since we would not be comparing apples to apples, as for one computation we would be in Coulomb gauge with respect to $A+\delta \xi^i_{t_0} a_0$ and for the other we would be in Coulomb gauge with respect to $A+\delta \xi'^j_{t_1} a_1$. Therefore, we will instead work in Coulomb gauge with respect to $A$. Because the first variations $\delta \xi^i_{t_0} a_0$ and $\delta \xi'^j_{t_1} a_1$ are already in this gauge, this simply means that the second variations (i.e., the first and third terms in \eqref{eq:commPaths}) must also be in the kernel of $d_{0,A}^*$.

We focus on $\delta h_{t_0}$, as the computation of $\delta h_{t_1}$ is identical. We write the first variation as $a' = \delta\xi'^{j}_{t_1} a_1$. As above, we begin by fixing $\xi^\CC$ and varying $\xi^\RR$ -- that is, we assume for now that $i=1$. To first order in $\delta \xi$ and $\delta \xi'$ (but including mixed terms of the form $\delta\xi \delta\xi'$), the second variation is of the form $a = \delta\xi^i_{t_0} \bar\partial_{0, A+a'} h$, where $h\in C^\infty(I,\mf{sl}(2,\CC))$, $h|_{\partial I}\in \mf{t}_\CC$, $\partial h|_{t_0}\in \mf{t}_\CC^\perp + \sigma_z$, $\partial h|_{L-t_0} \in \mf{t}_\CC^\perp$, but now we have $d_{0,A}^* a = d_{1, A+ a'} a = 0$. Due to the change in gauge choice, we cannot expect $h$ to be Hermitian. We recall from the comments below \eqref{eq:specRot} that the Hermitian and skew-Hermitian parts $h^H = \frac{h+h^\dagger}{2}$ and $h^A = \frac{h-h^\dagger}{2}$ are treated differently by the various hyperkahler-rotated $\bar\partial_0$ operators, and so we will work separately with $h^H$ and $h^A$.

As above, the equation $(d_{1,A+a'} a)^1 = 0$ yields $\Delta_{0, A+a'} h^H = 0$, since $\mu_\RR(A+a')$ is $O(\delta \xi'^2)$ and so its contribution may be neglected. So, the only change involves $h^A$; we now have
\be d_{0,A}^* a = d^*_{0,A+ a'} a + \sum_\mu [a'^\mu, a^\mu] = \delta\xi^i_{t_0}\brackets{\Delta_{0,A+a'} h^A - \frac{1}{2} [\mu_\RR(A+a'), h^H] + \sum_\mu [a'^\mu, (\bar\partial_{0,A+a'} h)^\mu]} \ . \ee
As above, we can neglect the term involving $\mu_\RR$, we can neglect the $a'$ contributions in the subscripts of the first and third terms, and in the final expression we can replace $h$ by $h_{t_0}$, so we find that $\Delta_{0, A+a'} h^A = - \sum_\mu [a'^\mu, (\bar\partial_{0,A} h_{t_0})^\mu]$. Writing $h^H = h_{t_0} + \delta h^H_{t_0}$ and $h^A = \delta h^A_{t_0}$, where $\delta h^H_{t_0}\in i\mf{g}$ and $\delta h^A_{t_0}\in \mf{g}$, and continuing to neglect higher order terms, we finally have
\be \Delta_{0,A} \delta h^H_{t_0} = 2 \sum_\mu [a'^\mu, (d_{0,A} h_{t_0})^\mu] \ , \quad \Delta_{0,A} \delta h^A_{t_0} = - \sum_\mu [a'^\mu, (\bar\partial_{0,A} h_{t_0})^\mu] \ . \label{eq:varEqns} \ee
Here, we have used the facts that $\Delta_{0,A} h_{t_0}=0$ and $d_0^* a' = 0$. The equations in \eqref{eq:varEqns} uniquely determine $\delta h^H_{t_0}$ and $\delta h^A_{t_0}$.

Now, we hyperkahler rotate -- i.e., we no longer assume that $i=1$. Using $\Delta_{0,R(A)}=\Delta_{0,A}$, $d_{0,R(A)} h_{t_0} = R(d_{0,A} h_{t_0})$, and $\bar\partial_{0,R(A)} h_{t_0} = R(\bar\partial_{0,A}^{(i)} h_{t_0})$, where $R\in SO(3)$ rotates the $i$ axis to the 1 axis, as in the proof of Theorem \ref{thm:conn}, we find that \eqref{eq:varEqns} is modified to
\be \Delta_{0,A} \delta h_{t_0}^H = 2 \sum_\mu [a'^\mu, (d_{0,A} h_{t_0})^\mu] \ , \quad \Delta_{0,A} \delta h^A_{t_0} = - \sum_\mu [a'^\mu, (\bar\partial^{(i)}_{0,A} h_{t_0})^\mu] \ , \ee
which coincides with the leftmost equations in \eqref{eq:hChange} after we factor out $\delta\xi'^j_{t_1}$ from both sides. Pursuant to the comments below \eqref{eq:specRot}, we have $\bar\partial^{(i)}_{0,A} \delta h_{t_0} = d_{0,A} \delta h^A_{t_0} + \bar\partial_{0,A}^{(i)} \delta h_{t_0}^H$. Plugging this into \eqref{eq:commPaths} leads immediately to \eqref{eq:curvature}.
\end{proof}

We conclude with a brief computation that demonstrates that this connection is not flat. Computations are easiest to perform when $\xi=0$, but this is of course the epitome of non-generic moment map parameters and so the connection is not defined at $\M_0 \subset \M_{\rm all}$. Nevertheless, the definition of the connection still gives a parallel transport map \emph{from} $\M_0^{\rm irr}$ \emph{to} nearby $\M_{\xi}$. We can therefore evaluate the above formula for the curvature at a point of $\M_0^{\rm irr}$ and conclude that, if this is nonvanishing, then by continuity the curvature of $\mathrm{Conn}_{\mc{M}}$ is similarly nonvanishing.

For simplicity, we focus on irreducible Nahm data of the form $A=(0,i c \sigma_x, 0,0)$ with $c>0$ and $i=2$, $j=3$, $t_0=t_1=L$. The relevant computations are as follows:
\begin{align}
h_L &= \frac{\cosh(2ct)}{2c \sinh(2cL)} \sigma_z \nonumber \\
a_0 &= \bar\partial^{(2)}_0 h_L = i(0, 0, -\partial h_L, ic [\sigma_x, h_L]) = \frac{i}{\sinh(2cL)} (0, 0, -\sinh(2ct) \sigma_z, \cosh(2ct) \sigma_y) \nonumber \\
a_1 &= \bar\partial_0^{(3)} h_L = i(0, 0, -ic[\sigma_x, h_L], -\partial h_L) = \frac{i}{\sinh(2cL)} (0, 0, -\cosh(2ct) \sigma_y, - \sinh(2ct) \sigma_z) \nonumber \\
d_0 h_L &= (\partial h_L, ic [\sigma_x, h_L], 0, 0) = \frac{1}{\sinh(2cL)} (\sinh(2ct) \sigma_z, \cosh(2ct) \sigma_y, 0, 0) \nonumber \\
2\sum_\mu &[a_1^\mu, (d_0 h_L)^\mu] = 2\sum_\mu [a_0^\mu, (d_0 h_L)^\mu] = h^H_0 = h^H_1 = 0 \nonumber \\
\sum_\mu &[a_0^\mu, a_1^\mu] = \sigma_x \frac{2i \sinh(4ct)}{\sinh^2(2cL)} \nonumber \\
h^A_0 &= - h^A_1 = \frac{i \sigma_x}{8 c^2 L \sinh^2(2cL)} \parens{- L \sinh(4 c t) + t \sinh(4 c L) } \nonumber \\
d_0 h_0^A &= \frac{i\sigma_x}{8 c^2 L \sinh^2(2cL)} \parens{ - 8 c L \cosh^2(2 c t) + 4 c L + \sinh(4 c L) }  (1,0,0,0) \nonumber \\
&i([a_1^2, h_L], -[a_1^3, h_L], -[a_1^0, h_L], [a_1^1, h_L]) = \sigma_x \frac{i \cosh^2(2ct)}{c \sinh^2(2cL)} (1, 0, 0, 0) \nonumber \\
&i([a_0^3, h_L], [a_0^2, h_L], -[a_0^1, h_L], -[a_0^0, h_L]) = -\sigma_x \frac{i \cosh^2(2ct)}{c \sinh^2(2cL)} (1, 0, 0, 0) \ .
\end{align}
The commutator \eqref{eq:curvature} is thus
\be \delta\xi^2_L \delta\xi'^3_L \cdot \frac{i\sigma_x (4cL + \sinh(4cL))}{4 c^2 L \sinh^2(2cL)} (1,0,0,0) \ , \ee
which is nonzero. As a check, we note that this expression is both in $\A_0$ and in $\ker d_0^*$.

Hence, $\mathrm{Conn}_{\mc{M}}$ is not flat. It is of interest to know what coarse-grained information may be extracted from this connection. A natural first attempt is to define a holonomy homomorphism from $\pi_1((\mf{z}_{\partial I}^3)^\circ)$ to the mapping class group $\pi_0(\mathrm{Diff}\mc{M}_{\xi})$, where $\mathrm{Diff}\mc{M}_{\xi}$ is the diffeomorphism group of $\M_\xi$. However, this map is trivial because $(\mf{z}_{\partial I}^3)^\circ$ is simply connected. We may rectify this in two different ways. First, we may quotient $\M_{\rm all}^\circ$ by the $Z_2\times Z_2$ action on it mentioned in Theorem \ref{thm:modHK}. As $\mathrm{Conn}_{\mc{M}}$ is $Z_2\times Z_2$-equivariant, we may equivalently think of it as a connection on the orbifold quotient $(\mf{z}^3_{\partial I})^{\circ} / (Z_2 \times Z_2)$ in order to obtain a homomorphism $\pi_1((\mf{z}^3_{\partial I})^\circ/(Z_2\times Z_2)) \simeq Z_2\times Z_2 \to \pi_0(\mathrm{Diff}\mc{M}_{\xi})$. Here, the source is the orbifold fundamental group $Z_2 \times Z_2$.

Alternatively, we may return to the holonomy map, which is naturally a map from the based loop space $\Omega (\mf{z}_{\partial I}^3)^{\circ}$ to $\mathrm{Diff}\mc{M}_{\xi}$; in fact, as $(\mf{z}^3_{\partial I})^{\circ}$ is simply-connected and hence $\Omega(\mf{z}^3_{\partial I})^{\circ}$ is connected, we may emphasize that this holonomy map has target $\mathrm{Diff}^0\mc{M}_{\xi}$, the connected component of $\mathrm{Diff}\mc{M}_{\xi}$ containing diffeomorphisms which are smoothly isotopic to the identity. We then naturally have homomorphisms $\pi_n(\Omega((\mf{z}_{\partial I}^3)^\circ)) \simeq \pi_{n+1}((\mf{z}_{\partial I}^3)^\circ) \to \pi_n(\mathrm{Diff}^0\mc{M}_{\xi})$ for all $n\ge 1$. In particular, for $n=1$ we obtain a canonical map $\mb{Z}^2 \to \pi_1(\mathrm{Diff}^0\mc{M}_{\xi})$ which may be of interest to describe; here, $\mb{Z}^2 \simeq \mathrm{H}_2((\mf{z}^3_{\partial I})^{\circ};\mb{Z}) \simeq \pi_2((\mf{z}_{\partial I}^3)^\circ)$ is a natural combinatorial invariant of the complement of the hyperplane arrangement determined by the non-generic loci in $\mf{z}_{\partial I}^3$, where here (and in other contexts in hyperkahler geometry where we obtain this canonical Ehresmann connection), we emphasize that we are removing real codimension three hyperplanes.

Finally, these two approaches (quotienting -- by both $Z_2\times Z_2$ and $SO(3)$ -- and considering higher homotopy groups) may be combined to define additional maps $\pi_{n+1}((Z_2\times Z_2)\backslash (\mf{z}_{\partial I}^3)^\circ / SO(3)) \to \pi_n({\rm Diff}(\M_\xi))$ for $n\ge 0$. It may be of interest to describe these maps.

\newpage
\bibliography{Refs}
\end{document}